\documentclass[a4paper]{article}
\usepackage{diagbox}

\usepackage{makecell}

\makeatletter
\let\c@figure\c@table
\let\ftype@figure\ftype@table

\usepackage{enumitem}
\usepackage{calc}

\usepackage{lmodern} 

\usepackage{graphicx} 

\usepackage{amsmath, accents}
\usepackage{amssymb,tikz}
\usepackage{pict2e}
\usepackage{amsthm}
\usepackage{amscd}
\usepackage{mathrsfs}
\usepackage{todonotes}
\usetikzlibrary{calc} 

\usepackage{yfonts}

\usepackage{marvosym}
\usepackage[percent]{overpic}

\newtheorem{theorem}{Theorem} [section]
\newtheorem{proposition}[theorem]{Proposition}	
\newtheorem{corollary}[theorem]{Corollary}
\newtheorem{lemma}[theorem]{Lemma}

\newtheorem{assumption}[theorem]{Assumption}
\newtheorem{conjecture}[theorem]{Conjecture}

\theoremstyle{definition}
\newtheorem{definition}[theorem]{Definition}
\newtheorem{remark}[theorem]{Remark}

\DeclareMathOperator\arctanh{arctanh}

\DeclareMathOperator{\dist}{dist}

\newcommand{\C}{\mathbb{C}}

\newcommand{\R}{\mathbb{R}}
\newcommand{\N}{\mathbb{N}}
\newcommand{\PP}{\mathbb{P}}
\newcommand{\re}{\text{\upshape Re\,}}
\newcommand{\im}{\text{\upshape Im\,}}

\let\oldbibliography\thebibliography
\renewcommand{\thebibliography}[1]{\oldbibliography{#1}
\setlength{\itemsep}{-0.5pt}}

\def\Xint#1{\mathchoice
{\XXint\displaystyle\textstyle{#1}}%
{\XXint\textstyle\scriptstyle{#1}}%
{\XXint\scriptstyle\scriptscriptstyle{#1}}%
{\XXint\scriptscriptstyle\scriptscriptstyle{#1}}%
\!\int}
\def\XXint#1#2#3{{\setbox0=\hbox{$#1{#2#3}{\int}$}
\vcenter{\hbox{$#2#3$}}\kern-.5\wd0}}

\def\dashint{\;\Xint-}

\usepackage{tikz}
\usetikzlibrary{arrows}
\usetikzlibrary{decorations.pathmorphing}
\usetikzlibrary{decorations.markings}
\usetikzlibrary{patterns}
\usetikzlibrary{automata}
\usetikzlibrary{positioning}
\usepackage{tikz-cd}
\tikzset{->-/.style={decoration={
				markings,
				mark=at position #1 with {\arrow{latex}}},postaction={decorate}}}
	
	\tikzset{-<-/.style={decoration={
				markings,
				mark=at position #1 with {\arrowreversed{latex}}},postaction={decorate}}}
				
\tikzset{
	master/.style={
		execute at end picture={
			\coordinate (lower right) at (current bounding box.south east);
			\coordinate (upper left) at (current bounding box.north west);
		}
	},
	slave/.style={
		execute at end picture={
			\pgfresetboundingbox
			\path (upper left) rectangle (lower right);
		}
	}
}

\usetikzlibrary{shapes.misc}\tikzset{cross/.style={cross out, draw, 
         minimum size=2*(#1-\pgflinewidth), 
         inner sep=0pt, outer sep=0pt}}

\usepackage{pgfplots}
	
\allowdisplaybreaks	

\numberwithin{equation}{section}

\def\ds{\displaystyle}
\def\bigO{{\cal O}}

\makeatletter
\newcommand{\oset}[3][0ex]{%
  \mathrel{\mathop{#3}\limits^{
    \vbox to#1{\kern-2\ex@
    \hbox{$\scriptstyle#2$}\vss}}}}
\makeatother

\usepackage{tocloft}
\usepackage[colorlinks=true]{hyperref}
\hypersetup{urlcolor=blue, citecolor=red, linkcolor=blue}

\begin{document}
\title{Hole probabilities and balayage of measures \\ for planar Coulomb gases}
\author{Christophe Charlier}

\maketitle

\begin{abstract}

We study hole probabilities of two-dimensional Coulomb gases with a general potential and arbitrary temperature. The hole region $U$ is assumed to satisfy $\partial U\subset S$, where $S$ is the support of the equilibrium measure $\mu$. Let $n$ be the number of points. As $n \to \infty$, we prove that the probability that no points lie in $U$ behaves like $\exp(-Cn^{2}+o(n^{2}))$. We determine $C$ in terms of $\mu$ and the balayage measure $\nu = \mathrm{Bal}(\mu|_{U},\partial U)$. If $U$ is unbounded, then $C$ also involves the Green function of $\Omega$ with pole at $\infty$, where $\Omega$ is the unbounded component of $U$. We also provide several examples where $\nu$ and $C$ admit explicit expressions: we consider several point processes, such as the elliptic Ginibre, Mittag-Leffler, and spherical point processes, and various hole regions, such as circular sectors, ellipses, rectangles, and the complement of an ellipse. This work generalizes previous results of Adhikari and Reddy in several directions.
\end{abstract}



{
\hypersetup{linkcolor=black}

\small \tableofcontents \addtocontents{toc}{\vspace{-0.2cm}} \normalsize

}

\section{Introduction}


\medskip \noindent The Coulomb gas model for $n$ points on the plane with external potential $Q:\C\to \R\cup\{+\infty\}$ is the probability measure
\begin{align}
& \frac{1}{Z_{n}}\prod_{1\leq j<k \leq n}|z_{j}-z_{k}|^{\beta} \prod_{j=1}^{n} e^{-n \frac{\beta}{2} Q(z_{j})}d^{2}z_{j}, & & z_{1},\ldots,z_{n}\in \C, \label{general density intro}
\end{align}
where $Z_{n}$ is the normalization constant, $d^{2}z$ is the two-dimensional Lebesgue measure, and $\beta >0$ is the inverse temperature. For $\beta=2$, \eqref{general density intro} is also the law of the complex eigenvalues of some random normal matrices (see e.g. \cite{CZ1998}). Under suitable assumptions on $Q$, as $n \to \infty$ the points $z_{1},\ldots,z_{n}$ will accumulate with high probability on the support $S$ of an equilibrium measure $\mu$ \cite{HM2013}. For example, the well-studied Ginibre ensemble corresponds to $Q(z)=|z|^{2}$, for which $\mu$ is the uniform measure supported on the unit disk, see Figure \ref{fig:hole regions ginibre intro} (left).

\medskip We are interested in the asymptotics as $n\to +\infty$ of the hole probability $\mathcal{P}_{n}:=\mathbb{P}(\#\{z_{j}\in U\}=0)$, where $U \subset \C$ is an open set satisfying $\partial U \subset S$. This is a classical problem in the theory of point processes with a long history, and it has been widely studied for particular potentials $Q$ and hole regions $U$. For example, determining the precise asymptotics of $\mathcal{P}_{n}$ in the case where $Q(z)=|z|^{2}$ and $U=\{z:|z|\leq r\}$ for some $r<1$ has received great attention over the years \cite{GHS1988, ForresterHoleProba, JLM1993, APS2009, L et al 2019, C2021}. In this paper, our focus is not on obtaining precise asymptotics, but rather on considering general potentials $Q$ and hole regions $U$. Of particular importance for us are the works \cite{AR2017, A2018} of Adhikari and Reddy, which we briefly discuss now:
\begin{itemize}[leftmargin=5mm]
\item For $Q(z) = |z|^{2}$, $\beta=2$ and sets $U\subset S$ satisfying the exterior ball condition (see \eqref{eqn:exterior ball condition lol}), it is proved in \cite{AR2017} that $\mathcal{P}_{n} = \exp (-Cn^{2}+o(n^{2}))$ as $n\to \infty$, where $C = \frac{1}{2}\big(  \int_{\partial U}Q(z)d\nu(z) - \int_{U}Q(z)d\mu(z) \big)$ and $\nu$ is the balayage measure of $\mu|_{U}$ on $\partial U$. In particular, $\nu$ has the same mass as $\mu|_{U}$ but $\nu$ is supported on $\partial U$ (the balayage operation is further explained in Subsection \ref{subsection:background potential} below). It is also proved in \cite{AR2017} that the equilibrium measure of the Ginibre process, when \textit{conditioned} on the hole event $\#\{z_{j}\in U\}=0$, is given by $\mu|_{S\setminus U} + \nu$. Thus, for $U\subset S$ and to a first order approximation, conditioning on the hole event (i) does not produce a macroscopic effect outside $\smash{\overline{U}}$ and (ii) forces roughly $\mu(U) \, n$ points to accumulate in a small interface around $\partial U$ (the width of this interface is expected to be  $\bigO(n^{-1})$ \cite{Seo, ACCL1, ACC2023}). This conditional point process is illustrated in Figure \ref{fig:hole regions ginibre intro} (right) when $U$ is a disk (in this case, due to the symmetry, $\nu$ is the uniform measure on $\partial U$).

\smallskip These results were then generalized in \cite{A2018} for arbitrary $\beta>0$ and for a class of rotation-invariant potentials, i.e.~when $Q(z)=g(|z|)$ for some $g:[0,\infty)\to \R\cup \{+\infty\}$. (It is also assumed in \cite{A2018} that $S$ is a disk, among other things.)
\item It is in general a challenging problem to compute $\nu$ (and therefore $C$) explicitly. However, when $Q(z)=|z|^{2}$ and when $U$ is either a disk, an annulus, a half-disk, a cardioid, an ellipse, or an equilateral triangle, $C$ was explicitly determined in \cite{AR2017} (for the equilateral triangle and the half-disk, $\nu$ was not explicitly determined, but the authors were still able to compute $C$ using clever arguments). In \cite{A2018}, the more general case $Q(z)=g(|z|)$ is considered, but explicit results for $\nu$ and $C$ were obtained only when $U$ is centered at $0$ and is either a  disk or an annulus.
\end{itemize}

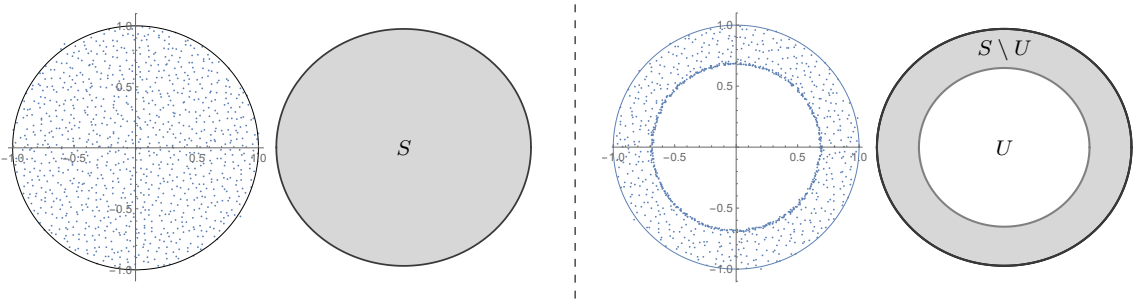
\begin{figure}
\begin{center}
\begin{tikzpicture}[master]
\node at (0,0) {\includegraphics[width=3.5cm]{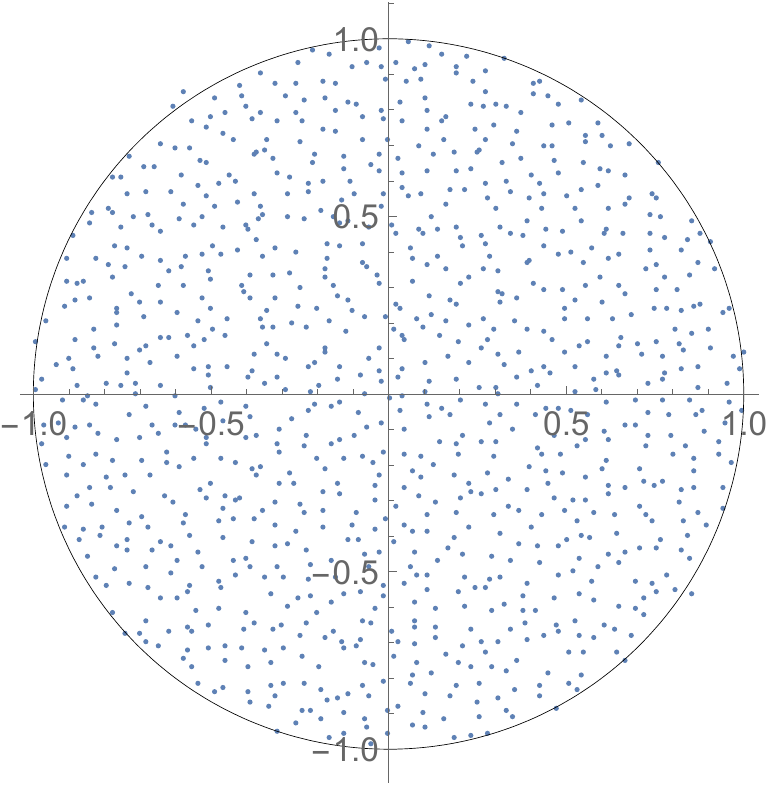}};
\end{tikzpicture}\hspace{-0.3cm}
\begin{tikzpicture}[slave]
\node at (0,0) {\includegraphics[width=3.5cm]{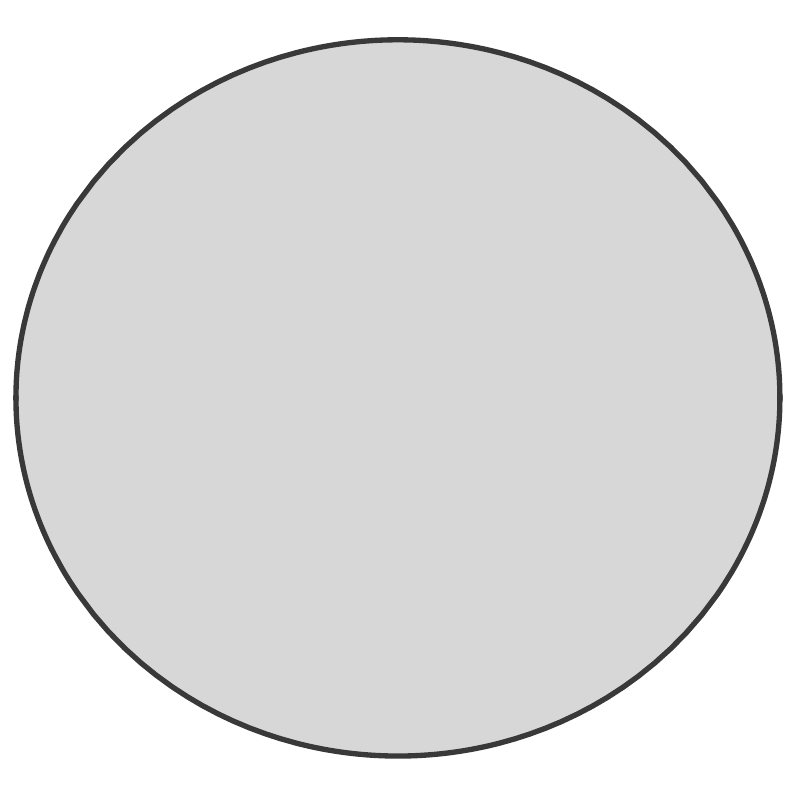}};
\node at (0,0) {\footnotesize $S$};
\end{tikzpicture}\hspace{0.5cm}
\begin{tikzpicture}[slave]
\node at (0,0) {\includegraphics[width=3.5cm]{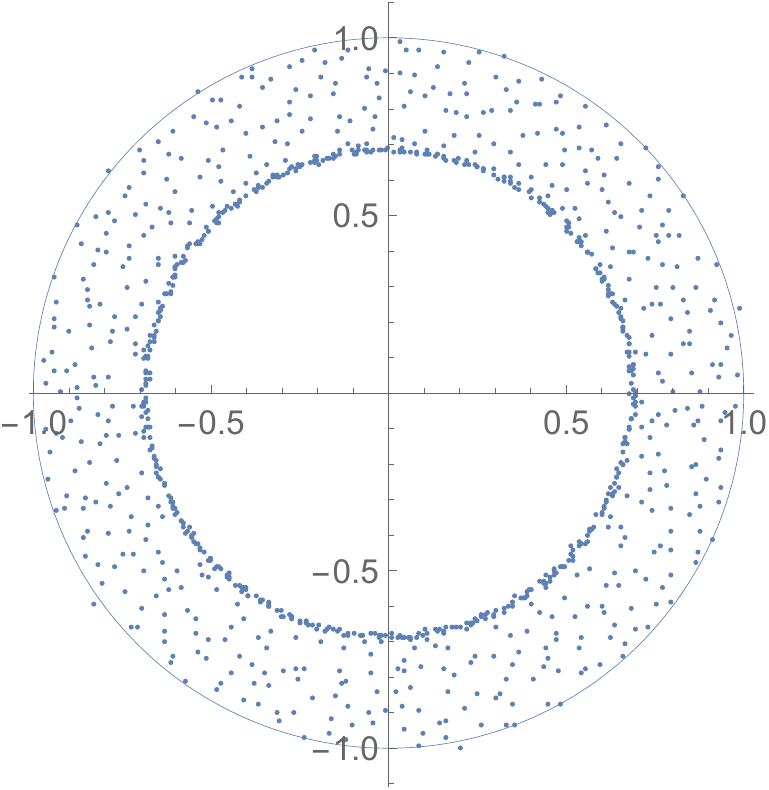}};
\draw[dashed] (-2.1,-2)--(-2.1,2);
\end{tikzpicture}\hspace{-0.3cm}
\begin{tikzpicture}[slave]
\node at (0,0) {\includegraphics[width=3.5cm]{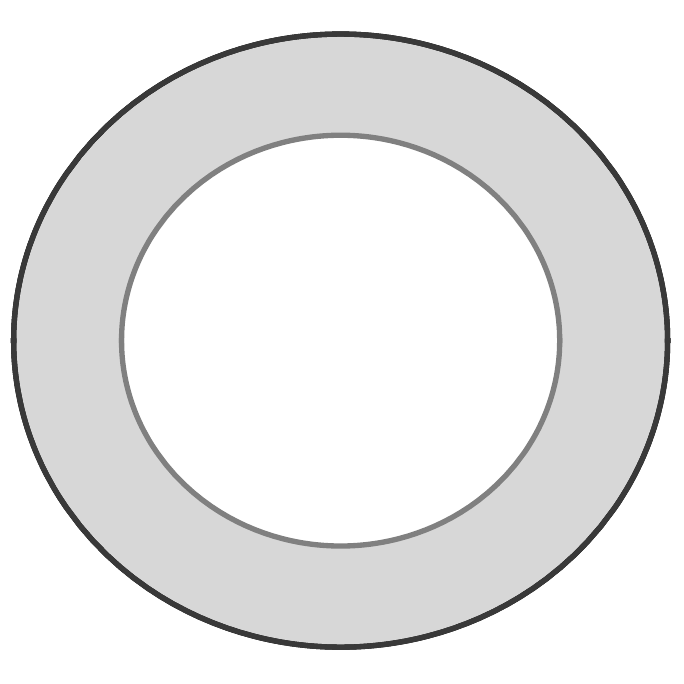}};
\node at (0,1.3) {\footnotesize $S\setminus U$};
\node at (0,0) {\footnotesize $U$};
\end{tikzpicture}
\end{center}
\vspace{-0.4cm}\caption{\label{fig:hole regions ginibre intro} Left: \eqref{general density intro} with $Q(z) = |z|^{2}$, $\beta=2$ and $n=1000$. $S$ is the unit circle. Right: the same point process, but \textit{conditioned} on the event that there is a hole on $U = \{z:|z|\leq 0.67\}$.}
\end{figure}

\noindent The results of this paper generalize the aforementioned results from \cite{AR2017, A2018} in several directions:
\begin{itemize}[leftmargin=5.5mm]
\item Theorem \ref{thm:general pot} establishes the asymptotic formula $\mathcal{P}_{n} = \exp (-Cn^{2}+o(n^{2}))$ as $n\to +\infty$ under weaker assumptions on $Q$ and $U$ than in \cite{AR2017, A2018}. In particular:
\begin{itemize}
\item \vspace{-0.1cm} $Q$ is not assumed to be rotation-invariant.
\item If $S$ is compact, then $U$ is allowed to have an unbounded component $\Omega$.
\end{itemize}
In the above formula, $C$ is given by
\begin{align}\label{C in intro}
C = \frac{\beta}{4}\bigg(  \int_{\partial U}Q(z)d\nu(z) + 2\, c^{\hspace{0.02cm}\mu}_{U} - \int_{U}Q(z)d\mu(z) \bigg).
\end{align}
Here $c^{\hspace{0.02cm}\mu}_{U}=0$ if $U$ is bounded, but if $U$ is unbounded then $c^{\hspace{0.02cm}\mu}_{U}$ is in general non-zero and involves the Green function of $\Omega$ with pole at $\infty$, where $\Omega$ is the unbounded component of $U$. 

Proposition \ref{prop:general pot} also establishes that $\mu|_{S\setminus U} + \nu$ is the equilibrium measure associated with \eqref{general density intro}, when \eqref{general density intro} is conditioned on the hole event $\#\{z_{j}\in U\}=0$.

Two examples of sets $S$ and $U$ covered by Theorem \ref{thm:general pot} are illustrated in Figure \ref{fig:hole regions examples}. 
\item We obtain explicit expressions for $\nu$ and $C$ for various choices of $Q$ and $U$. 

\smallskip Four examples of $Q$ that are considered are the following:
\begin{itemize}
\item $Q(z) = \frac{1}{1-\tau^{2}}\big( |z|^{2}-\tau \, \re z^{2} \big)$ for some $\tau \in [0,1)$. Then \eqref{general density intro} is known as the \textit{elliptic Ginibre point process} and $\mu$ is the uniform measure supported on an ellipse:
\begin{align}\label{mu S EG beginning of intro}
d\mu(z) = \frac{d^{2}z}{\pi(1-\tau^{2})}, \qquad S=\Big\{ z\in \C: \Big( \frac{\re z}{1+\tau} \Big)^{2} + \Big( \frac{\im z}{1-\tau} \Big)^{2} \leq 1 \Big\}.
\end{align}
The elliptic Ginibre point process is illustrated in Figure \ref{fig:Ginibre, ML and Spherical} (left).
\item $Q(z) = |z|^{2b}$ for some $b>0$. Then \eqref{general density intro} is known as the \textit{Mittag-Leffler point process}, and $\mu$ is supported on a disk:
\begin{align*}
d\mu(z) = \frac{b^{2}}{\pi}|z|^{2b-2}d^{2}z, \qquad S=\{z \in \C: |z| \leq b^{-\frac{1}{2b}}\}.
\end{align*}
The Mittag-Leffler point process is illustrated for $b=2$ in Figure \ref{fig:Ginibre, ML and Spherical} (middle), and for various values of $b$ in the first row of Figure \ref{fig:circular sector}.
\begin{figure}
\begin{center}
\begin{tikzpicture}[master]
\node at (0,0) {\includegraphics[width=5.63cm]{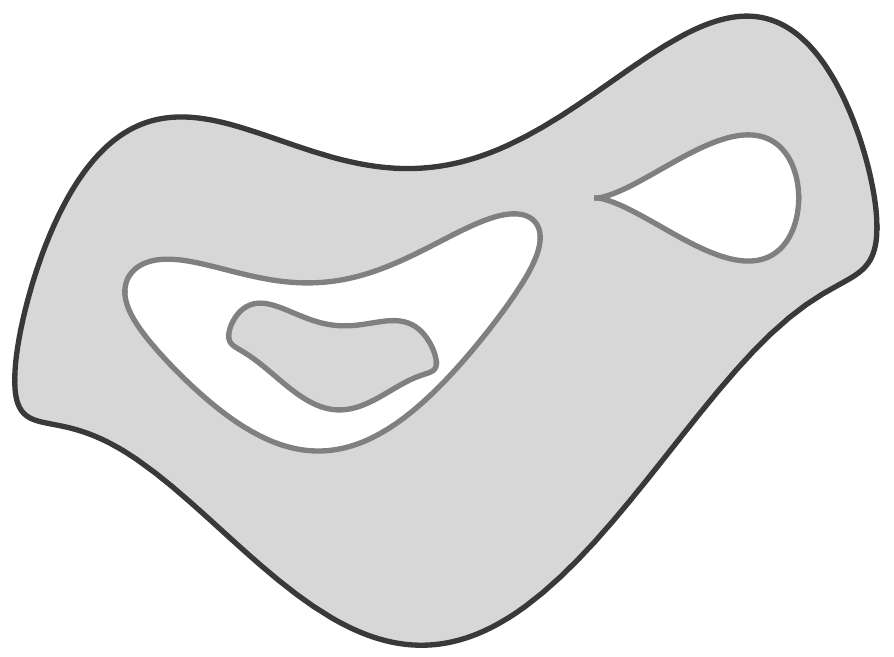}};
\end{tikzpicture} \hspace{0.51cm}
\begin{tikzpicture}[slave]
\node at (0,0) {\includegraphics[width=5.63cm]{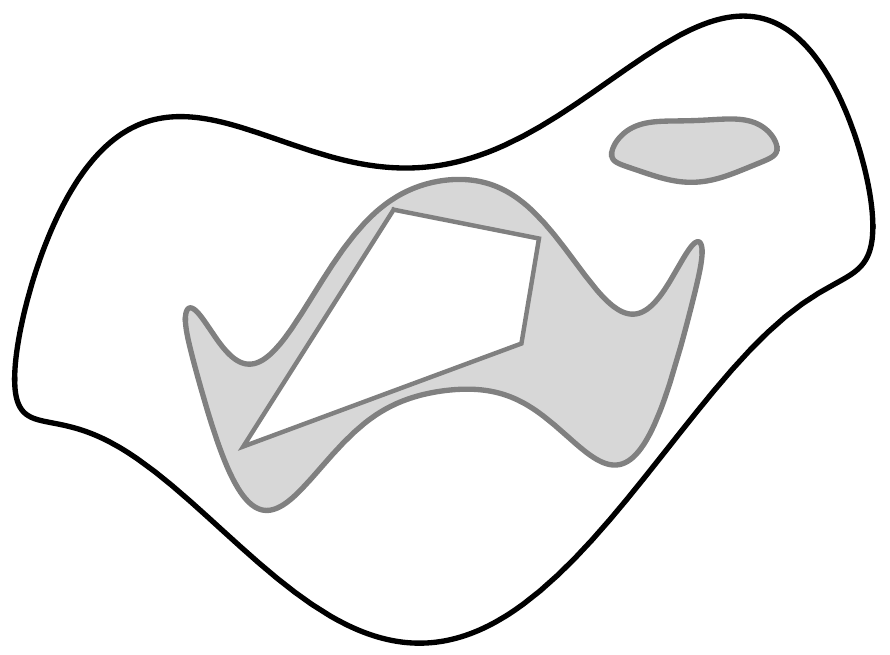}};
\end{tikzpicture}
\end{center}
\vspace{-0.4cm}\caption{\label{fig:hole regions examples}  Two examples of regions $U$ satisfying Assumption \ref{ass:U}. $U$ is bounded on the left, and unbounded on the right. $S\setminus U$ is the shaded region, $U$ is in white, $\partial S$ is in black, and $\partial U$ is in gray.} 
\end{figure}
\item $Q(z) = (1\hspace{-0.05cm}-\hspace{-0.05cm}\frac{1}{n}\hspace{-0.05cm}+\hspace{-0.05cm}\frac{4}{n\beta}) \log(1+|z|^{2})$. In this case \eqref{general density intro} becomes
\begin{align}\label{spherical ensemble}
& \frac{1}{Z_{n}}\prod_{1\leq j<k \leq n}|z_{j}-z_{k}|^{\beta} \prod_{j=1}^{n} e^{-(n-1) \frac{\beta}{2} \log(1+|z_{j}|^{2})}\frac{d^{2}z_{j}}{(1+|z_{j}|^{2})^{2}}, & & z_{1},\ldots,z_{n}\in \C.
\end{align}
The above potential is the only one considered in this paper for which $S$ is not compact:
\begin{align}\label{def of mu and S spherical}
d\mu(z) = \frac{d^{2}z}{\pi(1+|z|^{2})^{2}}, \qquad S=\C.
\end{align}
The point process \eqref{spherical ensemble} was first introduced in \cite{Kris2006}. It is known as the \textit{spherical point process}, because it can also be seen as a Coulomb gas on the unit sphere $\mathbb{S}^{2} \subset \R^{3}$ with a zero potential. To see this, let $\varphi:\mathbb{S}^{2}\to (\C\cup\{\infty\}), \varphi(u,v,w)=\frac{u+iv}{1-w}$ be the stereographic projection with respect to the north pole $(0,0,1)$. Then \eqref{spherical ensemble} is the pushforward measure by $\varphi$ of the following measure:
\begin{align}
\frac{1}{\hat{Z}_{n}}\prod_{1\leq j<k \leq n}|\hat{z}_{j}-\hat{z}_{k}|^{\beta} \prod_{j=1}^{n} d^{2}\hat{z}_{j}, & & \hat{z}_{1},\ldots,\hat{z}_{n}\in \mathbb{S}^{2}, \label{spherical ensemble on the sphere}
\end{align}
where $\hat{Z}_{n} = 2^{n[(n-1) \frac{\beta}{2}+2]}Z_{n}$ is the normalization constant, and $d^{2}\hat{z}$ is the area measure on $\mathbb{S}^{2}$. The point process \eqref{spherical ensemble on the sphere} is illustrated in Figure \ref{fig:Ginibre, ML and Spherical} (right).
\item $Q(z) = g(|z|)$ for some $g: [0,+\infty)\to \R \cup \{+\infty\}$. For such $Q$, \eqref{general density intro} is called a \textit{rotation-invariant point process}, because \eqref{general density intro} remains unchanged if each $z_{j}$'s is multiplied by the same constant $e^{i\gamma}$, $\gamma \in \R$. We also assume that $Q$ satisfies Assumption \ref{ass:Q} below. In particular, $g(r)-2\log r \to +\infty$ as $r\to +\infty$ and  the set $S$ is compact. We consider the case where $S$ consists of a finite number of annuli centered at $0$:
\begin{align}\label{support S when Q=g}
S = \{z : |z| \in \mathrm{S} \}, \qquad \mathrm{S} := [r_{0},r_{1}]\cup [r_{2},r_{3}] \ldots \cup [r_{2\ell},r_{2\ell+1}]
\end{align}
for some $0 \leq r_{0} < r_{1} < \ldots < r_{2\ell}<r_{2\ell+1}<+\infty$. We further assume that $g \in C^{2}(\mathrm{S}\setminus \{0\})$, so that by \cite[Theorem II.1.3]{SaTo} we have
\begin{align}\label{mu radially symmetric intro}
d\mu(z) = \frac{\Delta Q(z)}{4\pi} d^{2}z = d\mu_{\mathrm{rad}}(r) \frac{d\theta}{2\pi}, \qquad d\mu_{\mathrm{rad}}(r):=\frac{r}{2} \big( g''(r) + \frac{1}{r}g'(r) \big) dr,
\end{align}
where $z=re^{i\theta}$, $r\geq 0$, $\theta \in (-\pi,\pi]$, and $\Delta = \partial_{x}^{2}+\partial_{y}^{2}$ is the standard Laplacian. 

The Mittag-Leffler potential $Q(z)= |z|^{2b}$ corresponds to the special case $g(r) = r^{2b}$, but the potential $\log(1+|z|^{2})$ does not fall in this class because $\log(1+|z|^{2})-2\log|z| = \bigO(|z|^{-2})$ as $|z|\to + \infty$. 
\end{itemize}
\end{itemize}
\begin{figure}
\begin{tikzpicture}[master]
\node at (0,0) {\includegraphics[width=4.8cm]{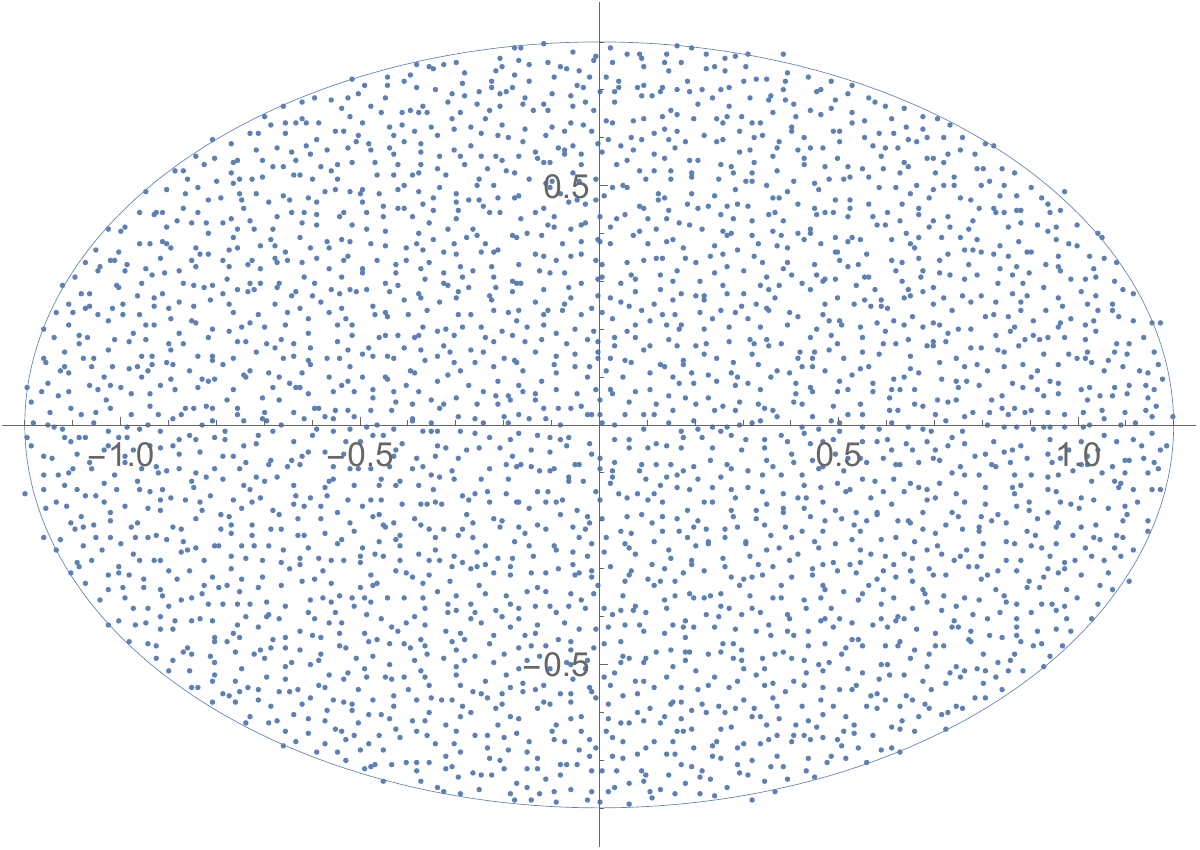}};
\end{tikzpicture}
\begin{tikzpicture}[slave]
\node at (0,0) {\includegraphics[width=4cm]{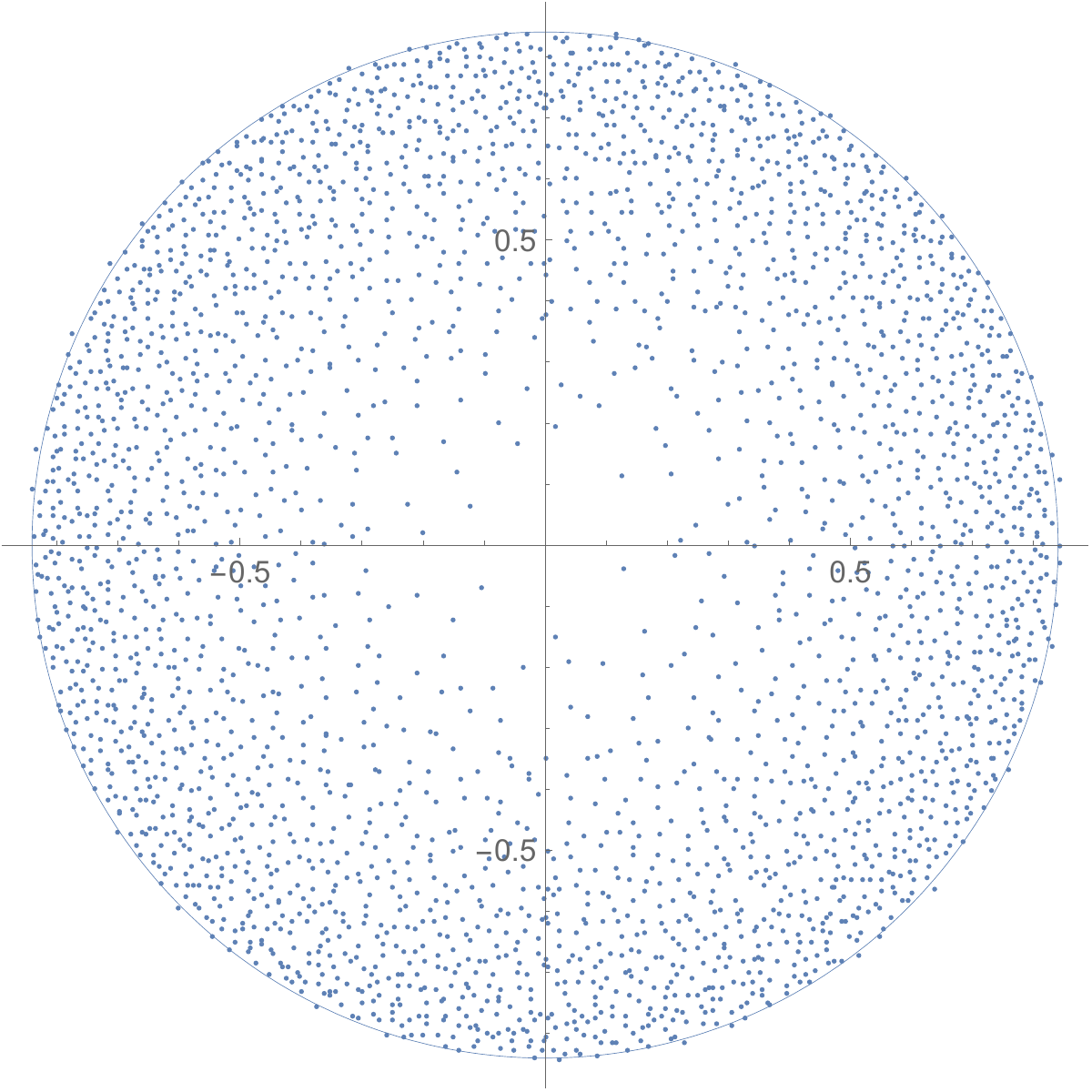}};
\end{tikzpicture}
\begin{tikzpicture}[slave]
\node at (0,0) {\includegraphics[width=4.5cm]{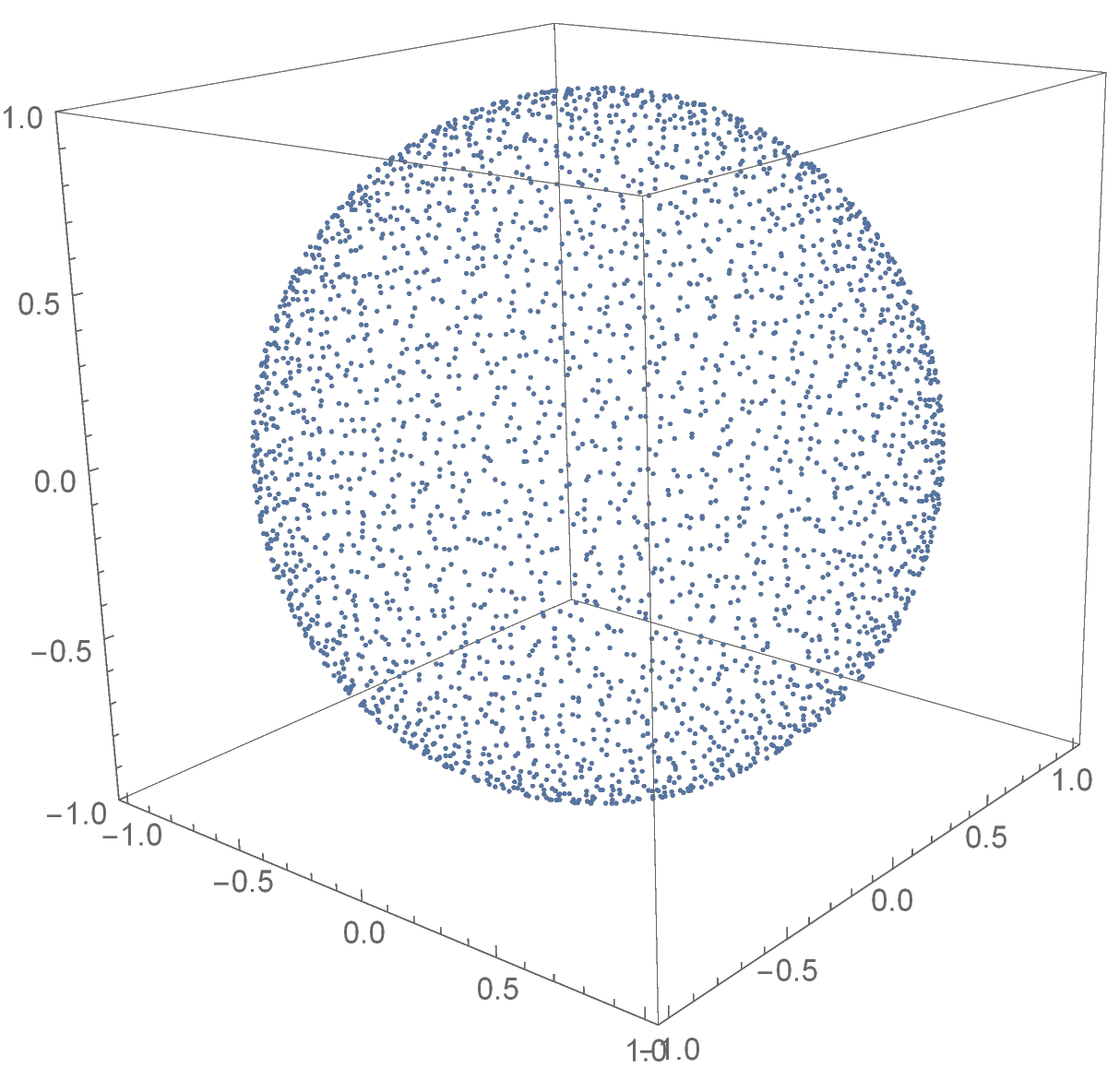}};
\end{tikzpicture}
\caption{\label{fig:Ginibre, ML and Spherical} Left: the elliptic Ginibre point process with $\tau=0.2$ and $n=3000$. Middle: the Mittag-Leffler  point process with $b=2$ and $n=3000$. Right: the spherical point process with $n=3000$.}
\end{figure}
We determine $\nu$ and $C$ explicitly for the four choices of $Q$ listed above, and for several hole regions $U$ as indicated in Tables \ref{table1} and \ref{table2}. These tables also summarize some results from \cite{AR2017, A2018}. 
\begin{table}[h]
\begin{center}
\begin{tikzpicture}
\node at (0,0) {};
\node at (0.8,0) {$\begin{array}{|l@{}|c|c|c|}\hline
\mbox{\backslashbox[3.35cm]{\hspace{0.1cm}$U$}{\vspace{-0.15cm}$Q$}} &  |z|^{2} & (1\hspace{-0.05cm}-\hspace{-0.05cm}\frac{1}{n}\hspace{-0.05cm}+\hspace{-0.05cm}\frac{4}{n\beta})\log(1+|z|^{2})  & g(|z|) \rule{0pt}{2.6ex} \\ \hline 
	\rule{0pt}{1.2\normalbaselineskip} \begin{subarray}{l} 
\mbox{disk centered at $0$:} \\[0.1cm] 
\ds \{z: |z| < a\} \end{subarray} \rule[-0.8\normalbaselineskip]{0pt}{0pt} & $\cite{AR2017}$ & \mbox{Thm} \, \ref{thm:centered disk} & \cite{A2018} \mbox{ and } \mbox{Thm} \, \ref{thm:centered disk} \\ \hline
	\rule{0pt}{1.2\normalbaselineskip} \begin{subarray}{l} 
\mbox{annulus centered at $0$:} \\[0.1cm] 
\ds \{z: a < |z| < c\} \end{subarray} \rule[-0.8\normalbaselineskip]{0pt}{0pt} & $\cite{AR2017}$ & \mbox{Thm} \, \ref{thm:regular annulus} & \cite{A2018} \mbox{ and } \mbox{Thm} \, \ref{thm:regular annulus} \\ \hline
	\rule{0pt}{1.7\normalbaselineskip} \begin{subarray}{l} 
\mbox{complement of a disk} \\[0.1cm]
\mbox{centered at $0$: } \\[0.1cm]
 \ds \{z: a < |z|\} \end{subarray} \rule[-1.3\normalbaselineskip]{0pt}{0pt}
 & \mbox{Thm} \, \ref{thm:unbounded annulus} &  & \mbox{Thm} \, \ref{thm:unbounded annulus} \\ \hline
 	\rule{0pt}{1.8\normalbaselineskip} \begin{subarray}{l} 
\mbox{circular sector} \\[0.1cm]
\mbox{centered at $0$:} \\[0.1cm]
\ds \{z: |z| < a, \theta \in (0,\tfrac{2\pi}{p}) \}\, \end{subarray} \rule[-0.9\normalbaselineskip]{0pt}{0pt}
 & \begin{subarray}{l} 
\mbox{for $p=2$: \cite{AR2017}} \\[0.1cm]
\mbox{for $p>1$:} \\[0.1cm]
\mbox{Thms \ref{thm:g sector nu}, \ref{thm:g sector C}}  \end{subarray} & \begin{subarray}{l} 
\ds \nu \mbox{ for $p\geq 1$:~Thm \ref{thm:g sector nu}} \\[0.1cm]
\ds C \mbox{ for $p\geq 2$:~Thm \ref{thm:g sector C}}  \end{subarray} & \begin{subarray}{l} 
\ds \nu \mbox{ for $p\geq 1$:~Thm \ref{thm:g sector nu}} \\[0.1cm]
\ds C \mbox{ for $p\geq 2$:~Thm~\ref{thm:g sector C}}  \end{subarray} \\ \hline 
\end{array}$};
\end{tikzpicture}
\end{center}
\vspace{-0.5cm}\caption{\label{table1} Choices of $Q$ and $U$ for which $\nu$ and $C$ are explicitly known: part 1.}
\end{table}

\newpage

\begin{table}[h]
\begin{center}
\begin{tikzpicture}
\node at (0,0) {$
\begin{array}{|l@{}|c|c|c|}\hline
	\mbox{\backslashbox[4cm]{\hspace{0.1cm}$U$}{\vspace{-0.15cm}$Q(z)$}} & |z|^{2} & \frac{|z|^{2}-\tau \, \re z^{2}}{1-\tau^{2}} & |z|^{2b}, \; b \in \N_{>0}  \\ \hline
	\rule{0pt}{0.9\normalbaselineskip} 
\mbox{disk}  \rule[-0.4\normalbaselineskip]{0pt}{0pt} & \cite{AR2017} & \mbox{Thm} \, \ref{thm:Elliptic disk C} & \mbox{Thms} \, \ref{thm:ML disk nu}, \ref{thm:ML disk C} \\ \hline
	\rule{0pt}{0.9\normalbaselineskip} 
\mbox{ellipse centered at $0$} \rule[-0.4\normalbaselineskip]{0pt}{0pt} & \cite{AR2017} & \mbox{Thm $\ref{thm:Elliptic disk C}$} & \mbox{Thms} \, \ref{thm:ML ellipse nu}, \ref{thm:ML ellipse C} \\ \hline
	\rule{0pt}{0.9\normalbaselineskip}
\mbox{ellipse, annulus, cardioid}
\rule[-0.4\normalbaselineskip]{0pt}{0pt} & \mbox{for $U$ centered at $0$: \cite{AR2017}} & \mbox{Thm $\ref{thm:Elliptic disk C}$} &  \\ \hline
	\rule{0pt}{0.9\normalbaselineskip}
\mbox{circular sector}
\rule[-0.4\normalbaselineskip]{0pt}{0pt} & \mbox{Thm $\ref{thm:Elliptic disk C}$} & \mbox{Thm $\ref{thm:Elliptic disk C}$} &  \\ \hline
	\rule{0pt}{1.1\normalbaselineskip} \begin{subarray}{l} 
\mbox{equilateral triangle}
\end{subarray} \rule[-0.8\normalbaselineskip]{0pt}{0pt} & \begin{subarray}{l} 
\mbox{For $C$: \cite{AR2017}} \\[0.05cm] 
\ds \mbox{For }\nu\mbox{: Thm \ref{thm:Ginibre triangle nu}} \end{subarray} & \mbox{Thms } \ref{thm:Ginibre triangle nu}, \ref{thm:Ginibre triangle C} &  \\ \hline
	\rule{0pt}{1.4\normalbaselineskip} 
\mbox{rectangle}  \rule[-0.9\normalbaselineskip]{0pt}{0pt} & \begin{subarray}{l} \mbox{Thms \ref{thm:Ginibre rectangle nu}, \ref{thm:Ginibre rectangle C}} \\[0.1cm]
\mbox{or Thms \ref{thm:ML rectangle nu}, \ref{thm:ML rectangle C}} \end{subarray} & \mbox{Thms \ref{thm:Ginibre rectangle nu}, \ref{thm:Ginibre rectangle C}} & \mbox{Thms \ref{thm:ML rectangle nu}, \ref{thm:ML rectangle C}} \\ \hline
	\rule{0pt}{1.3\normalbaselineskip} \begin{subarray}{l} 
\mbox{complement of an ellipse } \\[0.1cm]
\mbox{centered at $0$} 
\end{subarray} \rule[-0.9\normalbaselineskip]{0pt}{0pt} & \mbox{Thms \ref{thm:EG complement ellipse nu}, \ref{thm:EG complement ellipse C}} & \mbox{Thms \ref{thm:EG complement ellipse nu}, \ref{thm:EG complement ellipse C}} &  \\ \hline
	\rule{0pt}{0.9\normalbaselineskip}
\mbox{complement of a disk}  \rule[-0.4\normalbaselineskip]{0pt}{0pt} & \mbox{Thms \ref{thm:EG complement disk nu}, \ref{thm:EG complement disk C}} & \mbox{Thms \ref{thm:EG complement disk nu}, \ref{thm:EG complement disk C}} &  \\ \hline
\end{array}
$};
\end{tikzpicture}
\end{center}
\vspace{-0.5cm}\caption{\label{table2}Choices of $Q$ and $U$ for which $\nu$ and $C$ are explicitly known: part 2.}
\end{table}

%
%
%

\noindent Before stating our results in more detail, we briefly introduce some definitions and recall useful facts from potential theory. 

\subsection{Potential theoretic background}\label{subsection:background potential}
For standard books on logarithmic potential theory, we refer to e.g. \cite{SaTo, RFbook}.
\subsubsection*{Logarithmic energy, polar sets, capacity.}
Let $\sigma$ be a Borel measure on $\C$. The support of $\sigma$, denoted $\mathrm{supp} \, \sigma$, is the set of all $z\in \C$ such that $\sigma(\mathcal{N}) > 0$ for every open set $\mathcal{N}$ containing $z$. If $\sigma$ is finite and compactly supported, its logarithmic potential $p_{\sigma}: \C \to (-\infty,\infty]$ and logarithmic energy $I_{0}[\sigma]$ are defined by
\begin{align*}
p_{\sigma}(z) = \int \log\frac{1}{|z-w|}d\sigma(w), \qquad I_{0}[\sigma] = \iint \log\frac{1}{|z-w|}d\sigma(z)d\sigma(w) = \int p_{\sigma}(z) d\sigma(z).
\end{align*}
A set $P \subset \mathbb{C}$ is called \textit{polar} if $I_{0}[\sigma]=+\infty$ for every compactly supported finite measure $\sigma \neq 0$ with $\mathrm{supp}\,\sigma \subset P$. Polar sets play the role of negligible sets in potential theory. If something holds for all $z \in E\setminus P$ for some $E\subset \C$ and some Borel polar set $P$, then we say that it holds  {\it quasi-everywhere} (q.e.) on $E$. Given $E\subset \C$, let $\mathcal P(E)$ be the set of all Borel probability measures $\sigma$ with $\mathrm{supp}\,\sigma \subset E$. The quantity
\begin{align*}
\mathrm{cap}(E):=e^{-\inf\{I_{0}[\sigma]:\sigma \in \mathcal P(E)\}},
\end{align*}
is called the (logarithmic) \textit{capacity} of $E$. Thus $\mathrm{cap}(P)=0$ if and only if $P$ is a polar set. Every Borel polar set has Lebesgue measure zero (see e.g. \cite[Theorem 3.2.4]{RFbook}).

\subsubsection*{Potentials and equilibrium measures}
Throughout this paper, we will often (but not always) assume that $Q$ is admissible on $\C$. 
\begin{definition}\label{def:admissible}
A potential $Q:\C \to \R\cup\{+\infty\}$ is admissible on a closed set $E \subset \C$ if:
\begin{enumerate}
\item[(i)] \vspace{-0.1cm} $Q$ is lower semi-continuous on $E$.
\item[(ii)] \vspace{-0.1cm} The set $E_{0}:=\{z\in E: Q(z) < + \infty\}$ has positive capacity.
\item[(iii)] \vspace{-0.1cm} If $E$ is unbounded, then $Q(z)-(2+\epsilon)\log |z| \to +\infty$ as $|z|\to +\infty$, $z\in E$, for some $\epsilon >0$.
\end{enumerate}
\end{definition}
Given $\sigma \in \mathcal{P}(E)$ and a potential $Q$, the weighted energy is defined by 
\begin{align}\label{def of Rmu}
I_{Q}[\sigma] = \iint \log [(|z-w|e^{-\frac{Q(z)}{2}}e^{-\frac{Q(w)}{2}})^{-1}]d\sigma(z)d\sigma(w) = \int p_{\sigma}(z)d\sigma(z)+\int Q(z)d\sigma(z),
\end{align}
where the last equality holds whenever both integrals $\int p_{\sigma}(z)d\sigma(z)$, $\int Q(z)d\sigma(z)$ exist and are finite. The probability measure $\mu$ minimizing $\sigma \mapsto I_{Q}[\sigma]$ among all $\sigma\in \mathcal{P}(E)$ is called {\it the equilibrium measure} for $Q$ on $E$.

The equilibrium measure for $Q$ on $U^{c}$ plays an important role in the large $n$ behavior of the hole probability $\mathbb{P}(\#\{z_{j}\in U\}=0)$ for the following heuristic reason. From \eqref{general density intro}, we have
\begin{align*}
\mathbb{P}(\#\{z_{j}\in U\}=0) = \frac{1}{Z_{n}}\int_{U^{c}}\dots \int_{U^{c}} \prod_{1\leq j<k \leq n}|z_{j}-z_{k}|^{\beta} \prod_{j=1}^{n} e^{-n \frac{\beta}{2} Q(z_{j})}d^{2}z_{j}.
\end{align*}
Thus one expects the main contribution in the large $n$ asymptotics of $\mathbb{P}(\#\{z_{j}\in U\}=0)$ to come from the $n$-tuples $(z_{1},\ldots,z_{n})$ maximizing the above integrand, i.e. such that
\begin{align}\label{n tuple}
\frac{1}{n^{2}} \sum_{1 \leq j < k \leq n} \log \frac{1}{|z_{j}-z_{k}|^{2}} + \frac{1}{n} \sum_{j=1}^{n}Q(z_{j}) = \int_{z\neq w} \log \frac{1}{|z-w|^{2}} d\mu_{n}(z)d\mu_{n}(w) + \int Q(z) d\mu_{n}(z)
\end{align}
is minimal under the constraint that $z_{1},\ldots,z_{n}\in U^{c}$, where $\mu_{n}:=\frac{1}{n}\sum_{j=1}^{n}\delta_{z_{j}}$. The $n=+\infty$ analogue of this minimization problem is precisely the problem of finding $\mu$ minimizing $\sigma \mapsto I_{Q}[\sigma]$ among all $\sigma\in \mathcal{P}(U^{c})$.

The proof of the following proposition can be found in \cite[Theorems I.1.3 and I.3.3]{SaTo}.
\begin{proposition}\label{prop:eq measure}
Let $E\subset \C$ be a closed set, and suppose that $Q$ is an admissible potential on $E$. Then the equilibrium measure $\mu$ for $Q$ on $E$ exists, is unique, and $I_{Q}[\mu]$ is finite. Furthermore, $I_{0}[\mu]$ is finite, $\mathrm{supp}\, \mu$ is compact, $\mathrm{supp}\, \mu \subset E_{0}$, $\mathrm{cap}(\mathrm{supp}\, \mu)>0$, and there exists a constant $F\in \R$ such that
\begin{align}
& 2p_{\mu}(z)+Q(z) = F, & & \mbox{for q.e. } z \in \mathrm{supp}\, \mu, \label{EL1} \\
& 2p_{\mu}(z)+Q(z) \leq F, & & \mbox{for all } z \in \mathrm{supp}\, \mu, \label{EL2} \\
& 2p_{\mu}(z)+Q(z) \geq F, & & \mbox{for q.e. } z \in E\setminus \mathrm{supp}\, \mu. \label{EL3}
\end{align}
Some of the above conditions uniquely characterize $\mu$: if $Q$ is admissible on $E$, $\sigma\in \mathcal{P}(E)$ has compact support, $I_{0}[\sigma]$ is finite, and \eqref{EL1} and \eqref{EL3} are satisfied for some constant $F$ and with $\mu$ replaced by $\sigma$, then $\sigma = \mu$.
\end{proposition}

Proposition \ref{prop:eq measure} does not apply for $Q(z)=\log(1+|z|^{2})$ if $E$ is unbounded, because then $\log(1+|z|^{2})$ is not admissible on $E$. Instead, the following similar proposition holds.
\begin{proposition}\label{prop:eq measure spherical}
Let $Q(z)=\log(1+|z|^{2})$ and let $E\subset \C$ be a closed set with $\mathrm{cap}(E)>0$. Then the equilibrium measure $\mu$ for $Q$ on $E$ exists, is unique, and $I_{Q}[\mu]$ is finite. Furthermore, $\mathrm{cap}(\mathrm{supp}\, \mu)>0$, and there exists a constant $F\in \R$ such that \eqref{EL1}--\eqref{EL3} hold.

\smallskip \noindent Some of the above conditions uniquely characterize $\mu$: if $Q(z)=\log(1+|z|^{2})$,  $E\subset \C$ is a closed set with $\mathrm{cap}(E)>0$, $I_{Q}[\sigma]$ is finite, and \eqref{EL1} and \eqref{EL3} are satisfied for some constant $F$ and with $\mu$ replaced by $\sigma$, then $\sigma = \mu$.
\end{proposition}
Proposition \ref{prop:eq measure spherical} follows from \cite[Theorem 1.2 with $s=0$]{DS2007} (note that \cite[Theorem 1.2]{DS2007} is formulated on the unit sphere instead of $\C$).

If $E=\C$ in Proposition \ref{prop:eq measure spherical}, then $\mu$ and $S$ are given by \eqref{def of mu and S spherical}.

\subsubsection*{Green functions and balayage measures}
We start with the definition of the Green function of an unbounded domain with pole at infinity. 
\begin{definition}\label{def:green with pole at inf}
Suppose that $U\subset \C\cup \{\infty\}$ is open, connected, unbounded, $\infty \in U$ and $\mathrm{cap}(\partial U)>0$. Then $g_{U}(z,\infty)$ is the unique function satisfying the following conditions:
\begin{itemize}
\item \vspace{-0.15cm} $g_{U}(z,\infty)$ is harmonic in $U\setminus \{\infty\}$ and is bounded as $z$ stays away from $\infty$,
\item \vspace{-0.15cm} $g_{U}(z,\infty)-\frac{1}{2\pi}\log |z| = \bigO(1)$ as $z\to\infty$,
\item \vspace{-0.15cm} $\lim_{z\to x, z\in U}g_{U}(z,\infty)=0$ for q.e. $x\in \partial U$.
\end{itemize}
\end{definition}
\noindent For the existence and uniqueness of $g_{U}(z,\infty)$, see e.g. \cite[Section II.4]{SaTo}.

\medskip \noindent We now turn to the definition and properties of balayage measures. For the proof of the following proposition, see \cite[Theorem II.4.7]{SaTo} (see also \cite[Theorems II.4.1, II.4.4]{SaTo}).
\begin{proposition}[Definition and properties of balayage measures]\label{prop:def of bal}
Suppose that $U\subset (\C\cup \{\infty\})$ is open, $\partial U$ is compact in $\C$, $\mathrm{cap}(\partial U)>0$, and $U$ has finitely many connected components. Let $\sigma$ be a finite Borel measure with compact support in $\C$ and such that $\sigma(U^{c})=0$. Then there exists a unique measure $\nu$ supported on $\partial U$ such that $\nu(\partial U)=\sigma(U)$, $p_{\nu}$ is bounded on $\partial U$, $\nu(P)=0$ for every Borel polar set $P\subset \C$, and $p_{\nu}(z)=p_{\sigma}(z)+c^{\hspace{0.02cm}\sigma}_{U}$ for q.e. $z\in U^c$, where $c^{\hspace{0.02cm}\sigma}_{U}=2\pi\int_{\Omega}g_{\Omega}(t,\infty)d\sigma(t)$ and $\Omega$ is the unbounded component of $U$ ($c^{\hspace{0.02cm}\sigma}_{U}=0$ if $U$ has no unbounded component). $\nu$ is said to be the balayage measure of $\sigma$ on $\partial U$, and is also denoted $\nu = \mathrm{Bal}(\sigma,\partial U)$.
\end{proposition}

\section{Main results}\label{Subsection:general result}

Our main result is an asymptotic formula as $n\to + \infty$ for $\mathbb{P}(\# \{z_{j}\in U\} = 0)$, where $\mathbb{P}$ refers to either \eqref{general density intro} or to the spherical ensemble \eqref{spherical ensemble}. When $\PP$ refers to \eqref{general density intro}, our result will be valid for potentials $Q$ such that Assumption \ref{ass:Q} holds. For such $Q$, we let $\mu$ be the equilibrium measure for $Q$ on $\C$, and $S:=\mathrm{supp}\, \mu$ (here $S$ is compact by Proposition \ref{prop:eq measure}). When $\PP$ refers to \eqref{spherical ensemble}, we let the measure $\mu$ and its support $S$ be given by \eqref{def of mu and S spherical}. 

\begin{assumption}\label{ass:Q}
$Q$ is admissible on $\C$, and there exists $F\in \R$ such that 
\begin{align}\label{regularity condition}
2p_{\mu}(z) + Q(z) \geq F, \qquad \mbox{for all } z \in S^{c}.
\end{align}
Moreover, $Q$ is H\"{o}lder continuous in a neighborhood of $S$: there exist $\delta \hspace{-0.05cm}>\hspace{-0.05cm}0$, $K\hspace{-0.05cm}>\hspace{-0.05cm}0$, $\alpha \hspace{-0.05cm}\in\hspace{-0.05cm} (0,1]$ so that
\begin{align}\label{holder}
|Q(z)-Q(w)| \leq K |z-w|^{\alpha}, \qquad \mbox{for all } z,w \in \{s\in \C:\mathrm{dist}(s,S) \leq \delta\}.
\end{align}
\end{assumption}
\bigskip (Thus \eqref{regularity condition} requires that \eqref{EL3} holds for $\mu$ with ``q.e." replaced by ``for all".) Theorem \ref{thm:general pot} will be valid under the following assumptions on $U$.
\begin{assumption}\label{ass:U}
The set $U\subset (\C\cup \{\infty\})$ is open, $\partial U$ is compact in $\C$ and satisfies $\partial U \subset S$, $\mathrm{cap}(\partial U)>0$, and $U$ has finitely many connected components.
\end{assumption}
\begin{assumption}\label{ass:U2}
$U$ satisfies the exterior ball condition, i.e.
there exists $\epsilon_{0}>0$ such that for every $z\in \partial U$ there exists $\eta\in U^c$ such that
\begin{align}\label{eqn:exterior ball condition lol}
U^c \supset D(\eta,\epsilon_{0}) \;\; \mbox{and
}\; |z-\eta | = \epsilon_{0},
\end{align}
where $D(\eta,\epsilon_{0})$ is the open disk centered at $\eta$ of radius $\epsilon_{0}$.
\end{assumption}
\begin{remark}
Any open set $U$ with compact and $C^{2}$ boundary satisfies Assumption \ref{ass:U2}. A rectangle does not have $C^{2}$ boundary but satisfies Assumption \ref{ass:U2}. A circular sector $U=\{re^{i\theta}:r\in(0,a), \theta \in (0,\frac{2\pi}{p})\}$ with $a>0$ and $p\in (1,\infty)$ satisfies Assumption \ref{ass:U2} if and only if $p \geq 2$.
\end{remark}

We now state our first main result.
\begin{theorem}\label{thm:general pot} Fix $\beta>0$.

\medskip \noindent (i) Suppose $Q:\C \to \R\cup\{+\infty\}$ and $U\subset \C$ are such that Assumptions \ref{ass:Q}, \ref{ass:U} and \ref{ass:U2} hold. As $n \to +\infty$, we have $\mathbb{P}(\# \{z_{j}\in U\} = 0) = \exp \big( -Cn^{2}+o(n^{2}) \big)$, where $\mathbb{P}$ refers to \eqref{general density intro},
\begin{align}\label{C in main thm}
C = \frac{\beta}{4}\bigg(  \int_{\partial U}Q(z)d\nu(z) + 2c^{\hspace{0.02cm}\mu}_{U} - \int_{U}Q(z)d\mu(z) \bigg),
\end{align}
$\mu$ is the equilibrium measure for $Q$ on $\C$, $\nu = \mathrm{Bal}(\mu|_{U},\partial U)$, $c^{\hspace{0.02cm}\mu}_{U}=2\pi\int_{\Omega}g_{\Omega}(z,\infty)d\mu(z)$ and $\Omega$ is the unbounded component of $U$ ($c^{\hspace{0.02cm}\mu}_{U}=0$ if $U$ has no unbounded component).

\medskip \noindent (ii) The above statement also holds for the spherical ensemble \eqref{spherical ensemble}, provided that $U$ is also bounded. More precisely, let $U$ be bounded and such that Assumptions \ref{ass:U} and \ref{ass:U2} hold with $S=\C$. As $n \to +\infty$, we have $\mathbb{P}(\# \{z_{j}\in U\} = 0) = \exp \big( -Cn^{2}+o(n^{2}) \big)$, where $\mathbb{P}$ refers to \eqref{spherical ensemble}, $C$ is given by
\begin{align*}
C = \frac{\beta}{4}\bigg(  \int_{\partial U}\log(1+|z|^{2})d\nu(z) - \int_{U}\log(1+|z|^{2})d\mu(z) \bigg),
\end{align*}
$d\mu(z) = \frac{d^{2}z}{\pi(1+|z|^{2})^{2}}$, $\mathrm{supp} \, \mu = \C$, and $\nu = \mathrm{Bal}(\mu|_{U},\partial U)$.

\end{theorem}

\begin{remark}
Theorem \ref{thm:general pot} (i) extends \cite[Theorem 1.5]{AR2017} and \cite[Theorem 1.2 (B)]{A2018} for potentials $Q$ that are not assumed to be rotation-invariant, and for hole regions $U$ that are allowed to be unbounded. We also only assume that $Q$ is H\"{o}lder continuous in a neighborhood of $S$, which is an improvement over the assumption made in \cite[Theorem 1.2 (B)]{A2018} that $Q$ is Lipschitz continuous in a neighborhood of $S$. The H\"{o}lder continuity is important to handle potentials such as $Q(z) = |z|^{2b}$ with $b\in (0,\frac{1}{2})$.  
\end{remark}

\begin{remark}\label{remark:AR star domain}
In \cite{AR2017}, the authors studied hole probabilities for the point process \eqref{general density intro} with $Q(z) = |z|^{2}$ and $\beta=2$. They considered two types of hole regions $U$ (with $U\subset S$, $U$ open):
\begin{itemize}
\item[(a)] Hole regions $U$ satisfying the exterior ball conditions (see \cite[Theorem 1.5]{AR2017}).
\item[(b)] Hole regions $U$ for which there exists a decreasing sequence of open sets $U_{1} \supseteq U_{2} \supseteq \ldots$ with $\overline{U}\subset U_{m} \subseteq S$ for all $m\in \N_{>0}$ and such that the sequence $\nu_{m}:=\mathrm{Bal}(\mu|_{U_{m}},\partial U_{m})$ converges weakly to $\nu := \mathrm{Bal}(\mu|_{U},\partial U)$ (see \cite[Theorem 1.4]{AR2017}). 
\end{itemize}
Furthermore, using the fact that for $Q(z) = |z|^{2}$ the measure $\mu$ has uniform density, it is proved in \cite[Remark 5.2]{AR2017} that (b') implies that $\nu_{m}$ converges weakly to $\nu$ for some $U_{m}$ satisfying $U \subseteq U_{m}$ (but not necessarily $\overline{U}\subset U_{m}$), where (b') is the following condition:
\begin{itemize}
\item[(b')] The open set $U$ is a star domain. 
\end{itemize}
Indeed, if $U$ satisfies (b'), then $\nu_{m}$ converges weakly to $\nu$ with $U_{m}:= (U-z_{\star})(1+\frac{1}{m})+z_{\star}$, where $z_{\star} \in U$ is a star center of $U$ (for more details see \cite[Remark 5.2]{AR2017}). Criteria (b') allows to treat hole regions $U$ not covered by (a), and vice-versa. For example, circular sectors $\smash{U=\{re^{i\theta}:r\in(0,a), \theta \in (0,\frac{2\pi}{p})\}}$ with $a>0$ and $p\in (1,2)$ satisfy (b') but not (a), and annuli satisfy (a) but not (b'). Note also that circular sectors with $p=1$ satisfy (b') but not (b) (because $\overline{U}\not\subset U_{m}$). 

\medskip \noindent \cite[Theorems 1.5 and 1.4]{AR2017} were generalized in \cite[Theorem 1.2 (B) and (A)]{A2018} for general $\beta>0$ and for a class of rotation-invariant potentials $Q$. When $\mu$ is not uniform, (b') cannot be used and therefore condition (b) becomes difficult to verify in practice, even in the rotation-invariant setting \cite{A2018}. Hence we have not tried to extend  Theorem \ref{thm:general pot} for hole regions satisfying condition (b).
\end{remark}

\begin{remark}\label{remark:U is bounded is no loss of generality for spherical}
In Theorem \ref{thm:general pot} (ii), it is assumed that $U$ is bounded. This is without loss of generality. Indeed, assume that $U$ is unbounded, and take $z_{\star} \in \overline{U}^{c}$ ($\overline{U}^{c} \neq \emptyset$ by Assumption \ref{ass:U}). Let $\varphi:\mathbb{S}^{2}\to (\C\cup\{\infty\})$ be the stereographic projection with respect to the north pole. Suppose that $h:\mathbb{S}^{2}\to \mathbb{S}^{2}$ is a rotation of $\mathbb{S}^{2}$ that sends $\varphi^{-1}(z_{\star})$ to the north pole $(0,0,1)$. Then, since \eqref{spherical ensemble on the sphere} remains unchanged under the change of variables $\hat{z}_{j}\to h(\hat{z}_{j})$, $j=1,\ldots,n$, we have that $\mathbb{P}(\# \{z_{j}\in U\} = 0)=\mathbb{P}(\# \{z_{j}\in U_{\star}\} = 0)$, where $\mathbb{P}$ refers to \eqref{spherical ensemble} and $U_{\star} := \varphi \circ h \circ \varphi^{-1}(U)$ is bounded.
\end{remark}

\medskip Our next result is about the equilibrium measure for $Q$ on $U^{c}$. As mentioned in Subsection \ref{subsection:background potential}, this equilibrium measure plays an important role in the study of the hole probability $\mathbb{P}(\# \{z_{j}\in U\} = 0)$. This measure is also naturally related to the point process \eqref{general density intro}, when \eqref{general density intro} is conditioned on the hole event $\#\{z_{j}\in U\}=0$. Indeed, the $n$-tuples $(z_{1},\ldots,z_{n})$ maximizing the density of this conditional point process are precisely those $n$-tuples that minimize \eqref{n tuple} under the constraint that $z_{1},\ldots,z_{n}\in U^{c}$. Note that Proposition \ref{prop:general pot} is valid under weaker assumptions on $Q$ and $U$ than in Theorem \ref{thm:general pot}.

\begin{proposition}\label{prop:general pot}
(i) Suppose that $Q:\C \to \R\cup\{+\infty\}$ is admissible on $\C$. Let $\mu$ be the equilibrium measure for $Q$ on $\C$, and $S:=\mathrm{supp}\, \mu$. Suppose also that $U\subset \C$ is such that Assumption \ref{ass:U} holds. Then $\mu|_{S\setminus U} + \nu$ is the equilibrium measure for $Q$ on $U^{c}$, where $\nu = \mathrm{Bal}(\mu|_{U},\partial U)$.

\medskip \noindent (ii) The above statement also holds if $Q(z)=\log(1+|z|^{2})$, $\smash{d\mu(z) = \frac{d^{2}z}{\pi(1+|z|^{2})^{2}}}$, $\mathrm{supp} \, \mu = \C$, and $U\subset \C$ is such that Assumption \ref{ass:U} holds, provided that $U$ is also bounded. 
\end{proposition}

\subsection{Results for general rotation-invariant potentials}
\begin{figure}[h]
\begin{tikzpicture}[master]
\node at (0,0) {\includegraphics[width=3.63cm]{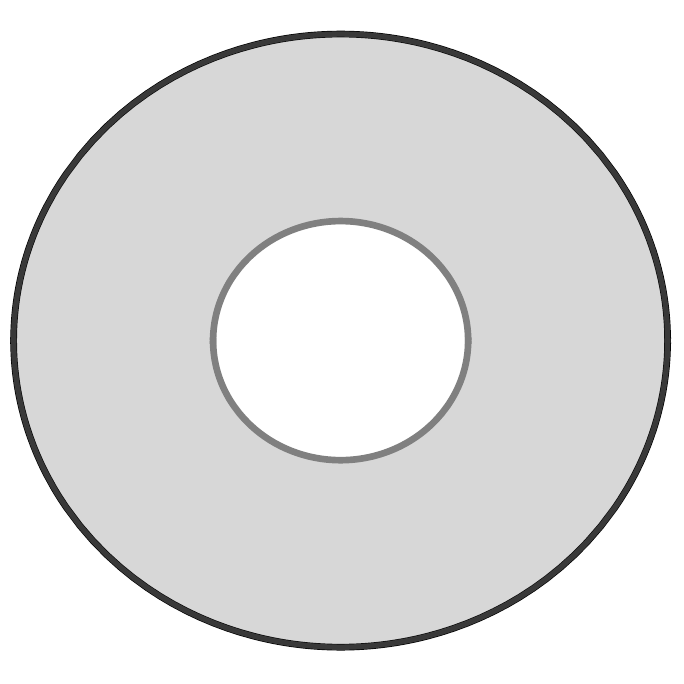}};
\node at (0,-2) {\footnotesize Theorem \ref{thm:centered disk}};
\end{tikzpicture} \hspace{-0.21cm}
\begin{tikzpicture}[slave]
\node at (0,0) {\includegraphics[width=3.63cm]{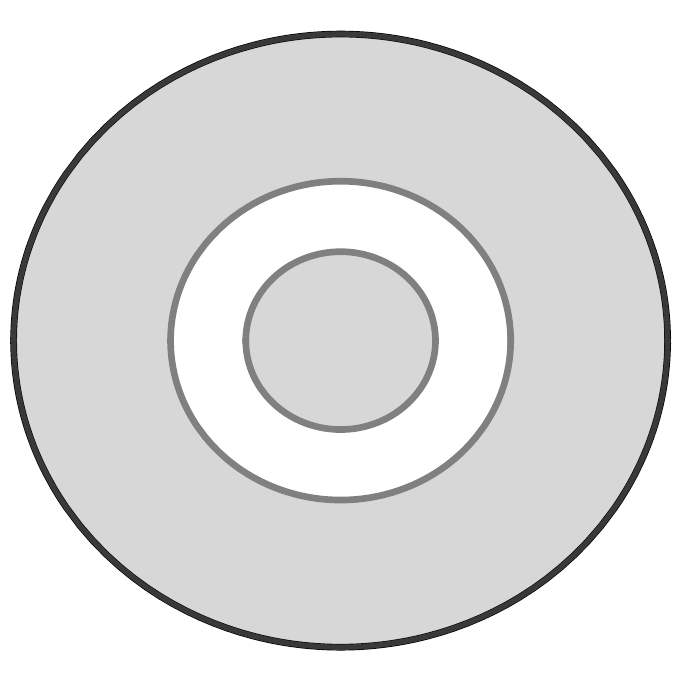}};
\node at (0,-2) {\footnotesize Theorem \ref{thm:regular annulus}};
\end{tikzpicture} \hspace{-0.21cm}
\begin{tikzpicture}[slave]
\node at (0,0) {\includegraphics[width=3.63cm]{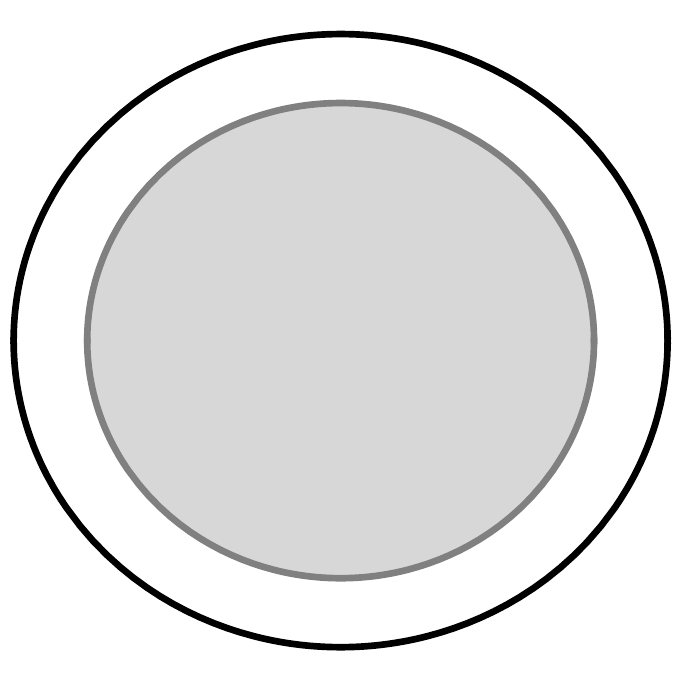}};
\node at (0,-2) {\footnotesize Theorem \ref{thm:unbounded annulus}};
\end{tikzpicture} \hspace{-0.21cm}
\begin{tikzpicture}[slave]
\node at (0,0) {\includegraphics[width=3.63cm]{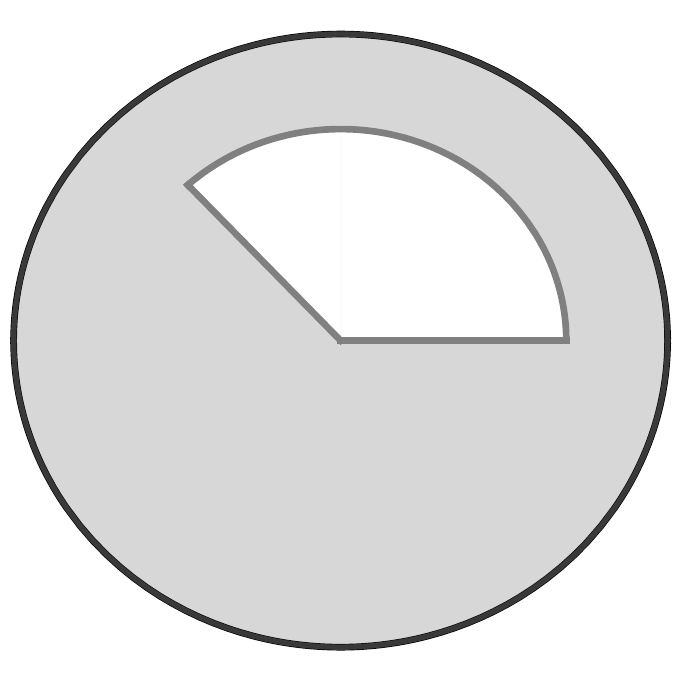}};
\node at (0,-2) {\footnotesize Theorems \ref{thm:g sector nu}, \ref{thm:g sector C}};
\end{tikzpicture}
\caption{\label{fig:hole prob for general rot invariant potential} Results for rotation-invariant $Q$, when $U$ is a disk (left), an annulus (middle-left), the complement of a disk (middle-right), and a circular sector (right). $S\setminus U$ is the shaded region, $U$ is in white, $\partial S$ is in black, and $\partial U$ is in gray. In the pictures, $S$ is a disk, but we also consider the more general situation \eqref{support S when Q=g 2}.}
\end{figure}
In this subsection, we study the balayage measure $\nu = \mathrm{Bal}(\mu|_{U},\partial U)$ and the hole probability $\mathbb{P}(\# \{z_{j}\in U\} = 0)$, where $\mathbb{P}$ refers to \eqref{general density intro}, in the case where $Q(z) = g(|z|)$ for some $g: [0,+\infty)\to \R \cup \{+\infty\}$ such that Assumption \ref{ass:Q} holds. We assume as in \eqref{support S when Q=g} that 
\begin{align}\label{support S when Q=g 2}
S := \mathrm{supp} \,\mu  = \{z : |z| \in \mathrm{S} \}, \qquad \mathrm{S} := [r_{0},r_{1}]\cup [r_{2},r_{3}] \ldots \cup [r_{2\ell},r_{2\ell+1}]
\end{align}
for some $0 \leq r_{0} < r_{1} < \ldots < r_{2\ell}<r_{2\ell+1}<+\infty$. We further assume that $g \in C^{2}(\mathrm{S}\setminus \{0\})$, so that as in \eqref{mu radially symmetric intro} we have 
\begin{align}\label{mu radially symmetric intro 2}
d\mu(z) = \frac{\Delta Q(z)}{4\pi} d^{2}z = d\mu_{\mathrm{rad}}(r) \frac{d\theta}{2\pi}, \qquad d\mu_{\mathrm{rad}}(r):=\frac{r}{2} \big( g''(r) + \frac{1}{r}g'(r) \big) dr,
\end{align}
where $z=re^{i\theta}$, $r\geq 0$, $\theta \in (-\pi,\pi]$, and $\mathrm{supp} \, \mu_{\mathrm{rad}} := \mathrm{S}$. We consider situations where the hole region $U$ is a disk, an annulus, the complement of a disk, and a circular sector. When $U$ is bounded, i.e. for the disk, the annulus and the circular sector, we also treat the spherical point process \eqref{spherical ensemble} (in this case, $\mu$ and $S$ are given by \eqref{def of mu and S spherical}, $g(r):=\log(1+r^{2})$, $d\mu_{\mathrm{rad}}(r) := \frac{2r}{(1+r^{2})^{2}}dr$ and $\mathrm{S}:=[0,+\infty)$). The results of this subsection are summarized in Figure \ref{fig:hole prob for general rot invariant potential} (with $\ell=0$ and $r_{0}=0$). 

\subsubsection{The disk}
\begin{theorem}\label{thm:centered disk}
Fix $\beta>0$.

\medskip \noindent (i) Suppose $Q(z) = g(|z|)$ for some $g: [0,+\infty)\to \R \cup \{+\infty\}$ such that Assumption \ref{ass:Q} holds, and that $S = \{z : |z| \in \mathrm{S} \}$ with $\mathrm{S} = [r_{0},r_{1}]\cup [r_{2},r_{3}] \cup \ldots \cup [r_{2\ell},r_{2\ell+1}]$ for some $0 \leq r_{0} < r_{1} < ... <r_{2\ell+1}<+\infty$. Suppose also that $g \in C^{2}(\mathrm{S}\setminus \{0\})$ and recall that $\mu_{\mathrm{rad}}$ is given by \eqref{mu radially symmetric intro 2}.

\medskip \noindent Let $a\in [r_{2k},r_{2k+1}]\setminus \{0\}$ for some $k\in \{0,\ldots,\ell\}$, and let $U = \{z: |z| < a\}$. Then $\nu := \mathrm{Bal}(\mu|_{U},\partial U)$ is given by
\begin{align}\label{nu for centered disk}
d\nu(z) = \frac{d\theta}{2\pi}\int_{r_{0}}^{a}d\mu_{\mathrm{rad}}(r)  = \bigg( \sum_{j=0}^{k-1} \frac{r_{2j+1}g'(r_{2j+1})-r_{2j}g'(r_{2j})}{2} + \frac{a g'(a)-r_{2k}g'(r_{2k})}{2} \bigg)\frac{d\theta}{2\pi}
\end{align}
where $z = ae^{i\theta}$, $\theta \in (-\pi,\pi]$. If $r_{0}=0$, then $r_{0}g'(r_{0})$ must be replaced by $0$ in \eqref{nu for centered disk}. Moreover, as $n \to +\infty$, we have $\mathbb{P}(\# \{z_{j}\in U\} = 0) = \exp \big( -Cn^{2}+o(n^{2}) \big)$, where $\mathbb{P}$ refers to \eqref{general density intro} and
\begin{align}
C & = \frac{\beta}{4} \int_{r_{0}}^{a}\big(g(a)-g(r)\big) d\mu_{\mathrm{rad}}(r). \label{C in thm disk rot inv}
\end{align}
(ii) A similar statement holds for the spherical point process \eqref{spherical ensemble}. More precisely, let $\mu$ and $S$ be given by \eqref{def of mu and S spherical}, $a>0$ and $U = \{z: |z| < a\}$.  Then $\nu := \mathrm{Bal}(\mu|_{U},\partial U)$ is given by $d\nu(z) = \frac{a^{2}}{1+a^{2}} \frac{d\theta}{2\pi}$ where $z = ae^{i\theta}$, $\theta \in (-\pi,\pi]$. Moreover, as $n \to +\infty$, we have $\mathbb{P}(\# \{z_{j}\in U\} = 0) = \exp \big( -Cn^{2}+o(n^{2}) \big)$, where $\mathbb{P}$ refers to \eqref{spherical ensemble} and $C$ is given by \eqref{C in thm disk rot inv} with $g(r)=\log(1+r^{2})$ and $d\mu_{\mathrm{rad}}(r) = \frac{2r}{(1+r^{2})^{2}}dr$, i.e.
\begin{align}\label{C in thm disk spherical}
C = \frac{\beta}{4} \bigg( \log(1+a^{2})- \frac{a^{2}}{1+a^{2}} \bigg).
\end{align}
\end{theorem}
\begin{remark}
Some of the results contained in Theorem \ref{thm:centered disk} were already known:
\begin{itemize}
\item For the Ginibre case $g(r)=r^{2}$, \eqref{C in thm disk rot inv} becomes $C=\frac{\beta}{8}a^{4}$. For $\beta=2$ this constant was already obtained in \cite{GHS1988, ForresterHoleProba, JLM1993, APS2009, AR2017, L et al 2019, C2021}, and for general $\beta>0$ in \cite{A2018}.
\item For the Mittag-Leffler case $g(r)=r^{2b}$ with $b>0$, \eqref{C in thm disk rot inv} becomes $C=\frac{\beta}{8}ba^{4b}$, which agrees with the results from \cite{A2018, C2021}.
\item For $\ell=k=0$ and $r_{0}=0$, \eqref{C in thm disk rot inv} becomes (after integration by parts) $C = \frac{\beta}{8} \int_{0}^{a}r \, g'(r)^{2} dr$, which recovers the result from \cite{A2018}.
\item The constant $C$ in \eqref{C in thm disk spherical} was already known for $\beta=2$, see \cite[Proposition 3.1]{AZ2015}. 
\end{itemize}
\end{remark}
\subsubsection{The annulus}

\begin{theorem}\label{thm:regular annulus}
Fix $\beta>0$.

\medskip \noindent (i) Suppose $Q(z) = g(|z|)$ for some $g: [0,+\infty)\to \R \cup \{+\infty\}$ such that Assumption \ref{ass:Q} holds, and that $S = \{z : |z| \in \mathrm{S} \}$ with $\mathrm{S} = [r_{0},r_{1}]\cup [r_{2},r_{3}] \cup \ldots \cup [r_{2\ell},r_{2\ell+1}]$ for some $0 \leq r_{0} < r_{1} < \ldots <r_{2\ell+1}<+\infty$. Suppose also that $g \in C^{2}(\mathrm{S}\setminus \{0\})$ and recall that $\mu_{\mathrm{rad}}$ is given by \eqref{mu radially symmetric intro 2}.

\medskip \noindent Let $\rho_{1}\in [r_{2k},r_{2k+1}]\setminus \{0\}$, $\rho_{2}\in [r_{2q},r_{2q+1}]$ for some $k,q\in \{0,\ldots,\ell\}$ and such that $\rho_{1}<\rho_{2}$, and let $U = \{z: \rho_{1} < |z| < \rho_{2} \}$. Then $\nu := \mathrm{Bal}(\mu|_{U},\partial U)$ is given by
\begin{align}\label{nu for centered annulus}
d\nu(z) = \begin{cases}
\lambda\kappa \frac{d\theta}{2\pi}, & \mbox{if } z = \rho_{1}e^{i\theta}, \\
(1-\lambda)\kappa \frac{d\theta}{2\pi}, & \mbox{if } z = \rho_{2}e^{i\theta},
\end{cases}
\end{align}
where $\theta \in (-\pi,\pi]$, $\kappa := \int_{\rho_{1}}^{\rho_{2}} d\mu_{\mathrm{rad}}(r)$ and $\lambda := \big( \log \rho_{2} - \frac{1}{\kappa} \int_{\rho_{1}}^{\rho_{2}} (\log r ) d\mu_{\mathrm{rad}}(r)\big) / \log(\rho_{2}/\rho_{1})$. Moreover, as $n \to +\infty$, we have $\mathbb{P}(\# \{z_{j}\in U\} = 0) = \exp \big( -Cn^{2}+o(n^{2}) \big)$, where $\mathbb{P}$ refers to \eqref{general density intro} and
\begin{align}
C = \frac{\beta}{4}  \int_{\rho_{1}}^{\rho_{2}} \big( (1-\lambda) g(\rho_{2})  + \lambda g(\rho_{1}) -g(r) \big)d\mu_{\mathrm{rad}}(r). \label{C in thm annulus rot inv}
\end{align}

\medskip \noindent (ii) A similar statement holds for the spherical point process \eqref{spherical ensemble}. More precisely, let $\mu$ and $S$ be given by \eqref{def of mu and S spherical}, $0<\rho_{1}<\rho_{2}<+\infty$ and $U = \{z : \rho_{1} < |z| < \rho_{2}\}$. Then $\nu := \mathrm{Bal}(\mu|_{U},\partial U)$ is given by \eqref{nu for centered annulus}. Moreover, as $n \to +\infty$, we have $\mathbb{P}(\# \{z_{j}\in U\} = 0) = \exp \big( -Cn^{2}+o(n^{2}) \big)$, where $\mathbb{P}$ refers to \eqref{spherical ensemble} and $C$ is given by \eqref{C in thm annulus rot inv} with $g(r)=\log(1+r^{2})$ and $d\mu_{\mathrm{rad}}(r) = \frac{2r}{(1+r^{2})^{2}}dr$, i.e.
\begin{align}\label{C in thm annulus spherical}
C = \frac{\beta}{4} \bigg( \log \frac{1+\rho_{2}^{2}}{1+\rho_{1}^{2}} - \frac{\rho_{2}^{2}-\rho_{1}^{2}}{(1+\rho_{1}^{2})(1+\rho_{2}^{2})} - \frac{\big(\log \frac{1+\rho_{2}^{2}}{1+\rho_{1}^{2}}\big)^{2}}{2\log(\rho_{2}/\rho_{1})} \bigg).
\end{align}
\end{theorem}
\begin{remark}Here are some particular cases of Theorem \ref{thm:regular annulus} (i):
\begin{itemize}
\item If $k=q$, then $\kappa$ and $\lambda$ simplifies to
\begin{align*}
\kappa = \frac{\rho_{2} g'(\rho_{2}) - \rho_{1} g'(\rho_{1})}{2}, \qquad \lambda  = \frac{g(\rho_{2})-g(\rho_{1})-\rho_{1}g'(\rho_{1})\log(\rho_{2}/\rho_{1})}{(\rho_{2}g'(\rho_{2})-\rho_{1}g'(\rho_{1}))\log(\rho_{2}/\rho_{1})},
\end{align*}
and $C$ in \eqref{C in thm annulus rot inv} can be rewritten (using integration by parts) as
\begin{align*}
C = \frac{\beta}{8} \bigg\{ \int_{\rho_{1}}^{\rho_{2}} rg'(r)^{2}dr - \big( (1-\lambda)\rho_{1} g'(\rho_{1}) + \lambda \rho_{2} g'(\rho_{2}) \big) (g(\rho_{2})-g(\rho_{1})) \bigg\}.
\end{align*}
For $\ell=0$ and $r_{0}=0$, the above formula was previously known from \cite{A2018}.
\item For the Mittag-Leffler case $g(r) = r^{2b}$, the constant $C$ in \eqref{C in thm annulus rot inv} becomes
\begin{align}\label{lol92}
C = \frac{\beta}{2} \bigg( \frac{b}{4}(\rho_{2}^{4b}-\rho_{1}^{4b}) - \frac{(\rho_{2}^{2b}-\rho_{1}^{2b})^{2}}{4\log(\rho_{2}/\rho_{1})} \bigg),
\end{align}
which was already known from \cite{A2018}.
\end{itemize}
\end{remark}
\subsubsection{The complement of a disk}
\begin{theorem}\label{thm:unbounded annulus}
Suppose $Q(z) = g(|z|)$ for some $g: [0,+\infty)\to \R \cup \{+\infty\}$ such that Assumption \ref{ass:Q} holds, and that $S = \{z : |z| \in \mathrm{S} \}$ with $\mathrm{S} = [r_{0},r_{1}]\cup [r_{2},r_{3}] \cup \ldots \cup [r_{2\ell},r_{2\ell+1}]$ for some $0 \leq r_{0} < r_{1} < \ldots <r_{2\ell+1}<+\infty$. Suppose also that $g \in C^{2}(\mathrm{S}\setminus \{0\})$ and recall that $\mu_{\mathrm{rad}}$ is given by \eqref{mu radially symmetric intro 2}.

\medskip \noindent Let $a \in [r_{2k},r_{2k+1}]\setminus \{0\}$ for some $k\in \{0,\ldots,\ell\}$, and let $U = \{z: a < |z| \}$. Then $\nu := \mathrm{Bal}(\mu|_{U},\partial U)$ is given by
\begin{align}\label{nu for unbounded annulus}
d\nu(z) = \frac{d\theta}{2\pi}\int_{a}^{r_{2\ell+1}}d\mu_{\mathrm{rad}}(r), \qquad z = ae^{i\theta}, \; \theta \in (-\pi,\pi].
\end{align}
Fix $\beta >0$. As $n \to +\infty$, $\mathbb{P}(\# \{z_{j}\in U\} = 0) = \exp \big( -C n^{2}+o(n^{2}) \big)$, where $\mathbb{P}$ refers to \eqref{general density intro} and
\begin{align}\label{C for unbounded annulus}
& C = \frac{\beta}{4} \int_{a}^{r_{2\ell+1}} \big(g(a)-g(r)+2\log \tfrac{r}{a}\big)d\mu_{\mathrm{rad}}(r).
\end{align}
\end{theorem}
\begin{remark}
For the Mittag-Leffler case $g(r) = r^{2b}$, the constant $C$ in \eqref{C for unbounded annulus} becomes
\begin{align*}
C = \frac{\beta}{2} \bigg( a^{2b} - \frac{ba^{4b}}{4} - \frac{1}{2b}\log(b\, a^{2b}) - \frac{3}{4b} \bigg).
\end{align*}
For $\beta=2$, this constant was already known from \cite{CMV2016} for $b=1$ and from \cite{C2021} for general $b>0$.
\end{remark}
\subsubsection{The circular sector}

\begin{theorem}\label{thm:g sector nu}
(i) Suppose that $Q(z) = g(|z|)$ for some $g: [0,+\infty)\to \R \cup \{+\infty\}$, $Q$ is admissible on $\C$, and that $S = \{z : |z| \in \mathrm{S} \}$ with $\mathrm{S} = [0,r_{1}]\cup [r_{2},r_{3}] \cup \ldots \cup [r_{2\ell},r_{2\ell+1}]$ for some $0 < r_{1} < \ldots <r_{2\ell+1}<+\infty$. Suppose also that $g \in C^{2}(\mathrm{S}\setminus \{0\})$ and recall that $\mu_{\mathrm{rad}}$ is given by \eqref{mu radially symmetric intro 2}.

\medskip \noindent Let $p\in [1,\infty)$, $a\in (0,r_{1}]$, $U=\{re^{i\theta} : 0<r<a, \, 0<\theta < \frac{2\pi}{p}\}$, and $\nu := \mathrm{Bal}(\mu|_{U},\partial U)$. 

\noindent For $r\in (0,a)$, we have
\begin{align}
\frac{d\nu(r)}{dr} & = \frac{d\nu(r e^{\frac{2\pi i}{p}})}{dr} = \frac{1}{\pi^{2}r} \int_{0}^{a} \bigg( \arctanh \frac{(\min\{r,x\})^{\frac{p}{2}}}{(\max\{r,x\})^{\frac{p}{2}}} - \arctanh \frac{r^{\frac{p}{2}}x^{\frac{p}{2}}}{a^{p}}   \bigg) d\mu_{\mathrm{rad}}(x) \label{g sector nu segment} \\
& = \frac{1}{\pi^{2}r} \sum_{m=0}^{+\infty} \frac{1}{2m+1} \int_{0}^{a} \bigg[\bigg(\frac{\min\{r,x\}}{\max\{r,x\}}\bigg)^{pm+\frac{p}{2}}-\bigg(\frac{r\, x}{a^{2}}\bigg)^{pm+\frac{p}{2}} \bigg]d\mu_{\mathrm{rad}}(x). \label{g sector nu segment 2}
\end{align}
For $z=ae^{i\theta}$, $\theta \in (0,\frac{2\pi}{p})$, we have
\begin{align}
\frac{d\nu(z)}{d\theta} & = \frac{2}{\pi^{2}} \int_{0}^{a} \im \arctanh \big( e^{\frac{ip\theta}{2}} ( \tfrac{x}{a} )^{\frac{p}{2}} \big) \; d\mu_{\mathrm{rad}}(x) \label{g sector nu arxtanh} \\
& = \frac{1}{\pi^{2}} \sum_{m=0}^{+\infty}  \frac{\sin((pm+\frac{p}{2})\theta)}{(m+\frac{1}{2}) \, a^{pm+\frac{p}{2}}}\int_{0}^{a} x^{pm+\frac{p}{2}} d\mu_{\mathrm{rad}}(x). \label{g sector nu}
\end{align}
In \eqref{g sector nu segment} and \eqref{g sector nu arxtanh}, $\arctanh z := \frac{1}{2}(\log (1+z) - \log(1-z))$ where the principal branch
is used for the logarithms.

\medskip \noindent (ii) A similar statement holds if $\mu$ and $S$ are defined by \eqref{def of mu and S spherical}. More precisely, let $p\in [1,\infty)$, $a\in (0,+\infty)$ and $U=\{re^{i\theta} : 0<r<a, \, 0<\theta < \frac{2\pi}{p}\}$.  Then $\nu := \mathrm{Bal}(\mu|_{U},\partial U)$ is given by \eqref{g sector nu segment}--\eqref{g sector nu} with $d\mu_{\mathrm{rad}}(x) = \frac{2x}{(1+x^{2})^{2}}dx$. 
\end{theorem}
\begin{remark}\label{boundary of U if p=1}
If $p=1$, then $\partial U$ should be understood as $\partial U = \big((0,a)+i0_{+}\big) \cup \big((0,a)-i0_{+}\big) \cup \{z: |z|=a\}$, and the density of $\nu$ along $(0,a)+i0_{+}$ is the same as the density of $\nu$ along $(0,a)-i0_{+}$ and is given by the right-hand sides of \eqref{g sector nu segment} and \eqref{g sector nu segment 2}.
\end{remark}
\begin{remark}\label{remark:circular sector ML}
For the Mittag-Leffler case $g(r) = r^{2b}$ with $b>0$, \eqref{g sector nu} becomes
\begin{align}\label{nu circular ML}
\frac{d\nu(z)}{d\theta} = \frac{2 b^{2}a^{2b}}{\pi^{2}} \sum_{m=0}^{+\infty}  \frac{\sin((pm+\frac{p}{2})\theta)}{(m+\frac{1}{2})(2b+pm+\frac{p}{2})}.
\end{align}
This formula was previously known from \cite{AR2017} for $b=1$ and $p=2$. If $2b \notin \frac{p}{2}+p \N$, then \eqref{g sector nu segment 2} becomes
\begin{align}\label{nu circular ML segment}
\frac{d\nu(r)}{dr} = \frac{d\nu(r e^{\frac{2\pi i}{p}})}{dr} = \frac{2p \, b^{2} a^{2b}}{\pi^{2}r} \sum_{m=0}^{+\infty} \frac{(\frac{r}{a})^{2b}-(\frac{r}{a})^{pm+\frac{p}{2}}}{(pm+\frac{p}{2}-2b)(pm+\frac{p}{2}+2b)}.
\end{align}
If $2b \in \frac{p}{2}+p \N$, then $\frac{d\nu(r)}{dr}$ and $\frac{d\nu(r e^{\frac{2\pi i}{p}})}{dr}$ can be obtained by taking limits in \eqref{nu circular ML segment}. In particular, $\frac{d\nu(r)}{dr}  \asymp r^{\min\{2b,\frac{p}{2}\}-1}$ as $r\to 0$ if $2b \neq \frac{p}{2}$, while $\frac{d\nu(r)}{dr}  \asymp r^{2b-1}\log \frac{1}{r}$  if $2b = \frac{p}{2}$ (see also Figure \ref{fig:circular sector}). 
\end{remark}

\begin{figure}
\begin{tikzpicture}[master]
\node at (0,0) {\includegraphics[width=3.63cm]{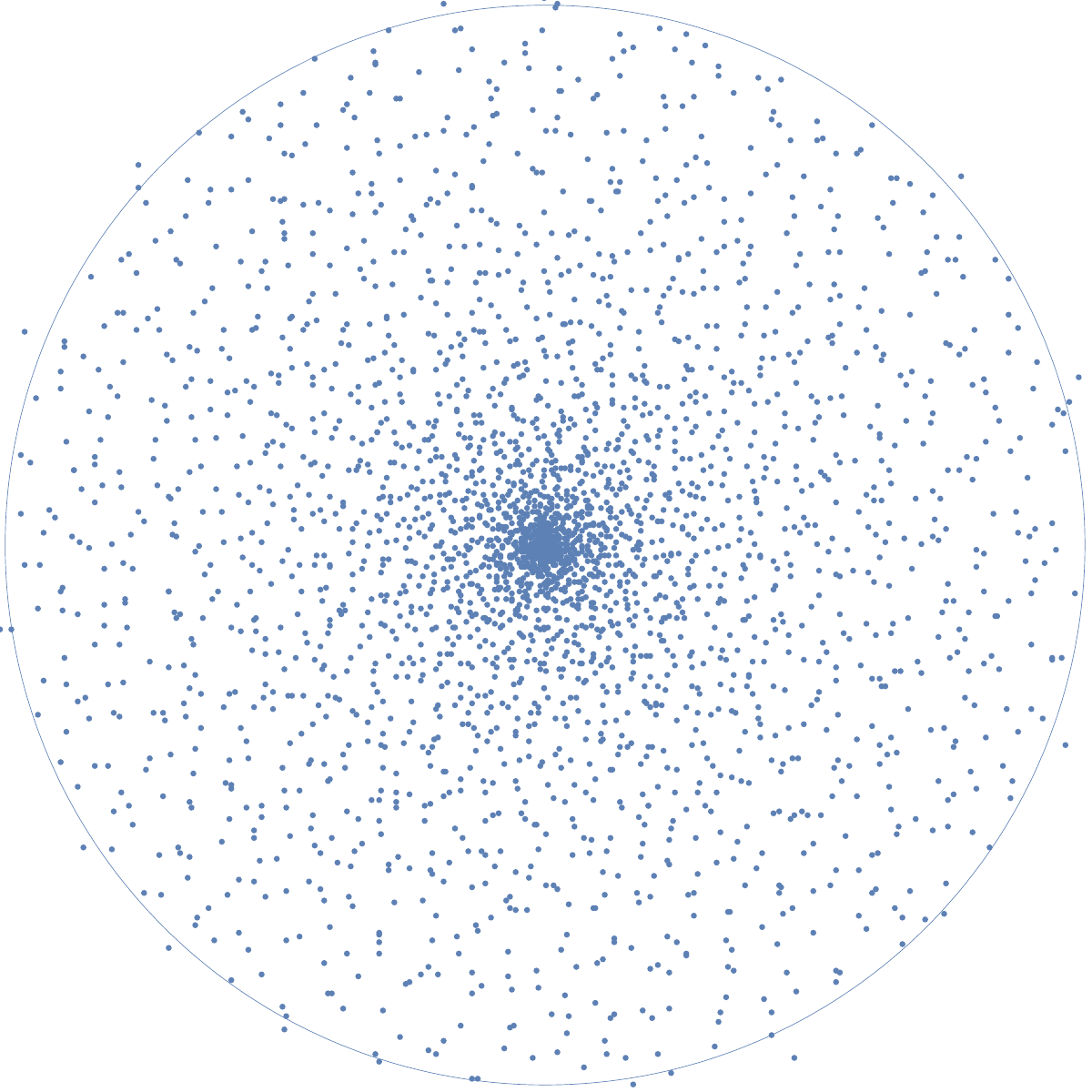}};
\node at (0,2) {\footnotesize $b=\frac{1}{3}$};
\end{tikzpicture} \hspace{-0.21cm}
\begin{tikzpicture}[slave]
\node at (0,0) {\includegraphics[width=3.63cm]{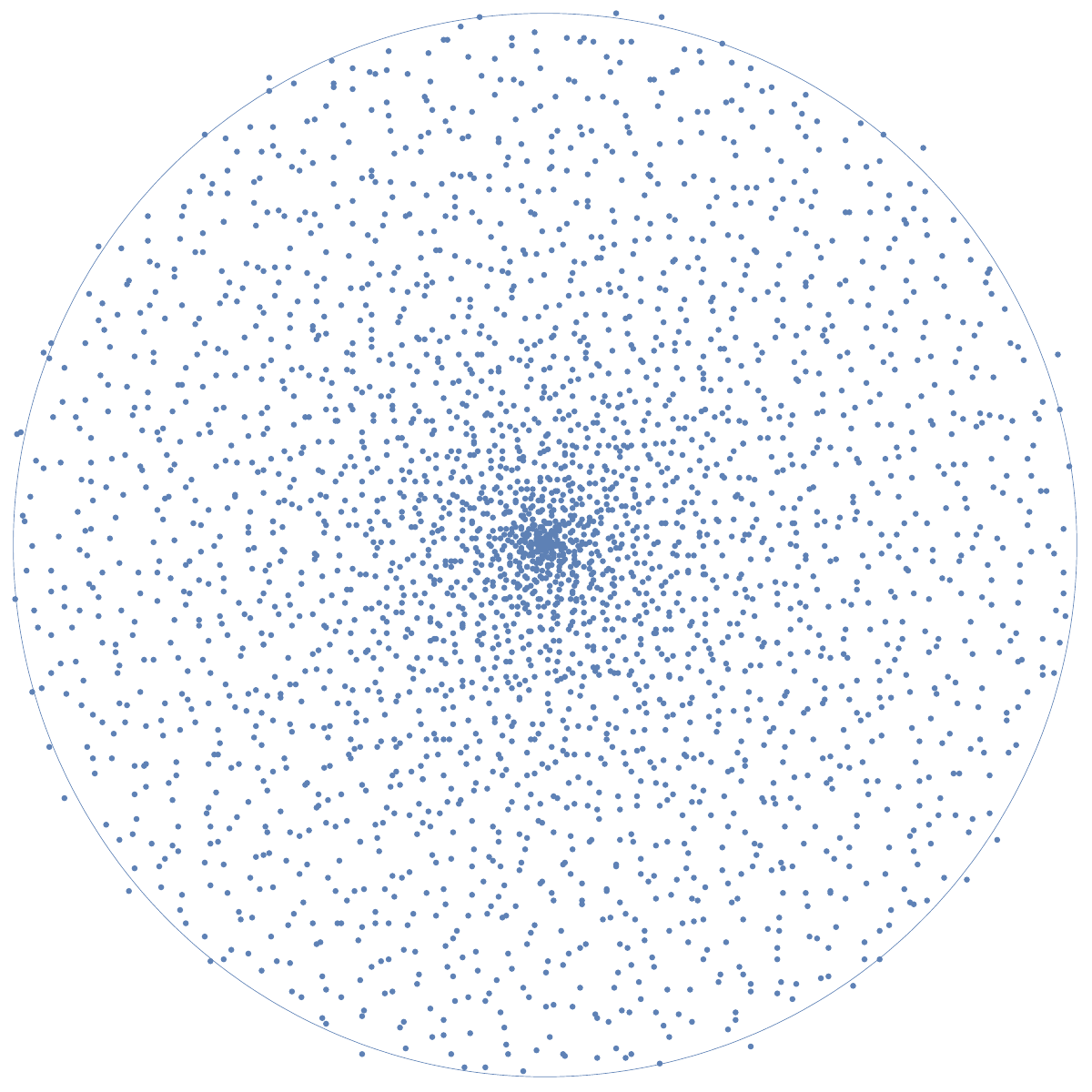}};
\node at (0,2) {\footnotesize $b=\frac{1}{2}$};
\end{tikzpicture} \hspace{-0.21cm}
\begin{tikzpicture}[slave]
\node at (0,0) {\includegraphics[width=3.63cm]{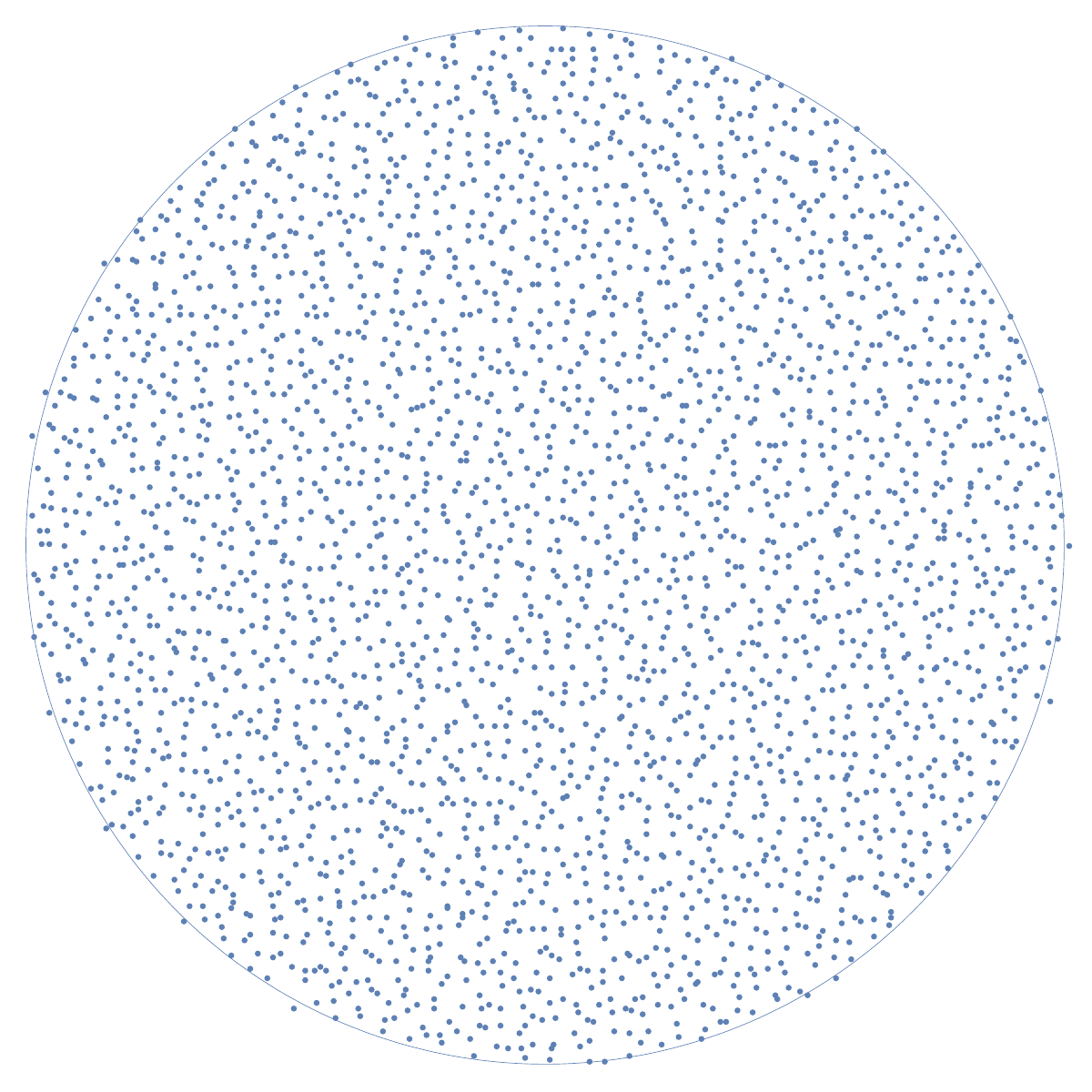}};
\node at (0,2) {\footnotesize $b=1$};
\draw[dashed] (-9.3,-1.9)--(5.3,-1.9);
\end{tikzpicture} \hspace{-0.21cm}
\begin{tikzpicture}[slave]
\node at (0,0) {\includegraphics[width=3.63cm]{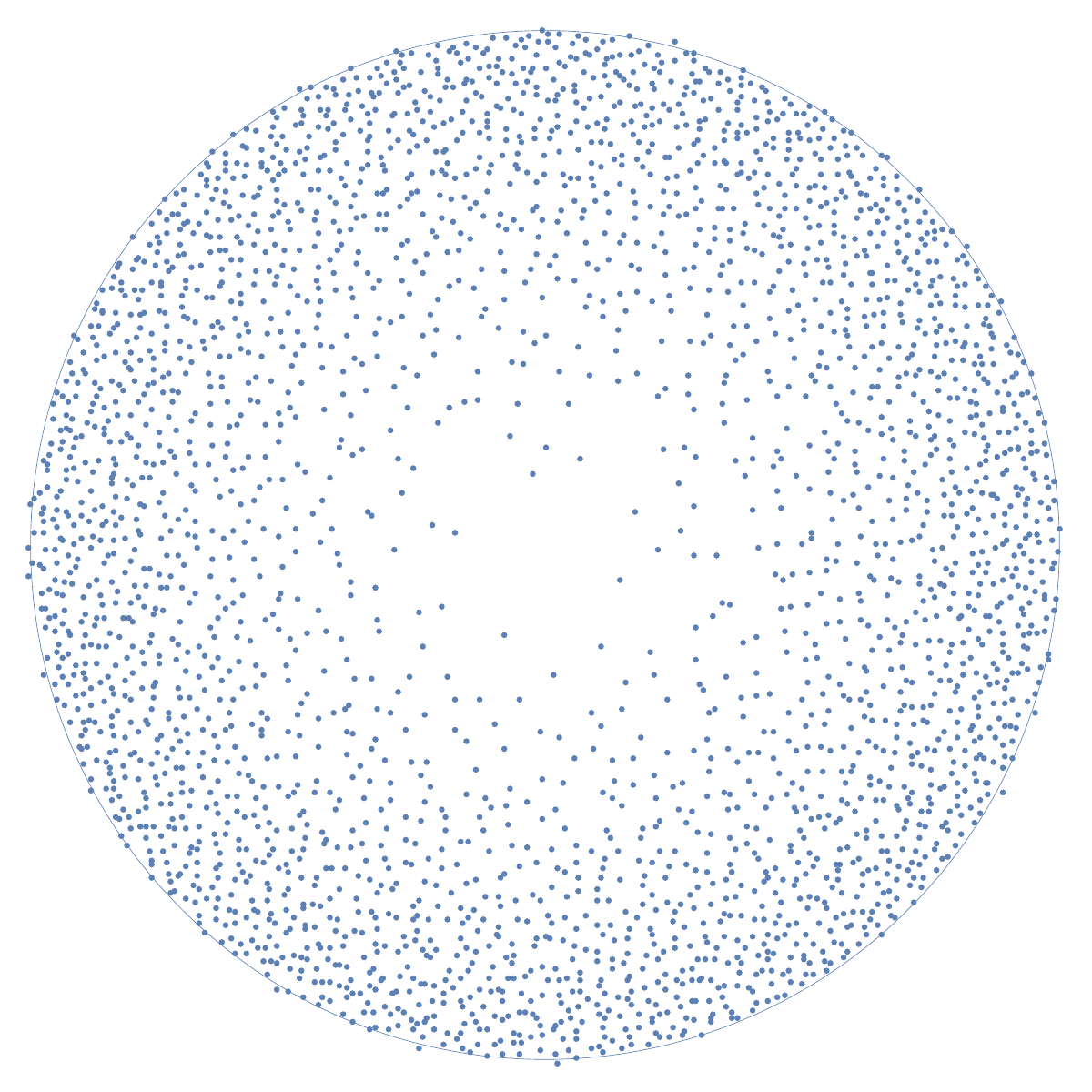}};
\node at (0,2) {\footnotesize $b=2$};
\end{tikzpicture} \\[-0.5cm]
\begin{tikzpicture}[slave]
\node at (0,0) {\includegraphics[width=3.63cm]{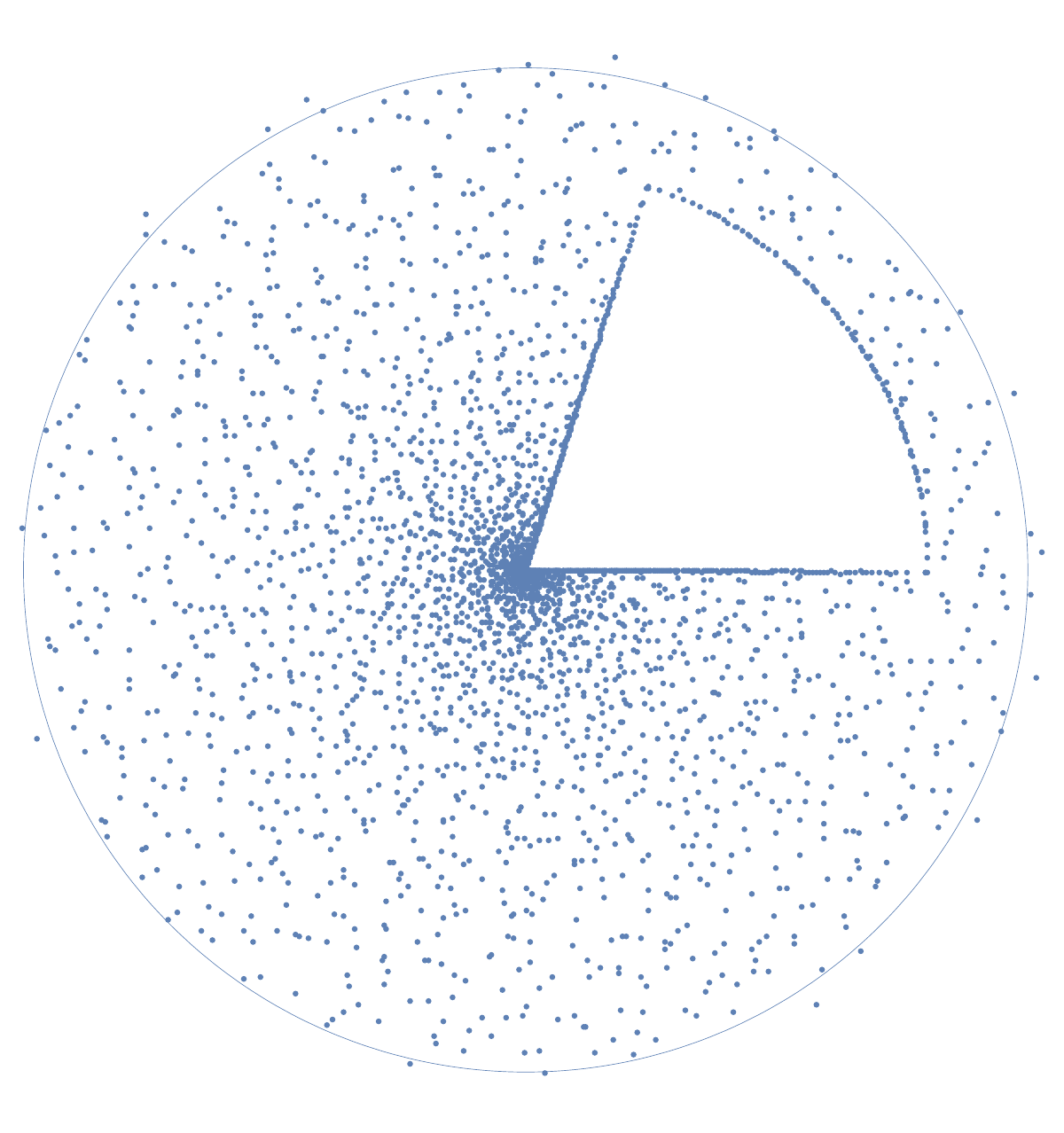}};
\end{tikzpicture} \hspace{-0.21cm}
\begin{tikzpicture}[slave]
\node at (0,0) {\includegraphics[width=3.63cm]{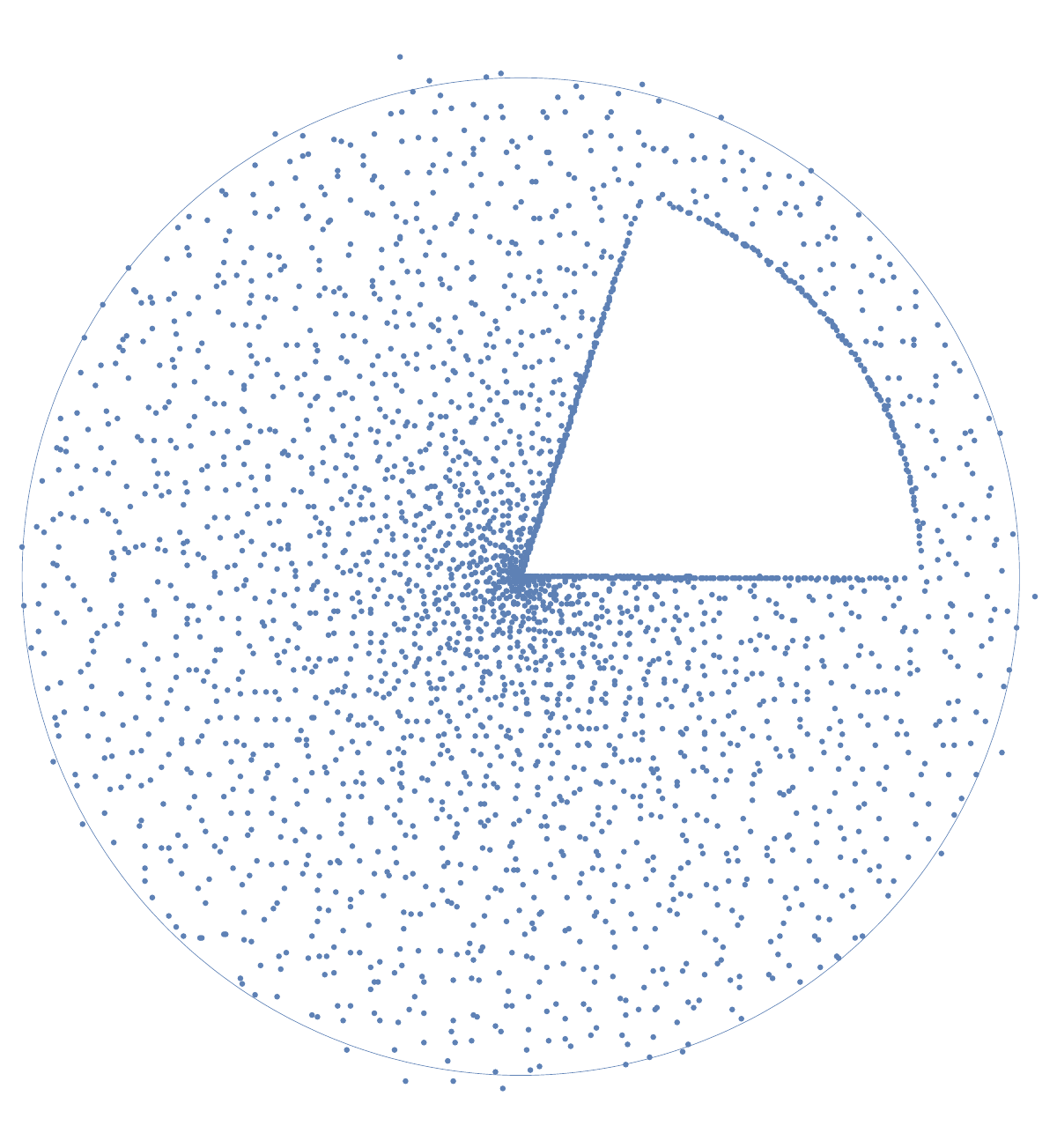}};
\end{tikzpicture} \hspace{-0.21cm}
\begin{tikzpicture}[slave]
\node at (0,0) {\includegraphics[width=3.63cm]{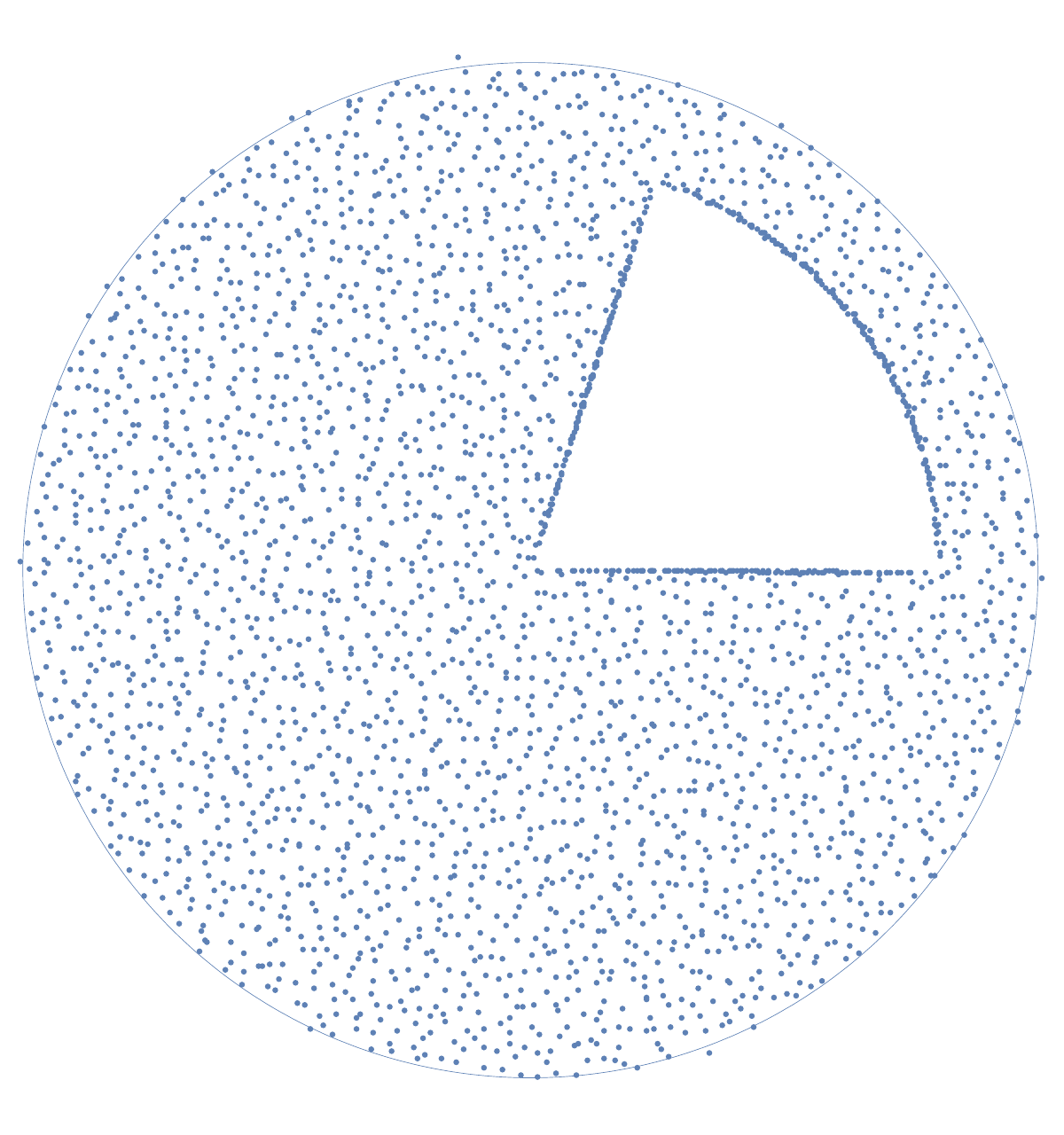}};
\node at (-1.9,1.6) {\footnotesize $p=5$};
\end{tikzpicture} \hspace{-0.21cm}
\begin{tikzpicture}[slave]
\node at (0,0) {\includegraphics[width=3.63cm]{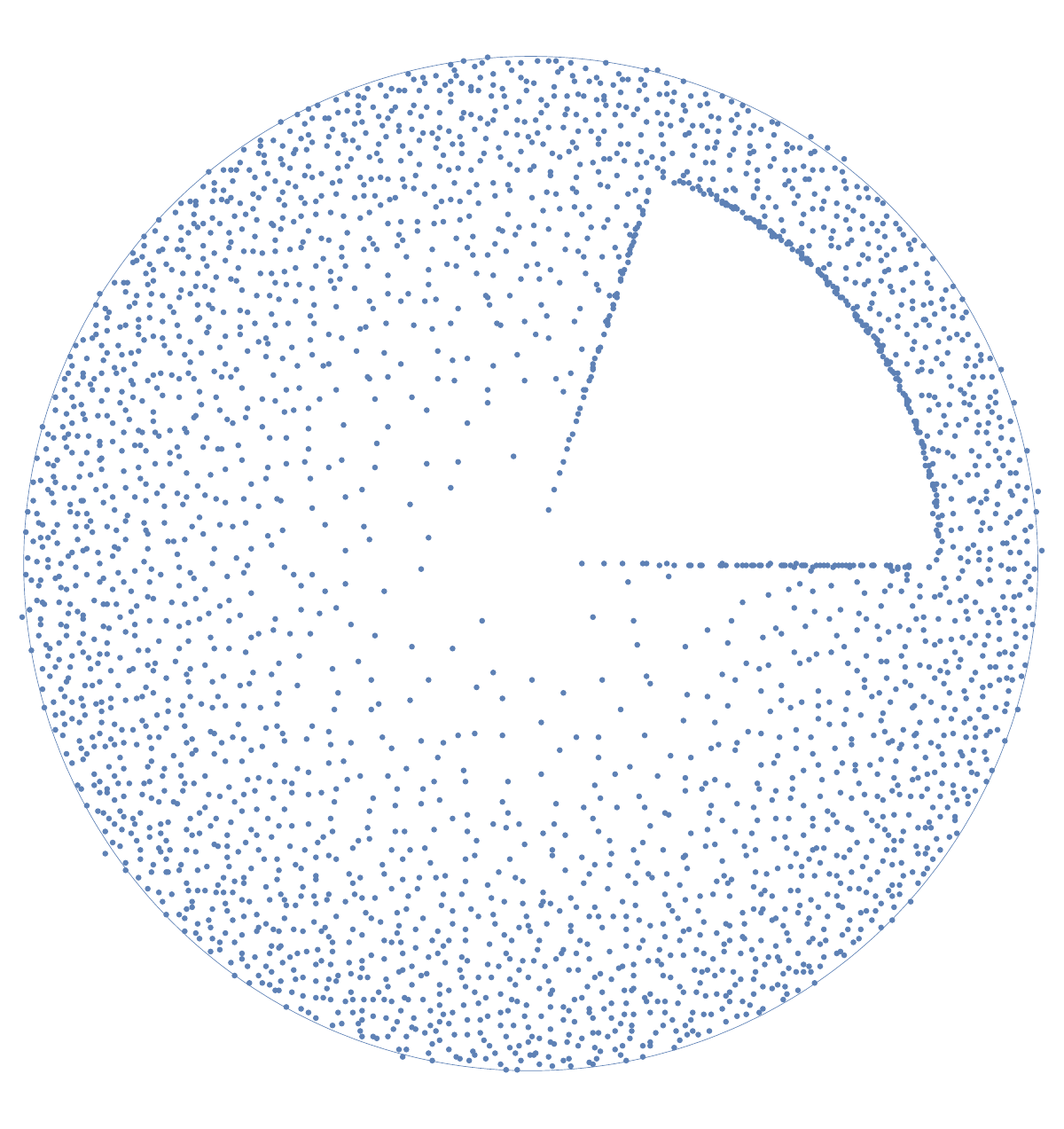}};
\end{tikzpicture} \\[-0.9cm]
\begin{center}
\begin{tikzpicture}[master]
\node at (0,0) {\includegraphics[width=14cm]{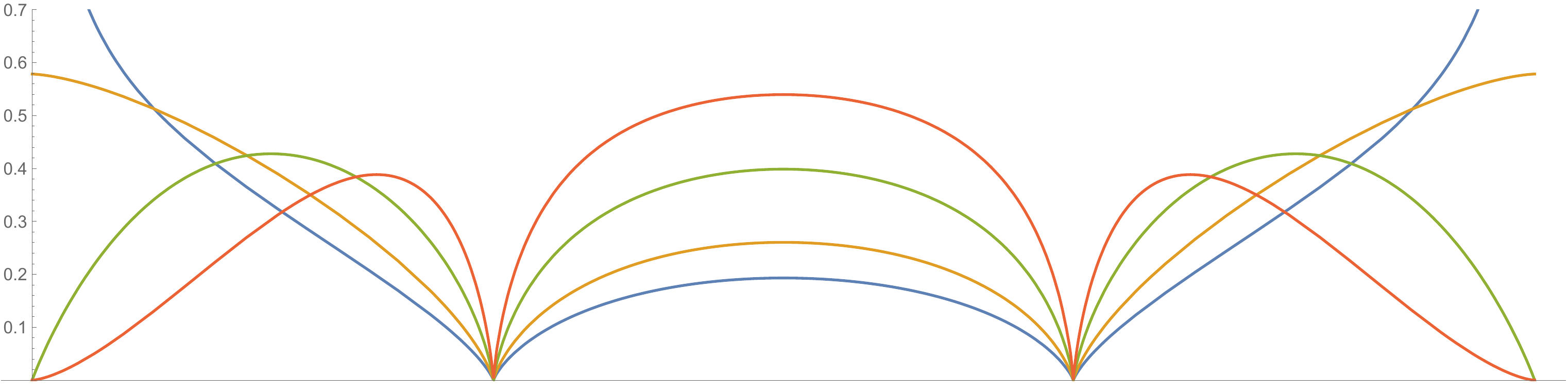}};
\node at (0,-1) {\footnotesize $b=\frac{1}{3}$};
\node at (0,-0.25) {\footnotesize $b=\frac{1}{2}$};
\node at (0,0.4) {\footnotesize $b=1$};
\node at (0,1.1) {\footnotesize $b=2$};
\node at (-5.5,1.5) {\footnotesize $\asymp x^{-\frac{1}{3}}$};
\node at (-6.35,1.2) {\footnotesize $\asymp 1$};
\node at (-6.35,-1) {\footnotesize $\asymp x$};
\node at (-6,-1.4) {\footnotesize $\asymp x^{\frac{3}{2}}$};
\node at (-6.35,1.2) {\footnotesize $\asymp 1$};
\draw[fill] (-6.71,-1.65) circle (0.04);
\node at (-6.71,-1.9) {\footnotesize $0$};
\draw[fill] (-2.59,-1.65) circle (0.04);
\node at (-2.59,-1.9) {\footnotesize $a$};
\draw[fill] (2.59,-1.65) circle (0.04);
\node at (2.61,-1.9) {\footnotesize $a e^{\frac{2\pi i}{p}}$};
\draw[fill] (6.71,-1.65) circle (0.04);
\node at (6.71,-1.9) {\footnotesize $0$};
\draw[dashed] (-7.3,-2.15)--(7.3,-2.15);
\end{tikzpicture}
\end{center}\,\\[-0.85cm]
\begin{tikzpicture}[master]
\node at (0,0) {\includegraphics[width=3.63cm]{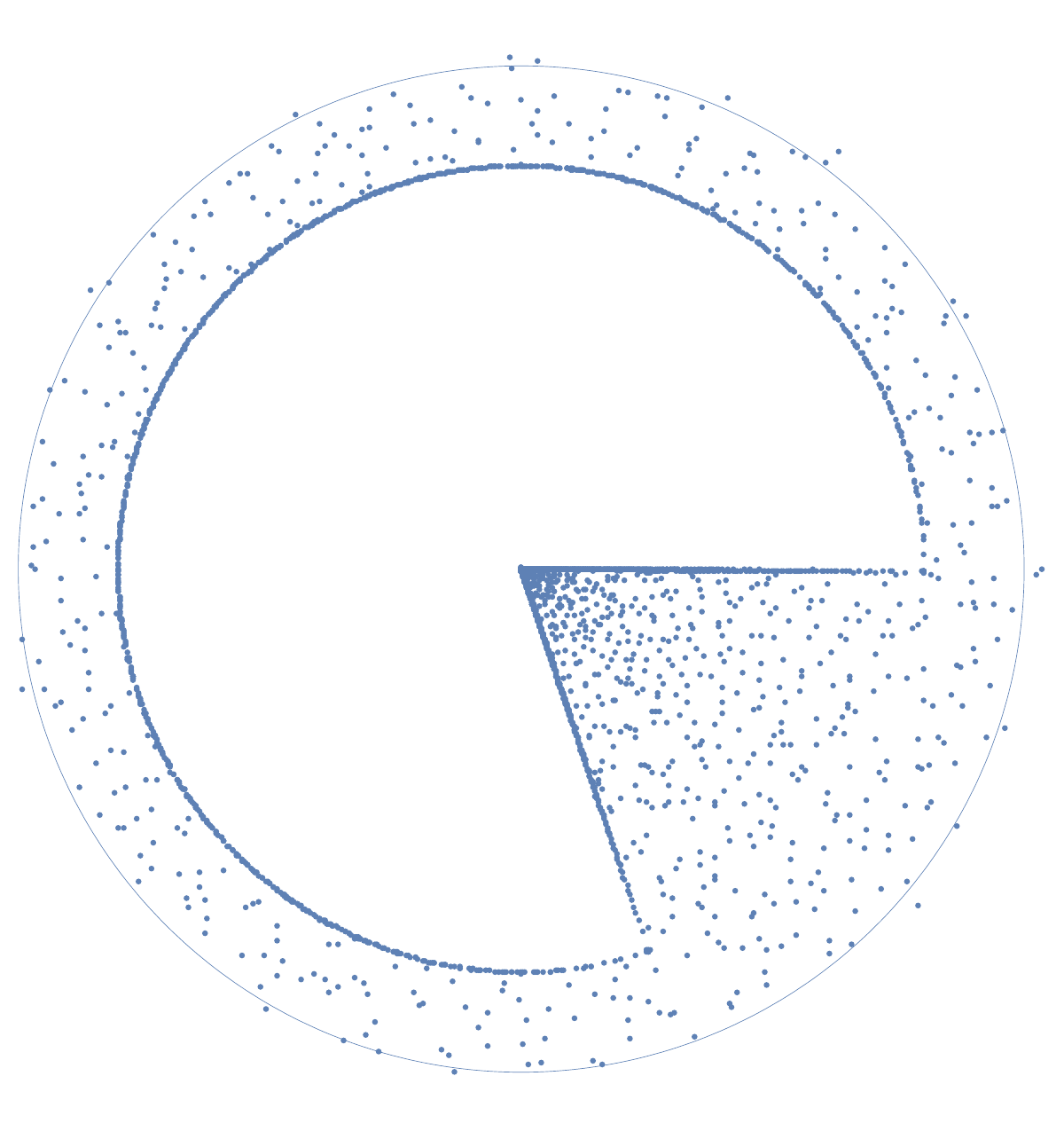}};
\end{tikzpicture} \hspace{-0.21cm}
\begin{tikzpicture}[slave]
\node at (0,0) {\includegraphics[width=3.63cm]{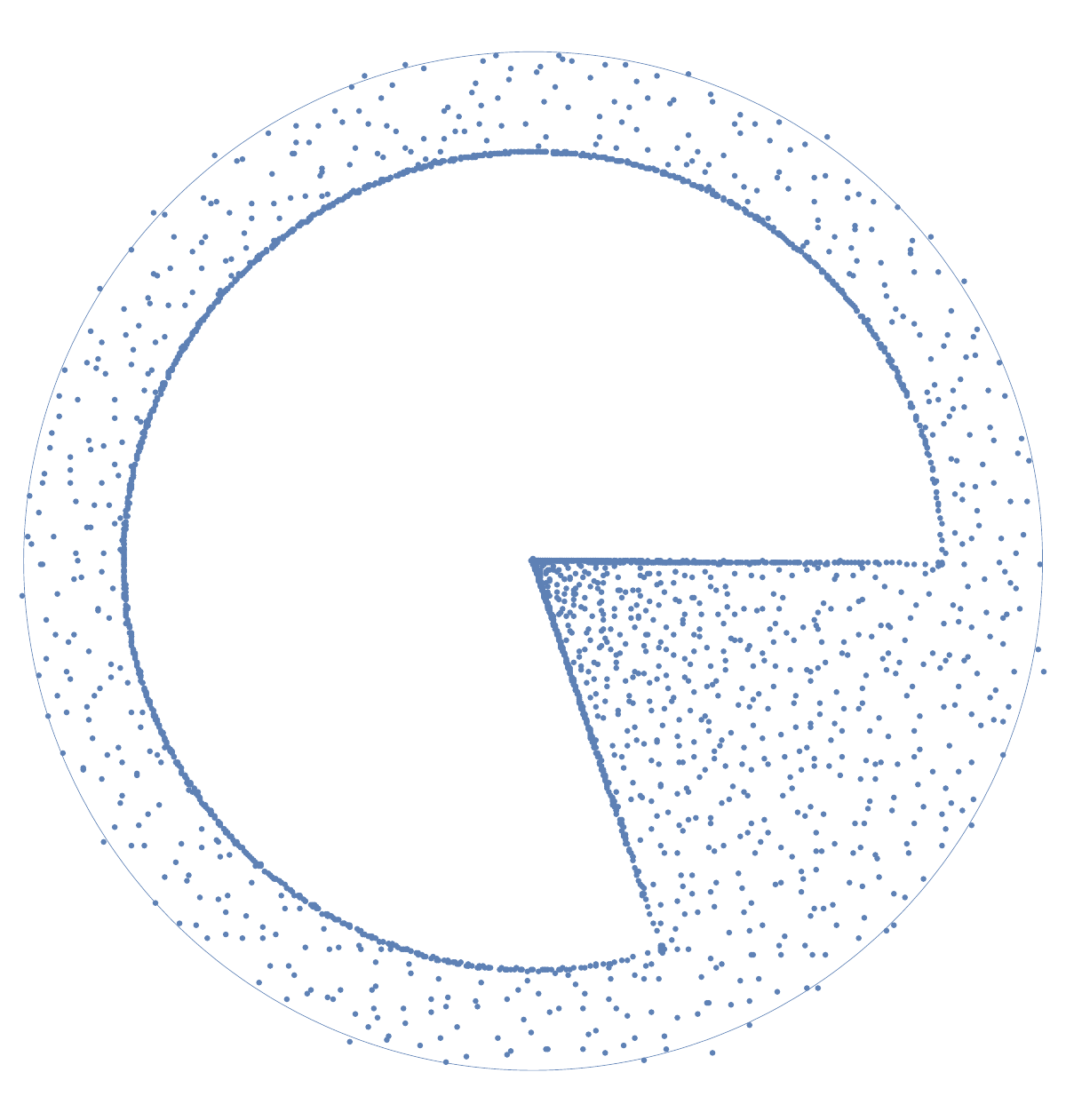}};
\end{tikzpicture} \hspace{-0.21cm}
\begin{tikzpicture}[slave]
\node at (0,0) {\includegraphics[width=3.63cm]{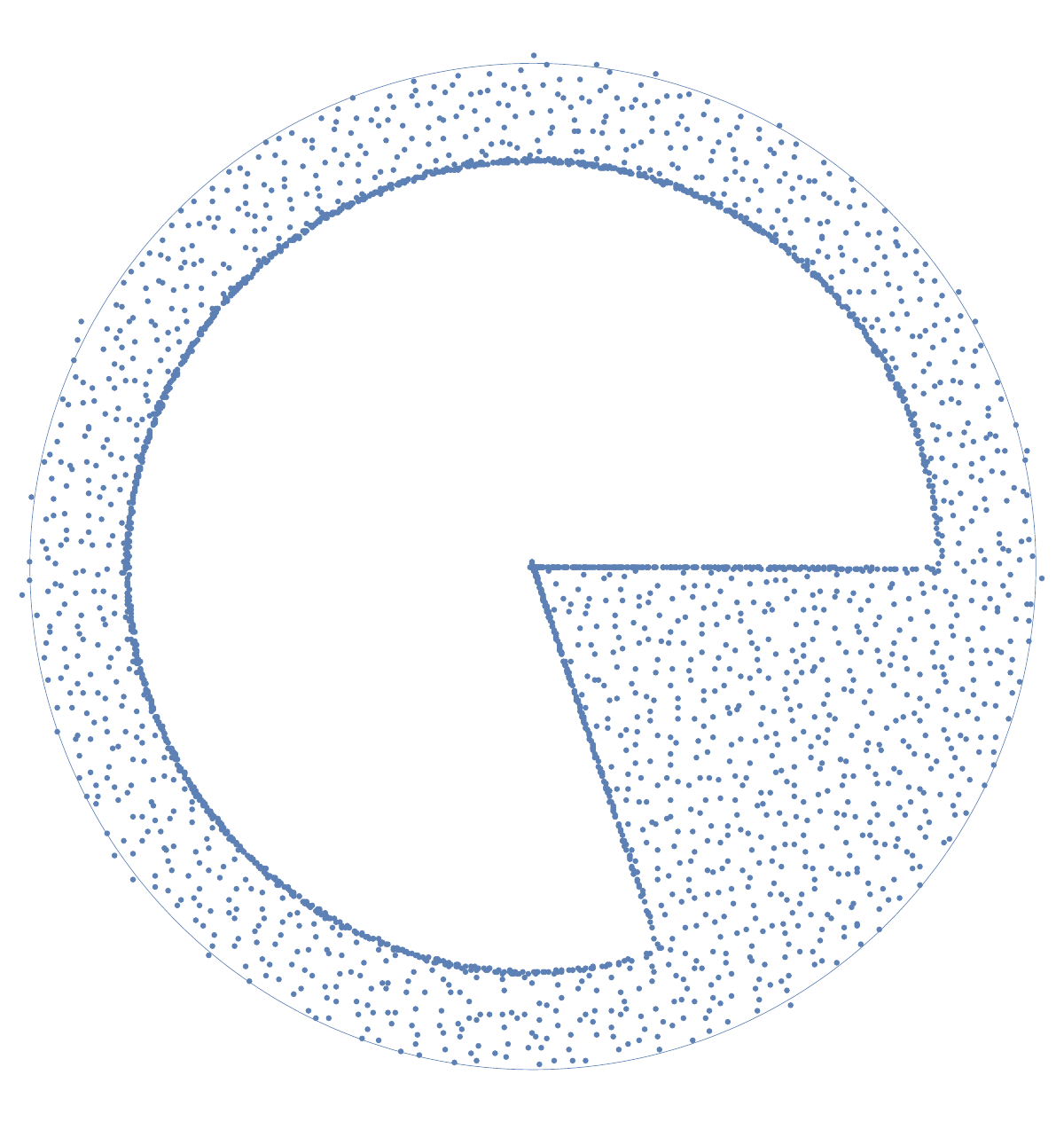}};
\node at (-1.9,1.6) {\footnotesize $p=\frac{5}{4}$};
\end{tikzpicture} \hspace{-0.21cm}
\begin{tikzpicture}[slave]
\node at (0,0) {\includegraphics[width=3.63cm]{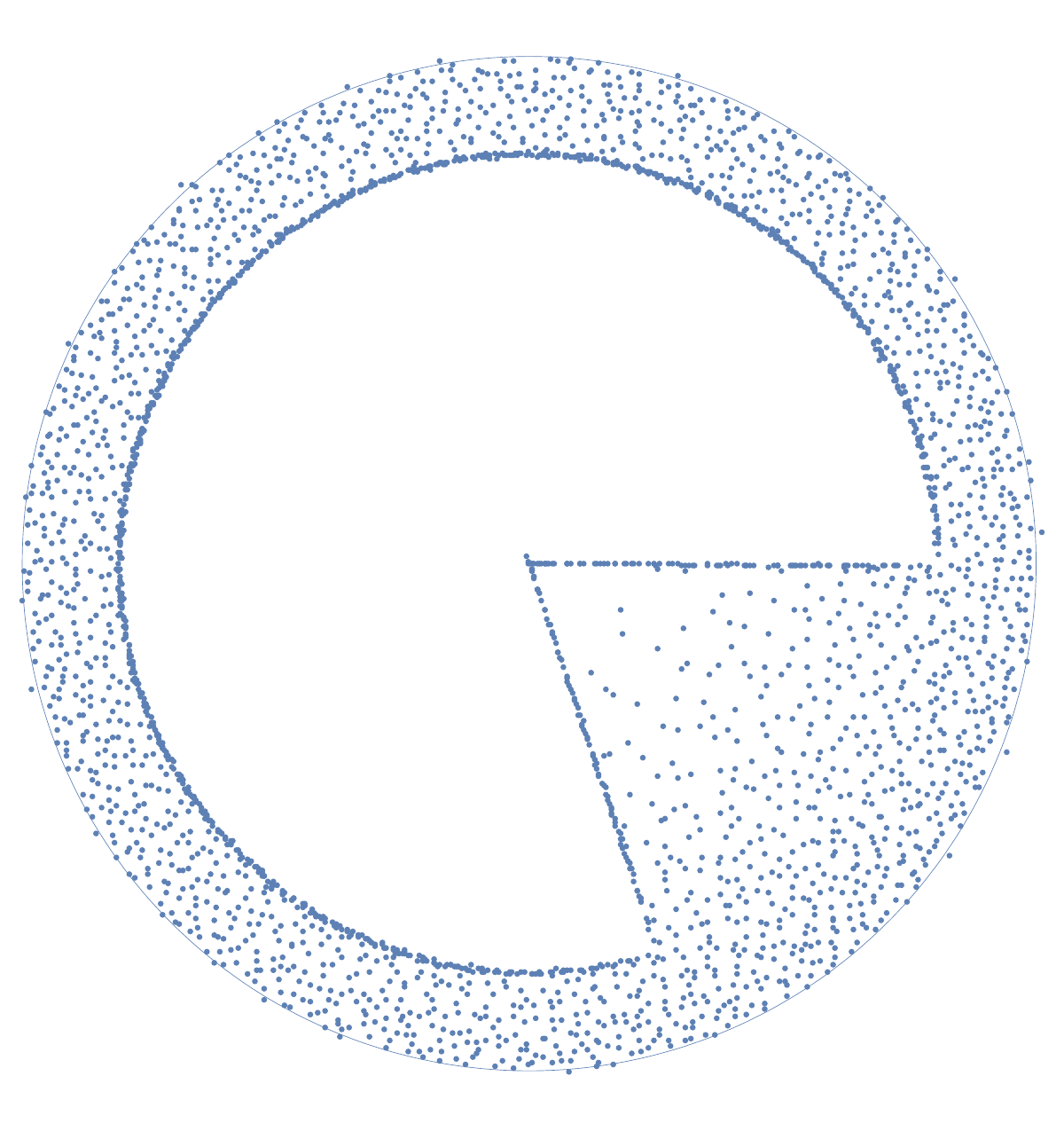}};
\end{tikzpicture}\\[-1cm]
\begin{center}
\begin{tikzpicture}[master]
\node at (0,0) {\includegraphics[width=14cm]{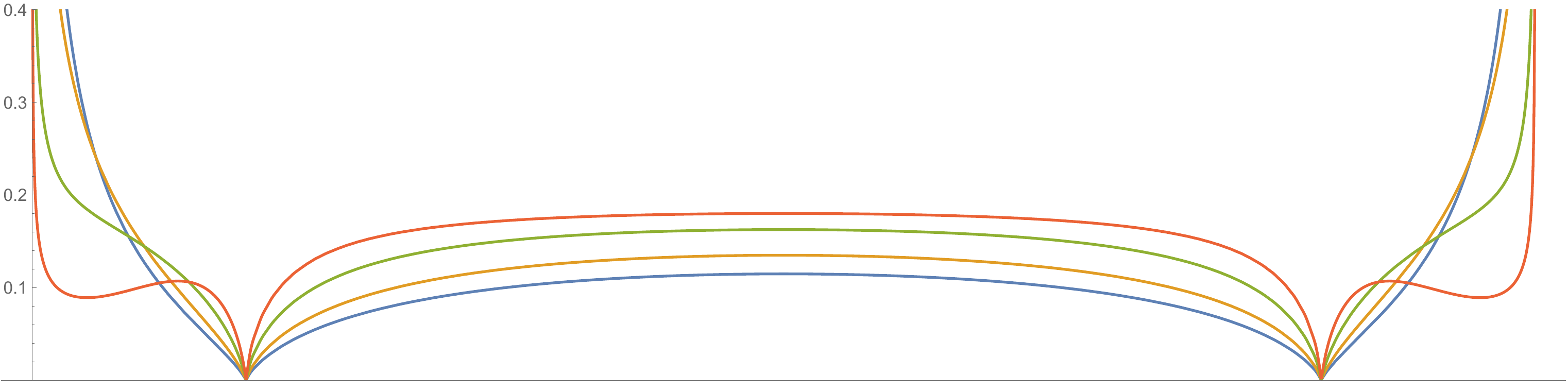}};
\node at (0,-1) {\footnotesize $b=\frac{1}{3}$};
\node at (-2,0.6) {\footnotesize $b=\frac{1}{2}$};
\draw[->-=1] (-2.05,0.5)--(-2.05,-0.6);
\node at (0,0.6) {\footnotesize $b=1$};
\draw[->-=1] (0,0.5)--(0,-0.3);
\node at (2,0.6) {\footnotesize $b=2$};
\draw[->-=1] (2,0.5)--(2,-0.2);
\node at (-5.3,1.52) {\footnotesize $\asymp x^{-\frac{1}{3}}$};
\draw[->-=1] (-5.8,1.45)--(-6.36,1.45);
\node at (-6.15,-1.3) {\footnotesize $\asymp x^{-\frac{1}{4}}$};
\draw[->-=1] (-6.5,-1.25)--(-6.5,-0.85);
\draw[->-=1] (-6.35,-1.25)--(-6.35,1);
\draw[->-=1] (-6.2,-1.25)--(-6.2,-0.18);
\draw[fill] (-6.71,-1.65) circle (0.04);
\node at (-6.71,-1.9) {\footnotesize $0$};
\draw[fill] (-4.8,-1.65) circle (0.04);
\node at (-4.8,-1.9) {\footnotesize $a$};
\draw[fill] (4.8,-1.65) circle (0.04);
\node at (4.8,-1.9) {\footnotesize $a e^{\frac{2\pi i}{p}}$};
\draw[fill] (6.71,-1.65) circle (0.04);
\node at (6.71,-1.9) {\footnotesize $0$};
\end{tikzpicture}
\end{center}
\vspace{-0.5cm}\caption{\label{fig:circular sector} Row 1: the Mittag-Leffler point process for the indicated values of $b$. In each plot, the thin blue circle is $\{z: |z|=b^{-\frac{1}{2b}}\}$. Row 2: the Mittag-Leffler ensemble conditioned on $\#\{z_{j}\in U\}=0$ with $U$ as in Theorem \ref{thm:g sector nu}, $p=5$, and $a=0.8b^{-\frac{1}{2b}}$. Row 3: the normalized density $\frac{d\nu(z)/|dz|}{\nu(\partial U)}$ for $p=5$ and $a=0.8b^{-\frac{1}{2b}}$. Rows 4 and 5: similar to rows 2 and 3, but with $p=\frac{5}{4}$ and $a=0.8b^{-\frac{1}{2b}}$.}
\end{figure}

In view of Remark \ref{remark:circular sector ML}, we are led to formulate the following conjecture.
\begin{conjecture}\label{conj:behavior of balayage measure}
Let $\mu$ be a probability measure and $S:=\mathrm{supp} \, \mu$. Suppose $S$ is compact, $U$ satisfies Assumption \ref{ass:U}, $\partial U$ is piecewise smooth, $U\subset S$, $\mu$ has a continuous density with respect to $d^{2}z$ on $U\setminus\{z_{0}\}$ for some $z_{0}\in \partial U$, and $d\mu(z)/d^{2}z \asymp |z-z_{0}|^{2b-2}$ as $z\to z_{0}$ for some $b>0$. At $z_{0}$, suppose that $\partial U$ makes an angle $\frac{2\pi}{p}$ pointing inwards $U$ for some $p \in [1,+\infty)$. If $2b \neq \frac{p}{2}$, then  $\nu := \mathrm{Bal}(\mu|_{U},\partial U)$ satisfies $d\nu(z)/|dz| \asymp |z-z_{0}|^{\min\{2b,\frac{p}{2}\}-1} $ as $z\to z_{0}$, where $|dz|$ is the arclength measure on $\partial U$. If $2b = \frac{p}{2}$, then instead we have $d\nu(z)/|dz| \asymp |z-z_{0}|^{2b-1}\log\frac{1}{|z-z_{0}|}$ as $z\to z_{0}$.
\end{conjecture}
\begin{remark}
\vspace{-0.05cm}This conjecture is verified in Theorem \ref{thm:g sector nu} for $g(r)=r^{2b}$ (see Remark \ref{remark:circular sector ML}), and also in all other cases considered below, such as in Theorem \ref{thm:Ginibre triangle nu} where $U$ is an equilateral triangle, $p\in \{2,6\}$ and $b=1$, and in Theorems \ref{thm:Ginibre rectangle nu} and \ref{thm:ML rectangle nu} where $U$ is a rectangle, $p\in \{2,4\}$ and $b \in \N_{>0}$. 
\end{remark}
\begin{remark}
For the spherical case $g(r) = \log(1+r^{2})$, \eqref{g sector nu} can be rewritten as
\vspace{-0.05cm}\begin{align}\label{nu circular S}
\frac{d\nu(z)}{d\theta} = \frac{1}{\pi^{2}} \sum_{m=0}^{+\infty} \frac{\mathrm{B}(\frac{a^{2}}{1+a^{2}}, 1+\frac{pm+p/2}{2},1-\frac{pm+p/2}{2})}{a^{pm+p/2}} \frac{\sin((pm+\frac{p}{2})\theta)}{m+\frac{1}{2}},
\end{align}
where $\mathrm{B}(z,\alpha,\beta) := \int_{0}^{z} x^{\alpha-1}(1-x)^{\beta-1}dx$ is the incomplete Beta function. The density $d\nu(r)/dr$ in \eqref{g sector nu segment 2} can similarly be expressed in terms of the generalized incomplete Beta function. 
\end{remark}
\begin{theorem}\label{thm:g sector C}
Let $\beta>0$. 

\medskip \noindent (i) Suppose $Q(z) = g(|z|)$ for some $g: [0,+\infty)\to \R \cup \{+\infty\}$ such that Assumption \ref{ass:Q} holds, and that $S = \{z : |z| \in \mathrm{S} \}$ with $\mathrm{S} = [r_{0},r_{1}]\cup [r_{2},r_{3}] \cup \ldots \cup [r_{2\ell},r_{2\ell+1}]$ for some $0 \leq r_{0} < r_{1} < \ldots <r_{2\ell+1}<+\infty$. Suppose also that $g \in C^{2}(\mathrm{S}\setminus \{0\})$ and recall that $\mu_{\mathrm{rad}}$ is given by \eqref{mu radially symmetric intro 2}.

\medskip \noindent Let $p\in [2,+\infty)$, $a\in (0,r_{1}]$ and $U=\{re^{i\theta} : 0<r<a, \, 0<\theta < \frac{2\pi}{p}\}$. As $n \to +\infty$, we have $\mathbb{P}(\# \{z_{j}\in U\} = 0) = \exp \big( -Cn^{2}+o(n^{2}) \big)$, where $\mathbb{P}$ refers to \eqref{general density intro} and
\begin{multline}
C = \frac{\beta}{4} \bigg( 2 \int_{0}^{a} g(r) d\nu(r) + g(a) \frac{2}{p\pi^{2}} \sum_{m=0}^{+\infty} \frac{1}{(m+\frac{1}{2})^{2}} \frac{\int_{0}^{a} r^{pm+\frac{p}{2}} d\mu_{\mathrm{rad}}(r)}{a^{pm+\frac{p}{2}}}   - \frac{1}{p} \int_{0}^{a} g(r) d\mu_{\mathrm{rad}}(r)\bigg), \label{C in thm cir sector}
\end{multline}
where $d\nu(r)$ is given for $r\in (0,a)$ by \eqref{g sector nu segment}. Moreover, if $g(r)$ can be written as a uniformly convergent series of the form $\sum_{k\in \mathcal{I}} g_{k}r^{k}$ for all $r \in [0,a]$, where $\mathcal{I}\subset [0,+\infty)$ is countable with no accumulation points and such that $\mathcal{I}\cap (\frac{p}{2}+p\N) = \emptyset$, then $C$ in \eqref{C in thm cir sector} can be further simplified as
\begin{align}
& C = \frac{\beta}{4}\bigg(  2\sum_{k\in \mathcal{I}} \frac{p g_{k}a^{k}}{\pi^{2}} \bigg[ \sum_{m=0}^{+\infty} \frac{1}{(k+pm+\frac{p}{2})(k-pm-\frac{p}{2})} \bigg( \frac{\int_{0}^{a} r^{pm+\frac{p}{2}} d\mu_{\mathrm{rad}}(r)}{a^{pm+\frac{p}{2}}} - \frac{\int_{0}^{a} r^{k} d\mu_{\mathrm{rad}}(r)}{a^{k}} \bigg) \bigg] \nonumber \\
& + g(a) \frac{2}{p \pi^{2}} \sum_{m=0}^{+\infty} \frac{1}{(m+\frac{1}{2})^{2}}  \frac{\int_{0}^{a} r^{pm+\frac{p}{2}} d\mu_{\mathrm{rad}}(r)}{a^{pm+\frac{p}{2}}} - \frac{1}{p} \int_{0}^{a} g(r) d\mu_{\mathrm{rad}}(r) \bigg).\label{g sector C}
\end{align}
If $\mathcal{I}\cap (\frac{p}{2}+p\N) \neq \emptyset$, then $C$ can be obtained by first replacing $p$ by $p'$ in the right-hand side of \eqref{g sector C}, and then taking $p'\to p$.

\medskip \noindent (ii) A similar statement holds for the spherical point process \eqref{spherical ensemble}. More precisely, let $p\in [2,\infty)$, $a\in (0,+\infty)$ and $U=\{re^{i\theta} : 0<r<a, \, 0<\theta < \frac{2\pi}{p}\}$. As $n \to +\infty$, we have $\mathbb{P}(\# \{z_{j}\in U\} = 0) = \exp \big( -Cn^{2}+o(n^{2}) \big)$, where $\mathbb{P}$ refers to \eqref{spherical ensemble} and $C$ is given by \eqref{C in thm cir sector} with $\mu$ and $S$ as in \eqref{def of mu and S spherical}, $\nu := \mathrm{Bal}(\mu|_{U},\partial U)$ is given by Theorem \ref{thm:g sector nu} (ii), $g(r) = \log(1+r^{2})$ and $d\mu_{\mathrm{rad}}(r) = \frac{2r}{(1+r^{2})^{2}}dr$. Moreover, if $a<1$ and $p$ is such that $(\frac{p}{2}+p\N)\cap 2\N = \emptyset$, then   $C$ can be rewritten as
\begin{align}\label{C circular S}
C & = \frac{\beta}{4}\bigg(  \sum_{k=1}^{+\infty} \frac{(-1)^{k+1}}{k}\frac{2p a^{2k}}{\pi^{2}} \bigg[ \sum_{m=0}^{+\infty} \frac{1}{(2k+pm+\frac{p}{2})(2k-pm-\frac{p}{2})} \nonumber \\
& \times \bigg( \frac{\mathrm{B}(\frac{a^{2}}{1+a^{2}}, 1+\frac{pm+p/2}{2},1-\frac{pm+p/2}{2})}{a^{pm+\frac{p}{2}}} - \frac{\mathrm{B}(\frac{a^{2}}{1+a^{2}}, 1+k,1-k)}{a^{2k}} \bigg) \bigg] \nonumber \\
& + \frac{2 \log(1+a^{2})}{p \pi^{2}} \sum_{m=0}^{+\infty} \frac{\mathrm{B}(\frac{a^{2}}{1+a^{2}}, 1+\frac{pm+p/2}{2},1-\frac{pm+p/2}{2})}{(m+\frac{1}{2})^{2}a^{pm+\frac{p}{2}}} - \frac{a^{2}-\log(1+a^{2})}{p(1+a^{2})}  \bigg).
\end{align}
If $(\frac{p}{2}+p\N)\cap 2\N \neq \emptyset$, then $C$ can be obtained by replacing $p$ by $p'$ in the right-hand side of \eqref{C circular S}, and then taking $p'\to p$. 

\medskip \noindent (iii) For $Q(z) = |z|^{2}$, statement (i) also holds if $p\in (1,2)$ and $a\neq r_{1} =1$.
\end{theorem}
\begin{remark}\label{remark:circular sector and ML}
For $g(r)=r^{2b}$ with $b >0$, $C$ in \eqref{g sector C} becomes
\begin{align}\label{C circular ML}
C = \beta  a^{4b} \frac{b}{p} \bigg( \frac{1}{4\pi^{2}}\sum_{m=0}^{+\infty} \frac{1}{(m+\frac{1}{2}+\frac{2b}{p})^{2}} + \frac{b}{p \pi^{2}} \sum_{m=0}^{+\infty} \frac{1}{(m+\frac{1}{2})^{2}(m+\frac{1}{2}+\frac{2b}{p})} - \frac{1}{8} \bigg).
\end{align}
This constant can also be rewritten in terms of the digamma function $\psi(z) := \frac{\Gamma'(z)}{\Gamma(z)}$, where $\Gamma(z) := \int_{0}^{+\infty}t^{z-1}e^{-t}dt$ is the Gamma function. Indeed, using the identities (see also \cite[5.2.2, 5.7.6, 5.15.1, 5.15.3, 25.6.1]{NIST})
\begin{align*}
& \psi(z) = -\gamma_{\mathrm{E}} + \sum_{m=0}^{+\infty} \frac{z-1}{(m+1)(m+z)}, \qquad \psi^{(1)}(z) = \sum_{m=0}^{+\infty} \frac{1}{(m+z)^{2}}, \qquad \mbox{for all } z>0, 
\end{align*}
and $\psi^{(1)}(\tfrac{1}{2}) = \frac{\pi^{2}}{2}$, where $\gamma_{\mathrm{E}}\approx 0.577$ is Euler's gamma constant, we get
\begin{align*}
C = \beta  a^{4b} \frac{\psi(\frac{1}{2})-\psi(\frac{1}{2}+\frac{2b}{p})+\frac{b}{p}\big(\psi^{(1)}(\frac{1}{2})+\psi^{(1)}(\frac{1}{2}+\frac{2b}{p})\big)}{4\pi^{2}}.
\end{align*}
Using well-known special values of $\psi$ and $\psi^{(1)}$ (see e.g. \cite[5.4.14, 5.4.15, 5.15.4]{NIST}), we get
\begin{align*}
& C|_{\frac{b}{p}=\frac{1}{4}} = \beta a^{4b} \bigg( \frac{1}{24} - \frac{\log 2}{2\pi^{2}} \bigg), \quad C|_{\frac{b}{p}=\frac{1}{2}} = \beta a^{4b} \bigg( \frac{1}{8} - \frac{1}{\pi^{2}} \bigg), \quad C|_{\frac{b}{p}=\frac{3}{4}} = \beta a^{4b} \bigg( \frac{1}{8} - \frac{\log 2}{2\pi^{2}} - \frac{7}{16\pi^{2}} \bigg), \\
& C|_{\frac{b}{p}=1} = \beta a^{4b} \bigg( \frac{1}{4} - \frac{16}{9\pi^{2}} \bigg), \quad C|_{\frac{b}{p}=\frac{3}{2}} = \beta a^{4b} \bigg( \frac{3}{8} - \frac{187}{75\pi^{2}} \bigg), \quad C|_{\frac{b}{p}=2} = \beta a^{4b} \bigg( \frac{1}{2} - \frac{35072}{11025\pi^{2}} \bigg).
\end{align*}
In particular, $C|_{b=1,p=2,\beta=2}=a^{4}(\frac{1}{4}-\frac{2}{\pi^{2}})$, which was previously known from \cite{AR2017}.
\end{remark}

\subsection{Results for the elliptic Ginibre point process}\label{subsection:elliptic Ginibre intro}
In this subsection, we consider the elliptic Ginibre point process, which we recall is defined by \eqref{general density intro} with $Q(z) = \frac{1}{1-\tau^{2}}\big( |z|^{2}-\tau \, \re z^{2} \big)$ and $\tau \in [0,1)$. Let $\mu$ be the equilibrium measure for $Q$ on $\C$, and $S:=\mathrm{supp} \, \mu$. Recall also from \eqref{mu S EG beginning of intro} that
\begin{align}\label{mu S EG}
d\mu(z) = \frac{d^{2}z}{\pi(1-\tau^{2})}, \qquad S=\Big\{ z\in \C: \Big( \frac{\re z}{1+\tau} \Big)^{2} + \Big( \frac{\im z}{1-\tau} \Big)^{2} \leq 1 \Big\}.
\end{align}
We refer to \cite{BFreview} for a survey of recent results on the elliptic Ginibre point process. Below, we state explicit results on the balayage measure $\nu := \mathrm{Bal}(\mu|_{U},\partial U)$ and the hole probability $\mathbb{P}(\# \{z_{j}\in U\} = 0)$, in the cases where $U$ is either an ellipse, an annulus, a cardioid, a circular sector, an equilateral triangle, a rectangle, the complement of an ellipse centered at $0$, and the complement of a disk. 

\begin{figure}[h]
\begin{tikzpicture}[master]
\node at (0,0) {\includegraphics[width=3.63cm]{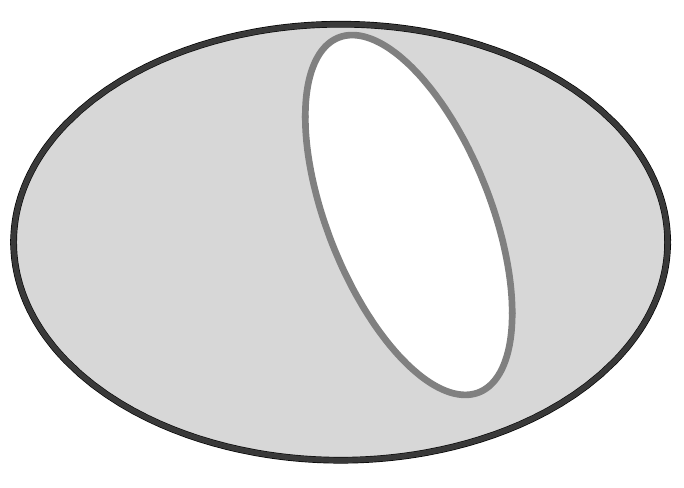}};
\node at (0,1.45) {\footnotesize The ellipse};
\end{tikzpicture} \hspace{-0.21cm}
\begin{tikzpicture}[slave]
\node at (0,0) {\includegraphics[width=3.63cm]{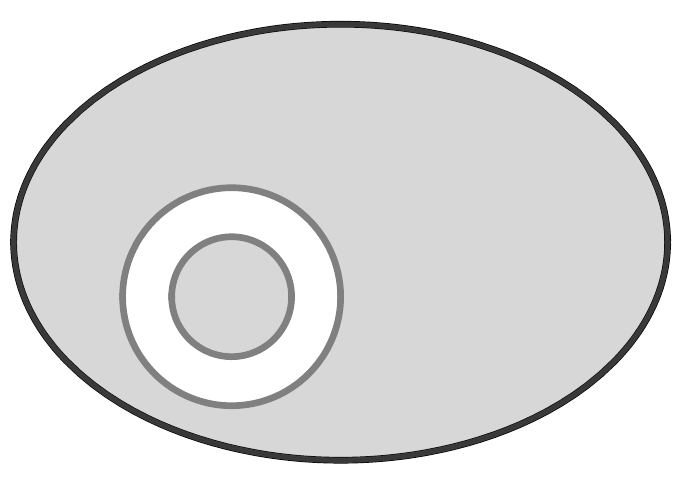}};
\node at (0,1.45) {\footnotesize The annulus};
\end{tikzpicture} \hspace{-0.21cm}
\begin{tikzpicture}[slave]
\node at (0,0) {\includegraphics[width=3.63cm]{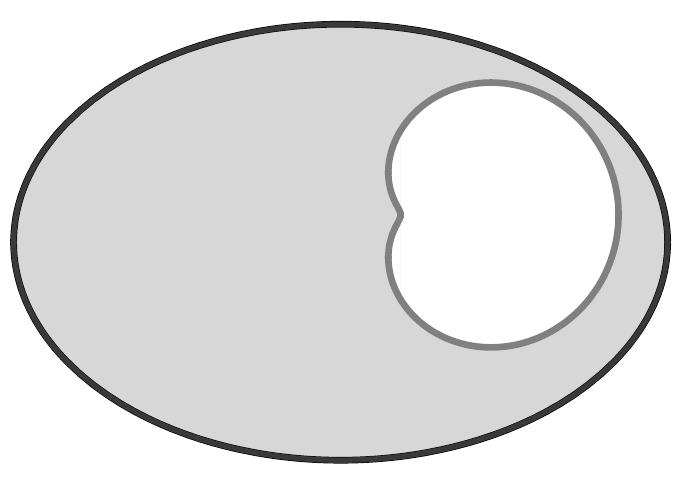}};
\node at (-2,-1.4) {\footnotesize Theorem \ref{thm:Elliptic disk C}};
\node at (0,1.45) {\footnotesize The cardioid};
\end{tikzpicture} \hspace{-0.21cm}
\begin{tikzpicture}[slave]
\node at (0,0) {\includegraphics[width=3.63cm]{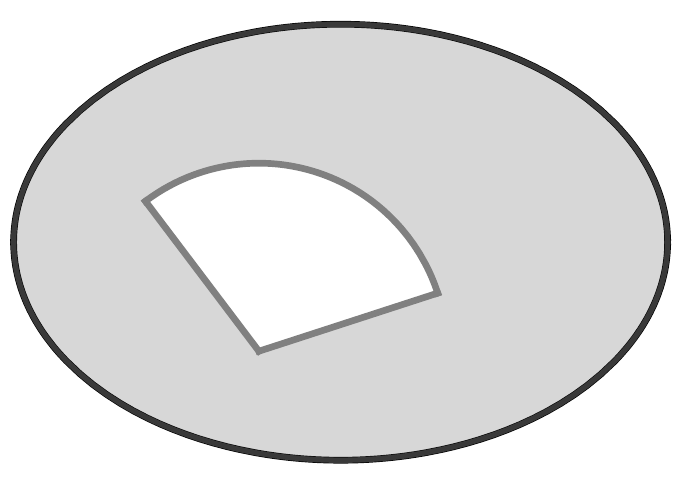}};
\node at (0,1.45) {\footnotesize The circular sector};
\end{tikzpicture}
\caption{\label{fig:hole prob for EG with disk and ellipse} $S\setminus U$ is the shaded region, $U$ is in white, $\partial S$ is in black, and $\partial U$ is in gray.}
\end{figure}
The next result is general and relates the leading order terms in the large $n$ asymptotics of $\mathbb{P}(\# \{z_{j}\in \zeta_{0}+\rho e^{i\theta_{0}}U\} = 0)$ and $\mathbb{P}(\# \{z_{j}\in U\} = 0)|_{\tau=0}$, where $\zeta_{0}\in \C$, $\rho>0$ and $\theta_{0}\in [0,2\pi)$, provided that $U\subset S$ and $\zeta_{0}+\rho e^{i\theta_{0}}U \subset S$.
\begin{theorem}\label{thm:general U EG} 
Let $\tau \in [0,1)$, $\zeta_{0}\in \C$, $\rho>0$, $\theta_{0}\in [0,2\pi)$ and $Q(z) = \frac{1}{1-\tau^{2}}\big( |z|^{2}-\tau \, \re z^{2} \big)$. Suppose $U\subset \C$ is bounded and such that Assumption \ref{ass:U} holds (with $S$ replaced by $\C$), and define
\begin{align}\label{C in main thm EG}
C(\tau,\zeta_{0},\rho,\theta_{0}) = \frac{\beta}{4}\bigg(  \int_{\zeta_{0}+\rho e^{i\theta_{0}}\partial U}Q(z)d\nu(z) - \int_{\zeta_{0}+\rho e^{i\theta_{0}}U}Q(z) \frac{d^{2}z}{\pi(1-\tau^{2})} \bigg),
\end{align}
where $\nu := \mathrm{Bal}(\frac{d^{2}z}{\pi(1-\tau^{2})}|_{\zeta_{0}+\rho e^{i\theta_{0}}U},\zeta_{0}+\rho e^{i\theta_{0}}\partial U)$. Then
\begin{align*}
C(\tau,\zeta_{0},\rho,\theta_{0}) = \frac{\rho^{4}}{(1-\tau^{2})^{2}}C(0,0,1,0).
\end{align*}
\end{theorem}

\subsubsection{The ellipse, the annulus, the cardioid, and the circular sector}
By combining Theorems \ref{thm:general U EG} and \ref{thm:general pot} with several findings from \cite{AR2017}, we readily obtain the following.

\begin{theorem}\label{thm:Elliptic disk C}
Fix $\beta>0$, $\tau \in [0,1)$, $\zeta_{0}\in \C$, $\theta_{0}\in [0,2\pi)$.

\medskip \noindent (The ellipse.) Let $a,c>0$ be such that $\zeta_{0}+e^{i\theta_{0}}U \subset S$, where $U=\{z:(\tfrac{\re z}{a})^{2}+(\tfrac{\im z}{c})^{2}<1 \}$. As $n \to +\infty$, we have $\mathbb{P}(\# \{z_{j}\in \zeta_{0}+e^{i\theta_{0}}U\} = 0) = \exp \big( -C n^{2}+o(n^{2}) \big)$, where
\begin{align}\label{C ellipse EG}
& C = \frac{1}{(1-\tau^{2})^{2}} \frac{\beta}{4}  \frac{a^{3}c^{3}}{a^{2}+c^{2}}.
\end{align}

\medskip \noindent (The annulus.) Let $0<\rho_{1}<\rho_{2}$ be such that $\zeta_{0}+U \subset S$, where $U=\{z: \rho_{1}<|z|<\rho_{2} \}$. As $n \to +\infty$, we have $\mathbb{P}(\# \{z_{j}\in \zeta_{0}+U\} = 0) = \exp \big( -C n^{2}+o(n^{2}) \big)$, where
\begin{align}\label{C annulus EG}
& C = \frac{1}{(1-\tau^{2})^{2}} \frac{\beta}{4} \bigg( \frac{\rho_{2}^{4}-\rho_{1}^{4}}{2} - \frac{(\rho_{2}^{2}-\rho_{1}^{2})^{2}}{2\log(\rho_{2}/\rho_{1})} \bigg).
\end{align}

\medskip \noindent (The cardioid.) Let $a>0$, $c\in [0,\frac{1}{2})$ be such that $\zeta_{0}+e^{i\theta_{0}}U \subset S$, where $U=\{re^{i\theta}: r<a(1+2c\cos\theta), \, \theta\in [0,2\pi) \}$. As $n \to +\infty$, $\mathbb{P}(\# \{z_{j}\in \zeta_{0}+e^{i\theta_{0}}U\} = 0) = \exp \big( -C n^{2}+o(n^{2}) \big)$, where
\begin{align}\label{C cardioid EG}
& C = \frac{1}{(1-\tau^{2})^{2}}\frac{\beta}{4}a^{4} \bigg( (c^{2}+1)^{2} - \frac{1}{2} \bigg).
\end{align}

\medskip \noindent (The circular sector.) Let $a>0$, $p\in [2,+\infty)$ be such that $\zeta_{0}+e^{i\theta_{0}}U \subset S$, where $U=\{re^{i\theta} : r\in (0,a), \, \theta \in (0,\tfrac{2\pi}{p})\}$. As $n \to +\infty$, we have $\mathbb{P}(\# \{z_{j}\in \zeta_{0}+e^{i\theta_{0}}U\} = 0) = \exp \big( -C n^{2}+o(n^{2}) \big)$, where
\begin{align}\label{C circular sector EG}
& C = \frac{1}{(1-\tau^{2})^{2}} \frac{\beta}{4}  a^{4} \frac{\psi(\frac{1}{2})-\psi(\frac{1}{2}+\frac{2}{p})+\frac{1}{p}\big(\psi^{(1)}(\frac{1}{2})+\psi^{(1)}(\frac{1}{2}+\frac{2}{p})\big)}{\pi^{2}}.
\end{align}
See also Figure \ref{fig:hole prob for EG with disk and ellipse}.
\end{theorem}
\begin{remark}\label{remark:tau=theta0:z0=0}
For $\tau=\theta_{0}=\zeta_{0}=0$ and $\beta=2=p$, the constants $C$ in \eqref{C ellipse EG}, \eqref{C annulus EG}, \eqref{C cardioid EG} and \eqref{C circular sector EG} were previously known from \cite{AR2017}. For the results \eqref{C cardioid EG} and \eqref{C circular sector EG}, we excluded $c=\frac{1}{2}$ and $p\in (1,2)$, since in these cases $U$ does not meet Assumption \ref{ass:U2}.
\end{remark}
\subsubsection{The triangle}

\begin{theorem}\label{thm:Ginibre triangle nu}
Let $\tau \in [0,1)$, $\theta_{0}\in [0,2\pi)$, $a>0$ and $\zeta_{0}\in \C$ be such that $U:=\zeta_{0}+e^{i\theta_{0}}aP \subset S$, where $P$ is the open equilateral triangle with vertices at $e^{-\frac{\pi i}{3}},e^{\frac{\pi i}{3}},-1$, and let $\nu = \mathrm{Bal}(\mu|_{U},\partial U)$. Then  
\begin{align}\label{Ginibre triangle nu}
\frac{d\nu(z)}{dy} = \frac{3a^{2}-4y^{2}}{8\pi a (1-\tau^{2})}, \qquad z=\zeta_{0}+e^{i\theta_{0}}(\tfrac{a}{2}+iy), \; y\in (-\tfrac{\sqrt{3}}{2}a, \tfrac{\sqrt{3}}{2}a),
\end{align}
and $d\nu(\zeta_{0}+(z-\zeta_{0})e^{\frac{2\pi i}{3}j})=d\nu(z)$ for all $j\in \{0,1,2\}$ and $z\in \zeta_{0}+e^{i\theta_{0}}\big(\frac{a}{2}+i(-\frac{\sqrt{3}}{2}a, \frac{\sqrt{3}}{2}a)\big)$ (see also Figure \ref{fig:density triangle Ginibre}). 
\end{theorem}
\begin{remark}
The density \eqref{Ginibre triangle nu} vanishes linearly as $y$ approaches $\pm \tfrac{\sqrt{3}}{2}a$. This is consistent with Conjecture \ref{conj:behavior of balayage measure} (with $p=6$ and $b=1$). 
\end{remark}
\begin{remark}\label{remark:triangle is polynomial}
The density of $\nu$ in \eqref{Ginibre triangle nu} is a (non-constant) polynomial with respect to the arclength measure on $\partial U$. This seems to be particular to this setting, i.e. that $U$ is an equilateral triangle and that $\mu$ has uniform density. Consider for example the situation where $U$ is a regular polygon centered at $0$ with $p$ vertices ($p\geq 3$) and $\mu|_{U} = |z|^{2b-2}d^{2}z$, $b \in \N_{>0}$. Conjecture \ref{conj:behavior of balayage measure}, if true, implies that $\mathrm{Bal}(\mu|_{U},\partial U)$ does not have a polynomial density with respect to the arclength measure on $\partial U$ for $p\geq 4$. Moreover, the analysis of Section \ref{section:triangle} suggests that $\mathrm{Bal}(\mu|_{U},\partial U)$ has a polynomial density  only if $p=3$ (the equilateral triangle) and $b=1$ (the uniform measure).
\end{remark}

\begin{figure}
\begin{center}
\hspace{-0.3cm}\begin{tikzpicture}[master]
\node at (0,0) {\includegraphics[width=4.8cm]{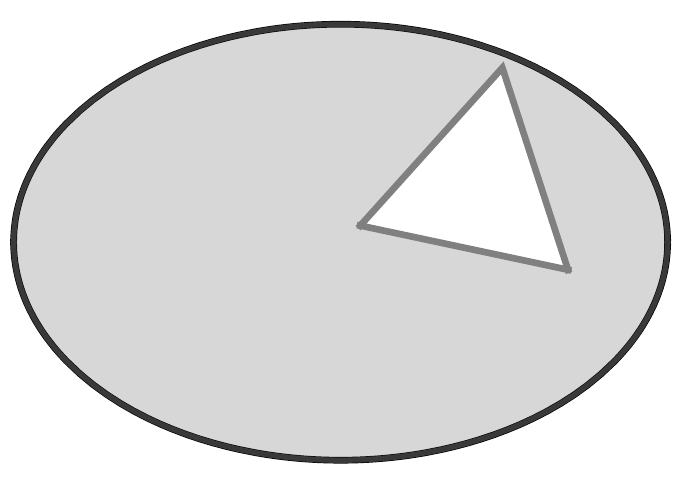}};
\draw[fill] (0.14,0.135) circle (0.04);
\node at (0.14,-0.05) {\footnotesize $v_{1}$};
\draw[fill] (1.58,-0.17) circle (0.04);
\node at (1.58,-0.35) {\footnotesize $v_{2}$};
\node at (0,-2) {\footnotesize Theorems \ref{thm:Ginibre triangle nu}, \ref{thm:Ginibre triangle C}};
\end{tikzpicture}  \hspace{-0.3cm}
\begin{tikzpicture}[slave]
\node at (0,0) {\includegraphics[width=4.8cm]{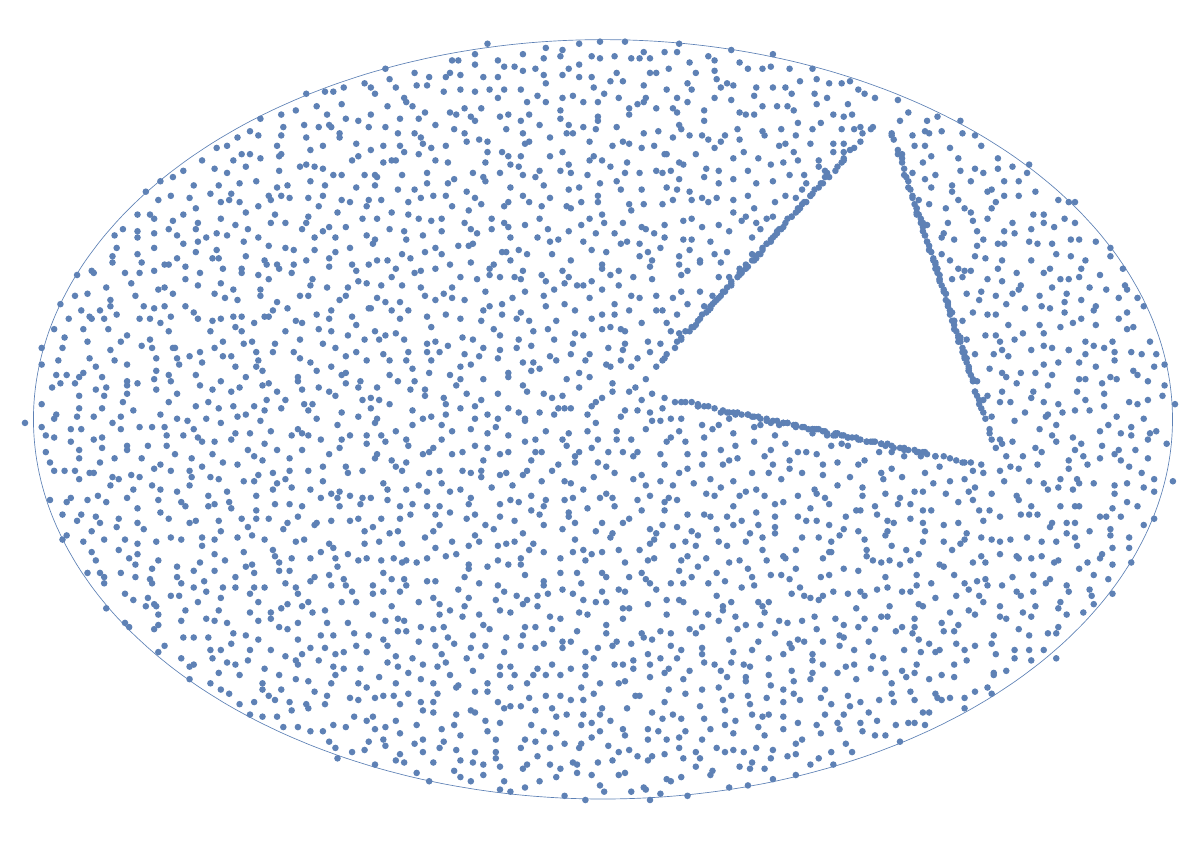}};
\end{tikzpicture} \hspace{-0.3cm}
\begin{tikzpicture}[slave]
\node at (0,0) {\includegraphics[width=4.8cm]{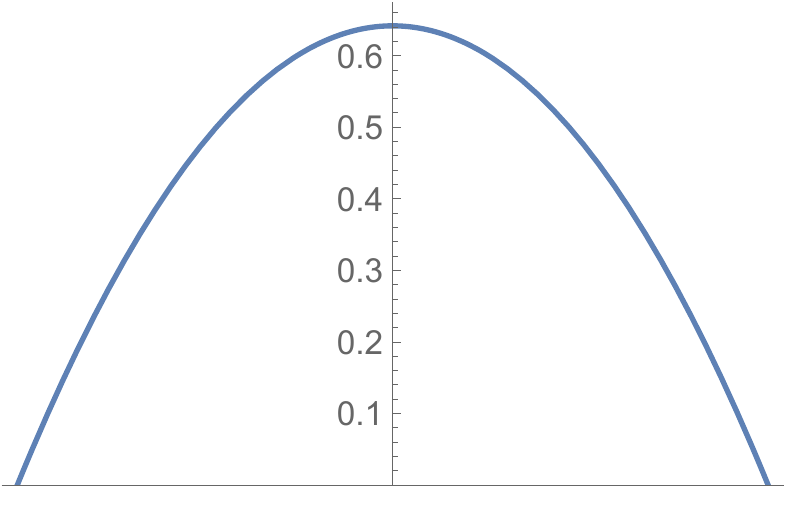}};

\node at (-2.25,-1.67) {\footnotesize $v_{1}$};
\draw[fill] (-2.29,-1.42) circle (0.04);
\node at (2.3,-1.67) {\footnotesize $v_{2}$};
\draw[fill] (2.29,-1.42) circle (0.04);
\node at (0,-1.9) {\footnotesize Theorem \ref{thm:Ginibre triangle nu}};
\end{tikzpicture}
\end{center}
\caption{\label{fig:density triangle Ginibre} In all figures, $\tau=0.2$ and $U$ is as in Theorem \ref{thm:Ginibre triangle nu} with $\zeta_{0}=0.5+i0.2$, $a=0.45$ and $\theta_{0} = \frac{\pi}{10}$. Middle: the elliptic Ginibre point process conditioned on $\# \{z_{j}\in U\} = 0$. Right: the normalized density $z \mapsto \frac{d\nu(z)/|dz|}{\nu(\partial U)}$ with $z\in (v_{1},v_{2})\subset \partial U$.}
\end{figure}

\begin{theorem}\label{thm:Ginibre triangle C}
Fix $\beta >0$, and let $\tau \in [0,1)$, $\theta_{0}\in [0,2\pi)$, $a >0$ and $\zeta_{0}\in \C$ be such that $U=\zeta_{0}+e^{i\theta_{0}}aP \subset S$, where $P$ is the open equilateral triangle with vertices at $e^{-\frac{\pi i}{3}},e^{\frac{\pi i}{3}},-1$. As $n \to +\infty$, we have $\mathbb{P}(\# \{z_{j}\in U\} = 0) = \exp \big( -C n^{2}+o(n^{2}) \big)$, where
\begin{align}\label{C triangle EG}
& C = \frac{9\sqrt{3} \beta a^{4}}{320\pi (1-\tau^{2})^{2}}.
\end{align}
\end{theorem}
\begin{remark}
For $\tau=0$, $\beta=2$, $\zeta_{0}=0$ and $\theta_{0} = \frac{\pi}{3}$, Theorem \ref{thm:Ginibre triangle C} was previously known from \cite{AR2017}.
\end{remark}

\begin{figure}[h]
\begin{center}
\hspace{-0.7cm}\begin{tikzpicture}[master]
\node at (0,0) {\includegraphics[width=4.5cm]{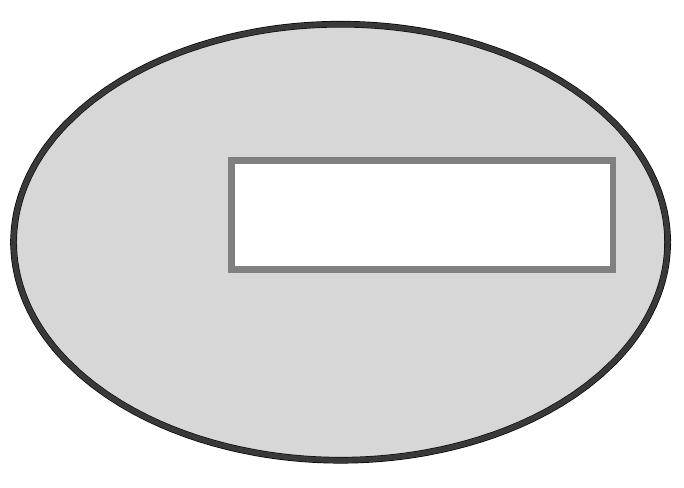}};
\draw[fill] (-0.725,0.55) circle (0.04);
\node at (-0.725,0.72) {\footnotesize $a_{1}\hspace{-0.06cm}+\hspace{-0.05cm}ic_{2}$};
\draw[fill] (1.79,-0.18) circle (0.04);
\node at (1.5,-0.4) {\footnotesize $a_{2}\hspace{-0.06cm}+\hspace{-0.05cm}ic_{1}$};
\node at (0,-2) {\footnotesize Theorems \ref{thm:Ginibre rectangle nu}, \ref{thm:Ginibre rectangle C}};
\end{tikzpicture}  \hspace{-0.35cm}
\begin{tikzpicture}[slave]
\node at (0,0) {\includegraphics[width=4.5cm]{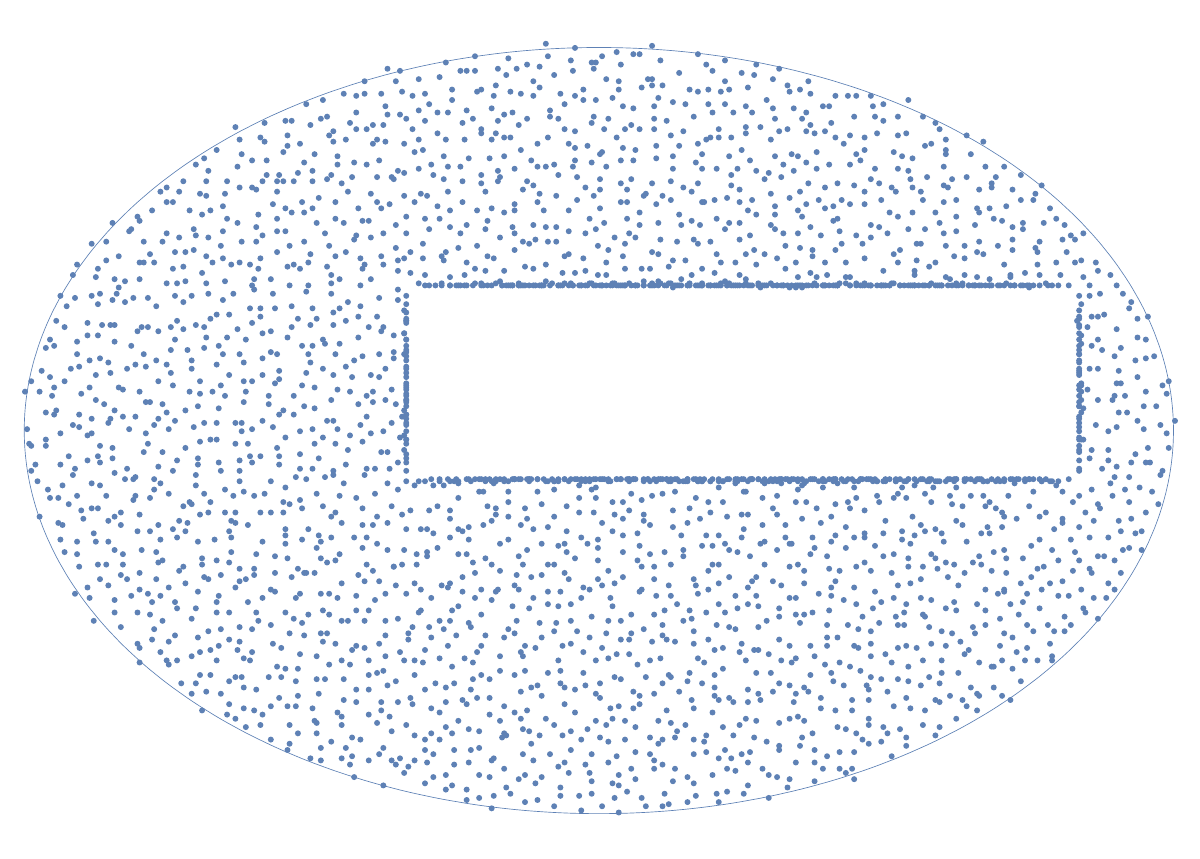}};
\end{tikzpicture} \hspace{0.25cm}
\begin{tikzpicture}[slave]
\node at (0,0) {\includegraphics[width=5.4cm]{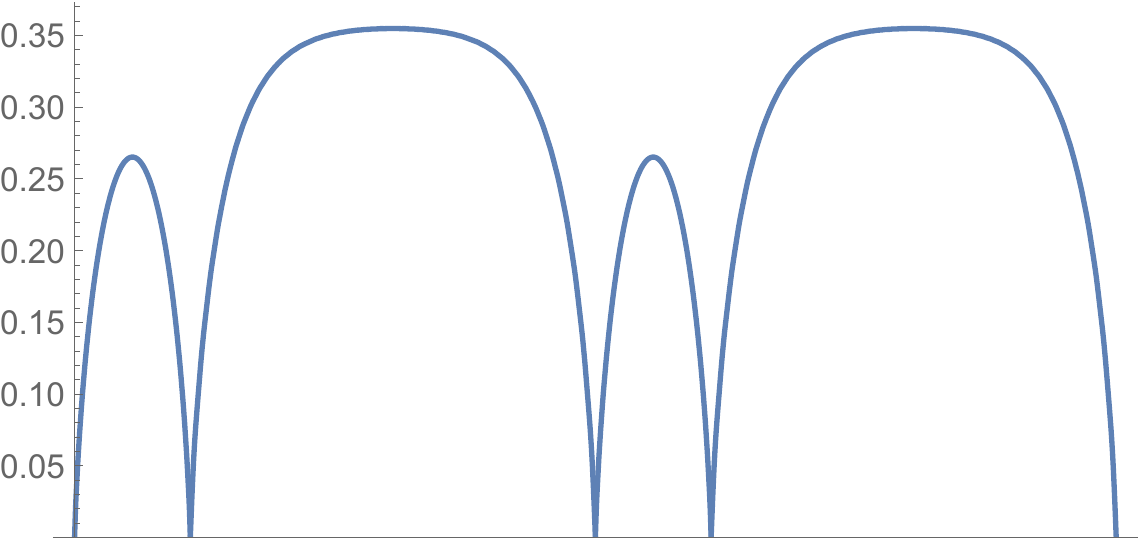}};

\node at (0,-1.9) {\footnotesize Theorem \ref{thm:Ginibre rectangle nu}};
\draw[fill] (-2.35,-1.228) circle (0.04);
\node at (-2.35,-1.45) {\footnotesize $a_{2}\hspace{-0.06cm}+\hspace{-0.05cm}ic_{1}$};


\draw[fill] (0.12,-1.228) circle (0.04);
\node at (0.12,-1.45) {\footnotesize $a_{1}\hspace{-0.06cm}+\hspace{-0.05cm}ic_{2}$};


\draw[fill] (2.57,-1.228) circle (0.04);
\node at (2.57,-1.45) {\footnotesize $a_{2}\hspace{-0.06cm}+\hspace{-0.05cm}ic_{1}$};
\end{tikzpicture}
\end{center}
\caption{\label{fig:density rectangle Ginibre} In all figures, $\tau=0.2$ and $U=\{z:\re z \in (a_{1},a_{2}), \im z \in (c_{1}, c_{2})\}$ with $a_{1}=-0.4$, $a_{2}=1$, $c_{1}=-0.1$, $c_{2}=0.3$. Middle: the elliptic Ginibre point process conditioned on $\# \{z_{j}\in U\} = 0$. Right: the normalized density $z \mapsto \frac{d\nu(z)/|dz|}{\nu(\partial U)}$ with $z \in \partial U$ going in the counterclockwise direction.}
\end{figure}

\subsubsection{The rectangle}

\begin{theorem}\label{thm:Ginibre rectangle nu}
\noindent \smallskip Let $\tau\in [0,1)$, $a_{2}>a_{1}$, $c_{2}>c_{1}$ be such that $U:=\{z: \re z \in (a_{1},a_{2}), \im z \in (c_{1}, c_{2})\} \subset S$, and let $\nu = \mathrm{Bal}(\mu|_{U},\partial U)$. 

\medskip \noindent (Right and left sides.) For $z=a_{2}+iy$, $y\in (c_{1},c_{2})$, and also for $z=a_{1}+iy$, $y\in (c_{1},c_{2})$, we have
\begin{align}\label{square right left b1}
\frac{d\nu(z)}{dy}= \frac{1}{1-\tau^{2}} \sum_{m=0}^{+\infty} \frac{4(c_{2}-c_{1})}{\pi^{3}} \frac{\tanh(\frac{a_{2}-a_{1}}{c_{2}-c_{1}}(1+2m)\frac{\pi}{2})}{(1+2m)^{2}} \sin \bigg( \frac{y-c_{1}}{c_{2}-c_{1}} (1+2m) \pi \bigg) .
\end{align}
(Top and bottom sides.) For $z=x+ic_{2}$, $x\in (a_{1},a_{2})$, and also for $z=x+ic_{1}$, $x\in (a_{1},a_{2})$, we have
\begin{align}\label{square top bottom b1}
\frac{d\nu(z)}{dx}= \frac{1}{1-\tau^{2}}\sum_{m=0}^{+\infty} \frac{4(a_{2}-a_{1})}{\pi^{3}} \frac{\tanh(\frac{c_{2}-c_{1}}{a_{2}-a_{1}}(1+2m)\frac{\pi}{2})}{(1+2m)^{2}} \sin \bigg( \frac{x-a_{1}}{a_{2}-a_{1}} (1+2m) \pi \bigg).
\end{align}
See also Figure \ref{fig:density rectangle Ginibre}.
\end{theorem}
\begin{remark}\label{remark:vanishing of balayage near a corner}
In Subsection \ref{subsection: behavior near a corner}, we show that the density \eqref{square right left b1} satisfies $\frac{d\nu(z)}{dy} = \frac{2}{\pi^{2}}\frac{y-c_{1}}{1-\tau^{2}}\log\frac{1}{y-c_{1}} + \bigO(y-c_{1})$ as $y \to c_{1}$, $y>c_{1}$. This is consistent with Conjecture \ref{conj:behavior of balayage measure} (with $p=4$ and $b=1$).
\end{remark}
\begin{remark}\label{remark: elliptic for square}
In Subsection \ref{subsection:balayage square conformal}, we provide another (less explicit) expression for $\nu$ in terms of elliptic functions, in the special case where $a_{2}=c_{2}=-a_{1}=-c_{1}=\frac{c}{2}$ for some $c>0$.
\end{remark}

\begin{theorem}\label{thm:Ginibre rectangle C}
Fix $\beta >0$, and let $\tau\in [0,1)$, $a_{2}>a_{1}$, $c_{2}>c_{1}$ be such that $U:=\{z: \re z \in (a_{1},a_{2}),  \im z \in (c_{1}, c_{2})\} \subset S$.  As $n \to +\infty$, we have $\mathbb{P}(\# \{z_{j}\in U\} = 0) = \exp \big( -C n^{2}+o(n^{2}) \big)$, where
\begin{align*}
& C = \frac{\beta(a_{2}-a_{1})^{2}(c_{2}-c_{1})^{2}}{4\pi (1-\tau^{2})^{2}} \bigg( \frac{\alpha + \alpha^{-1}}{6} - \frac{32}{\pi^{5}} \sum_{m=0}^{+\infty} \frac{\frac{1}{\alpha^{2}} \tanh(\alpha (1+2m)\frac{\pi}{2}) + \alpha^{2} \tanh(\frac{1}{\alpha}(1+2m)\frac{\pi}{2})}{(1+2m)^{5}} \bigg)
\end{align*}
with $\alpha := \frac{a_{2}-a_{1}}{c_{2}-c_{1}}$.
\end{theorem}
\begin{remark}
For the square of side length $c>0$, i.e. $a_{2}-a_{1}=c_{2}-c_{1}=c$, the constant $C$ of Theorem \ref{thm:Ginibre rectangle C} further simplifies as
\begin{align}\label{C for EG square}
& C = \frac{1}{(1-\tau^{2})^{2}} \frac{\beta c^{4}}{\pi}\bigg( \frac{1}{12} -  \frac{16}{\pi^{5}} \sum_{m=0}^{+\infty} \frac{\tanh(\frac{(1+2m)\pi}{2})}{(1+2m)^{5}} \bigg) \approx (1.1187 \cdot 10^{-2}) \frac{\beta c^{4}}{(1-\tau^{2})^{2}}.
\end{align}
\end{remark}

Define $T_{v,\alpha}$ and $T_{v}$ for $v\in [0,+\infty)$ and $\alpha \in (0,+\infty)$ by
\begin{align}
& T_{v,\alpha} := \frac{\alpha^{-1-\frac{v}{2}}}{2} \sum_{m=0}^{+\infty} \frac{\tanh(\alpha (1+2m)\frac{\pi}{2})}{(1+2m)^{3+v}} + e^{\pi i \frac{v}{2}} \frac{\alpha^{1+\frac{v}{2}}}{2} \sum_{m=0}^{+\infty} \frac{\tanh(\alpha^{-1} (1+2m)\frac{\pi}{2})}{(1+2m)^{3+v}}, \label{def of Tva} \\
& T_{v} := \sum_{m=0}^{+\infty} \frac{\tanh(\frac{(1+2m)\pi}{2})}{(1+2m)^{3+v}}. \label{def of Tv}
\end{align}
If $v\in 4\N$, then $T_{v}= T_{v,1}$, while if $v\in 2\N  \setminus 4\N$, then $T_{v} \neq T_{v,1} = 0$. Note that $T_{2}$ already appeared in \eqref{C for EG square}. To prove Theorem \ref{thm:Ginibre rectangle C}, we will need simplified expressions for $T_{0,\alpha}$ and $T_{2,\alpha}$. The identity $T_{0,\alpha} = \frac{\pi^{3}}{32}$ was previously known (see e.g. \cite[0.243.3]{GRtable}), but we did not find in the literature simplified expressions for $T_{v,\alpha}$ when $v \neq 0$. The next theorem, which we consider of independent interest, gives a simplified expression for $T_{v,\alpha}$ when $v\in 2\N$.
\begin{theorem}\label{thm: some nice series}
Let $\alpha \in (0,+\infty)$. The sequence $\{T_{2v,\alpha}\}_{v\in \N}$ satisfies the recursive formula 
\begin{align}
T_{2v,\alpha} & = \frac{1}{(2v)!} \frac{\pi^{3+2v}}{4^{v+3}(v+1)(2v+1)} \frac{(\sqrt{\alpha}+\frac{i}{\sqrt{\alpha}})^{2+2v} - (\sqrt{\alpha} - \frac{i}{\sqrt{\alpha}})^{2+2v}}{2i} \nonumber \\
& - \sum_{q=0}^{v-1} \frac{\pi^{2v-2q}}{(2v-2q)!\, 4^{v-q}} \frac{(\sqrt{\alpha}+\frac{i}{\sqrt{\alpha}})^{2v-2q} + (\sqrt{\alpha} - \frac{i}{\sqrt{\alpha}})^{2v-2q}}{2}T_{2q,\alpha}, \label{recursion Tva}
\end{align}
and the initial condition $T_{0,\alpha} = \frac{\pi^{3}}{32}$. In particular,
\begin{align*}
& T_{2,\alpha} = \pi^{5} \frac{\frac{1}{\alpha}-\alpha}{384}, \qquad T_{4,\alpha} = \pi^{7} \frac{6(\frac{1}{\alpha^{2}}+\alpha^{2})-5}{23040}, \qquad T_{6,\alpha} = \pi^{9} \frac{17(\frac{1}{\alpha^{3}}-\alpha^{3})-14(\frac{1}{\alpha}-\alpha)}{645120}, \\
& T_{8,\alpha} =  \frac{310(\frac{1}{\alpha^{4}}+\alpha^{4}) - 255(\frac{1}{\alpha^{2}} + \alpha^{2}) + 252}{\pi^{-11} \; 116121600}, \; T_{10,\alpha} =  \frac{2073(\frac{1}{\alpha^{5}}-\alpha^{5}) - 1705 (\frac{1}{\alpha^{3}}-\alpha^{3}) + 1683(\frac{1}{\alpha}-\alpha)}{\pi^{-13} \; 7664025600}.
\end{align*}
If $\alpha=1$, \eqref{recursion Tva} simplifies as follows: the sequence $\{T_{4v}\}_{v\in \N}$ satisfies
\begin{align*}
T_{4v} = \frac{(-1)^{v}\pi^{4v+3}}{(4v)! \, 4^{v+2}(4v+1)(4v+2)} - \sum_{q=0}^{v-1} \frac{(-1)^{v-q}\pi^{4(v-q)}}{(4(v-q))!\, 4^{v-q}}T_{4q},
\end{align*}
and the initial condition $T_{0} = \frac{\pi^{3}}{32}$. In particular, 
\begin{align*}
T_{0} = \frac{\pi^{3}}{32}, \qquad T_{4} = \frac{7\pi^{7}}{23040}, \qquad T_{8} = \frac{181 \pi^{11}}{58060800}, \qquad T_{12} = \frac{178559 \pi^{15}}{5579410636800}.
\end{align*}
\end{theorem}
Theorem \ref{thm: some nice series} will also be useful to prove Corollary \ref{coro:ML square C} below.
\subsubsection{The complement of an ellipse centered at $0$}
\begin{figure}
\begin{center}
\hspace{-0.3cm}\begin{tikzpicture}[master]
\node at (0,0) {\includegraphics[width=4.8cm]{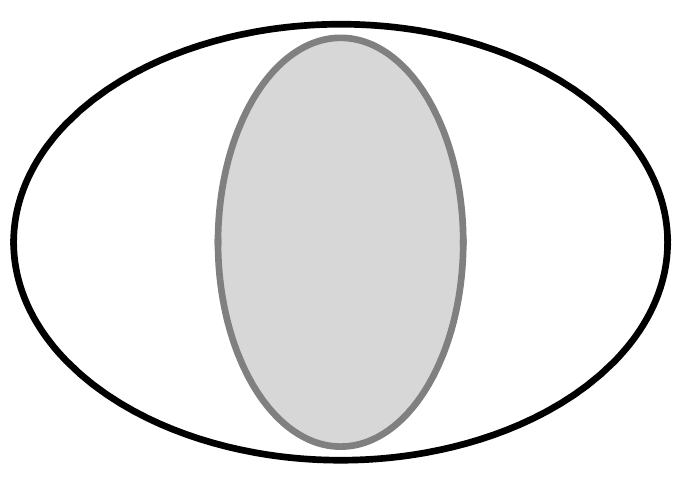}};
\draw[fill] (0.857,0) circle (0.04);
\node at (1.03,0) {\footnotesize $a$};
\draw[fill] (0,1.436) circle (0.04);
\node at (0,1.28) {\footnotesize $ic$};
\node at (0,-2) {\footnotesize Theorems \ref{thm:EG complement ellipse nu}, \ref{thm:EG complement ellipse C}};
\end{tikzpicture}  \hspace{-0.3cm}
\begin{tikzpicture}[slave]
\node at (0,0) {\includegraphics[width=4.8cm]{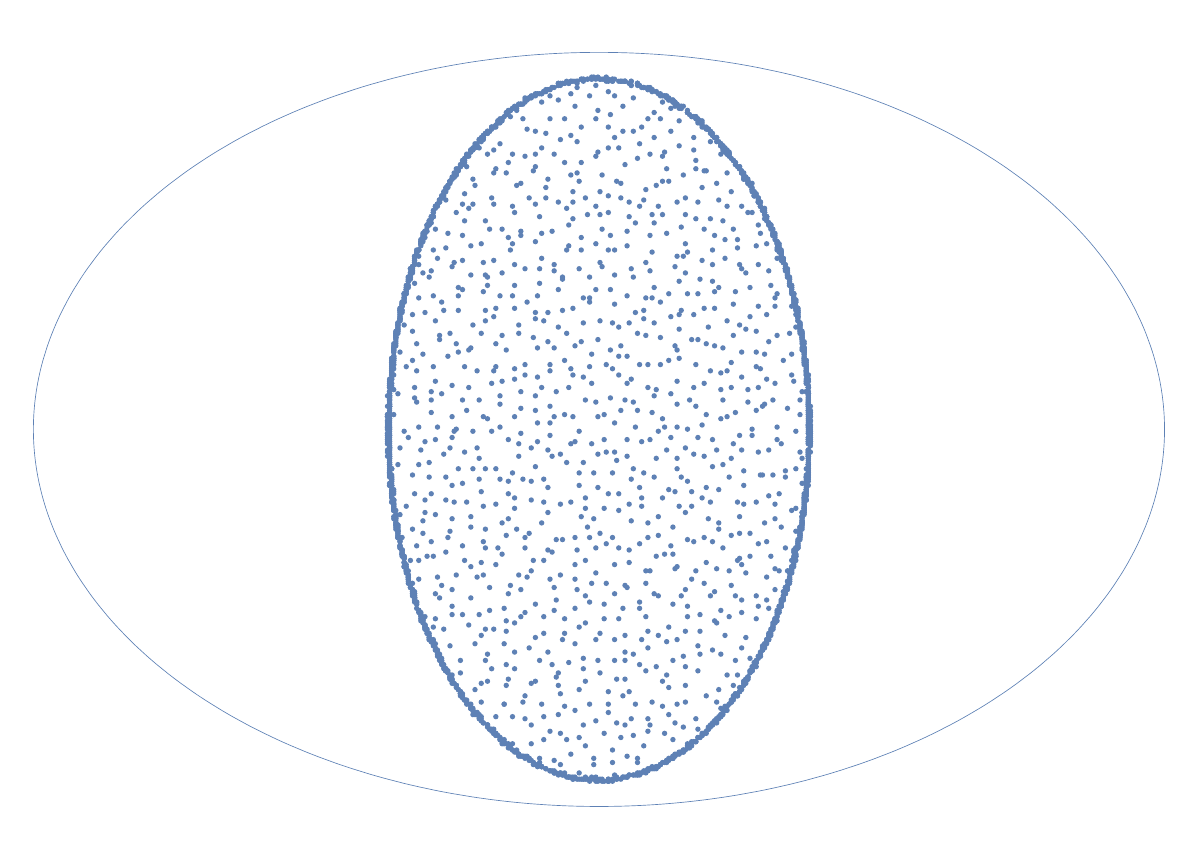}};
\end{tikzpicture} \hspace{-0.3cm}
\begin{tikzpicture}[slave]
\node at (0,0) {\includegraphics[width=4.8cm]{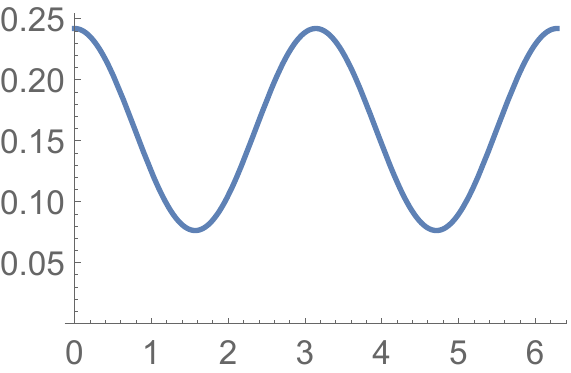}};

\node at (0,-1.9) {\footnotesize Theorem \ref{thm:EG complement ellipse nu}};
\end{tikzpicture}
\end{center}
\caption{\label{fig:density complement ellipse Ginibre} In all figures, $\tau=0.2$ and $U=\{z:(\frac{\re z}{a})^{2}+(\frac{\im z}{c})^{2}>1\}$ with $a=0.45$ and $c = 0.75$. Middle: the elliptic Ginibre point process conditioned on $\# \{z_{j}\in U\} = 0$. Right: the normalized density $\theta \mapsto \frac{d\nu(z)/d\theta}{\nu(\partial U)}$ with $z=a\cos \theta + i c \sin \theta \in \partial U$, $\theta \in [0,2\pi)$.}
\end{figure}

\begin{theorem}\label{thm:EG complement ellipse nu}
Let $\tau\in [0,1)$, $a\in (0,1+\tau]$, and $c\in (0,1-\tau]$. Then $U^{c}\subset S$, where $U:=\{z:(\frac{\re z}{a})^{2}+(\frac{\im z}{c})^{2}>1\}$.
Let $\nu=\mathrm{Bal}(\mu|_{U},\partial U)$. For $z=a \cos \theta + i c \sin \theta$, $\theta \in (-\pi,\pi]$, we have $d\nu(z) = \big( c_{0} + 2 c_{1} \cos (2 \theta) \big)d\theta$, where
\begin{align}\label{def of c0 c1 complement ellipse intro}
& c_{0} = \frac{1}{2\pi} \frac{1-\tau^{2}-a c}{1-\tau^{2}}, \qquad c_{1} = \frac{a+c}{8\pi} \frac{(a+c)\tau - a + c}{1-\tau^{2}},
\end{align}
see also Figure \ref{fig:density complement ellipse Ginibre}.
\end{theorem}
\begin{remark}
If the ellipse $\partial U = \{z:(\frac{\re z}{a})^{2}+(\frac{\im z}{c})^{2}=1\}$ has the same eccentricity as the larger ellipse $\partial S = \{z:(\frac{\re z}{1+\tau})^{2}+(\frac{\im z}{1-\tau})^{2}=1\}$, then \eqref{def of c0 c1 complement ellipse intro} yields $c_{1}=0$, which means that $\nu$ is the uniform measure on $\partial U$ with respect to $d\theta$ (but not with respect to the arclength measure, except if $a=c$).
\end{remark}

\begin{theorem}\label{thm:EG complement ellipse C}
Let $\beta>0$, $\tau\in [0,1)$, $a\in (0,1+\tau]$, and $c\in (0,1-\tau]$. Then $U^{c}\subset S$, where $U:=\{z:(\frac{\re z}{a})^{2}+(\frac{\im z}{c})^{2}>1\}$. As $n \to +\infty$, $\mathbb{P}(\# \{z_{j}\in U\} = 0) = \exp \big( -C n^{2}+o(n^{2}) \big)$, where
\begin{align}
C = \frac{\beta}{4(1-\tau^{2})} \bigg\{ & a^{2}(1-\tau) \bigg( 1 - \frac{a^{2}+2ac}{8(1+\tau)} \bigg) + c^{2} (1+\tau) \bigg( 1 - \frac{c^{2} + 2ac}{8(1-\tau)} \bigg) \nonumber \\
& + \frac{a^{2}c^{2}}{4} + (1-\tau^{2}) \bigg( 2 \log \Big( \frac{2}{a+c} \Big) - \frac{3}{2} \bigg) \bigg\}. \label{EG complement ellipse C}
\end{align}
\end{theorem}
\begin{remark}
If $\beta=2$, $\tau=0$ and $c=a$, then $C$ in \eqref{EG complement ellipse C} becomes $C=a^{2}-\frac{a^{4}}{4}-\frac{3}{4}-\log a$, which recovers the result \cite[(51)]{CMV2016}. Also, if $a=1+\tau$ and $c=1-\tau$, then \eqref{EG complement ellipse C} yields $C=0$, and thus $\mathbb{P}(\# \{z_{j}\in U\} = 0) = \exp \big( o(n^{2}) \big)$ as $n\to + \infty$.
\end{remark}

\subsubsection{The complement of a disk}

\begin{theorem}\label{thm:EG complement disk nu}
Let $\tau \in [0,1)$, $x_{0},y_{0}\in \R$ and $a>0$ be such that $U^{c}\subset S$, where $U:=\{z:|z-(x_{0}+iy_{0})|>a\}$. Let $\nu=\mathrm{Bal}(\mu|_{U},\partial U)$. For $z=x_{0}+iy_{0}+a e^{i\theta}$, $\theta \in (-\pi,\pi]$, we have 
\begin{align}\label{def of dnu complement disk}
& d\nu(z) = \frac{a}{\pi(1-\tau^{2})}\bigg( \frac{1-\tau^{2}}{2a} - \frac{a}{2} - x_{0}(1-\tau)\cos \theta - y_{0} (1+\tau) \sin \theta + a\tau \cos(2\theta) \bigg)d\theta,
\end{align}
see also Figure \ref{fig:density complement disk Ginibre}.
\end{theorem}
\begin{remark}
If $x_{0}=y_{0}=0$, then $\nu$ in \eqref{def of dnu complement disk} is equal to $\nu$ in Theorem \ref{thm:EG complement ellipse nu} with $c=a$, as it must.
\end{remark}
\begin{figure}[h]
\begin{center}
\hspace{-0.3cm}\begin{tikzpicture}[master]
\node at (0,0) {\includegraphics[width=4.8cm]{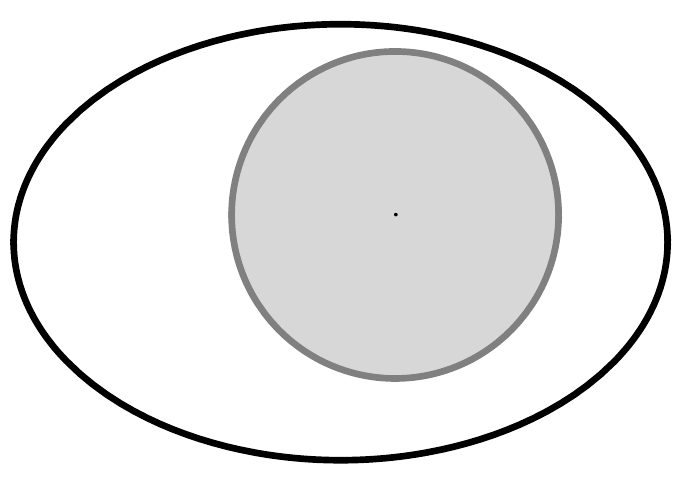}};
\draw[fill] (0.38,0.2) circle (0.04);
\node at (0.38,0) {\footnotesize $z_{0}$};
\node at (0,-2) {\footnotesize Theorems \ref{thm:EG complement disk nu}, \ref{thm:EG complement disk C}};
\end{tikzpicture}  \hspace{-0.3cm}
\begin{tikzpicture}[slave]
\node at (0,0) {\includegraphics[width=4.8cm]{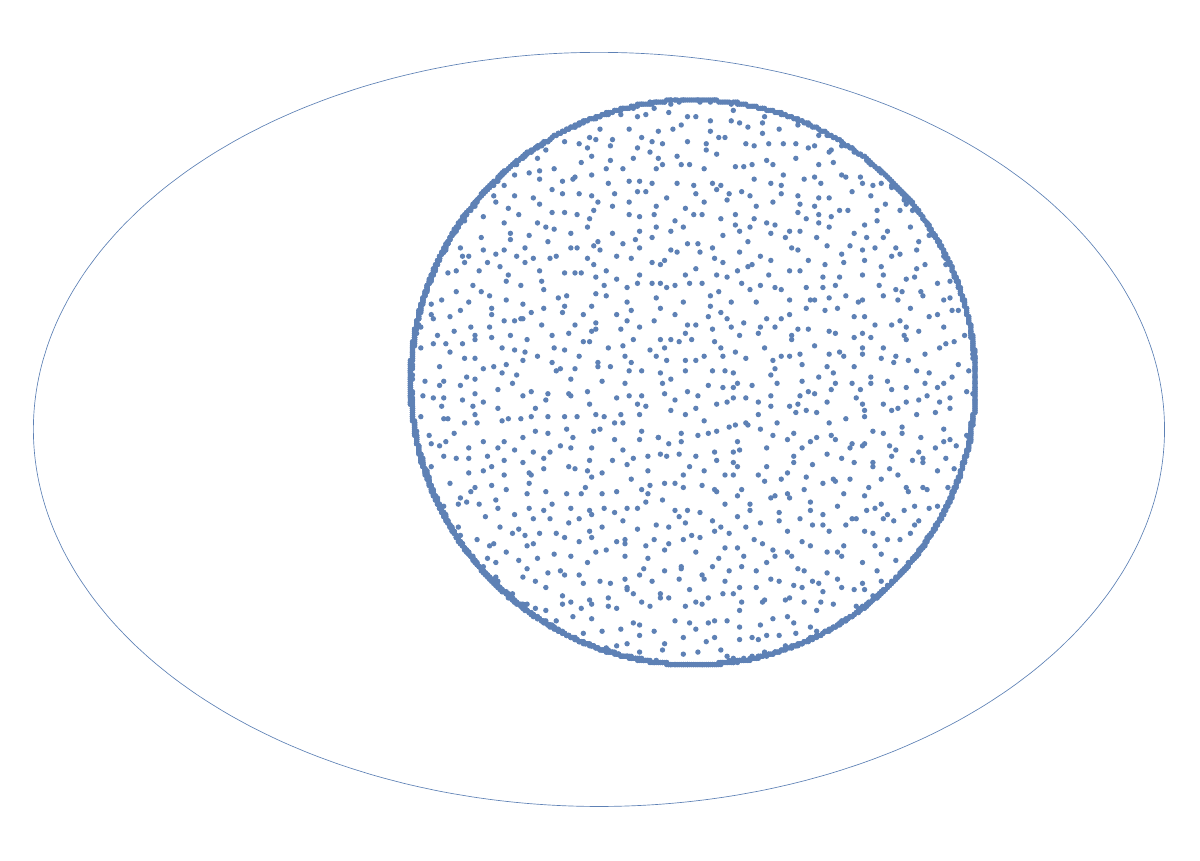}};
\end{tikzpicture} \hspace{-0.3cm}
\begin{tikzpicture}[slave]
\node at (0,0) {\includegraphics[width=4.8cm]{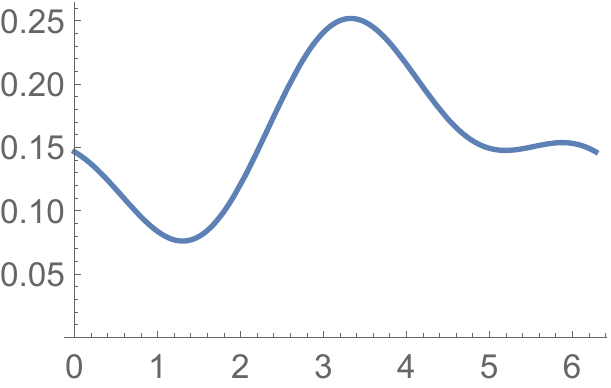}};

\node at (0,-1.9) {\footnotesize Theorem \ref{thm:EG complement disk nu}};
\end{tikzpicture}
\end{center}
\vspace{-0.5cm}\caption{\label{fig:density complement disk Ginibre} In all figures, $U=\{z:|z-z_{0}|> a\}$ with $z_{0}=0.2+i\, 0.1$ and $a=0.6$. Middle: the elliptic Ginibre point process conditioned on $\# \{z_{j}\in U\} = 0$. Right: the normalized density $\theta \mapsto \frac{d\nu(z)/d\theta}{\nu(\partial U)}$ with $z = z_{0}+ae^{i\theta}\subset \partial U$, $\theta \in [0,2\pi)$.}
\end{figure}

\begin{theorem}\label{thm:EG complement disk C}
Fix $\beta>0$, and let $\tau \in [0,1)$, $x_{0},y_{0}\in \R$ and $a>0$ be such that $U^{c}\subset S$, where $U:=\{z:|z-(x_{0}+iy_{0})|>a\}$. As $n \to +\infty$, $\mathbb{P}(\# \{z_{j}\in U\} = 0) = \exp \big( -C n^{2}+o(n^{2}) \big)$, where
\begin{align}
C = \frac{\beta}{4} \bigg( & \frac{1}{\pi(1-\tau^{2})} \int_{-\pi}^{\pi} R(\theta)^{2}\log R(\theta) \; d\theta - \frac{3}{2} + \frac{x_{0}^{2}}{1+\tau} + \frac{y_{0}^{2}}{1-\tau} \nonumber \\ 
& -2 \log a + 2a^{2}\bigg( \frac{1}{1-\tau^{2}} - \frac{x_{0}^{2}}{(1+\tau)^{2}} - \frac{y_{0}^{2}}{(1-\tau)^{2}} \bigg) - \frac{(1+2\tau^{2})a^{4}}{2(1-\tau^{2})^{2}} \bigg), \label{EG complement disk C}
\end{align}
where $R$ is independent of $a$ and given by
\begin{align}\label{def of R Uc is a disk}
R(\theta) := \frac{(1-\tau^{2})^{2}( \frac{\sqrt{2+2\tau^{2}-x_{0}^{2}-y_{0}^{2}+(x_{0}^{2}-y_{0}^{2}-4\tau)\cos (2\theta) + 2x_{0}y_{0}\sin(2\theta)}}{\sqrt{2}(1-\tau^{2})} - \frac{x_{0}\cos \theta}{(1+\tau)^{2}} - \frac{y_{0}\sin\theta}{(1-\tau)^{2}})}{1+\tau^{2}-2\tau\cos(2\theta)}.
\end{align}
Moreover, if $x_{0}=y_{0}=0$, then $\int_{-\pi}^{\pi} R(\theta)^{2}\log R(\theta) \; d\theta = 0$, and 
\begin{align}\label{C EG centered disk}
C = \frac{\beta}{2} \bigg( \frac{a^{2}}{1-\tau^{2}} - \frac{(1+2\tau^{2})a^{4}}{4(1-\tau^{2})^{2}} - \frac{3}{4} -  \log a \bigg).
\end{align}
\end{theorem}
\begin{remark}
It is easy to check that $C$ in \eqref{C EG centered disk} is equal to $C$ in \eqref{EG complement ellipse C} when $c=a$, as it must. Also, if $\tau = 0 = x_{0}=y_{0}$, then \eqref{C EG centered disk} yields $C=a^{2}-\frac{a^{4}}{4}-\frac{3}{4}-\log a$, which recovers \cite[(51)]{CMV2016}. 
\end{remark}
\begin{remark}
For $\theta \in (-\pi,\pi]$, $R(\theta)$ in \eqref{def of R Uc is a disk} is the only positive number such that $x_{0}+iy_{0}+R(\theta)e^{i\theta} \in \partial S$, i.e. 
\begin{align*}
\bigg( \frac{x_{0}+R(\theta)\cos \theta}{1+\tau} \bigg)^{2} + \bigg( \frac{y_{0}+R(\theta)\sin \theta}{1-\tau} \bigg)^{2} =1.
\end{align*}
Since $\tau,x_{0},y_{0},a$ are such that $U^{c}\subset S$, we have $R(\theta) \geq a$ for all $\theta \in (-\pi,\pi]$.
\end{remark}

\subsection{Results for the Mittag-Leffler point process}
Recall that the Mittag-Leffler point process is defined by \eqref{general density intro} with $Q(z) = |z|^{2b}$ for some $b>0$, and that $\mu$ and $S:=\mathrm{supp} \, \mu$ are given by
\begin{align}\label{mu S ML}
d\mu(z) = \frac{b^{2}}{\pi}|z|^{2b-2}d^{2}z, \qquad S=\{z \in \C: |z| \leq b^{-\frac{1}{2b}}\}.
\end{align}
This point process has attracted a lot of interest recently \cite{A2018, AKS2023, C2021 FH, C2021, CL2023, ACCL1, ACCL2, BC2022, ACC2023, Berezin2023}. In this subsection, we focus on the case $b\in \N_{>0}$. We obtain explicit results for $\nu := \mathrm{Bal}(\mu|_{U},\partial U)$ and the hole probability $\mathbb{P}(\# \{z_{j}\in U\} = 0)$, in the cases where $U$ is either a disk, an ellipse centered at $0$, and a rectangle.

\medskip Already for $b\in \N_{>0}$, obtaining explicit results for $\nu=\mathrm{Bal}(\mu|_{U},\partial U)$ turns out to be much more difficult than for $\mathrm{Bal}(d^{2}z|_{U},\partial U)$. This explains why we cover less hole regions here for the Mittag-Leffler point process than what we did in Subsection \ref{subsection:elliptic Ginibre intro} for the elliptic Ginibre point process.
\subsubsection{The disk}
\begin{figure}
\begin{center}
\hspace{-1cm}\begin{tikzpicture}[master]
\node at (0,0) {\includegraphics[width=3.4cm]{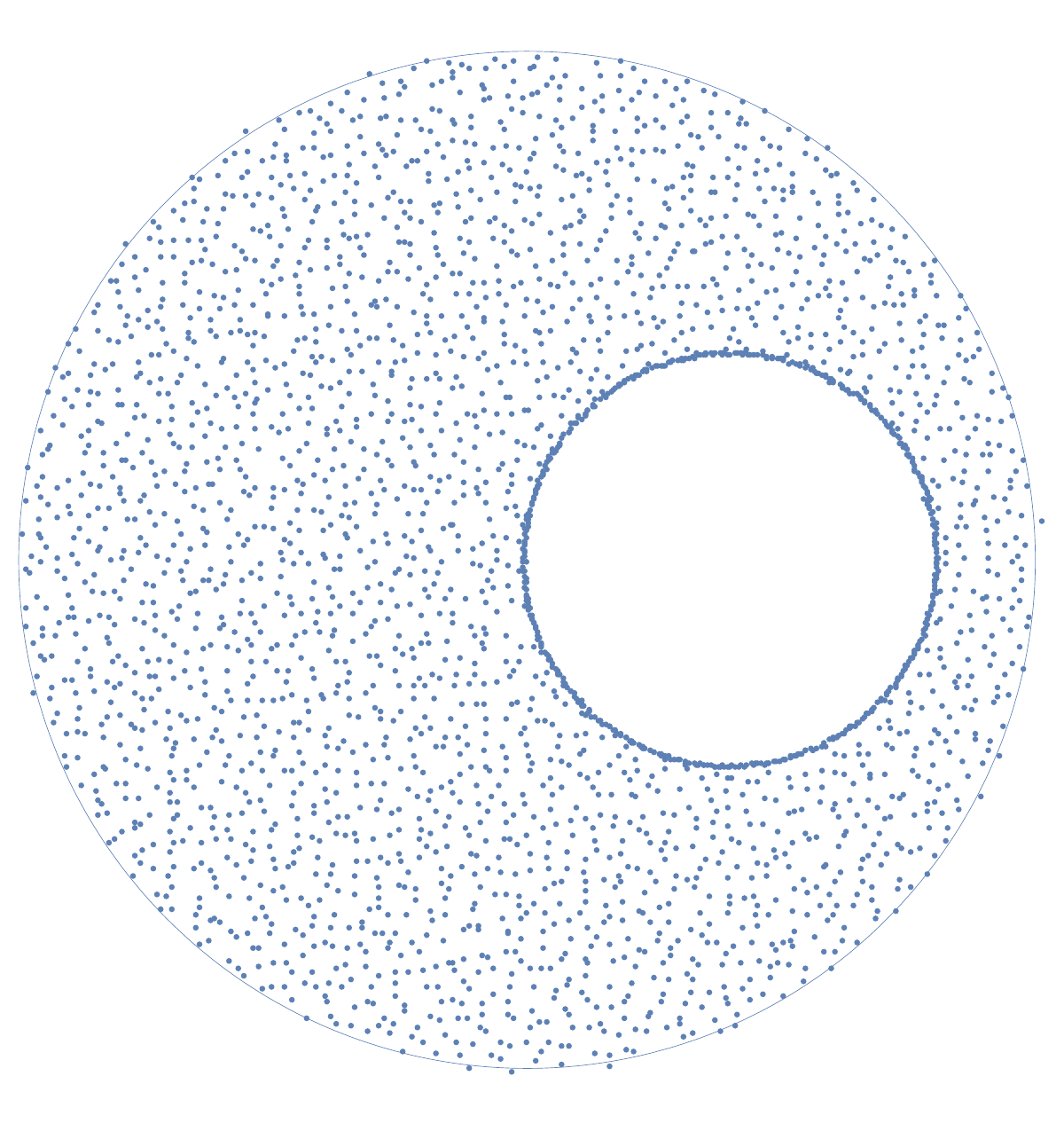}};
\node at (0,1.8) {\footnotesize $b=1$};
\end{tikzpicture} \hspace{-0.4cm}
\begin{tikzpicture}[slave]
\node at (0,0) {\includegraphics[width=3.4cm]{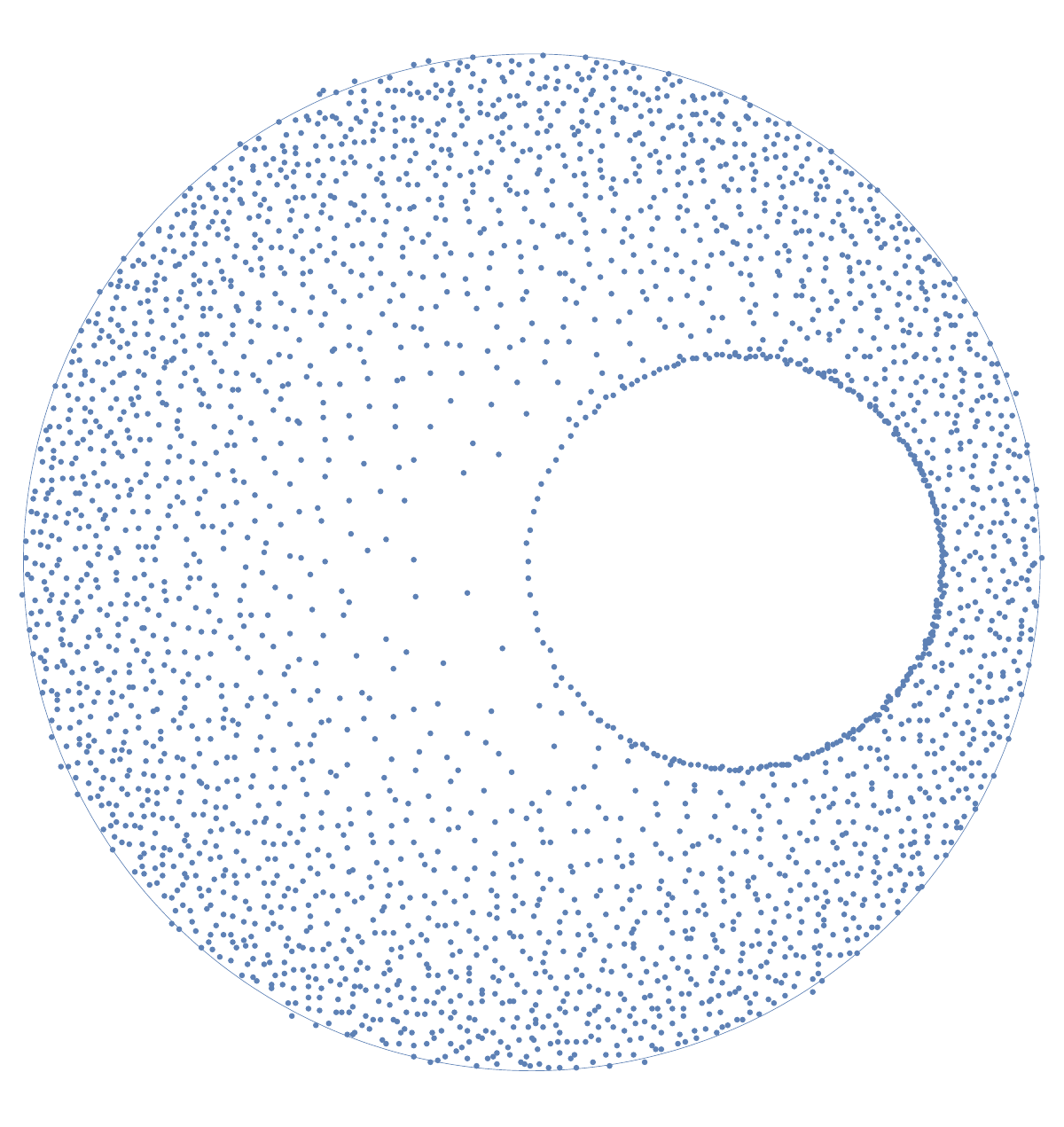}};
\node at (0,1.8) {\footnotesize $b=2$};
\end{tikzpicture} \hspace{-0.4cm}
\begin{tikzpicture}[slave]
\node at (0,0) {\includegraphics[width=3.4cm]{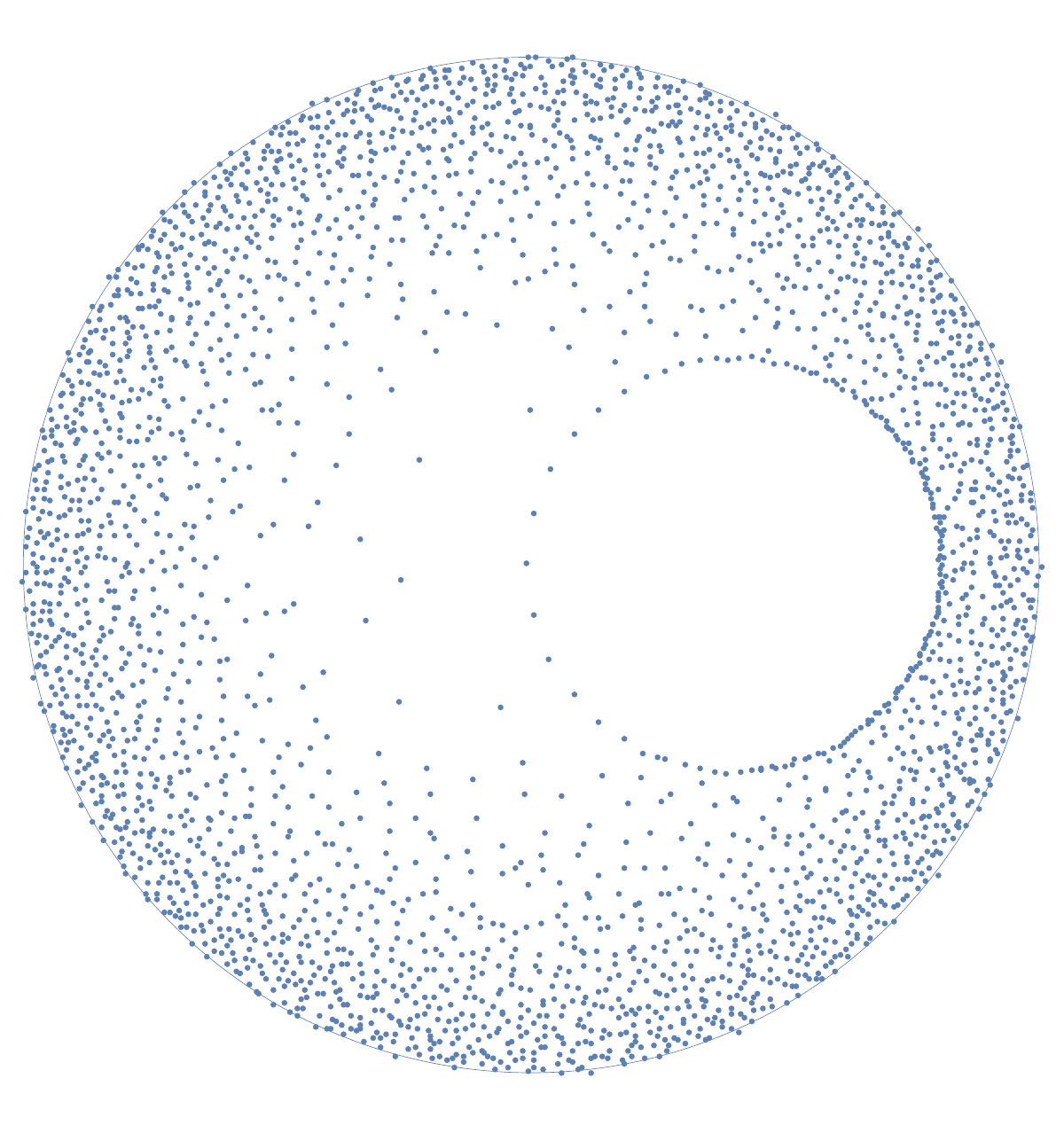}};
\node at (0,1.8) {\footnotesize $b=3$};
\end{tikzpicture}  \hspace{0.5cm}
\begin{tikzpicture}[slave]
\node at (0,0) {\includegraphics[width=4.9cm]{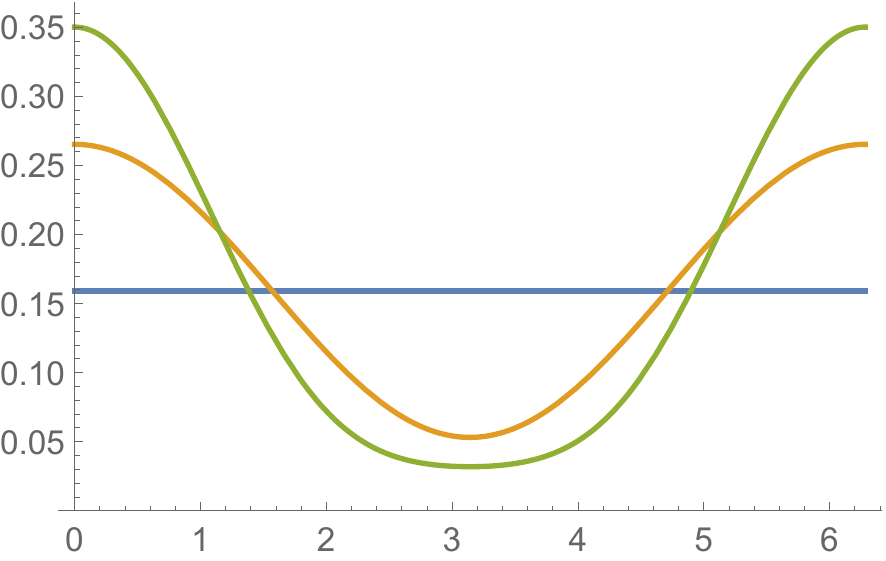}};
\node at (0.1,0.15) {\footnotesize $b=1$};
\node at (0.1,-0.58) {\footnotesize $b=2$};
\node at (0.1,-1.1) {\footnotesize $b=3$};
\node at (0.1,1.8) {\footnotesize Theorem \ref{thm:ML disk nu}};
\end{tikzpicture}
\end{center}
\caption{\label{fig:ML disk points} Left: The Mittag-Leffler ensemble conditioned on $\# \{z_{j}\in U\} = 0$ with $U=\{z:|z-x_{0}|<a\}$, $x_{0}=a=0.4b^{-\frac{1}{2b}}$ and $b\in \{1,2,3\}$. Right: the normalized density $\theta \mapsto \frac{d\nu(z)/d\theta}{\nu(\partial U)}$ with $z=x_{0}+ae^{i\theta}$, $x_{0}=a=0.4\smash{b^{-\frac{1}{2b}}}$ and $b=1,2,3$.}
\end{figure}
We start with the case where $U$ is a disk, i.e. $U:=\{z:|z-(x_{0}+i y_{0})|<a \}$ for some $x_{0},y_{0}\in \R$. Since $\mu$ is rotation-invariant, we assume without loss of generality that $y_{0}=0$ and $x_{0}\geq 0$.
\begin{theorem}\label{thm:ML disk nu}
Let $b\in \N_{>0}$, $x_{0} \geq 0$ and $a >0$ be such that $U:=\{z:|z-x_{0}|<a \} \subset S$ (i.e. $x_{0}+a \leq b^{-\frac{1}{2b}}$), and let $\nu=\mathrm{Bal}(\mu|_{U},\partial U)$. Then for $z=x_{0}+a e^{i\theta}$, $\theta \in [0,2\pi)$, we have $d\nu(z) = \big( c_{0} + 2 \sum_{\ell=1}^{b-1} c_{\ell} \cos (\ell \theta) \big)\frac{d\theta}{\pi}$, where
\begin{align}\label{def of cj DISK intro}
c_{\ell} = \frac{b^{2}}{a^{\ell}} \sum_{\substack{k=0 \\ k-\ell \, \mathrm{even}}}^{b-1} \binom{b-1}{k} \binom{k}{\frac{k-\ell}{2}} x_{0}^{k} \sum_{m=0}^{b-1-k}\binom{b-1-k}{m}\frac{x_{0}^{2(b-1-k-m)}a^{2+\ell+k+2m}}{2+\ell+k+2m},
\end{align}
see also Figure \ref{fig:ML disk points}.
\end{theorem}
\begin{remark}
In the particular case where $x_{0}=a$, we have $0\in \partial U$. Since $b \geq 1$, from \eqref{mu S ML} one expects $d\nu(z)/d\theta$ to attain its minimum at $z=0$, and the numerics suggest that $d\nu(z)/d\theta|_{z=0}>0$ (see also Figure \ref{fig:ML disk points}). This inequality is consistent with Conjecture \ref{conj:behavior of balayage measure} (with $p=2$ and $b\in \N_{>0}$). 
\end{remark}

\begin{theorem}\label{thm:ML disk C}
Fix $\beta>0$. Let $b\in \N_{>0}$, $x_{0} \geq 0$ and $a >0$ be such that $x_{0}+a \leq b^{-\frac{1}{2b}}$, so that $U:=\{z:|z-x_{0}|<a \} \subset S$. As $n \to +\infty$, we have $\mathbb{P}(\# \{z_{j}\in U\} = 0) = \exp \big( -C n^{2}+o(n^{2}) \big)$, where
\begin{align*}
C = \frac{\beta}{4} \bigg\{ & 2\sum_{j=0}^{b} \binom{b}{j}a^{j}x_{0}^{j}(x_{0}^{2}+a^{2})^{b-j} \bigg[ c_{0} \binom{j}{\frac{j}{2}} \mathbf{1}_{j \, \mathrm{even}} + 2 \sum_{\ell=1}^{b-1} c_{\ell} \binom{j}{\frac{j+\ell}{2}} \mathbf{1}_{j+\ell \, \mathrm{even}} \bigg] \\
& - b^{2} \sum_{j=0}^{b-1} \binom{2b-1}{2j} \binom{2j}{j} \sum_{k=0}^{2b-1-2j} \binom{2b-1-2j}{k} \frac{x_{0}^{2(2b-1-j-k)}a^{2(1+j+k)}}{1+j+k} \bigg\}.
\end{align*}
($\mathbf{1}_{j \, \mathrm{even}}:=1$ if $j$ is even and $0$ otherwise.) In particular, for $b=1,2,3$, we have $C|_{b=1} = \frac{\beta a^{4}}{8}$ and
\begin{align*}
C|_{b=2} = \frac{\beta a^{4}}{4}(a^{4}+8a^{2}x_{0}^{2}+8x_{0}^{4}), \quad C|_{b=3} = \frac{3\beta a^{4}}{8}(a^{8}+24a^{6}x_{0}^{2}+102a^{4}x_{0}^{4}+108a^{2}x_{0}^{6}+27x_{0}^{8}).
\end{align*}
\end{theorem}


\subsubsection{The ellipse}

Recall that $\mu$ and $S$ are given by \eqref{mu S ML}.
\begin{theorem}\label{thm:ML ellipse nu}
Let $b\in \N_{>0}$, and $a,c \in (0, b^{-\frac{1}{2b}}]$. Then $U := \{z:(\frac{\re z}{a})^{2}+(\frac{\im z}{c})^{2}<1\} \subset S$, and the measure $\nu:=\mathrm{Bal}(\mu|_{U},\partial U)$ is given by $d\nu(z) = \big( c_{0} + 2 \sum_{\ell=1}^{b} c_{\ell} \cos (2\ell \theta) \big)\frac{d\theta}{\pi}$ for $z = a \cos \theta + i c \sin \theta$, $\theta \in [0,2\pi)$, where
\begin{align}
& c_{0} = e_{0}, \qquad c_{\ell} = \frac{e_{\ell}}{\alpha^{\ell}\gamma^{-\ell}+\alpha^{-\ell}\gamma^{\ell}} \mbox{ for } \ell=1,\ldots,b, \label{def of ej}  \\
& e_{\ell} = \frac{(\alpha^{2}-\gamma^{2})b^{2}}{2} \sum_{j=0}^{b-1} \binom{b-1}{j} (\alpha^{2}+\gamma^{2})^{b-1-j} \sum_{\substack{k=0 \, \mathrm{or} \, 1 \\ j-k \, \mathrm{even}}}^{j} \frac{1+\mathbf{1}_{k \neq 0}}{2} (\alpha^{j+k} \gamma^{j-k} + \alpha^{j-k} \gamma^{j+k}) \binom{j}{\frac{j-k}{2}} d_{\ell}^{(k)}, \nonumber
\end{align}
with $\alpha = \frac{a+c}{2}$, $\gamma=\frac{a-c}{2}$, $\mathbf{1}_{k \neq 0}:=1$ if $k \neq 0$ and $\mathbf{1}_{k \neq 0}:=0$ otherwise, $d_{0}^{(k)} = d_{1}^{(k)} = \ldots = d_{k-1}^{(k)} = 0$, and for $\ell \in \{k,k+1,\ldots,b\}$ we have
\begin{align}\label{def of d ell pkp}
d_{\ell}^{(k)} = \frac{\binom{2\ell}{\ell-k}}{b+\ell} + 2\ell  \sum_{m=0}^{\ell-k-1}  \frac{(-1)^{\ell-k-m}}{b+k+m} \frac{(k+\ell+m-1)!}{m!(2k+m)!(\ell-k-m)!}.
\end{align}
In particular, with $\rho := \gamma/\alpha \in (-1,1)$ (if $\rho=0$ then $\rho^{-1}:=+\infty$), we have
\begin{align*}
& \frac{d\nu(z)|_{b=1}}{d\theta} = \frac{\alpha^{2}}{2\pi}(1-\rho^{2})\bigg( 1 - \frac{2\cos(2\theta)}{\rho^{-1}+\rho} \bigg), \qquad \frac{d\nu(z)|_{b=2}}{d\theta} = \frac{\alpha^{4}}{\pi}(1-\rho^{4}) \bigg( 1 - \frac{2 \cos(4\theta)}{\rho^{-2}+\rho^{2}} \bigg), \\
& \frac{d\nu(z)|_{b=3}}{d\theta} = \frac{3\alpha^{6}}{\pi} \bigg\{ \frac{1+3\rho^{2}-3\rho^{4}-\rho^{6}}{2} \bigg( 1 - \frac{2\cos(4\theta)}{\rho^{-2}+\rho^{2}} \bigg) + (1-\rho^{6}) \bigg( \frac{\cos(2\theta)}{\rho^{-1}+\rho} - \frac{\cos(6\theta)}{\rho^{-3}+\rho^{3}} \bigg) \bigg\},
\end{align*}
see also Figure \ref{fig:ML ellipse points}.
\end{theorem}

\begin{figure}
\begin{center}
\hspace{-1cm}\begin{tikzpicture}[master]
\node at (0,0) {\includegraphics[width=3.4cm]{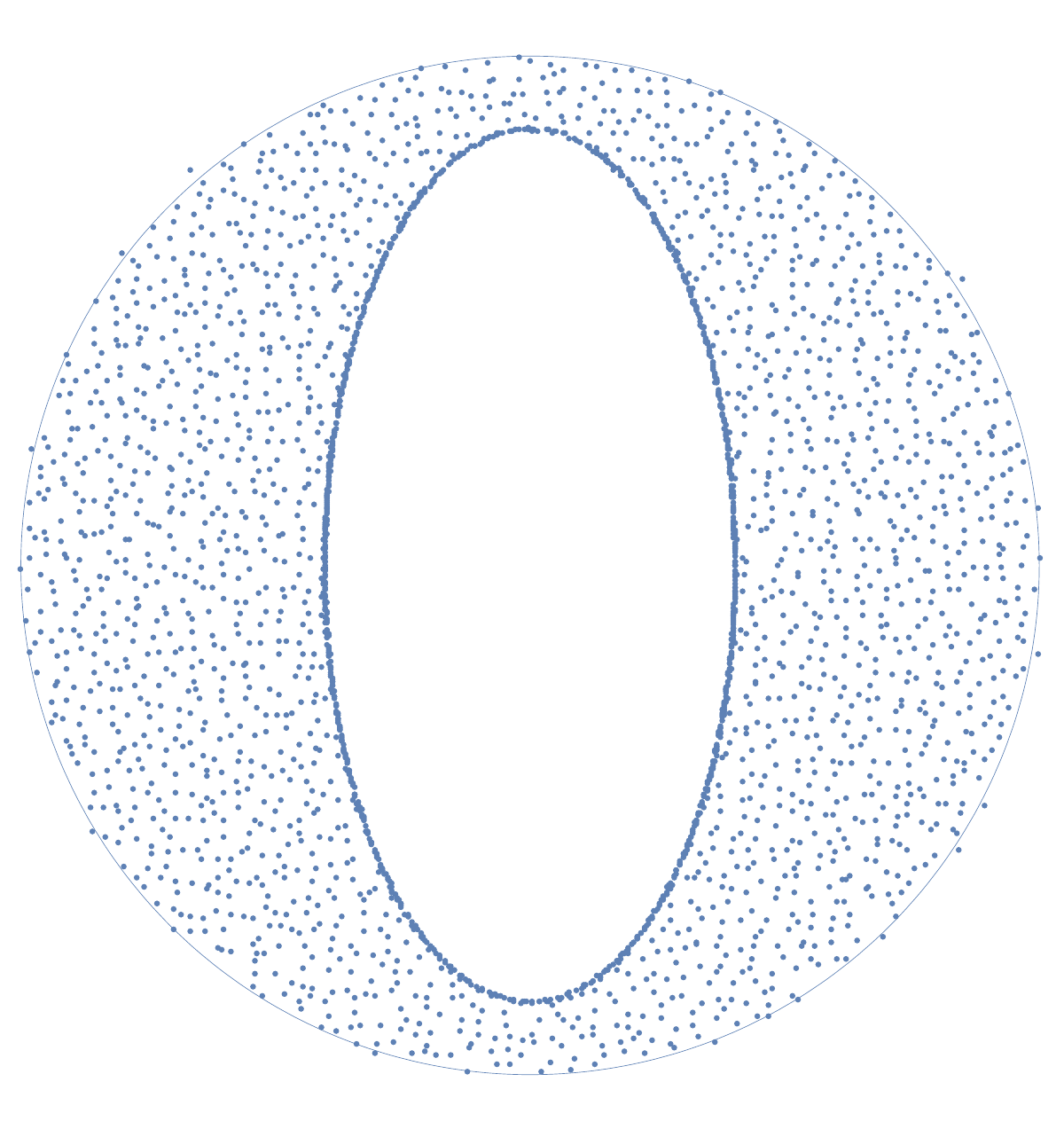}};
\node at (0,1.8) {\footnotesize $b=1$};
\end{tikzpicture}\hspace{-0.4cm}
\begin{tikzpicture}[slave]
\node at (0,0) {\includegraphics[width=3.4cm]{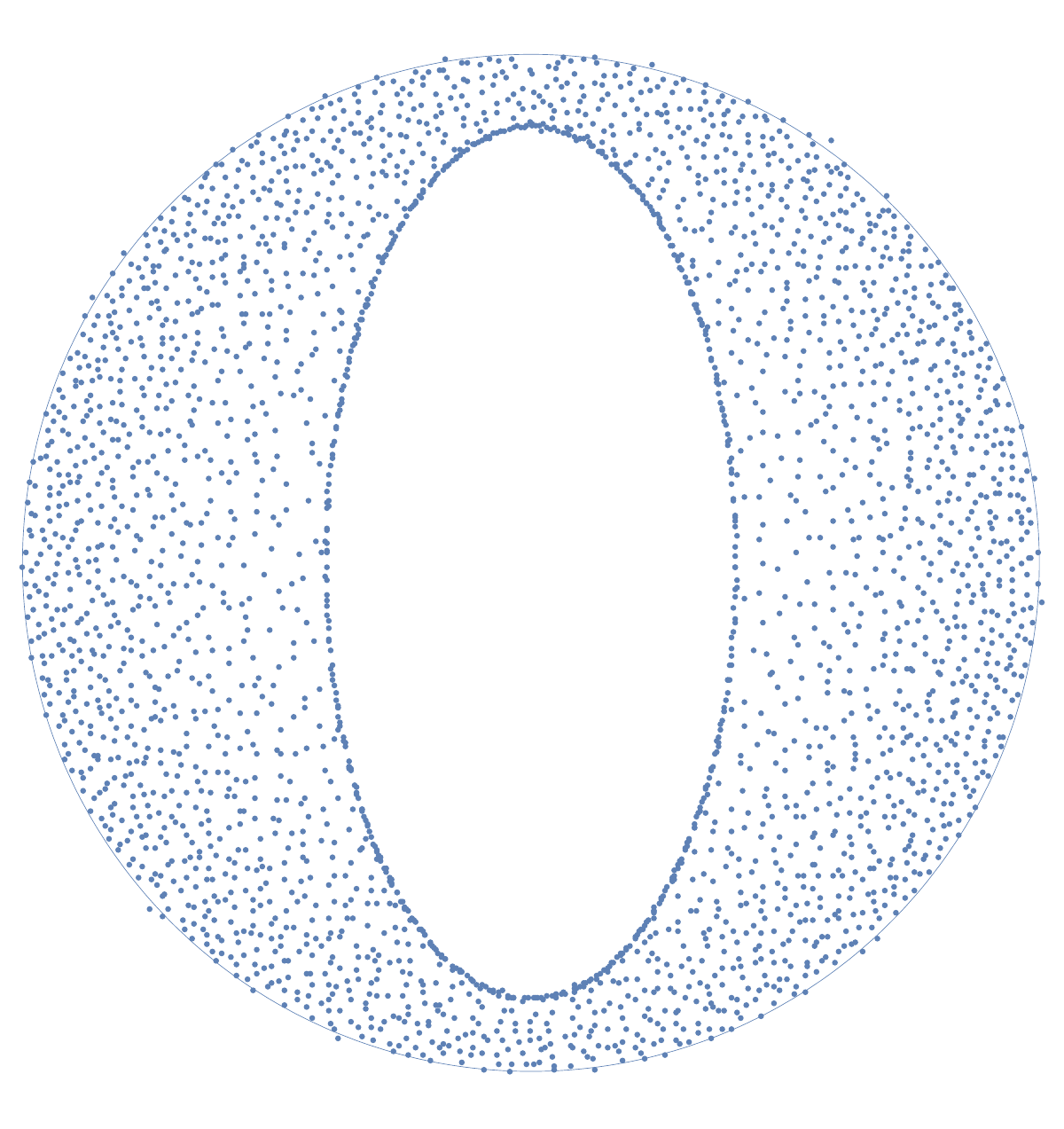}};
\node at (0,1.8) {\footnotesize $b=2$};
\end{tikzpicture}\hspace{-0.4cm}
\begin{tikzpicture}[slave]
\node at (0,0) {\includegraphics[width=3.4cm]{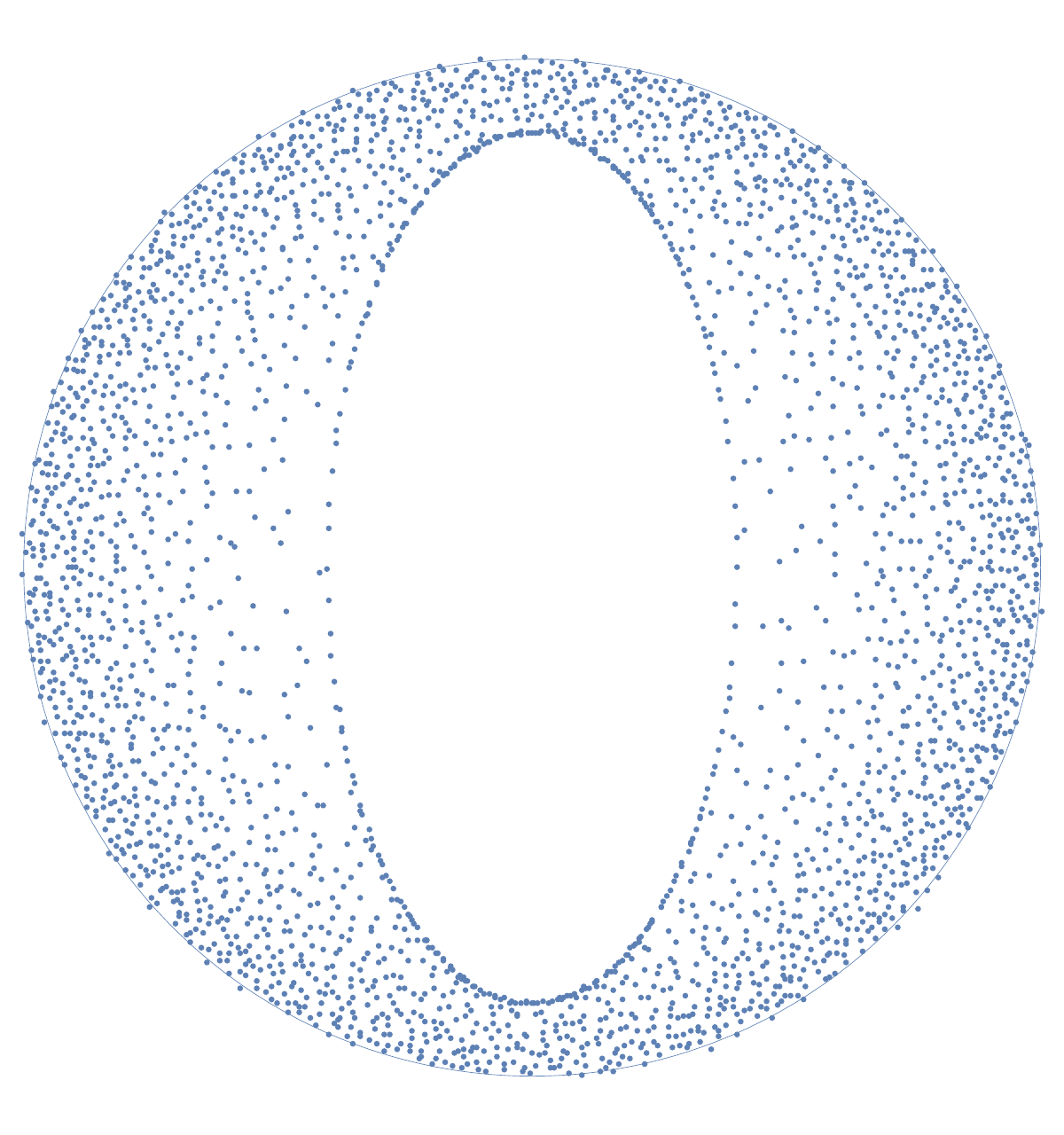}};
\node at (0,1.8) {\footnotesize $b=3$};
\end{tikzpicture}\hspace{0.5cm}
\begin{tikzpicture}[slave]
\node at (0,0) {\includegraphics[width=4.9cm]{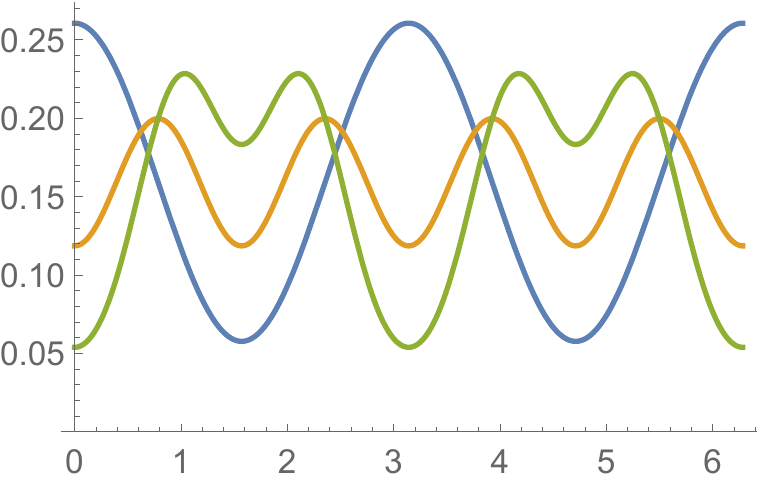}};
\node at (-0.9,-0.78) {\footnotesize $b=1$};
\node at (1.8,-1) {\footnotesize $b=2$};
\draw[->-=1] (1.8,-0.9)--(2.3,-0.05);
\node at (0.2,-0.78) {\footnotesize $b=3$};
\node at (0.2,1.8) {\footnotesize Theorem \ref{thm:ML ellipse nu}};
\end{tikzpicture}
\end{center}
\caption{\label{fig:ML ellipse points} Left: The Mittag-Leffler ensemble conditioned on $\# \{z_{j}\in U\} = 0$ with $U=\{z:(\frac{\re z}{a})^{2}+(\frac{\im z}{c})^{2}<1\}$, $a=0.4b^{-\frac{1}{2b}}$, $c=0.85b^{-\frac{1}{2b}}$ and $b\in \{1,2,3\}$. Right: the normalized density $\theta \mapsto \frac{d\nu(z)/d\theta}{\nu(\partial U)}$ with $z=a \cos \theta + i \, c \sin \theta$, $a=0.4b^{-\frac{1}{2b}}$, $c=0.85b^{-\frac{1}{2b}}$ and $b=1,2,3$.}
\end{figure}

\begin{theorem}\label{thm:ML ellipse C}
Fix $\beta>0$. Let $b\in \N_{>0}$ and $a,c \in (0,b^{-\frac{1}{2b}}]$. Then $U := \{z:(\frac{\re z}{a})^{2}+(\frac{\im z}{c})^{2}<1\} \subset S$, and as $n \to +\infty$, we have $\mathbb{P}(\# \{z_{j}\in U\} = 0) = \exp \big( -C n^{2}+o(n^{2}) \big)$, where
\begin{align*}
C = \frac{\beta}{4} & \bigg\{ \sum_{j=0}^{b} \binom{b}{j} \alpha^{j} \gamma^{j} (\alpha^{2}+\gamma^{2})^{b-j} 2 \bigg( c_{0} \binom{j}{\frac{j}{2}} \mathbf{1}_{j \, \mathrm{even}} + \sum_{\ell=1}^{b} 2 c_{\ell} \binom{j}{\frac{j+\ell}{2}} \mathbf{1}_{j+\ell \, \mathrm{even}} \bigg) \\
& - (\alpha^{2}-\gamma^{2}) \frac{b}{2} \sum_{j=0}^{b-1} \binom{2b-1}{2j}\binom{2j}{j} \alpha^{2j}\gamma^{2j}(\alpha^{2}+\gamma^{2})^{2b-1-2j} \bigg\},
\end{align*}
where $c_{0},\ldots,c_{b}$ are given by \eqref{def of ej}, $\alpha = \frac{a+c}{2}$, $\gamma=\frac{a-c}{2}$, $\mathbf{1}_{j \, \mathrm{even}}:=1$ if $j$ is even and $\mathbf{1}_{j \, \mathrm{even}}:=0$ otherwise. In particular,
\begin{align*}
& C|_{b=1} = \frac{\beta (\alpha^{2}-\gamma^{2})^{3}}{8(\alpha^{2}+\gamma^{2})}, \qquad C|_{b=2} = \frac{\beta (\alpha^{4}-\gamma^{4})^{3}}{4(\alpha^{4}+\gamma^{4})}, \\
& C|_{b=3}  =  3\beta\frac{\alpha^{16} + \gamma^{16}  +  15(\alpha^{14}\gamma^{2}  + \alpha^{2}\gamma^{14})  +  46(\alpha^{12}\gamma^{4}  +  \alpha^{4}\gamma^{12})  +  43(\alpha^{10}\gamma^{6}  +  \alpha^{6}\gamma^{10})  +  42\alpha^{8}\gamma^{8}}{8(\alpha^{2}  -  \gamma^{2})^{-3}(\alpha^{10} + \alpha^{6}\gamma^{4} + \alpha^{4}\gamma^{6} + \gamma^{10} )}.
\end{align*}
\end{theorem}

\subsubsection{The square}
Recall that $\mu$ and $S$ are given by \eqref{mu S ML}.

\medskip Theorems \ref{thm:ML rectangle nu} and \ref{thm:ML rectangle C} below state results about the balayage measure $\mathrm{Bal}(\mu|_{U},\partial U)$ and the hole probability $\mathbb{P}(\# \{z_{j}\in U\} = 0)$ in the case where $U\subset S$ is an arbitrary rectangle. The formulas in the statements of Theorems \ref{thm:ML rectangle nu} and \ref{thm:ML rectangle C} are rather long. Corollaries \ref{coro:ML square nu} and \ref{coro:ML square C} are about the special case where $U$ is a square centered at $0$. This special case is worth being stated separately, because the corresponding formulas for $\mathrm{Bal}(\mu|_{U},\partial U)$ and $\mathbb{P}(\# \{z_{j}\in U\} = 0)$ are much simpler (due to significant simplifications).
\begin{corollary}\label{coro:ML square nu}
Let $b\in \N_{>0}$, $c>0$ be such that $U:=\{z:\re z \in (-\frac{c}{2},\frac{c}{2}), \im z \in (-\frac{c}{2},\frac{c}{2})\} \subset S$, and let $\nu = \mathrm{Bal}(\mu|_{U},\partial U)$. For $z=\frac{c}{2}+iy$, $y \in (-\frac{c}{2},\frac{c}{2})$, we have
\begin{align}\label{density square general b}
d\nu(z) & = \sum_{m=0}^{+\infty} \mathcal{C}_{b,m} \cos\bigg( \frac{y}{c}(1+2m)\pi \bigg)dy, 
\end{align}
where
\begin{align}
\mathcal{C}_{b,m} = & \frac{16 b^{2} c^{2b-1}}{\pi^{3} 2^{2b}}(-1)^{m} \sum_{\ell=0}^{b-1} \binom{b-1}{\ell} \sum_{v_{1}=0}^{\ell} \sum_{v_{2}=0}^{b-1-\ell} \frac{(2\ell)!}{(2\ell-2v_{1})!} \frac{(2b-2-2\ell)!}{(2b-2-2\ell-2v_{2})!} \frac{2^{2v_{1}+2v_{2}}(-1)^{v_{2}}}{\pi^{2v_{1}+2v_{2}}} \nonumber \\
\times & \bigg( \mathbf{1}_{v_{1}+v_{2} \, \mathrm{even}} \frac{\tanh(\frac{(1+2m)\pi}{2})}{(1+2m)^{2+2v_{1}+2v_{2}}} + \sum_{q=0}^{v_{1}-1} \frac{4(-1)^{v_{1}+q}}{\pi(1+2m)^{3+2v_{2}+2q}} \bigg( 1-\frac{1}{2^{2v_{1}-2q}} \bigg)\zeta(2v_{1}-2q) \bigg) \label{def Cbm}
\end{align}
and $d\nu(i^{j}z)=d\nu(z)$ for all $j\in \{0,1,2,3\}$ and $z\in \frac{c}{2}+ i (-\frac{c}{2}, \frac{c}{2})$, see also Figure \ref{fig:density square}. In \eqref{def Cbm}, $\zeta$ is the Riemann zeta function (see e.g. \cite[Chapter 25]{NIST} for the definition of $\zeta$).
\end{corollary}
\begin{remark}
Note that formula \eqref{def Cbm} only involves the Riemann zeta function at positive even integer, for which we have the explicit formula (see e.g. \cite[25.6.2]{NIST})
\begin{align*}
\zeta(2m) = \frac{(2\pi)^{2m}}{(2m)!2}|B_{2m}|, \qquad m \in \N_{>0},
\end{align*}
where $B_{m}$ are the Bernoulli numbers defined via the generating function (see e.g. \cite[24.2.1]{NIST}) 
\begin{align*}
\frac{t}{e^{t}-1} = \sum_{m=0}^{+\infty} B_{m} \frac{t^{m}}{m!}, \qquad |t|<2\pi.
\end{align*}
\end{remark}

\begin{figure}
\begin{center}
\hspace{-1cm}\begin{tikzpicture}[master]
\node at (0,0) {\includegraphics[width=3.4cm]{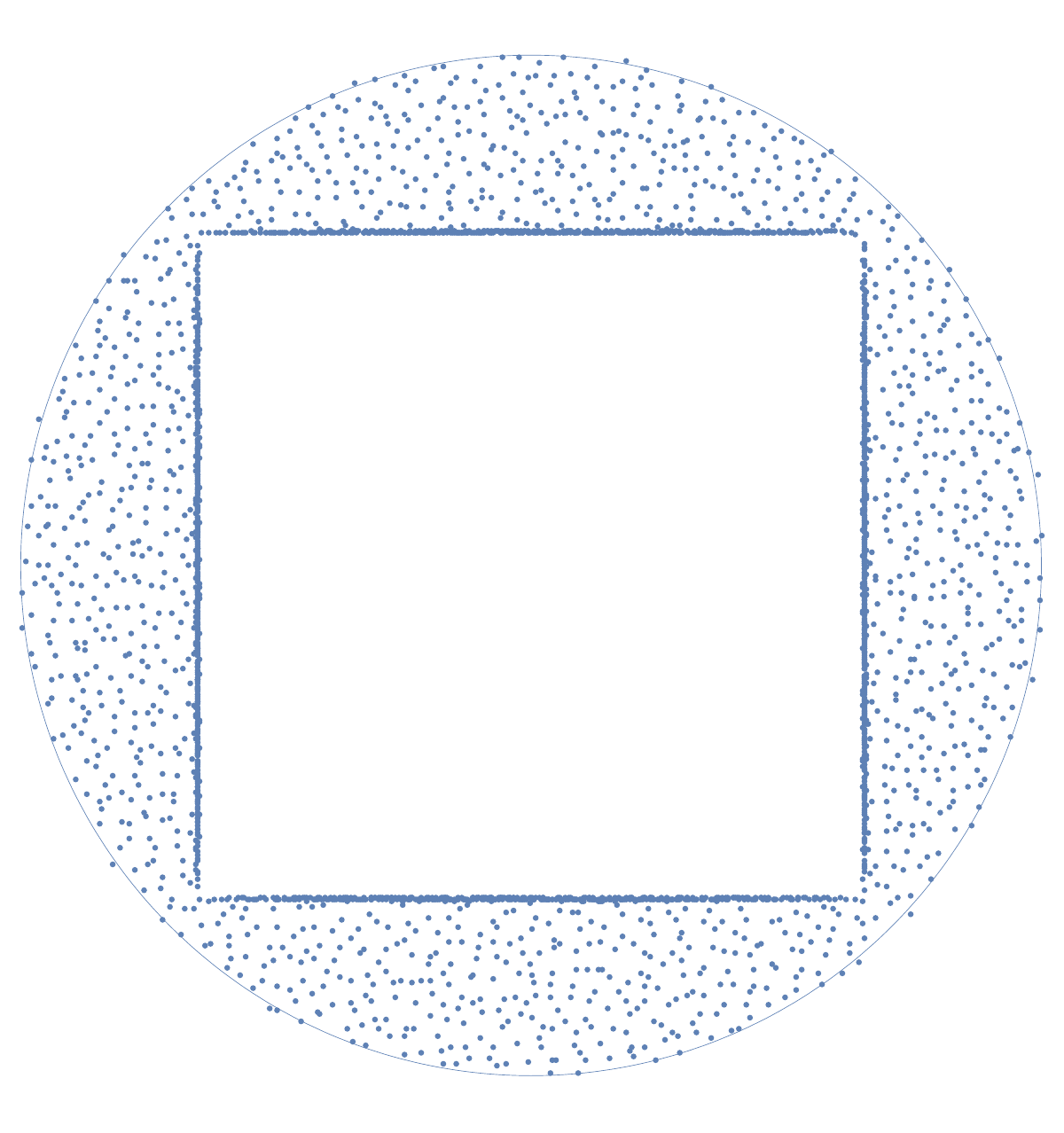}};
\node at (0,1.8) {\footnotesize $b=1$};
\end{tikzpicture}\hspace{-0.4cm}
\begin{tikzpicture}[slave]
\node at (0,0) {\includegraphics[width=3.4cm]{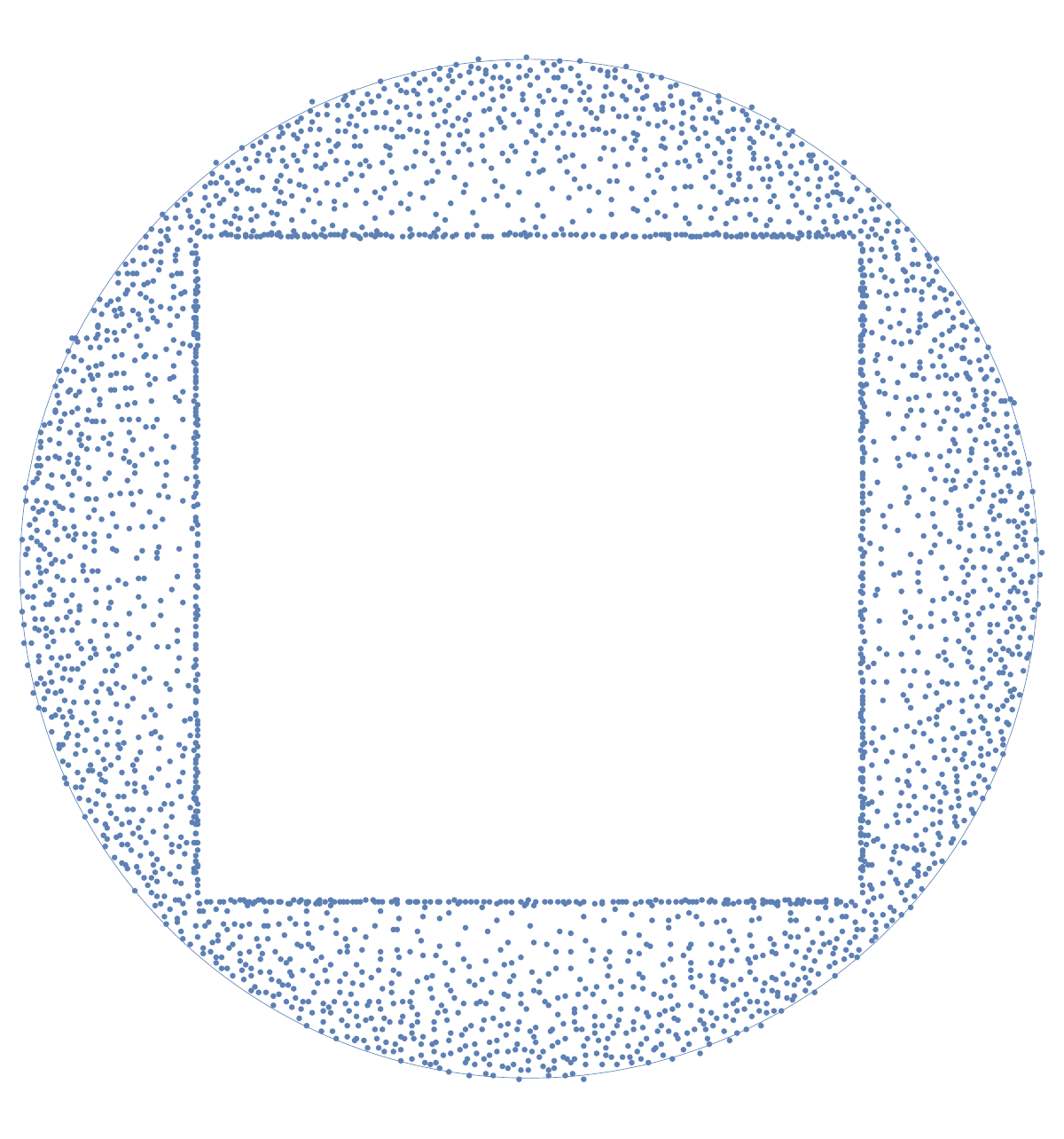}};
\node at (0,1.8) {\footnotesize $b=3$};
\end{tikzpicture}\hspace{-0.4cm}
\begin{tikzpicture}[slave]
\node at (0,0) {\includegraphics[width=3.4cm]{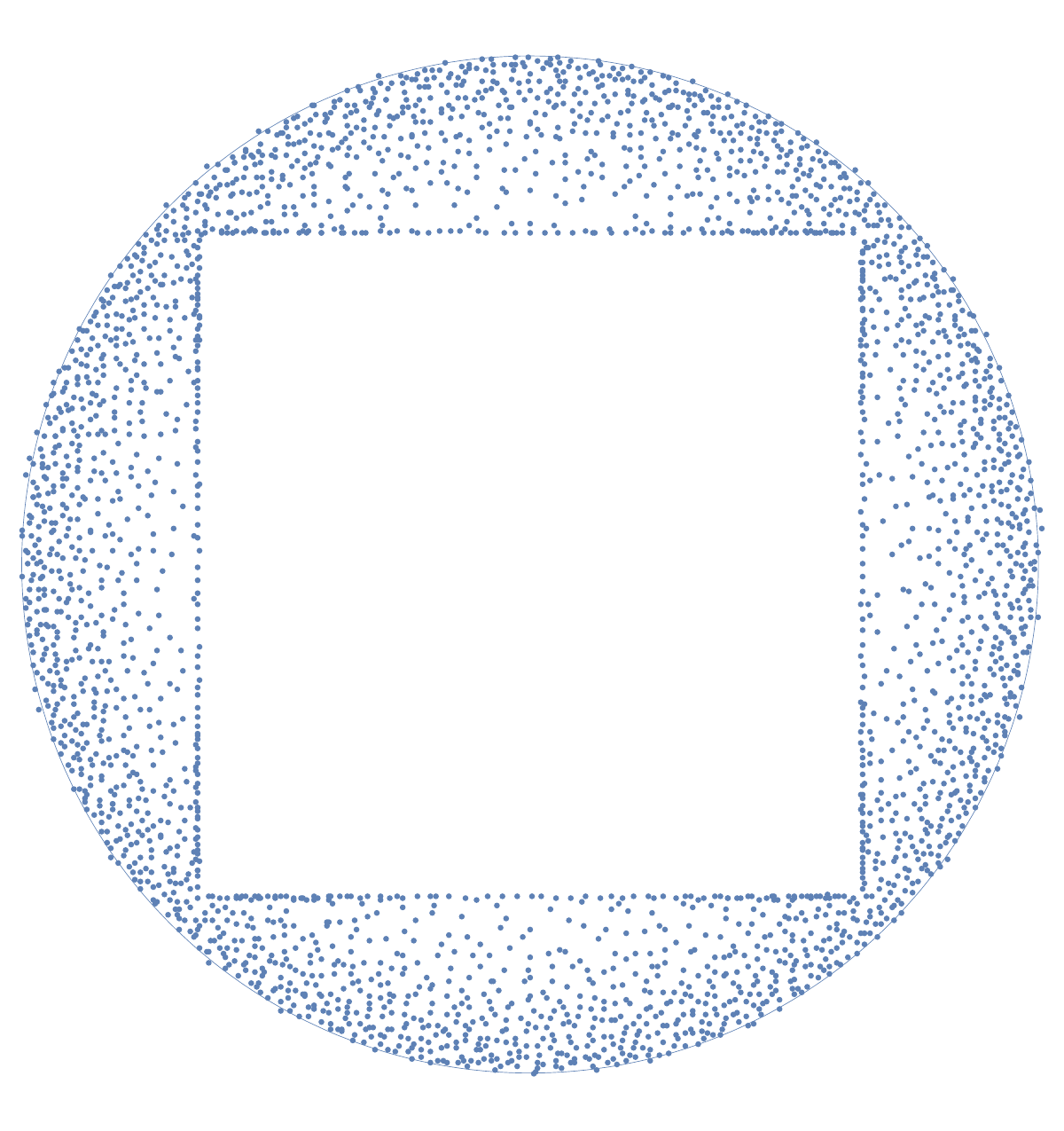}};
\node at (0,1.8) {\footnotesize $b=4$};
\end{tikzpicture}\hspace{0.5cm}
\begin{tikzpicture}[slave]
\node at (0,0) {\includegraphics[width=4.9cm]{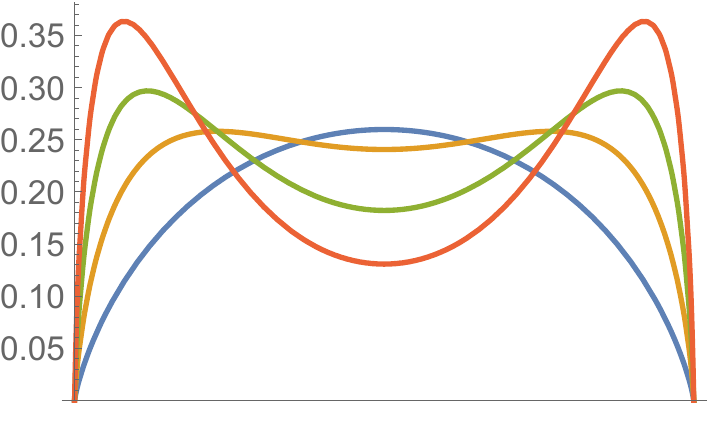}};
\node at (0.2,0.7) {\footnotesize $b \hspace{-0.05cm}=\hspace{-0.05cm}1$};
\node at (0.2,0.3) {\footnotesize $b \hspace{-0.05cm}=\hspace{-0.05cm}2$};
\node at (0.2,-0.1) {\footnotesize $b \hspace{-0.05cm}=\hspace{-0.05cm}3$};
\node at (0.2,-0.45) {\footnotesize $b \hspace{-0.05cm}=\hspace{-0.05cm}4$};
\node at (0.2,1.8) {\footnotesize Corollary \ref{coro:ML square nu}};
\draw[fill] (-1.925,-1.315) circle (0.04);
\node at (-1.925,-1.55) {\footnotesize  $\frac{c}{2}\hspace{-0.05cm}-\hspace{-0.05cm}i\frac{c}{2}$};
\draw[fill] (2.33,-1.315) circle (0.04);
\node at (2.33,-1.55) {\footnotesize  $\frac{c}{2}\hspace{-0.05cm}+\hspace{-0.05cm}i\frac{c}{2}$};
\end{tikzpicture}
\end{center}
\caption{\label{fig:density square} Left: The Mittag-Leffler ensemble conditioned on $\# \{z_{j}\in U\} = 0$ with $U=\{z:\re z \in (a_{1},a_{2}), \im z \in (c_{1}, c_{2})\}$, $a_{2}=c_{2}=-a_{1}=-c_{1}=\frac{c}{2}$, $c=1.3b^{-\frac{1}{2b}}$, and $b\in \{1,3,4\}$. Right: the normalized density $z \mapsto \frac{d\nu(z)/|dz|}{\nu(\partial U)}$ with $z\in \frac{c}{2}+i(-\frac{c}{2},\frac{c}{2}) \subset \partial U$, $c=1.3b^{-\frac{1}{2b}}$ and $b=1,2,3,4$.}
\end{figure}

\begin{corollary}\label{coro:ML square C}
Fix $\beta>0$, and let $b\in \N_{>0}$, $c>0$ be such that $U:=\{z:\re z \in (-\frac{c}{2},\frac{c}{2}), \im z \in (-\frac{c}{2},\frac{c}{2})\} \subset S$. As $n \to +\infty$, we have $\mathbb{P}(\# \{z_{j}\in U\} = 0) = \exp \big( -C n^{2}+o(n^{2}) \big)$, where
\begin{align*}
& C = \frac{\beta b^{2}c^{4b}}{4^{2b}\pi} \bigg\{ \frac{32}{\pi^{3}} \sum_{j=0}^{b} \binom{b}{j} \sum_{v=0}^{j} \frac{(2j)!(-1)^{v}4^{v}}{(2j-2v)!\pi^{2v}} \sum_{\ell=0}^{b-1} \binom{b-1}{\ell} \sum_{v_{1}=0}^{\ell}\sum_{v_{2}=0}^{b-1-\ell} \frac{(2\ell)!}{(2\ell-2v_{1})!} \nonumber \\
& \times \frac{(2b-2-2\ell)!}{(2b-2-2\ell-2v_{2})!} \frac{4^{v_{1}+v_{2}}(-1)^{v_{2}}}{\pi^{2v_{1}+2v_{2}}} \bigg\{ \mathbf{1}_{v_{1}+v_{2}\, \mathrm{even}} T_{2(v_{1}+v_{2}+v)} + \sum_{q=0}^{v_{1}-1} \frac{4(-1)^{v_{1}+q}}{\pi}\bigg( 1- \frac{1}{4^{2+v_{2}+q+v}} \bigg) \nonumber \\
& \times \zeta\big(2(2+v_{2}+q+v)\big) \bigg( 1- \frac{1}{4^{v_{1}-q}} \bigg) \zeta\big(2(v_{1}-q)\big) \bigg\} -  \sum_{j=0}^{2b-1} \frac{\binom{2b-1}{j}}{(2j+1)(4b-1-2j)} \bigg\},
\end{align*}
where the $T_{v}$'s are defined in \eqref{def of Tv}, and $\mathbf{1}_{v \, \mathrm{even}}:=1$ if $v$ is even and $\mathbf{1}_{v \, \mathrm{even}}:=0$ otherwise. In particular, substituting also the simplified expressions for $\{T_{4v}\}_{v\in \N}$ from Theorem \ref{thm: some nice series}, we get
\begin{align*}
& C|_{b=1} = \frac{\beta c^{4}}{\pi}\bigg( \frac{1}{12} - 16 \frac{T_{2}}{\pi^{5}} \bigg) \approx (1.1187 \cdot 10^{-2}) \beta c^{4}, \quad C|_{b=2} = \frac{\beta c^{8}}{\pi}\bigg( \frac{1}{5} - 64 \frac{T_{2}}{\pi^{5}} \bigg) \approx (2.3057 \cdot 10^{-3}) \beta c^{8}, \\
& C|_{b=3} = \frac{\beta c^{12}}{\pi} \bigg( \frac{125501}{73920} - 81 \frac{T_{2}}{\pi^{5}} - 25 920 \frac{T_{6}}{\pi^{9}} - 2073600 \frac{T_{10}}{\pi^{13}} \bigg) \approx (4.2438 \cdot 10^{-4}) \beta c^{12}, \\
& C|_{b=4} = \frac{\beta c^{16}}{\pi} \bigg( \frac{2152349}{45045} - 64 \frac{T_{2}}{\pi^{5}} - 184320 \frac{T_{6}}{\pi^{9}} - 132710400 \frac{T_{10}}{\pi^{13}} \bigg) \approx (8.2742 \cdot 10^{-5}) \beta c^{16}.
\end{align*}
\end{corollary}

\subsubsection{The rectangle}

Recall that $\mu$ and $S$ are given by \eqref{mu S ML}. We now consider the case where $U\subset S$ is an arbitrary rectangle.

\begin{theorem}\label{thm:ML rectangle nu}
Let $b \in \N_{>0}$, $a_{2}>a_{1}$, $c_{2}>c_{1}$ be such that $U:=\{z:\re z \in (a_{1},a_{2}), \im z \in (c_{1},c_{2})\} \subset S$, and let $\nu = \mathrm{Bal}(\mu|_{U},\partial U)$. 

\medskip  \noindent (Right side.) For $z=a_{2}+iy$, $y\in (c_{1},c_{2})$, we have
\begin{align}\label{square right}
d\nu(z)= \sum_{m=1}^{+\infty} \mathcal{C}_{b,m}^{R}\sin \bigg( \frac{y-c_{1}}{c_{2}-c_{1}} m \pi \bigg) dy,
\end{align}
where
\begin{align*}
& \mathcal{C}_{b,m}^{R} = \frac{-4b^{2}}{(c_{2}-c_{1})\pi^{2}} \sum_{\ell=0}^{b-1}  \binom{b-1}{\ell} \bigg\{ \sum_{v_{1}=0}^{\ell} \frac{(2\ell)!(a_{2}-a_{1})^{2v_{1}+1}}{(2\ell-2v_{1})!\pi^{2v_{1}+1}} \bigg( a_{2}^{2\ell-2v_{1}} \bigg[  \frac{\frac{(a_{2}-a_{1})m}{(c_{2}-c_{1})}\pi \coth (\frac{(a_{2}-a_{1})m}{(c_{2}-c_{1})}\pi) - 1}{2(\frac{(a_{2}-a_{1})m}{(c_{2}-c_{1})})^{2v_{1}+2}} \\
& + \sum_{j=0}^{v_{1}-1} \frac{(-1)^{v_{1}+j}}{(\frac{(a_{2}-a_{1})m}{(c_{2}-c_{1})})^{2+2j}}\zeta(2v_{1}-2j) \bigg] -a_{1}^{2\ell-2v_{1}} \bigg[  \frac{\frac{(a_{2}-a_{1})m}{(c_{2}-c_{1})}\pi \frac{1}{\sinh (\frac{(a_{2}-a_{1})m}{(c_{2}-c_{1})}\pi)} - 1}{2(\frac{(a_{2}-a_{1})m}{(c_{2}-c_{1})})^{2v_{1}+2}} \\
& + \sum_{j=0}^{v_{1}-1} \frac{(-1)^{v_{1}+j}}{(\frac{(a_{2}-a_{1})m}{(c_{2}-c_{1})})^{2+2j}} \bigg( \frac{1}{2^{2v_{1}-2j-1}}-1 \bigg)\zeta(2v_{1}-2j) \bigg] \bigg) \bigg\} \times \bigg\{ \sum_{v_{2}=0}^{b-1-\ell} \frac{(2(b-1-\ell))!(c_{2}-c_{1})^{2v_{2}+1}}{(2(b-1-\ell-v_{2}))!\pi^{2v_{2}+1}} \\
& \times (-1)^{v_{2}} \bigg( \frac{(-1)^{m}}{m^{2v_{2}+1}} c_{2}^{2(b-1-\ell-v_{2})}-\frac{1}{m^{2v_{2}+1}} c_{1}^{2(b-1-\ell-v_{2})} \bigg) \bigg\}.
\end{align*}
(Left side.) For $z=a_{1}+iy$, $y\in (c_{1},c_{2})$, we have
\begin{align}\label{square left}
d\nu(z)= \sum_{m=1}^{+\infty} \mathcal{C}_{b,m}^{L}\sin \bigg( \frac{y-c_{1}}{c_{2}-c_{1}} m \pi \bigg) dy,
\end{align}
where
\begin{align*}
& \mathcal{C}_{b,m}^{L} = \frac{4b^{2}}{(c_{2}-c_{1})\pi^{2}} \sum_{\ell=0}^{b-1}  \binom{b-1}{\ell} \bigg\{ \sum_{v_{1}=0}^{\ell} \frac{(2\ell)!(a_{2}-a_{1})^{2v_{1}+1}}{(2\ell-2v_{1})!\pi^{2v_{1}+1}} \bigg( a_{2}^{2\ell-2v_{1}} \bigg[  \frac{\frac{(a_{2}-a_{1})m}{(c_{2}-c_{1})}\pi \frac{1}{\sinh (\frac{(a_{2}-a_{1})m}{(c_{2}-c_{1})}\pi)} - 1}{2(\frac{(a_{2}-a_{1})m}{(c_{2}-c_{1})})^{2v_{1}+2}} \\
& + \sum_{j=0}^{v_{1}-1} \frac{(-1)^{v_{1}+j}}{(\frac{(a_{2}-a_{1})m}{(c_{2}-c_{1})})^{2+2j}} \bigg( \frac{1}{2^{2v_{1}-2j-1}}-1 \bigg) \zeta(2v_{1}-2j) \bigg] -a_{1}^{2\ell-2v_{1}} \bigg[  \frac{\frac{(a_{2}-a_{1})m}{(c_{2}-c_{1})}\pi \coth (\frac{(a_{2}-a_{1})m}{(c_{2}-c_{1})}\pi) - 1}{2(\frac{(a_{2}-a_{1})m}{(c_{2}-c_{1})})^{2v_{1}+2}} \\
& + \sum_{j=0}^{v_{1}-1} \frac{(-1)^{v_{1}+j}}{(\frac{(a_{2}-a_{1})m}{(c_{2}-c_{1})})^{2+2j}}  \zeta(2v_{1}-2j) \bigg] \bigg) \bigg\} \times \bigg\{ \sum_{v_{2}=0}^{b-1-\ell} \frac{(2(b-1-\ell))!(c_{2}-c_{1})^{2v_{2}+1}}{(2(b-1-\ell-v_{2}))!\pi^{2v_{2}+1}} \\
& \times (-1)^{v_{2}} \bigg( \frac{(-1)^{m}}{m^{2v_{2}+1}} c_{2}^{2(b-1-\ell-v_{2})}-\frac{1}{m^{2v_{2}+1}} c_{1}^{2(b-1-\ell-v_{2})} \bigg) \bigg\}.
\end{align*}
(Top side.) For $z=x+ic_{2}$, $x\in (a_{1},a_{2})$, we have
\begin{align}\label{square top}
d\nu(z)= \sum_{m=1}^{+\infty} \mathcal{C}_{b,m}^{T}\sin \bigg( \frac{x-a_{1}}{a_{2}-a_{1}} m \pi \bigg)  dx,
\end{align}
where
\begin{align*}
& \mathcal{C}_{b,m}^{T} = \frac{-4b^{2}}{(a_{2}-a_{1})\pi^{2}} \sum_{\ell=0}^{b-1}  \binom{b-1}{\ell} \bigg\{ \sum_{v_{1}=0}^{\ell} \frac{(2\ell)!(a_{2}-a_{1})^{2v_{1}+1}}{(2\ell-2v_{1})!\pi^{2v_{1}+1}}(-1)^{v_{1}} \bigg( \frac{(-1)^{m}}{m^{2v_{1}+1}}a_{2}^{2\ell-2v_{1}} - \frac{a_{1}^{2\ell-2v_{1}}}{m^{2v_{1}+1}} \bigg) \bigg\} \\
& \times \bigg\{ \sum_{v_{2}=0}^{b-1-\ell} \frac{(2(b-1-\ell))!(c_{2}-c_{1})^{2v_{2}+1}}{(2(b-1-\ell-v_{2}))!\pi^{2v_{2}+1}} \bigg( c_{2}^{2(b-1-\ell-v_{2})} \bigg[ \frac{\frac{(c_{2}-c_{1})m}{(a_{2}-a_{1})}\pi \coth(\frac{(c_{2}-c_{1})m}{(a_{2}-a_{1})}\pi)-1}{2(\frac{(c_{2}-c_{1})m}{(a_{2}-a_{1})})^{2v_{2}+2}} \\
& + \sum_{j=0}^{v_{2}-1} \frac{(-1)^{v_{2}+j}}{(\frac{(c_{2}-c_{1})m}{(a_{2}-a_{1})})^{2+2j}}\zeta(2v_{2}-2j) \bigg] - c_{1}^{2(b-1-\ell-v_{2})} \bigg[ \frac{\frac{(c_{2}-c_{1})m}{(a_{2}-a_{1})}\pi \frac{1}{\sinh(\frac{(c_{2}-c_{1})m}{(a_{2}-a_{1})}\pi)}-1}{2(\frac{(c_{2}-c_{1})m}{(a_{2}-a_{1})})^{2v_{2}+2}} \\
& + \sum_{j=0}^{v_{2}-1} \frac{(-1)^{v_{2}+j}}{(\frac{(c_{2}-c_{1})m}{(a_{2}-a_{1})})^{2+2j}} \bigg( \frac{1}{2^{2v_{2}-2j-1}} - 1 \bigg) \zeta(2v_{2}-2j) \bigg] \bigg) \bigg\}.
\end{align*}
(Bottom side.) For $z=x+ic_{1}$, $x\in (a_{1},a_{2})$, we have
\begin{align}\label{square bottom}
d\nu(z)= \sum_{m=1}^{+\infty} \mathcal{C}_{b,m}^{B}\sin \bigg( \frac{x-a_{1}}{a_{2}-a_{1}} m \pi \bigg) dx,
\end{align}
where
\begin{align*}
& \mathcal{C}_{b,m}^{B} = \frac{4b^{2}}{(a_{2}-a_{1})\pi^{2}} \sum_{\ell=0}^{b-1}  \binom{b-1}{\ell} \bigg\{ \sum_{v_{1}=0}^{\ell} \frac{(2\ell)!(a_{2}-a_{1})^{2v_{1}+1}}{(2\ell-2v_{1})!\pi^{2v_{1}+1}}(-1)^{v_{1}} \bigg( \frac{(-1)^{m}}{m^{2v_{1}+1}}a_{2}^{2\ell-2v_{1}} - \frac{a_{1}^{2\ell-2v_{1}}}{m^{2v_{1}+1}} \bigg) \bigg\} \\
& \times \bigg\{ \sum_{v_{2}=0}^{b-1-\ell} \frac{(2(b-1-\ell))!(c_{2}-c_{1})^{2v_{2}+1}}{(2(b-1-\ell-v_{2}))!\pi^{2v_{2}+1}} \bigg( c_{2}^{2(b-1-\ell-v_{2})} \bigg[ \frac{\frac{(c_{2}-c_{1})m}{(a_{2}-a_{1})}\pi \frac{1}{\sinh(\frac{(c_{2}-c_{1})m}{(a_{2}-a_{1})}\pi)}-1}{2(\frac{(c_{2}-c_{1})m}{(a_{2}-a_{1})})^{2v_{2}+2}} \\
& + \sum_{j=0}^{v_{2}-1} \frac{(-1)^{v_{2}+j}}{(\frac{(c_{2}-c_{1})m}{(a_{2}-a_{1})})^{2+2j}} \bigg( \frac{1}{2^{2v_{2}-2j-1}} - 1 \bigg) \zeta(2v_{2}-2j) \bigg] - c_{1}^{2(b-1-\ell-v_{2})} \bigg[ \frac{\frac{(c_{2}-c_{1})m}{(a_{2}-a_{1})}\pi \coth(\frac{(c_{2}-c_{1})m}{(a_{2}-a_{1})}\pi)-1}{2(\frac{(c_{2}-c_{1})m}{(a_{2}-a_{1})})^{2v_{2}+2}} \\
& + \sum_{j=0}^{v_{2}-1} \frac{(-1)^{v_{2}+j}}{(\frac{(c_{2}-c_{1})m}{(a_{2}-a_{1})})^{2+2j}}\zeta(2v_{2}-2j)      \bigg] \bigg) \bigg\}.
\end{align*}
\end{theorem}

\begin{figure}
\begin{center}
\begin{tikzpicture}[master]
\node at (0,0) {\includegraphics[width=3.63cm]{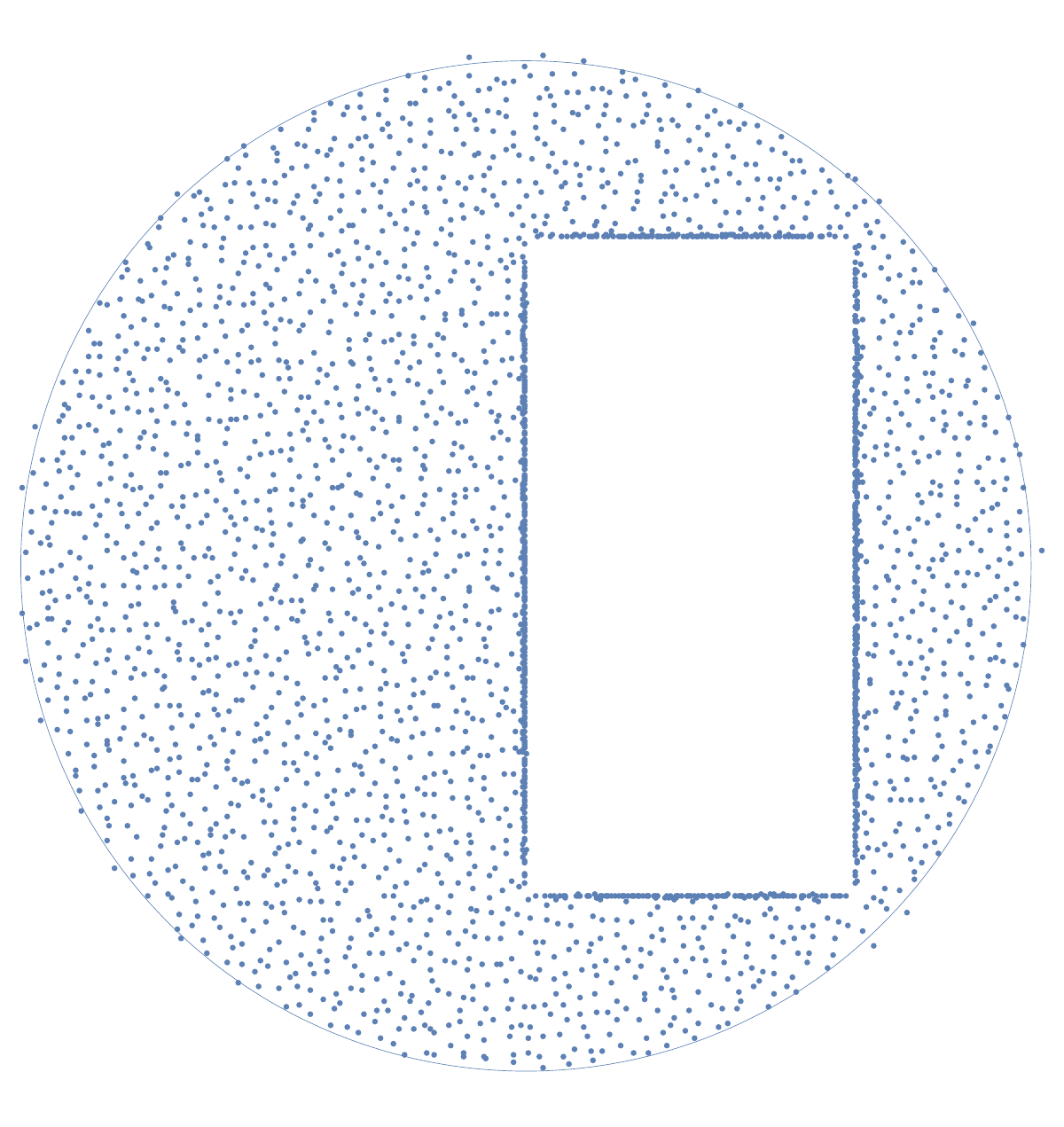}};
\node at (0,1.9) {\footnotesize $b=1$};
\end{tikzpicture}\hspace{-0.21cm}
\begin{tikzpicture}[slave]
\node at (0,0) {\includegraphics[width=3.63cm]{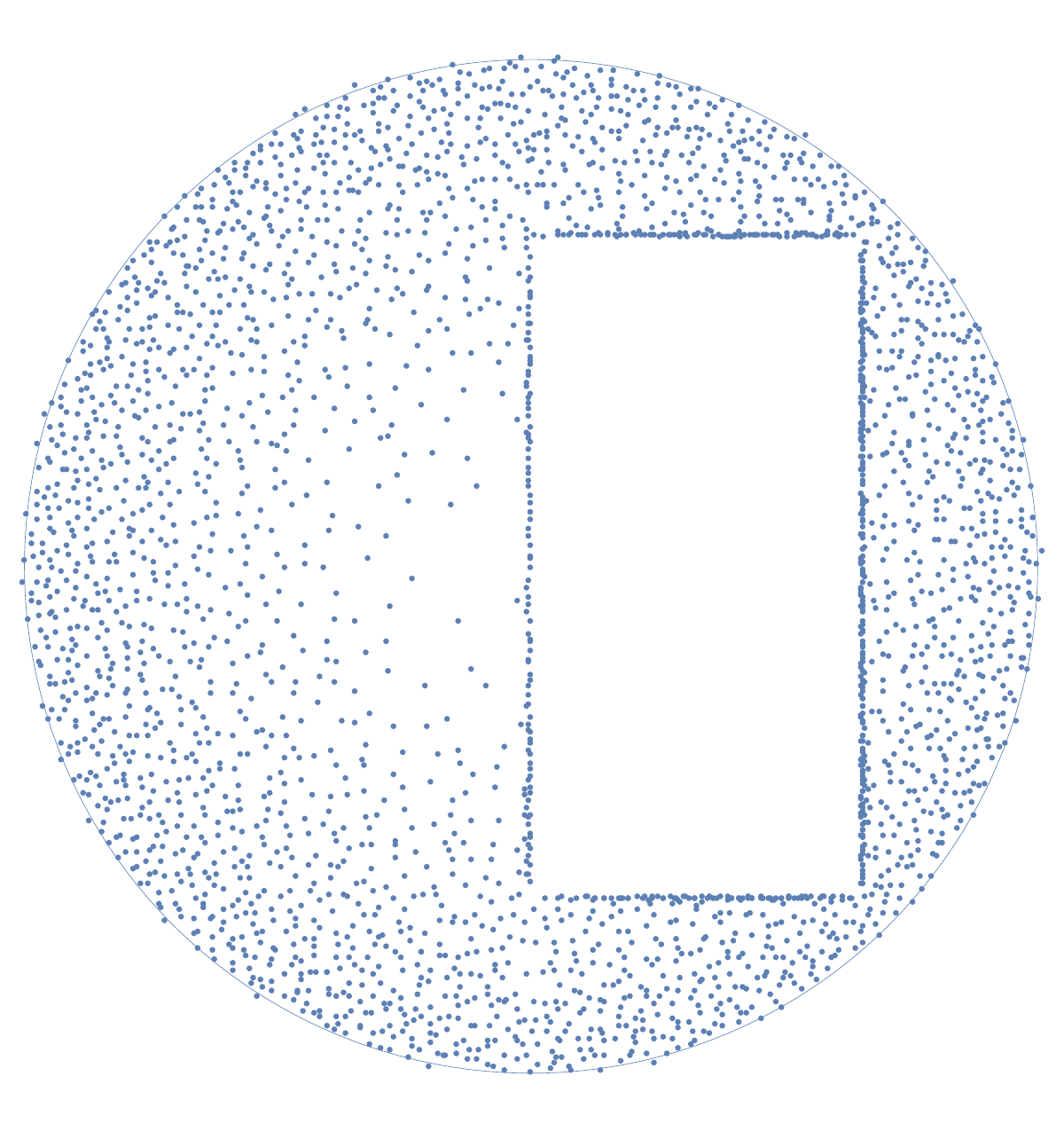}};
\node at (0,1.9) {\footnotesize $b=2$};
\end{tikzpicture}\hspace{-0.21cm}
\begin{tikzpicture}[slave]
\node at (0,0) {\includegraphics[width=3.63cm]{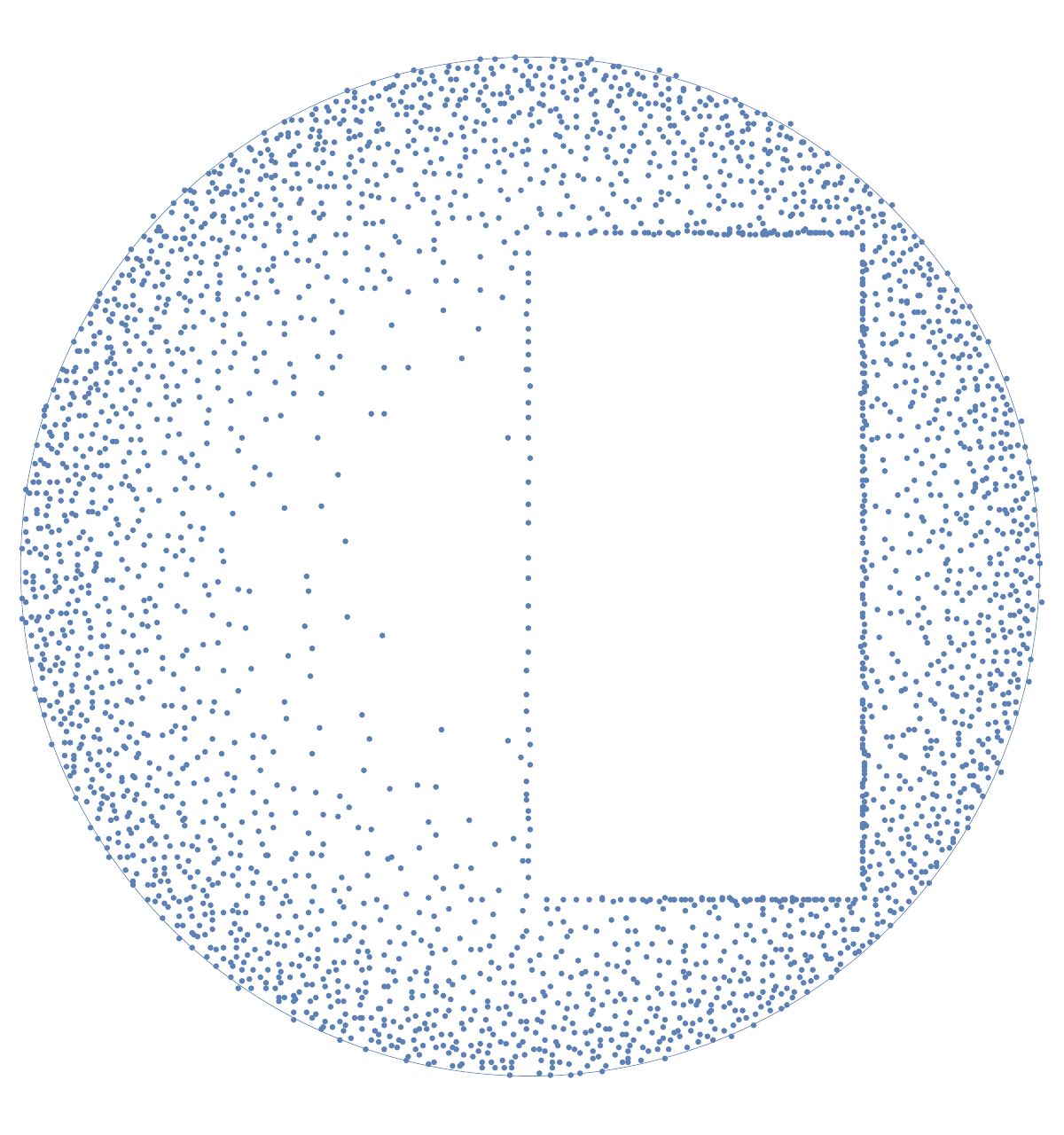}};
\node at (0,1.9) {\footnotesize $b=3$};
\end{tikzpicture}\hspace{-0.21cm}
\begin{tikzpicture}[slave]
\node at (0,0) {\includegraphics[width=3.63cm]{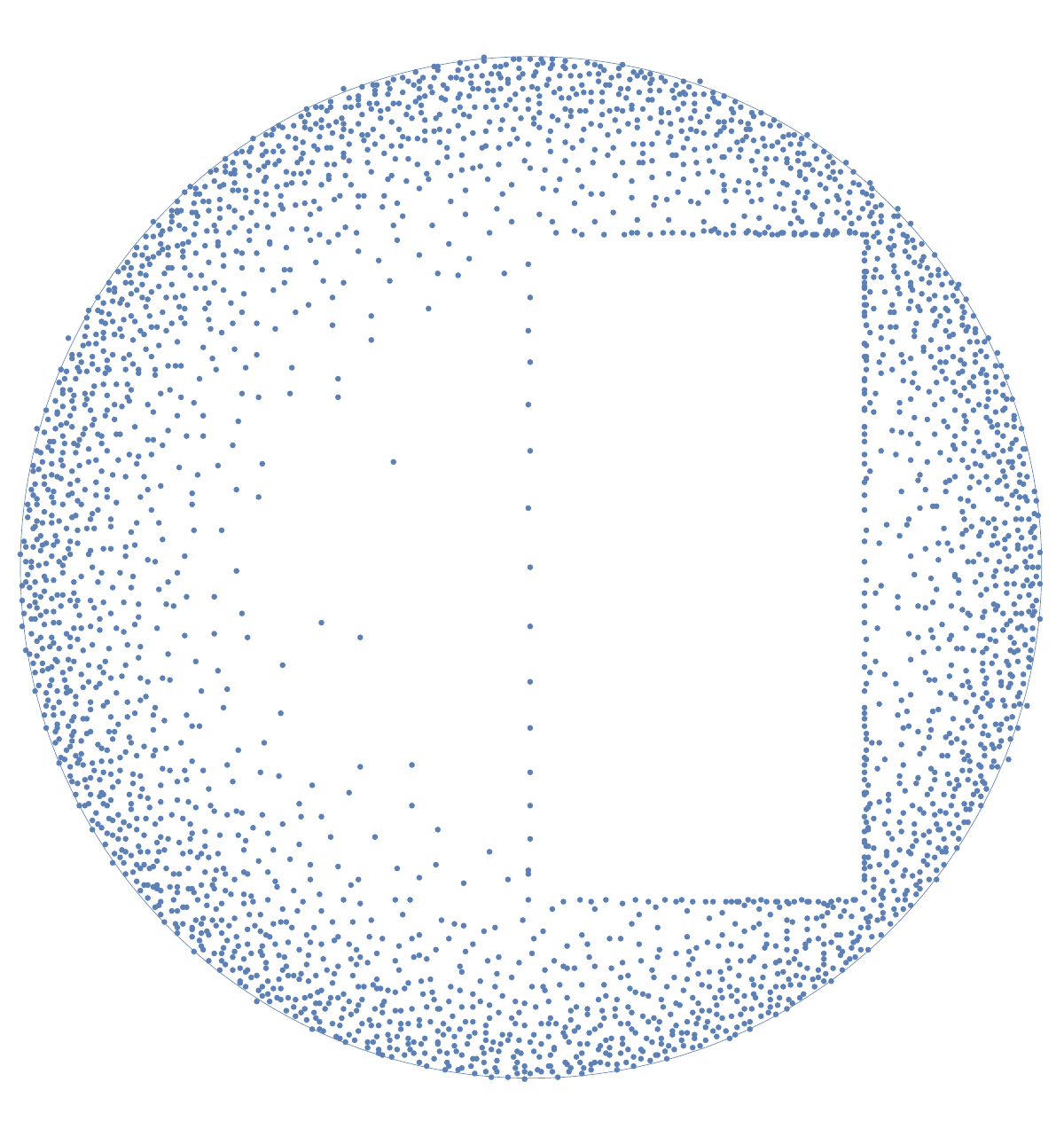}};
\node at (0,1.9) {\footnotesize $b=4$};
\end{tikzpicture}\\
\begin{tikzpicture}[master]
\node at (0,0) {\includegraphics[width=14cm]{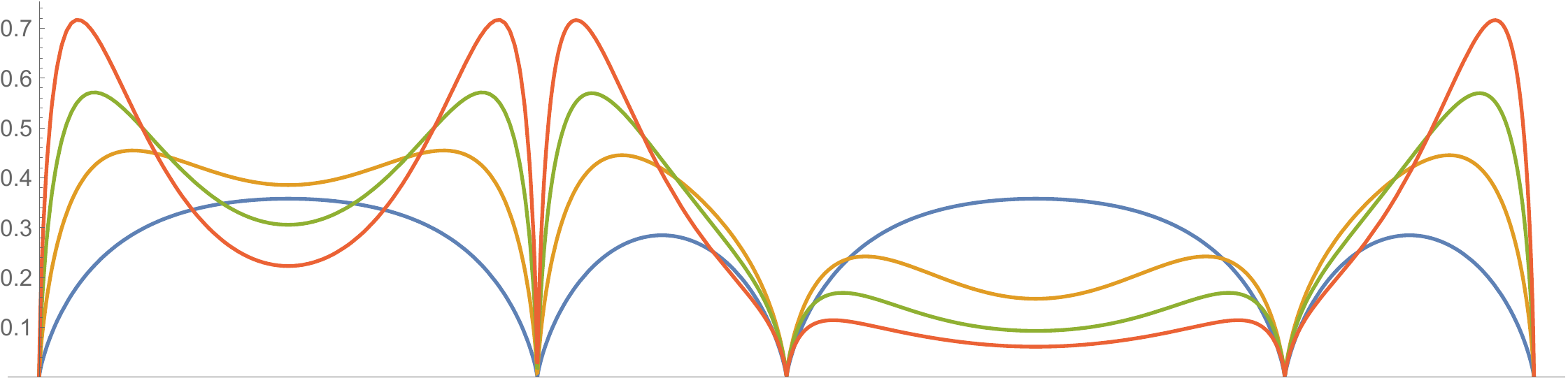}};
\node at (2.3,0.15) {\footnotesize $b \hspace{-0.05cm}=\hspace{-0.05cm}1$};
\node at (2.3,-0.75) {\footnotesize $b \hspace{-0.05cm}=\hspace{-0.05cm}2$};
\node at (2.3,-1.1) {\footnotesize $b \hspace{-0.05cm}=\hspace{-0.05cm}3$};
\node at (2.3,-1.45) {\footnotesize $b \hspace{-0.05cm}=\hspace{-0.05cm}4$};
\node at (0.2,1.8) {\footnotesize Theorem \ref{thm:ML rectangle nu}};
\draw[fill] (-6.65,-1.625) circle (0.04);
\node at (-6.5,-1.93) {\footnotesize  $a_{2}\hspace{-0.05cm}+\hspace{-0.05cm}ic_{1}$};
\draw[fill] (-2.205,-1.625) circle (0.04);
\node at (-2.205,-1.93) {\footnotesize  $a_{2}\hspace{-0.05cm}+\hspace{-0.05cm}ic_{2}$};
\draw[fill] (0.025,-1.625) circle (0.04);
\node at (0.025,-1.93) {\footnotesize  $a_{1}\hspace{-0.05cm}+\hspace{-0.05cm}ic_{2}$};
\draw[fill] (4.475,-1.625) circle (0.04);
\node at (4.475,-1.93) {\footnotesize  $a_{1}\hspace{-0.05cm}+\hspace{-0.05cm}ic_{1}$};
\draw[fill] (6.7,-1.625) circle (0.04);
\node at (6.55,-1.93) {\footnotesize  $a_{2}\hspace{-0.05cm}+\hspace{-0.05cm}ic_{1}$};
\end{tikzpicture}
\end{center}
\caption{\label{fig:density rectangle} Row 1: The Mittag-Leffler ensemble conditioned on $\# \{z_{j}\in U\} = 0$ with $U=\{z:\re z \in (a_{1},a_{2}), \im z \in (c_{1}, c_{2})\}$, $a_{1}=0$, $a_{2}=c_{2}=-c_{1}=\frac{c}{2}$, $c=1.3b^{-\frac{1}{2b}}$, and $b\in \{1,2,3,4\}$. Row 2: the normalized density $z \mapsto \frac{d\nu(z)/|dz|}{\nu(\partial U)}$ with $z\in \partial U$ and $b=1,2,3,4$.}
\end{figure}

\begin{remark}
In the case where $a_{1}=0$, $a_{2}=c_{2}=-c_{1}$, we have $0\in \partial U$. Since $b \geq 1$, by \eqref{mu S ML} we expect $d\nu(z)/|dz|$ to attain a local minimum at $z=0$, and the numerics suggest that $(d\nu(z)/|dz|)|_{z=0}>0$ (see also Figure \ref{fig:density rectangle}). This strict inequality is consistent with Conjecture \ref{conj:behavior of balayage measure} (with $p=2$ and $b\in \N_{>0}$). 
\end{remark}

\begin{theorem}\label{thm:ML rectangle C}
Fix $\beta>0$. Let $b \in \N_{>0}$ and $a_{2}>a_{1}$, $c_{2}>c_{1}$ be such that $U:=\{z:\re z \in (a_{1},a_{2}), \im z \in (c_{1},c_{2})\} \subset S$. As $n \to +\infty$, we have $\mathbb{P}(\# \{z_{j}\in U\} = 0) = \exp \big( -C n^{2}+o(n^{2}) \big)$, where
\begin{align}
& C = \frac{\beta}{4} \bigg\{ \sum_{k=0}^{b} \binom{b}{k} \sum_{v=0}^{k} \frac{(2k)!(-1)^{1+v}}{(2k-2v)!}  \bigg[ \bigg( \frac{c_{2}-c_{1}}{\pi} \bigg)^{2v+1} \nonumber \\
& \times \sum_{m=1}^{+\infty} \Big( a_{2}^{2(b-k)} \mathcal{C}_{b,m}^{R} + a_{1}^{2(b-k)} \mathcal{C}_{b,m}^{L}   \Big) \bigg( c_{2}^{2k-2v} \frac{(-1)^{m}}{m^{2v+1}} - \frac{c_{1}^{2k-2v}}{m^{2v+1}} \bigg) \nonumber \\
& + \bigg( \frac{a_{2}-a_{1}}{\pi} \bigg)^{2v+1} \sum_{m=1}^{+\infty} \Big( c_{2}^{2(b-k)} \mathcal{C}_{b,m}^{T}  + c_{1}^{2(b-k)} \mathcal{C}_{b,m}^{B}   \Big) \bigg( a_{2}^{2k-2v} \frac{(-1)^{m}}{m^{2v+1}} - \frac{a_{1}^{2k-2v}}{m^{2v+1}} \bigg) \bigg] \nonumber \\
& - \frac{b^{2}}{\pi} \sum_{j=0}^{2b-1} \binom{2b-1}{j} \frac{a_{2}^{2j+1}-a_{1}^{2j+1}}{2j+1} \frac{c_{2}^{2(2b-1-j)+1}-c_{1}^{2(2b-1-j)+1}}{2(2b-1-j)+1} \bigg\}, \label{lol76}
\end{align}
where the coefficients $\mathcal{C}_{b,m}^{R}$, $\mathcal{C}_{b,m}^{L}$, $\mathcal{C}_{b,m}^{T}$ and $\mathcal{C}_{b,m}^{B}$ are defined in the statement of Theorem \ref{thm:ML rectangle nu}. 
\end{theorem}

\subsubsection*{Related works and discussion}
Hole probabilities of rare events as considered in this paper are also called \textit{large gaps} in the literature. In dimension one, large gap problems have been investigated by many authors, see e.g. \cite{DIZ1997, FK2020} for the sine point process, \cite{DXZ2020} for the Pearcey point process, \cite{MNSV2011, DS2017, CGS2019} for the eigenvalues of some Hermitian random matrices, and \cite{KPTTZ, BMS2023} for the real eigenvalues of some non-Hermitian random matrices. For a longer (but still far from exhaustive) list of references, see the introduction of \cite{C2021}. 

\medskip The works \cite{A2018, AR2017, AZ2015, APS2009, C2021, CMV2016, ForresterHoleProba, GHS1988, JLM1993, L et al 2019} have already been mentioned earlier in the introduction and treat large gap problems of two-dimensional Coulomb gases. We also mention the works \cite{AS2013, AIS2014} on hole probabilities for the eigenvalues of products of non-Hermitian random matrices. It is proved in \cite{C2021} that when $U$ is a union of annuli, $Q(z) = |z|^{2b}+\frac{2\alpha}{n} \log |z|$ for some $b>0$, $\alpha>-1$, and $\beta=2$, the large $n$ asymptotics of $\mathbb{P}(\# \{z_{j}\in U\} = 0)$, where $\mathbb{P}$ refers to \eqref{general density intro}, are of the form
\begin{align}\label{precise asymp}
\exp \bigg( C_{1} n^{2} + C_{2} n \log n + C_{3} n +  C_{4} \sqrt{n} + C_{5}\log n + C_{6} + \mathcal{F}_{n} + \bigO\big( n^{-\frac{1}{12}}\big)\bigg),
\end{align}
for some explicit constants $C_{1},\ldots,C_{6}$ and where $\mathcal{F}_{n}$ is of order $1$. Extending this result to general rotation-invariant potentials is still an open problem (see however \cite{ABE2023, ABES2023, BKS2023} for related results). More generally, when $U$ and $Q$ are such that Assumptions \ref{ass:Q}, \ref{ass:U} and \ref{ass:U2} hold, we expect that the large $n$ asymptotics of $\mathbb{P}(\# \{z_{j}\in U\} = 0)$ are also of the form \eqref{precise asymp}. However, even in the determinantal case $\beta=2$, we are not aware of a work where subleading terms have been obtained for hole probabilities when either $Q$ or $U$ are not rotation-invariant.

\medskip For clarity, let us denote the partition function (=normalization constant) of \eqref{general density intro} by $Z_{n}^{Q}$, where the dependence in $Q$ has been made explicit. It is easy to see that
\begin{align}\label{lol93}
\mathbb{P}(\# \{z_{j}\in U\} = 0) = \frac{Z_{n}^{\tilde{Q}}}{Z_{n}^{Q}}, \qquad \mbox{where } \quad  \tilde{Q}(z) := \begin{cases}
Q(z), & \mbox{if } z \notin U, \\
+\infty, & \mbox{if } z \in U.
\end{cases}
\end{align}
(Note that $Z_{n}^{\tilde{Q}}$ can also be interpreted as the partition function of the point process \eqref{general density intro} associated with the potential $Q$ and conditioned on the hole event $\#\{z_{j}\in U\}=0$.) It follows from \eqref{lol93} that large gap problems can also be seen as asymptotic problems for partition functions with hard edges along $\partial U$ and with non-zero potentials (see e.g. \cite{Seo, HW2020, ACCL1, ACCL2, ACC2023} for studies on hard edges).  The method developed in e.g. \cite{BBLM2015, BGM2017, LY2023, BKP2023} uses matrix-valued Riemann-Hilbert problems and allows to obtain precise asymptotics for partition functions in the soft edge case where $Q$ is a Gaussian weight perturbed by a finite number of point charges. Developing an analogous Riemann-Hilbert method to treat a large gap problem where either $Q$ or $U$ is not rotation-invariant would be of high interest.

\medskip It is worth to mention that the difficulties encountered when analyzing asymptotics of partition functions with hard edges and soft edges are drastically different. Such differences can already be seen at the level of the equilibrium measure:
\begin{itemize}[leftmargin=4.5mm]
\item In the hard edge cases considered here, by Proposition \ref{prop:general pot}, the relevant equilibrium measure is $\mu|_{S\setminus U} + \nu$. Determining the support of $\mu|_{S\setminus U} + \nu$ is trivial (it is obviously $(S\setminus U) \cup \partial U = S\setminus U$), and the main challenge is to determine the density of $\nu$.
\item In soft edge cases such as in \cite{CK2022, Byun}, the difficulties are exactly the opposite: determining the support of the equilibrium measure is challenging, but the density of the equilibrium measure is trivial (when $Q$ is smooth the density is simply $d\mu(z) = \frac{\Delta Q(z)}{4\pi} d^{2}z$, see \cite[Theorem II.1.3]{SaTo}).
\end{itemize}
 
\medskip Another interesting line of research, which is not studied in this work, concerns large gap problems of two-dimensional Coulomb gases in the case where $\emptyset \subsetneq U \cap \partial S \subsetneq \partial S$. This problem was considered in \cite{ASZ2014} for the Gaussian potential $Q(z) = |z|^{2}$, where several properties of the equilibrium measure are given. However, to our knowledge, even in the case where $Q(z) = |z|^{2}$ and $U$ is a half-plane, the equilibrium measure for $Q$ on $U^{c}$ is not explicitly known. 
 
\medskip Large gap problems in dimension two have been studied for other models than the Coulomb gas model, such as the zeros of random entire functions \cite{Nishry}. In this model, conditioning on the hole event always produces a macroscopic gap outside the hole region \cite{GN2019} (in contrast with the two-dimensional Coulomb gas), and the associated conditional equilibrium measure is studied in \cite{NW2023}. We also refer to \cite{GN2018} for a recent survey on the hole event for two-dimensional point processes. 

\subsubsection*{Outline of the rest of the paper}

\begin{itemize}[leftmargin=4mm]
\item In Section \ref{section: general Q and U}, we prove Theorem \ref{thm:general pot} and Proposition \ref{prop:general pot}.
\item \vspace{-0.2cm} In Section \ref{section: method to obtain balayage}, we describe several methods that will be used to obtain explicit balayage measures.
\item \vspace{-0.2cm} In Section \ref{section:general rotation-invariant pot}, we prove Theorems \ref{thm:centered disk}, \ref{thm:regular annulus}, \ref{thm:unbounded annulus}, \ref{thm:g sector nu}, \ref{thm:g sector C} for general rotation-invariant potentials.
\item \vspace{-0.2cm} In Section \ref{section:elliptic Ginibre and several hole regions}, we prove Theorems \ref{thm:general U EG}, \ref{thm:Elliptic disk C} and \ref{thm:Ginibre triangle C}. Here $Q(z)=\frac{1}{1-\tau^{2}}\big( |z|^{2}-\tau \, \re z^{2} \big)$ and bounded hole regions $U$ are considered.
\item \vspace{-0.2cm} In Section \ref{section:disk}, we prove Theorems  \ref{thm:ML disk nu}, \ref{thm:ML disk C}. Here the hole region is a disk.
\item \vspace{-0.2cm} In Section \ref{section:ellipse centered}, we prove Theorems \ref{thm:ML ellipse nu}, \ref{thm:ML ellipse C}. Here the hole region is an ellipse centered at $0$.
\item \vspace{-0.2cm} In Section \ref{section:triangle}, we prove Theorem \ref{thm:Ginibre triangle nu}. Here the hole region is an equilateral triangle.
\item \vspace{-0.2cm} In Section \ref{section: rectangle}, we prove Theorems \ref{thm:Ginibre rectangle nu}, \ref{thm:ML rectangle nu}, \ref{thm:Ginibre rectangle C}, \ref{thm:ML rectangle C} and Corollaries \ref{coro:ML square nu}, \ref{coro:ML square C}. Here the hole region is a rectangle. We also prove Theorem \ref{thm: some nice series} which gives simple formulas for the series $T_{2v,\alpha}$.
\item \vspace{-0.2cm} In Section \ref{section:complement ellipse}, we prove Theorems \ref{thm:EG complement ellipse nu}, \ref{thm:EG complement ellipse C}. Here the hole region is the complement of an ellipse centered at $0$.
\item \vspace{-0.2cm} In Section \ref{section:complement disk}, we prove Theorems \ref{thm:EG complement disk nu}, \ref{thm:EG complement disk C}. Here the hole region is the complement of a disk.
\end{itemize}

\section{Hole probabilities for general potentials and hole regions}\label{section: general Q and U}
In this section, we prove our most general results Theorem \ref{thm:general pot} and Proposition \ref{prop:general pot}. The proof extends ideas from \cite{AR2017,A2018} to handle hole regions $U$ that are possibly unbounded, and potentials $Q$ satisfying the H\"{o}lder condition \eqref{holder} and that are not necessarily rotation-invariant. We will first prove Theorem \ref{thm:general pot} (i) and Proposition \ref{prop:general pot} (i) when $Q$ is admissible in the sense of Definition \ref{def:admissible}. Since the potential $\log(1+|z|^{2})$ is not admissible, the proofs of Theorem \ref{thm:general pot} (ii) and Proposition \ref{prop:general pot} (ii) are slightly different and we provide the details in Subsection \ref{subsection:spherical fekete}.

\subsection{Preliminaries}\label{subsection:prelim}

We start by briefly recalling well-known facts about Fekete points and balayage measures that will be used to prove Theorem \ref{thm:general pot} and Proposition \ref{prop:general pot}.

\subsubsection*{Fekete points}
For $n \in \N$, consider
\begin{align}\label{fekete delta n Q}
\delta_n^{Q}(E):=\sup_{z_1,z_2,\ldots,z_n\in E}\Bigg\{\prod_{1\leq j<k \leq n}|z_j-z_k|^{2}e^{-Q(z_j)}e^{-Q(z_k)}\Bigg\}^{\frac{1}{n(n-1)}}.
\end{align}
A set $\mathcal{F}_n = \{z_1^*,z_2^*,\ldots,z_n^* \}\subset E$ is said to be a Fekete set for $Q$ on $E$ if it minimizes \eqref{n tuple} under the constraint that $z_{1},\ldots,z_{n}\in E$, i.e. 
\begin{align}\label{fekete zstar}
\delta_n^{Q}(E)=\Bigg\{ \prod_{1\leq j<k \leq n}|z_j^{*}-z_k^{*}|^{2}e^{-Q(z_j^{*})}e^{-Q(z_k^{*})} 
\Bigg\}^{\frac{1}{n(n-1)}}.
\end{align}
\begin{remark}
If $\eta \in E$ and $\mathcal{F}_n = \{z_1^*,z_2^*,\ldots,z_n^* \}\subset E$ is a Fekete set for $Q$ on $E$, then by \eqref{fekete delta n Q}--\eqref{fekete zstar} we have
\begin{align*}
\prod_{k = 2}^{n} |\eta -z_k^{*}|^{2} e^{-Q(\eta)}e^{-Q(z_k^{*})} \prod_{2\leq j<k \leq n}|z_j^{*}-z_k^{*}|^{2}e^{-Q(z_j^{*})}e^{-Q(z_k^{*})}  \leq \prod_{1\leq j<k \leq n}|z_j^{*}-z_k^{*}|^{2}e^{-Q(z_j^{*})}e^{-Q(z_k^{*})},
\end{align*}
which yields after simplification
\begin{align}\label{fekete inequality}
e^{-(n-1)Q(\eta)} \prod_{k = 2}^{n} |\eta -z_k^{*}|^{2} \leq e^{-(n-1)Q(z_1^{*})}\prod_{k = 2}^{n} |z_1^{*}-z_k^{*}|^{2}.
\end{align}
\end{remark}

Finding $-\log \delta_n^{Q}(E)$ can be seen as a discrete version of the minimization problem $\inf_{\sigma \in \mathcal{P}(E)} I_{Q}[\sigma]$, where $I_{Q}[\sigma]$ is given by \eqref{def of Rmu}. In fact, if $Q$ is quasi-admissible on $E$,\footnote{A potential $Q$ is quasi-admissible on $E$ if it satisfies properties (i) and (iii) of Definition \ref{def:admissible}.} then by \cite[Theorems III.1.1 and III.1.3, (I.6.1) and (I.1.13)]{SaTo} the sequence $\{\delta_n^{Q}(E)\}_{n=2}^{\infty}$ decreases to $e^{-\mathcal{I}_{Q}[E]}$, where 
\begin{align}\label{def of Ical}
\mathcal{I}_{Q}[E]:=\inf_{\sigma \in \mathcal{P}(E)} I_{Q}[\sigma] = I_{Q}[\mu_{E}],
\end{align}
and $\mu_{E}$ is the equilibrium measure for $Q$ on $E$. In particular,
\begin{align}\label{eqn:limit}
\lim_{n \to \infty} \delta_n^{Q}(E) = e^{-\mathcal{I}_{Q}[E]} = e^{-I_{Q}[\mu_{E}]}.
\end{align}

\subsubsection*{Balayage measures}
The definition and some properties of the balayage measure are listed in Proposition \ref{prop:def of bal}. In what follows we will also need the following.
\begin{lemma}\label{lemma:further properties of nu}
Suppose that $U$, $\sigma$ and $c^{\hspace{0.02cm}\sigma}_{U}$ are as in Proposition \ref{prop:def of bal}. The measure $\nu = \mathrm{Bal}(\sigma,\partial U)$ satisfies the following properties:
\begin{enumerate}
\item[(i)] $p_{\nu}(z)=p_{\sigma}(z)+c^{\hspace{0.02cm}\sigma}_{U}$ for every $z \notin \overline{U}$ and for q.e. $z\in \partial U$.
\item[(ii)] $p_{\nu}(z)\leq p_{\sigma}(z)+c^{\hspace{0.02cm}\sigma}_{U}$ for every $z \in \C$.
\item[(iii)] If $h$ is a continuous function on $\overline{U}$ that is harmonic on $U$, then
\begin{align*}
\int_{\partial U} h d\nu = \int_{U} h d\sigma.
\end{align*}
\end{enumerate}
\end{lemma}
\begin{proof}
See \cite[Theorem II.4.7]{SaTo} (see also \cite[Theorems II.4.1 and II.4.4]{SaTo}).
\end{proof}

Following \cite{AR2017, A2018}, we will prove Theorem \ref{thm:general pot} (i) in two steps:
\begin{itemize}[leftmargin=*]
\item In Subsection \ref{subsection: eq measure}, we establish Proposition \ref{prop:general pot} (i), as well as the identity
\begin{align}\label{lol87}
\mathcal{I}_{Q}[U^{c}] - \mathcal{I}_{Q}[\C] = \frac{1}{2}\bigg(  \int_{\partial U}Q(z)d\nu(z) + 2c^{\hspace{0.02cm}\mu}_{U} - \int_{U}Q(z)d\mu(z) \bigg).
\end{align}
\item Next, in Subsection \ref{gap}, we prove that as $n\to +\infty$
\begin{align}\label{lol86}
\mathcal{P}_{n} = \exp \big( -Cn^{2}+o(n^{2}) \big), \qquad \mbox{where } \quad C=\frac{\beta}{2}(\mathcal{I}_{Q}[U^{c}] - \mathcal{I}_{Q}[\C]),
\end{align}
and where $\mathcal{P}_{n}:=\mathbb{P}(\# \{z_{j}\in U\} = 0)$. The proof of \eqref{lol86} is subdivided into two parts:
\begin{itemize}[leftmargin=*]
\item The following upper bounds are proved in Lemma \ref{lemma:upper bound}:
\begin{align}\label{upper bounds outline}
\hspace{-0.3cm} \limsup_{n\to \infty}\frac{1}{n^2}\log \mathcal{P}_{n} \leq -\frac{\beta}{2}\mathcal{I}_{Q}[U^{c}] - \liminf_{n\to \infty} \frac{1}{n^2}\log Z_{n} \quad \mbox{and} \quad \limsup_{n\to \infty}\frac{1}{n^2}\log Z_{n} \leq -\frac{\beta}{2}\mathcal{I}_{Q}[\C].
\end{align}
\item The following lower bounds are proved in Lemma \ref{lemma:lower bound}:
\begin{align}\label{lower bounds outline}
\hspace{-0.3cm} \liminf_{n\to \infty}\frac{1}{n^2}\log\mathcal{P}_{n} \geq - \frac{\beta}{2} \mathcal{I}_{Q}[U^{c}] - \limsup_{n\to\infty} \frac{1}{n^{2}}\log Z_{n} \quad \mbox{and} \quad \liminf_{n\to \infty}\frac{1}{n^2}\log Z_{n} \geq -\frac{\beta}{2}\mathcal{I}_{Q}[\C].
\end{align}
\end{itemize}
Taken together, the inequalities \eqref{upper bounds outline} and \eqref{lower bounds outline} imply \eqref{lol86}.
\end{itemize}
Theorem \ref{thm:general pot} (i) then directly follows by combining the results of the two steps outlined above.

\subsection{Proof of Proposition \ref{prop:general pot} (i)}\label{subsection: eq measure}
In this subsection, we prove Proposition \ref{prop:generalformula} (which, in turn, implies Proposition \ref{prop:general pot} (i)).
\begin{proposition}\label{prop:generalformula}
Let $Q$ be an admissible potential on $\C$. Let $\mu$ be the equilibrium measure for $Q$ on $\C$,  $S:=\mathrm{supp}\, \mu$, and let $U\subset \C$ be such that Assumption \ref{ass:U} holds.  Then the equilibrium measure for $Q$ on $U^c$ is $\nu+\mu|_{S\setminus U}$ and
\begin{align}\label{RU-Rempty}
\mathcal{I}_{Q}[U^{c}] - \mathcal{I}_{Q}[\C] = \frac{1}{2}\bigg(  \int_{\partial U}Q(z)d\nu(z) + 2c^{\hspace{0.02cm}\mu}_{U} - \int_{U}Q(z)d\mu(z) \bigg),
\end{align}
where $\nu = \mathrm{Bal}(\mu|_{U},\partial U)$, $c^{\hspace{0.02cm}\mu}_{U}=2\pi\int_{\Omega}g_{\Omega}(z,\infty)d\mu(z)$ and $\Omega$ is the unbounded component of $U$ ($c^{\hspace{0.02cm}\mu}_{U}=0$ if $U$ has no unbounded component). 
\end{proposition}
\begin{proof}
By Assumption \ref{ass:U}, $U$ is open, $\partial U$ is compact in $\C$, $\partial U \subset S$, and $\mathrm{cap}(\partial U)>0$. Since $Q$ is admissible on $\C$, this ensures that $\mathrm{cap}(\{z \in U^{c}: Q(z) < + \infty\})>0$, and thus $Q$ is admissible on $U^{c}$. 

Let $\mu$ be the equilibrium measure for $Q$ on $\C$, and define $\mu_{1}:= \mu|_{U^{c}}$ and $\mu_{2}:=\mu|_{U}$. By Proposition \ref{prop:eq measure}, $I_{0}[\mu]<+\infty$, and thus by \cite[Theorem 3.2.3]{RFbook} $\mu(P)=0$ for every Borel polar set $P$. 


By Proposition \ref{prop:def of bal}, the measure $\nu=\mathrm{Bal}(\mu_{2},\partial U)$ is supported on $\partial U$, and satisfies $\nu(\partial U)=\mu(U)$, $\nu(P)=0$ for every Borel polar set $P$ and 
\begin{align*}
p_{\nu}(z)=p_{\mu_2}(z)+c^{\hspace{0.02cm}\mu}_{U}\;\; \mbox{ for q.e. } z \in U^c.
\end{align*}
Define $\tilde{\mu}=\mu_{1}+\nu$. We have that
\begin{align}\label{lol17}
p_{\tilde{\mu}}(z)=p_{\mu_1}(z)+p_{\nu}(z)=p_{\mu_1}(z) + p_{\mu_2}(z) + c^{\hspace{0.02cm}\mu}_{U} = p_{\mu}(z) + c^{\hspace{0.02cm}\mu}_{U}\;\; \mbox{ for q.e. } z\in U^c.
\end{align}
(In fact, by Lemma \ref{lemma:further properties of nu} (i), \eqref{lol17} holds for all $z\in \overline{U}^{c}$ and for q.e. $z\in \partial U$.)
Since $\mu$ is the equilibrium measure for $Q$ on $\C$, by Proposition \ref{prop:eq measure} we have 
\begin{align}
2p_{\mu}(z)+Q(z) & = F \;\;\mbox{ for q.e. } z\in S, \label{lol14} \\ 
2p_{\mu}(z)+Q(z) & \geq F \;\;\mbox{ for q.e. } z \in S^c, \nonumber
\end{align}
for some constant $F\in \R$. Hence, by \eqref{lol17},
\begin{align}
2 p_{\tilde{\mu}}(z) - 2c^{\hspace{0.02cm}\mu}_{U} + Q(z) & = F \;\;\mbox{ for q.e. } z\in S \setminus U, \label{EL elliptic Ginibre} \\ 
2p_{\tilde{\mu}}(z) - 2c^{\hspace{0.02cm}\mu}_{U} + Q(z) & \geq F \;\;\mbox{ for q.e. } z \in S^{c}\setminus U. \label{lol16}
\end{align}
Since $\partial U \subset S$, we have $\mathrm{supp} \,\tilde{\mu} = S \setminus U$ (we also have either $S^{c}\setminus U=S^{c}$ or $S^{c}\setminus U=\emptyset$). Since $Q$ is lower-continuous on the compact set $S$, $Q$ is bounded from below on $S$, and by \cite[Proof of Theorem I.1.3 (b)]{SaTo} $Q$ is also bounded from above on $S$. Since $\tilde{\mu}$ is finite, we conclude that
\begin{align*}
-\infty < I_{0}[\tilde{\mu}] = \int p_{\tilde{\mu}}(z)d\tilde{\mu}(z) = \int \left(\frac{F+2c^{\hspace{0.02cm}\mu}_{U}}{2}-\frac{Q(z)}{2}\right) d\tilde{\mu}(z) < + \infty,
\end{align*}
where for the second equality we have used that $\tilde{\mu}(P)=0$ for all Borel polar sets $P$. So, $I_{0}[\tilde{\mu}]$ is finite and $\tilde{\mu}$ satisfies \eqref{EL elliptic Ginibre}, \eqref{lol16}, $\mathrm{supp} \,\tilde{\mu} = S \setminus U$ and $\tilde{\mu}(S \setminus U) = 1$. By Proposition \ref{prop:eq measure}, $\tilde{\mu}$ is the equilibrium measure for $Q$ on $U^c$. This already proves Proposition \ref{prop:general pot} (i), and it only remains to compute $\mathcal{I}_{Q}[U^{c}] - \mathcal{I}_{Q}[\C]$.

\medskip Since $\tilde{\mu}$ is the equilibrium measure for $Q$ on $U^c$ and satisfies \eqref{EL elliptic Ginibre}, by \eqref{def of Ical} we have
\begin{align}
\mathcal{I}_{Q}[U^{c}] = I_{Q}[\tilde{\mu}] & = \int p_{\tilde{\mu}}(z)d\tilde{\mu}(z)+\int Q(z) d\tilde{\mu}(z) = \frac{F+2c^{\hspace{0.02cm}\mu}_{U}}{2} + \frac{1}{2} \int Q(z) d\tilde{\mu}(z) \nonumber \\
& = \frac{F+2c^{\hspace{0.02cm}\mu}_{U}}{2} + \frac{1}{2} \int Q(z) d\mu_1(z) + \frac{1}{2} \int Q(z) d\nu(z) \nonumber \\
& = \frac{F+2c^{\hspace{0.02cm}\mu}_{U}}{2} + \frac{1}{2} \int Q(z)d\mu(z) +\frac{1}{2}\left[\int Q(z) d\nu(z) - \int Q(z) d\mu_2(z) \right]. \label{lol18}
\end{align}
On the other hand, using that $\mu$ is the equilibrium measure for $Q$ on $\C$ and satisfies \eqref{lol14}, we get
\begin{align}\label{lol19}
\mathcal{I}_{Q}[\C] = I_{Q}[\mu] = \int p_{\mu}(z)d\mu(z)+\int Q(z) d\mu(z) = \frac{F}{2} + \frac{1}{2} \int Q(z) d\mu(z).
\end{align}
Combining \eqref{lol18} and \eqref{lol19}, we obtain \eqref{RU-Rempty}.
\end{proof}
\begin{remark}
The following fact will be used later: if \eqref{regularity condition} holds, then following the proof of Proposition \ref{prop:generalformula} and using also Lemma \ref{lemma:further properties of nu} (i) we conclude that \eqref{lol16} holds with ``for q.e." replaced by ``for all", i.e. we have
\begin{align}\label{lol20}
2p_{\tilde{\mu}}(z) - 2c^{\hspace{0.02cm}\mu}_{U} + Q(z)  \geq F \qquad \mbox{ for all } z \in S^{c}\setminus U.
\end{align}
\end{remark}

\subsection{Leading order asymptotics for $\mathcal{P}_{n}$}\label{gap}
The goal of this subsection is to prove the following proposition. 
\begin{proposition}\label{prop:holeprobability1}
Fix $\beta >0$, and let $Q$ and $U$ be such that Assumptions \ref{ass:Q} and \ref{ass:U} hold. Then
\begin{align}\label{lol15}
\lim_{n\to \infty}\frac{1}{n^2}\log\mathcal{P}_{n}=-\frac{\beta}{2}\bigg(  \mathcal{I}_{Q}[U^{c}] - \mathcal{I}_{Q}[\C] \bigg).
\end{align}
\end{proposition}
Note that Theorem \ref{thm:general pot} (i) directly follows by combining Propositions \ref{prop:generalformula} and \ref{prop:holeprobability1}. As explained at the end of Subsection \ref{subsection:prelim}, we will prove Proposition \ref{prop:holeprobability1} in two steps. First, we establish the following lemma.

\begin{lemma}[Upper bounds]\label{lemma:upper bound}
Let $Q$ and $U$ be such that Assumptions \ref{ass:Q} and \ref{ass:U} hold. Then
\begin{align}
& \limsup_{n\to \infty}\frac{1}{n^2}\log \mathcal{P}_{n} \leq -\frac{\beta}{2}\mathcal{I}_{Q}[U^{c}] - \liminf_{n\to \infty} \frac{1}{n^2}\log Z_{n}, \label{inequality} \\
& \limsup_{n\to \infty}\frac{1}{n^2}\log Z_{n,\beta} \leq -\frac{\beta}{2}\mathcal{I}_{Q}[\C]. \label{inequality 2}
\end{align}
\end{lemma}
\begin{proof}
Since $\mathcal{P}_{n}=\mathbb{P}(\# \{z_{j}\in U\} = 0)$, where $\mathbb{P}$ refers to \eqref{general density intro}, we have
\begin{align}
\mathcal{P}_{n} & = \frac{1}{Z_n} \int_{U^c}\ldots \int_{U^c} \prod_{1 \leq j < k \leq n}|z_j-z_k|^{\beta} \prod_{j=1}^{n} e^{-n \frac{\beta}{2} Q(z_{j})} d^{2}z_{j} \nonumber \\
& = \frac{1}{Z_n}\int_{U^c}\ldots \int_{U^c}\Bigg\{ \prod_{1 \leq j<k \leq n}|z_j-z_k|^{2}e^{-Q(z_{j})}e^{-Q(z_{k})}\Bigg\}^{\frac{\beta}{2}} \prod_{j=1}^{n} e^{-\frac{\beta}{2}Q(z_{j})} d^{2}z_{j}. \label{eqn:feket}
\end{align}
Let $\{z_1^*,z_2^*,\ldots,z_n^*\}$ be a Fekete set for $Q$ on $U^c$, i.e. 
\begin{align*}
\delta_n^{Q}(U^c)=\Bigg\{
\prod_{1 \leq j < k \leq n}|z_j^*-z_k^*|^{2} e^{-Q(z_{j}^{*})}e^{-Q(z_{k}^{*})}\Bigg\}^{\frac{1}{n(n-1)}}.
\end{align*}
Then using \eqref{eqn:feket} we obtain
\begin{align}\label{lol88}
\mathcal{P}_{n} & \leq (\delta_n^{Q}(U^c))^{\frac{\beta}{2} n(n-1)} \frac{1}{Z_{n}} \left(\int_{U^c}e^{-Q(z)}d^{2}z \right)^{n}.
\end{align}
Since $Q$ is admissible on $U^{c}$ (this follows from Assumptions \ref{ass:Q} and \ref{ass:U} as explained in the proof of Proposition \ref{prop:generalformula}), \eqref{eqn:limit} holds with $E$ replaced by $U^{c}$, and thus by taking $\limsup_{n\to \infty}\frac{1}{n^2}\log(\cdot)$ on both sides of \eqref{lol88} we find \eqref{inequality}.

By a similar argument,
\begin{align}\label{lol89}
Z_{n} \leq (\delta_n^{Q}(\C))^{\frac{\beta}{2} n(n-1)} \left(\int_{U^c}e^{-Q(z)}d^{2}z \right)^{n}.
\end{align}
Since $Q$ is admissible on $\C$, \eqref{eqn:limit} holds with $E$ replaced by $\C$, and by taking $\limsup_{n\to \infty}\frac{1}{n^2}\log(\cdot)$ on both sides of \eqref{lol89} we find \eqref{inequality 2}.
\end{proof}

The next goal is to establish the complementary lower bounds \eqref{lower bounds outline}. For this, we will rely on the following lemma, whose proof is given after the proof of Lemma \ref{lemma:lower bound}.

\begin{lemma}\label{lemma:dist fekete}
Let $Q$ be such that Assumption \ref{ass:Q} holds, and let $U$ be either empty or such that Assumption \ref{ass:U} holds. Let $\mathcal{F}_{n}=\{z_1^*,z_2^*,\ldots,z_n^*\}$ be a Fekete set for $Q$ on $U^c$ such that $\mathcal{F}_{n} \subset S\setminus U$. Then, for all large enough $n$, 
\begin{align}\label{fekete lower bound}
\min\{|z_j^*-z_k^*|: 1\leq j\neq k \leq n\}\geq \frac{C_{1}}{n^{1+\frac{2}{\alpha}}}
\end{align}
for some constant $C_{1}>0$, and where $\alpha\in (0,1]$ is the constant appearing in \eqref{holder}.
\end{lemma}

\noindent Using Lemma \ref{lemma:dist fekete} we now prove the following. 
\begin{lemma}[Lower bounds]\label{lemma:lower bound}
Let $Q$ and $U$ be such that Assumptions \ref{ass:Q} and \ref{ass:U} hold. Then
\begin{align}
& \liminf_{n\to \infty}\frac{1}{n^2}\log\mathcal{P}_{n} \geq - \frac{\beta}{2} \mathcal{I}_{Q}[U^{c}] - \limsup_{n\to\infty} \frac{1}{n^{2}}\log Z_{n}, \label{lower bound} \\
& \liminf_{n\to \infty}\frac{1}{n^2}\log Z_{n} \geq -\frac{\beta}{2}\mathcal{I}_{Q}[\C]. \label{lower bound 2}
\end{align}
\end{lemma}
\begin{proof}
The proof uses ideas from \cite[Proof of Theorem 1.5]{AR2017} and \cite[Proof of (III)]{A2018}.
Recall from Proposition \ref{prop:generalformula} that $\tilde{\mu} := \nu+\mu|_{S\setminus U}$ is the equilibrium measure for $Q$ on $U^{c}$, where $\mu$ is the equilibrium measure for $Q$ on $\C$ and $\nu = \mathrm{Bal}(\mu|_{U},\partial U)$. Let $ \{ z_1^*,z_2^*,\ldots,z_n^*\}$ be a Fekete set for $Q$ on $U^c$. Since $Q$ is admissible on $U^{c}$, by \cite[Theorem III.2.8]{SaTo} we can (and do) choose $z_{\ell}^*  \in \mathrm{supp}(\tilde{\mu}) = S\setminus U$ for $\ell=1,2,\ldots, n$ (indeed, by \eqref{lol20}, $\{z\in S^{c}\setminus U: 2p_{\tilde{\mu}}(z) - 2c^{\hspace{0.02cm}\mu}_{U} + Q(z) < F\}=\emptyset$, and thus, in the notation of \cite[Theorem III.2.8]{SaTo}, we have $R_{w}\subset \mathcal{S}_{w}$). Thus
\begin{align}
\mathcal{P}_{n} & = \frac{1}{Z_n}\int_{U^c}\ldots \int_{U^c} \prod_{1 \leq j < k \leq n}|z_j-z_k|^{\beta} \prod_{j=1}^n
e^{-n \frac{\beta}{2} Q(z_{j})}d^{2}z_{j} \nonumber \\
& \geq \frac{1}{Z_n}\int_{B_1}\ldots \int_{B_n}
\Bigg\{ \prod_{1 \leq j<k \leq n}|z_j-z_k|^{2}e^{-Q(z_{j})}e^{-Q(z_{k})}\Bigg\}^{\frac{\beta}{2}} \prod_{j=1}^n e^{-\frac{\beta}{2} Q(z_{j})}d^{2}z_{j} \nonumber \\
& \geq \frac{e^{-\mathfrak{m}n}}{Z_n}\int_{B_1}\ldots \int_{B_n}
\Bigg\{ \prod_{1 \leq j<k \leq n}|z_j-z_k|^{2}e^{-Q(z_{j})}e^{-Q(z_{k})}\Bigg\}^{\frac{\beta}{2}} \prod_{j=1}^n d^{2}z_{j}, \label{lol23}
\end{align}
where $\mathfrak{m}:=\max\{Q(z): \dist(z,S)\leq \delta\}\in \R$, $\delta>0$ is as in \eqref{holder}, $B_\ell:=U^c\cap D(z_{\ell}^*,\frac{C_{1}}{n^{2+\frac{2}{\alpha}}})$ for $\ell=1,2,\ldots, n$, and $C_{1}>0$ is as in Lemma \ref{lemma:dist fekete}. By Lemma \ref{lemma:dist fekete}, the inequality $|z_j^*-z_k^*|\geq \frac{C_{1}}{n^{1+\frac{2}{\alpha}}}$ holds for all $1\leq j \neq k \leq n$. If $z_j \in D(z_j^*,\frac{C_{1}}{n^{2+\frac{2}{\alpha}}})$ and $z_k\in D(z_k^*,\frac{C_{1}}{n^{2+\frac{2}{\alpha}}})$ for $j\neq k$, then 
\begin{align}\label{lol21}
|z_j-z_k| \geq |z_j^*-z_k^*|-\frac{2C_{1}}{n^{2+\frac{2}{\alpha}}} \geq
|z_j^*-z_k^*| - \frac{2}{n}  |z_j^*-z_k^*| =
\Big(1-\frac{2}{n}\Big)|z_j^*-z_k^*|.
\end{align}
Also, if $z_j\in D(z_j^*,\frac{C_{1}}{n^{2+\frac{2}{\alpha}}})$, then by \eqref{holder} we have
\begin{align}\label{lol22}
e^{-Q(z_{j})} \geq e^{-Q(z_{j}^{*})} e^{-K|z_{j}-z_{j}^{*}|^{\alpha}} \geq e^{-Q(z_{j}^{*})} e^{-\frac{c_{1}}{n^{2+2\alpha}}}, \qquad \mbox{where } c_{1} := C_{1}^{\alpha}K>0.
\end{align}
Using \eqref{lol21} and \eqref{lol22} in \eqref{lol23}, we get
\begin{align*}
& \mathcal{P}_{n} \geq \frac{e^{-\mathfrak{m}n}}{Z_n}\int_{B_1} \ldots \int_{B_n} \Bigg\{ \prod_{1 \leq j<k \leq n} \left(1-\frac{2}{n}\right)^{2} |z_j^{*}-z_k^{*}|^{2}  e^{-Q(z_{j})}e^{-Q(z_{k})}\Bigg\}^{\frac{\beta}{2}} \prod_{j=1}^n d^{2}z_{j} \\
& \geq \frac{e^{-\mathfrak{m}n}}{Z_n} e^{-\frac{\beta}{2}\frac{c_{1}}{n^{2\alpha}}} \left(1-\frac{2}{n}\right)^{\frac{\beta}{2} n(n-1)} \Bigg\{ \prod_{1 \leq j<k \leq n}|z_j^{*}-z_k^{*}|^{2} e^{-Q(z_{j}^{*})}e^{-Q(z_{k}^{*})}\Bigg\}^{\frac{\beta}{2}} \prod_{j=1}^n\int_{B_j} d^{2}z_{j}.
\end{align*}
From \eqref{eqn:exterior ball condition lol}, we infer that
$\int_{B_j} d^{2}z_{j}\geq c_{2} (\frac{C_{1}}{n^{2+\frac{2}{\alpha}}})^2$ holds for some $c_{2}>0$, for all large enough $n$ and all $j\in\{1,2,\ldots,n\}$. Thus
\begin{align*}
\mathcal{P}_{n} \geq 
\frac{e^{-\mathfrak{m}n}}{Z_n} e^{-\frac{\beta}{2}\frac{c_{1}}{n^{2\alpha}}} \left(1-\frac{2}{n}\right)^{\frac{\beta}{2} n(n-1)} (\delta_n^{Q}(U^c))^{\frac{\beta}{2} n(n-1)} \bigg(\frac{
\sqrt{c_{2}}C_{1}}{n^{2+\frac{2}{\alpha}}}\bigg)^{2n}.
\end{align*}
Taking $\liminf_{n\to \infty}\frac{1}{n^2}\log(\cdot)$ on both sides and using \eqref{eqn:limit}, we find \eqref{lower bound}.

By a similar analysis (using now Lemma \ref{lemma:dist fekete} with $U=\emptyset$), 
\begin{align*}
Z_{n} \geq e^{-\mathfrak{m}n} e^{-\frac{\beta}{2}\frac{\tilde{c}_{1}}{n^{2\alpha}}} \left(1-\frac{2}{n}\right)^{\frac{\beta}{2} n(n-1)}    (\delta_n^{Q}(\C))^{\frac{\beta}{2} n(n-1)} \bigg(\frac{
\sqrt{\pi}\tilde{C}_{1}}{n^{2+\frac{2}{\alpha}}}\bigg)^{2n},
\end{align*}
for some $\tilde{C}_{1},\tilde{c}_{1}>0$, and \eqref{lower bound 2} follows.
\end{proof}
We now finish the proof of Proposition \ref{prop:holeprobability1}.
\begin{proof}[Proof of Proposition \ref{prop:holeprobability1}]
The inequalities \eqref{inequality 2} and \eqref{lower bound 2} imply that $\lim_{n\to \infty}\frac{1}{n^2}\log Z_{n} = -\frac{\beta}{2}\mathcal{I}_{Q}[\C]$. This, together with \eqref{inequality} and \eqref{lower bound}, directly yields \eqref{lol15}.
\end{proof}

It remains to prove Lemma \ref{lemma:dist fekete}.

\begin{proof}[Proof of Lemma \ref{lemma:dist fekete}]
The proof uses ideas from \cite[Proof of Lemma 5.1]{AR2017} and \cite[Proof of Lemma 4.2]{A2018}.
Suppose $|z_1^*-z_2^*|\leq \frac{1}{n^{2/\alpha}}$. We will prove that $|z_1^*-z_2^*| \geq \frac{e^{-2K-2}}{2  n^{1+\frac{2}{\alpha}}}$, where $K$ is as in \eqref{holder}. Consider $P(z)=(z-z_2^*)\cdots (z-z_n^*)$. By the residue theorem,
\begin{align*}
& |P(z_1^*)|  = |P(z_1^*)-P(z_2^*)|  = \bigg|\frac{1}{2\pi i}\int_{|w-z_1^*|=\frac{2}{n^{2/\alpha}}} \bigg( \frac{P(w)}{(w-z_1^*)} - \frac{P(w)}{(w-z_2^*)} \bigg) dw\bigg| \\
& \leq \frac{1}{2\pi}\int_{|w-z_1^*|=\frac{2}{n^{2/\alpha}}}\frac{|z_1^*-z_2^*| \, }{|w-z_1^*||w-z_2^*|}|P(w)|\, |dw| \\
& \leq \frac{1}{2\pi}  \frac{n^{2/\alpha}}{2} n^{2/\alpha} |z_1^*-z_2^*| \, |P(w^*)| \, 2\pi \frac{2}{n^{2/\alpha}},
\end{align*}
where $w^*\in \{w :|w-z_1^*|=\frac{2}{n^{2/\alpha}}\}$ is such that $|P(w^*)|=\sup\{|P(w)|:|z_1^*-w|=\frac{2}{n^{2/\alpha}}\}$, and where for the last inequality we have used that $|w-z_2^*| \geq \tfrac{1}{n^{2/\alpha}}$ for all $w \in \{w: |w - z_{1}^{*}|=\frac{2}{n^{2/\alpha}}\}$. Therefore we have
\begin{align}\label{eqn:z1}
|P(z_1^*)|\leq n^{2/\alpha}  |P(w^*)| \, |z_1^*-z_2^*|.
\end{align}
There are two cases to consider: $w^{*} \in U^{c}$ and $w^{*}\in U$. To treat both cases, we will need the following inequality: if $z,w\in S$ and $|z-w|\leq \frac{2}{n^{1/\alpha}}$, then by \eqref{holder} we have
\begin{align}\label{eqn:exp}
e^{-(n-1)Q(z)}\leq e^{-(n-1)Q(w)}e^{(n-1)K|z-w|^{\alpha}} \leq c_{2}e^{-(n-1)Q(w)}, \qquad \mbox{where } c_{2} := e^{2K}>0.
\end{align}

\noindent \underline{Case I: $w^*\in U^c$.} Since  $\{z_1^*,z_2^*,\ldots,z_n^*\}$ is a Fekete set for $Q$ on $U^c$, and since $w^*\in U^c$,  \eqref{fekete inequality} implies that
\begin{align*}
|P(w^{*})|e^{-(n-1)\frac{Q(w^{*})}{2}} \leq |P(z_{1}^{*})| e^{-(n-1)\frac{Q(z_1^{*})}{2}}.
\end{align*}
Using also \eqref{eqn:z1} and \eqref{eqn:exp}, we infer that
\begin{align*}
|P(z_1^*)|e^{-(n-1)\frac{Q(z_1^*)}{2}} & \leq n^{2/\alpha} |z_1^*-z_2^*| \, |P(w^*)| c_{2} e^{-(n-1)\frac{Q(w^*)}{2}} \\
& \leq c_{2} n^{2/\alpha} |z_1^*-z_2^*| \, |P(z_1^*)| \, e^{-(n-1)\frac{Q(z_1^*)}{2}},
\end{align*}
and thus $|z_1^*-z_2^*|\geq \frac{1}{c_{2} n^{2/\alpha}}$.

\noindent \underline{Case II: $w^*\in U$}. Since $z_{1}^{*}\in U^{c}$, we have $\mathrm{dist}(z_1^*,\partial U):=\inf \{|z-z_1^*|\, :\,z\in \partial U\}<|w^{*}-z_{1}^{*}|=\frac{2}{n^{2/\alpha}}$. Let $\epsilon_{0} >0$ be as in Assumption \ref{ass:U2}, and let $n$ be large enough such that $\frac{2}{n^{2/\alpha}}<\frac{1}{n^{1/\alpha}}\leq \epsilon_{0}$. Since $U$ satisfies Assumption \ref{ass:U2}, we can choose $\eta\in U^c$ such that  $D(\eta,\frac{1}{n^{1/\alpha}}) \subset U^c $ and $|\eta-z_{1}^{*}|\leq \frac{1}{n^{1/\alpha}}$. Since $P$ is a degree $n-1$ polynomial, we have
\begin{align}\label{eqn: this is zeta}
|P(\zeta^*)|\le
|P(\eta)|+ \sum_{j=1}^{n-1} \bigg|\frac{P^{(j)}(\eta)}{j!}\bigg| \, |\zeta^*-\eta|^{j}.
\end{align}
Using again the residue theorem, we obtain
\begin{align*}
\bigg|\frac{P^{(j)}(\eta)}{j!}\bigg|\leq
 \int_{|z-\eta|=\frac{1}{n^{1/\alpha}}}\frac{|P(z)|}{|z-\eta|^{j+1}}\frac{|dz|}{2\pi}\leq
n^{j/\alpha} \, |P(\eta^*)|, \quad \mbox{for all } j\in \{0,\ldots,n-1\},
\end{align*}
where $\eta^*\in \{z:|z-\eta|=\frac{1}{n^{1/\alpha}}\}$ is such that $|P(\eta^*)|=\sup\{|P(z)|:|z-\eta|=\frac{1}{n^{1/\alpha}}\}$. We thus obtain, for any $j\in \{0,\ldots,n-1\}$,
\begin{align}\label{lol90}
\frac{|P^{(j)}(\eta)|}{j!} \, |w^*-\eta|^{j} \leq  |P(\eta^*)| \, \bigg(1+\frac{2}{n^{1/\alpha}}\bigg)^j \leq |P(\eta^*)| \, e^{2},
\end{align}
where for the first inequality we have used $|w^*-\eta|\leq |w^*-z_1^*|+|z_1^*-\eta|\leq
\frac{2}{n^{2/\alpha}}+\frac{1}{n^{1/\alpha}}$. Combining \eqref{lol90} with \eqref{eqn: this is zeta}, we find
\begin{align*}
|P(\zeta^*)|\leq |P(\eta)| +  n |P(\eta^*)| e^{2}.
\end{align*}
Therefore, using also \eqref{eqn:z1} and \eqref{eqn:exp}, we get
\begin{align}
|P(z_1^*)|e^{-(n-1)\frac{Q(z_1^*)}{2}} & \leq n^{2/\alpha} \,  c_{2} \, \left(|P(\eta)|e^{-(n-1)\frac{Q(\eta)}{2}}+n  |P(\eta^*)|e^{2} e^{-(n-1)\frac{Q(\eta^*)}{2}}\right) |z_1^*-z_2^*| \nonumber \\
& \leq n^{2/\alpha}  c_{2} \left(1+n e^{2}\right)|P(z_1^*)|e^{-(n-1)\frac{Q(z_1^*)}{2}} |z_1^*-z_2^*|, \label{lol91}
\end{align}
where for the last inequality we have used \eqref{fekete inequality} (recall that $\{z_1^*,z_2^*,\ldots,z_n^*\}$ is a Fekete set for $Q$ on $U^c$ and $\eta,\eta^*\in U^c$). Simplifying \eqref{lol91} yields
\begin{align*}
|z_1^*-z_2^*|\geq \frac{1}{2 c_{2} e^{2}n^{1+\frac{2}{\alpha}}}.
\end{align*}
We deduce from Case I and Case II above that if $|z_1^*-z_2^*|\leq \frac{1}{n^{2/\alpha}}$, then $|z_1^*-z_2^*|\geq \frac{1}{2 c_{2} e^{2}n^{1+\frac{2}{\alpha}}}$. By replacing $z_{1}^{*}$ and $z_{2}^{*}$ by respectively $z_{\ell}^{*}$ and $z_{k}^{*}$ for some arbitrary $1\leq \ell\neq k \leq n$, we conclude similarly that $|z_\ell^*-z_k^*|\leq \frac{1}{n^{2/\alpha}}$ implies $|z_{\ell}^*-z_k^*|\geq \frac{1}{2 c_{2} e^{2}n^{1+\frac{2}{\alpha}}}$. This finishes the proof.
\end{proof}
\subsection{The spherical ensemble}\label{subsection:spherical fekete}
In this subsection, we prove Theorem \ref{thm:general pot} (ii) and Proposition \ref{prop:general pot} (ii) about the spherical ensemble \eqref{spherical ensemble}. Thus, in this subsection, we take $Q(z) = \log(1+|z|^{2})$, $\mu$ is supported on the whole complex plane, $d\mu(z) = \frac{d^{2}z}{\pi(1+|z|^{2})^{2}}$, and $U$ is bounded and satisfies Assumption \ref{ass:U} (with $S=\C$). The analysis here follows the one done in the previous subsections but with an important difference: $Q(z) = \log(1+|z|^{2})$ does not meet criteria (iii) of Definition \ref{def:admissible}, and therefore the theory from \cite{SaTo} does not directly apply. 

\medskip Throughout this subsection, we will use $z,\zeta$ to denote variables on the plane, and $\hat{z},\hat{\zeta}$ for variables on the unit sphere $\mathbb{S}^{2}$. We will also use $|d\hat{z}|$, $d^{2}\hat{z}$ for the length and area measures on $\mathbb{S}^{2}$. If $z$ and $\hat{z}$ appear in a same equation (and similarly for $\zeta,\hat{\zeta}$), then it is assumed that $\varphi(\hat{z})=z$, where $\varphi:\mathbb{S}^{2}\to \C\cup\{\infty\}, \; \varphi(u,v,w)=\frac{u+iv}{1-w}$ is the stereographic projection from the north pole $(0,0,1)$. In particular,
\begin{align}\label{change of var on the sphere}
& |d\zeta| = \frac{1+|\zeta|^{2}}{2}|d\hat{\zeta}| = \frac{e^{Q(\zeta)}}{2}|d\hat{\zeta}|, \qquad d^{2}\zeta = \frac{e^{2Q(\zeta)}}{4}d^{2}\hat{\zeta}, \qquad |\zeta-z| = \frac{1}{2}|\hat{\zeta}-\hat{z}|e^{\frac{Q(\zeta)}{2}}e^{\frac{Q(z)}{2}}.
\end{align} 
Let $E\subset \C$ be a closed set. For $Q(z) = \log(1+|z|^{2})$, \eqref{fekete delta n Q} can be rewritten as
\begin{align}\label{lol24}
\delta_n^{Q}(E) = 2^{-1} \sup_{\hat{z}_1,\hat{z}_2,\ldots,\hat{z}_n\in \hat{E}}\Bigg\{\prod_{1\leq j<k \leq n}|\hat{z}_j-\hat{z}_k|^{2}\Bigg\}^{\frac{1}{n(n-1)}} =: 2^{-1}\hat{\delta}_n^{Q}(\hat{E}),
\end{align}
where $\hat{E}:=\varphi^{-1}(E) \subset \mathbb{S}^{2}$. From \eqref{lol24}, the existence of Fekete sets readily follows from the fact that $\hat{E}$ is compact. Define $\mathcal{I}_{Q}[E]$ as in \eqref{def of Ical}. The same proof as \cite[Theorem III.1.1]{SaTo} shows that $\{\delta_n^{Q}(E)\}_{n=2}^{\infty}$ is a decreasing sequence. Thus $\delta^{Q}(E) := \lim_{n\to\infty}\delta_n^{Q}(E)$ exists.

\begin{lemma}\label{lemma:limit fekete spherical}
$\delta^{Q}(E) = e^{-\mathcal{I}_{Q}[E]} = e^{-I_{Q}[E]}$.
\end{lemma}
\begin{proof}
The inequality $e^{-\mathcal{I}_{Q}[E]} \leq \delta^{Q}(E)$ follows exactly as in \cite[Proof of Theorem III.1.3]{SaTo}, so we omit the details here. We now prove the inequality $\delta^{Q}(E) \leq e^{-\mathcal{I}_{Q}[E]}$ by adapting \cite[proof of Theorem III.1.3]{SaTo}. For any $n\in \N_{\geq 2}$, let $\mathcal{F}_{n}=\{z_{1},\ldots,z_{n}\}$ be a Fekete set for $Q$ on $E$, $\hat{\mathcal{F}}_{n}:=\{ \hat{z}_{1},\ldots,\hat{z}_{n}\}$, and let $\hat{\mu}_{n} := \smash{\frac{1}{n}\sum_{j=1}^{n} \delta_{\hat{z}_{j}}}$ be the normalized counting measure. The measures $\{\hat{\mu}_{n}\}_{n\geq 2}$ are all supported in the compact set $\hat{E}$. Thus we can extract a subsequence $\{\hat{\mu}_{\ell}\}$ converging to some $\hat{\sigma}$ in the weak$^*$ topology on $\mathcal{P}(\hat{E})$. Set $h_{M}(\hat{z},\hat{t}):=\min\{M,\log|\hat{z}-\hat{t}|^{-1}\}$ for $\hat{z},\hat{t}\in \mathbb{S}^{2}$. We have
\begin{align*}
\hat{I}[\hat{\sigma}] & := \iint \log(|\hat{z}-\hat{t}|^{-1}) \; d\hat{\sigma}(\hat{z})d\hat{\sigma}(\hat{t}) = \lim_{M\to \infty} \iint h_{M}(\hat{z},\hat{t}) d\hat{\sigma}(\hat{z})d\hat{\sigma}(\hat{t}) \\
& = \lim_{M\to \infty} \lim_{\ell\to\infty} \iint h_{M}(\hat{z},\hat{t}) d\hat{\mu}_{\ell}(\hat{z})d\hat{\mu}_{\ell}(\hat{t}) \leq \lim_{M\to \infty} \lim_{\ell\to\infty} \bigg( \frac{M}{\ell} + \iint_{\hat{z}\neq \hat{t}} \log|\hat{z}-\hat{t}|^{-1} d\hat{\mu}_{\ell}(\hat{z})d\hat{\mu}_{\ell}(\hat{t}) \bigg) \\
& = \lim_{\ell\to\infty} \frac{\ell(\ell-1)}{\ell^{2}} \log \frac{1}{2 \, \delta_{\ell}^{Q}(E)} = \log \frac{1}{2 \, \delta^{Q}(E)}.
\end{align*}
On the other hand, changing variables using the stereographic projection $\varphi:\mathbb{S}^{2}\to \C\cup \{\infty\}$, we get
\begin{align*}
\hat{I}[\hat{\sigma}] = \iint \log(2^{-1}|z-t|^{-1}e^{\frac{Q(z)}{2}}e^{\frac{Q(w)}{2}}) \; d(\varphi_{*}\hat{\sigma})(z)d(\varphi_{*}\hat{\sigma})(t) = I_{Q}[\varphi_{*}\hat{\sigma}]-\log 2,
\end{align*}
where $\varphi_{*}\hat{\sigma}$ is the pushforward measure of $\hat{\sigma}$ by $\varphi$. Hence $\mathcal{I}_{Q}[E] := \inf_{\sigma \in \mathcal{P}(E)} I_{Q}[\sigma] \leq I_{Q}[\varphi_{*}\hat{\sigma}] = \log \frac{1}{\delta^{Q}(E)}$, which is equivalent to $\delta^{Q}(E) \leq e^{-\mathcal{I}_{Q}[E]}$.
\end{proof}
Here are the analogues of Proposition \ref{prop:generalformula} and Lemmas \ref{lemma:upper bound}--\ref{lemma:lower bound}.
\begin{proposition}\label{prop:generalformula spherical}
Let $Q(z):= \log(1+|z|^{2})$, let $\mu$, $S:=\mathrm{supp} \, \mu$ be as in \eqref{def of mu and S spherical}, and let $U$ be bounded and such that Assumption \ref{ass:U} holds. Then the equilibrium measure for $Q$ on $U^c$ is $\nu+\mu|_{\C\setminus U}$ and
\begin{align}\label{lol25}
\mathcal{I}_{Q}[U^{c}] - \mathcal{I}_{Q}[\C] = \frac{1}{2}\bigg(  \int_{\partial U}Q(z)d\nu(z) - \int_{U}Q(z)d\mu(z) \bigg),
\end{align}
where $\nu = \mathrm{Bal}(\mu|_{U},\partial U)$. 
\end{proposition}
\begin{proof}
Following the proof of Proposition \ref{prop:generalformula}, we infer that $\tilde{\mu}:=\nu+\mu|_{S\setminus U}$ satisfies $2p_{\tilde{\mu}}(z)+Q(z)=0$ for q.e. $z\in \C\setminus U$. Furthermore, since $\nu$ is finite and $Q(z) = \log(1+|z|^{2})$, $I_{Q}[\tilde{\mu}]$ is finite. By Proposition \ref{prop:eq measure spherical}, $\tilde{\mu}$ is the equilibrium measure for $Q$ on $U^{c}$. Then \eqref{lol18}--\eqref{lol19} hold with $c^{\hspace{0.02cm}\mu}_{U}=0$, which implies \eqref{lol25}.
\end{proof}
\begin{lemma}[Upper bounds]\label{lemma:upper bound spherical}
Let $Q(z):= \log(1+|z|^{2})$ and let $U$ be bounded and such that Assumptions \ref{ass:U} holds. Then the inequalities \eqref{inequality}--\eqref{inequality 2} hold.
\end{lemma}
\begin{proof}
The proof is almost identical to the proof of Lemma \ref{lemma:upper bound}, except that one must use Lemma \ref{lemma:limit fekete spherical} instead of \eqref{eqn:limit}. We omit further details.
\end{proof}
Note that since $S=\C$, all Fekete sets for $Q$ on $U^{c}$ are such that $\mathcal{F}_{n}\subset S\setminus U = U^{c}$.
\begin{lemma}\label{lemma:dist fekete spherical}
Let $Q(z):= \log(1+|z|^{2})$, let $U$ be either empty or such that Assumption \ref{ass:U} holds, and let $\mathcal{F}_{n}=\{z_1^*,z_2^*,\ldots,z_n^*\}$ be a Fekete set for $Q$ on $U^c$. Then, for all large enough $n$, 
\begin{align}\label{fekete lower bound spherical}
\min\{|\hat{z}_j^*-\hat{z}_k^*|: 1\leq j\neq k \leq n\}\geq \frac{C_{1}}{n^{3}},
\end{align}
for some $C_{1}>0$, and where $\hat{z}_{j}^{*}:=\varphi^{-1}(z_{j}^{*})$.
\end{lemma}
\begin{proof}
If $U$ is empty, this follows from \cite{HW1951} (see also the introduction of \cite{SK1997}). Suppose now that $U$ is not empty. Let $\hat{U}:=\varphi^{-1}(U)$, and recall from \eqref{lol24} that $\hat{\mathcal{F}}_{n}:=\{\hat{z}_1^*,\hat{z}_2^*,\ldots,\hat{z}_n^*\} \subset \hat{U}^{c}:=S^{2}\setminus \hat{U}$ is such that
\begin{align*}
\prod_{1\leq j<k \leq n}|\hat{z}_j^{*}-\hat{z}_k^{*}|^{2} = \sup_{\hat{z}_1,\hat{z}_2,\ldots,\hat{z}_n\in  \hat{U}^{c}}\prod_{1\leq j<k \leq n}|\hat{z}_j-\hat{z}_k|^{2}.
\end{align*}
Since the sphere $\mathbb{S}^{2}$ is rotation-invariant, and since $U$ is open, we can (and do) assume without loss of generality that $\hat{U}^{c}$ does not contain a neighborhood of the north pole, i.e. that $U$ is unbounded and that $U^{c}$ is bounded. 

\medskip \noindent Since $Q$ satisfies \eqref{holder} with $\alpha=1$ and $K=1$, and since $U^{c}$ is bounded, a straightforward adaptation of the proof of Lemma \ref{lemma:dist fekete} yields 
\begin{align*}
\min\{|z_j^*-z_k^*|: 1\leq j\neq k \leq n\}\geq \frac{C_{2}}{n^{3}},
\end{align*}
for some $C_{2}>0$ and for all large enough $n$, which implies \eqref{fekete lower bound spherical} (using again that $U^{c}$ is bounded).
\end{proof}

\begin{lemma}[Lower bounds]\label{lemma:lower bound spherical}
Let $Q(z):= \log(1+|z|^{2})$ and $U$ be bounded and such that Assumption \ref{ass:U} holds. Then the inequalities \eqref{lower bound}--\eqref{lower bound 2} hold.
\end{lemma}
\begin{proof}
Recall from Proposition \ref{prop:generalformula spherical} that $\tilde{\mu} := \nu+\mu|_{S\setminus U}$ is the equilibrium measure for $Q$ on $U^{c}$, where $\mu$ is as in \eqref{def of mu and S spherical} and $\nu = \mathrm{Bal}(\mu|_{U},\partial U)$. Let $ \{ z_1^*,z_2^*,\ldots,z_n^*\}\subset U^{c}$ be a Fekete set for $Q$ on $U^c$. Then, in a similar way as \eqref{lol23} (but using \eqref{spherical ensemble} instead of \eqref{general density intro}), we obtain 
\begin{align}
\mathcal{P}_{n} & = \frac{1}{Z_n}\int_{U^c}\ldots \int_{U^c} \prod_{1 \leq j < k \leq n}|z_j-z_k|^{\beta} \prod_{j=1}^n
e^{-[(n-1) \frac{\beta}{2}+2] Q(z_{j})}d^{2}z_{j} \nonumber \\
& \geq \frac{1}{Z_n}\int_{B_1}\ldots \int_{B_n}
\Bigg\{ \prod_{1 \leq j<k \leq n}|z_j-z_k|^{2}e^{-Q(z_{j})}e^{-Q(z_{k})}\Bigg\}^{\frac{\beta}{2}} \prod_{j=1}^n e^{-2 Q(z_{j})}d^{2}z_{j}, \label{lol23 bis}
\end{align}
where $B_\ell:=\varphi\big(\hat{U}^c \cap B(\hat{z}_{\ell}^*,\frac{C_{1}}{n^{4}}) \big)$ for $\ell=1,2,\ldots, n$, $B(\hat{z},r)\subset \R^{3}$ is the ball centered at $\hat{z}$ of radius $r$, $\hat{U}^c:= \varphi^{-1}(U^{c})$ and $\hat{z}_{\ell}:=\varphi^{-1}(z_{\ell})$. Changing variables using \eqref{change of var on the sphere}, we get
\begin{align*}
\mathcal{P}_{n} \geq \frac{1}{2^{\frac{\beta}{2} n(n-1)+2n}Z_n}\int_{\varphi^{-1}(B_1)}\ldots \int_{\varphi^{-1}(B_n)}
\Bigg\{ \prod_{1 \leq j<k \leq n}|\hat{z}_j-\hat{z}_k|^{2}\Bigg\}^{\frac{\beta}{2}} \prod_{j=1}^n d^{2}\hat{z}_{j}.
\end{align*}
If $\hat{z}_j \in B(\hat{z}_j^*,\frac{C_{1}}{n^{4}})$ and $\hat{z}_k\in B(\hat{z}_k^*,\frac{C_{1}}{n^{4}})$ for $j\neq k$, then by Lemma \ref{lemma:dist fekete spherical},
\begin{align}\label{lol21 bis}
|\hat{z}_j-\hat{z}_k| \geq |\hat{z}_j^*-\hat{z}_k^*|-\frac{2C_{1}}{n^{4}} \geq |\hat{z}_j^*-\hat{z}_k^*| - \frac{2}{n}  |\hat{z}_j^*-\hat{z}_k^*| = \Big(1-\frac{2}{n}\Big) |\hat{z}_j^*-\hat{z}_k^*| .
\end{align}
Using \eqref{lol21 bis} in \eqref{lol23 bis}, we get
\begin{align*}
& \mathcal{P}_{n} \geq \frac{1}{2^{\frac{\beta}{2} n(n-1)+2n}Z_n}\int_{\varphi^{-1}(B_1)}\ldots \int_{\varphi^{-1}(B_n)} \Bigg\{ \prod_{1 \leq j<k \leq n} \bigg(1-\frac{2}{n}\bigg)^{2} |\hat{z}_j^{*}-\hat{z}_k^{*}|^{2}  \Bigg\}^{\frac{\beta}{2}} \prod_{j=1}^n d^{2}\hat{z}_{j} \\
& \geq \frac{1}{2^{\frac{\beta}{2} n(n-1)+2n}Z_n} \left(1-\frac{2}{n}\right)^{\frac{\beta}{2} n(n-1)}  \Bigg\{ \prod_{1 \leq j<k \leq n}|\hat{z}_j^{*}-\hat{z}_k^{*}|^{2} \Bigg\}^{\frac{\beta}{2}} \prod_{j=1}^n\int_{\varphi^{-1}(B_j)} d^{2}\hat{z}_{j}.
\end{align*}
From \eqref{eqn:exterior ball condition lol}, we infer that $\int_{\varphi^{-1}(B_j)} d^{2}\hat{z}_{j}\geq c_{2} (\frac{C_{1}}{n^{4}})^2$ holds for some $c_{2}>0$, for all large enough $n$ and all $j\in\{1,2,\ldots,n\}$. Thus, by \eqref{lol24},
\begin{align*}
\mathcal{P}_{n} \geq 
\frac{1}{2^{2n}Z_n} \left(1-\frac{2}{n}\right)^{\frac{\beta}{2} n(n-1)}  (\delta_n^{Q}(U^c))^{\frac{\beta}{2} n(n-1)} \bigg(\frac{
\sqrt{c_{2}}C_{1}}{n^{4}}\bigg)^{2n}.
\end{align*}
Taking $\liminf_{n\to \infty}\frac{1}{n^2}\log(\cdot)$ on both sides and using Lemma \ref{lemma:limit fekete spherical} and \eqref{lol24}, we find \eqref{lower bound}.

By a similar analysis (using now Lemma \ref{lemma:dist fekete spherical} with $U=\emptyset$), 
\begin{align*}
Z_{n} \geq \frac{1}{2^{2n}} \left(1-\frac{2}{n}\right)^{\frac{\beta}{2} n(n-1)}   (\delta_n^{Q}(\C))^{\frac{\beta}{2} n(n-1)} \bigg(\frac{\sqrt{\pi}\tilde{C}_{1}}{n^{4}}\bigg)^{2n},
\end{align*}
for some $\tilde{C}_{1}>0$, and \eqref{lower bound 2} follows.
\end{proof}
\section{Methods to obtain balayage measures}\label{section: method to obtain balayage}

Let $U$ be an open set, and let $\mu$ be a finite measure with compact support. In this section, we briefly outline several general strategies that will be used to obtain explicitly the balayage measure $\mathrm{Bal}(\mu|_{U},\partial U)$. 
\subsection{Ansatz}\label{section:ansatz method}
Suppose that we have a candidate measure $\nu$ supported on $\partial U$, and we would like to prove that $\nu = \mathrm{Bal}(\mu|_{U},\partial U)$. The uniqueness part of Proposition \ref{prop:def of bal} tells us that if $\nu(\partial U)=\mu(U)$, $p_{\nu}$ is bounded on $\partial U$, $\nu(P)=0$ for every Borel polar set $P\subset \C$, and $p_{\nu}(z)=p_{\mu|_{U}}(z)+c^{\hspace{0.02cm}\mu}_{U}$ for q.e. $z\in U^c$, then $\nu=\mathrm{Bal}(\mu|_{U},\partial U)$. However, verifying $p_{\nu}(z)=p_{\mu|_{U}}(z)+c^{\hspace{0.02cm}\mu}_{U}$ for all $z\in \partial U$ is in general a difficult task. 

\medskip If $U$ is bounded, the simple but useful idea used in \cite{AR2017} is to instead verify that the moments of $\nu$ and $\mu|_{U}$ are the same.
\begin{lemma}\label{lemma: moment when U is bounded}
Suppose $U$ is bounded and satisfies Assumption \ref{ass:U} (with $S$ replaced by $\C$), $U^{c}$ is connected, $\partial U$ is piecewise smooth, and that $p_{\mu|_{U}}$ is continuous in a neighborhood of $\partial U$. Suppose also that $\nu$ is supported on $\partial U$ with a bounded and piecewise continuous density with respect to the arclength measure on $\partial U$ and
\begin{align}\label{matching moments general U bounded}
\int_{\partial U}z^n d\nu(z) = \int_{ U}z^nd\mu(z), \qquad \mbox{for all } n \in \N.
\end{align}
Then $\nu=\mathrm{Bal}(\mu|_{U},\partial U)$.
\end{lemma}
\begin{remark}\label{remark: pmu bounded on dU}
The assumption ``$p_{\mu|_{U}}$ is continuous in a neighborhood of $\partial U$" will always be verified in later sections. For example, this assumption is automatically satisfied if $\mu$ has a continuous density on $\mathcal{N}\cap \overline{U}$ with respect to $d^{2}z$, for some open set $\mathcal{N}$ such that $\partial U\subset \mathcal{N}$.
\end{remark}
\begin{proof}
The proof is an extension of \cite[Remark 4.6]{AR2017}. Since $U$ is bounded, there exists $R>0$ such that $U\subset \{z:|z|< R\}$. Given $z \in \C$, let $\log_{z}(\cdot)$ be a branch of the logarithm that is analytic on $\C\setminus (-z \infty,0]$. Then, for any fixed $z$ with $|z|>R$, $z-w$ remains bounded away from the cut $(-z \infty,0]$ for all $w \in U$. Hence, using \eqref{matching moments general U bounded} (with $z$ replaced by $w$) and the fact that the series $\log_{z}(z-w) = \log_{z} z + \sum_{m=1}^{+\infty} \frac{w^{m}}{m z^{m}}$ is uniformly convergent for $w\in U$, we get
\begin{align*}
\int_{\partial U}\log_{z}(z-w)d\nu(w) = \int_{ U}\log_{z}(z-w)d\mu(w), \qquad \mbox{for all } |z|>R.
\end{align*}
In particular, taking the real part on both sides, we get $p_{\nu}(z)=p_{\mu|_{U}}(z)$ for every $|z|>R$. Since $p_{\nu}$, $p_{\mu|_{U}}$ are harmonic in $\overline{U}^{c}$ (see \cite[Theorem 0.5.6]{SaTo}) and $U^{c}$ is connected, $p_{\nu}(z)=p_{\mu|_{U}}(z)$ for all $z \notin \overline{U}$. Recall that $p_{\mu|_{U}}$ is assumed to be continuous in a neighborhood of $\partial U$. Also, by assumption, $d\nu(z) = f(z)|dz|$ for $z\in \partial U$, where $f$ is bounded and piecewise continuous and $|dz|$ is the arclength measure on $\partial U$. This implies by \cite[Theorem II.1.5 (ii)]{SaTo} that $p_{\nu}$ has continuous boundary values on the domain of continuity of $f$. Hence $p_{\mu|_{U}} = p_{\nu}$ for q.e. $z\in \partial U$. Equation \eqref{matching moments general U bounded} with $n=0$ implies $\nu(\partial U)=\mu(U)$. Also, since $d\nu(z) = f(z)|dz|$, clearly $\nu(P)=0$ for every Borel polar set and $p_{\nu}$ is bounded on the compact $\partial U$. By Proposition \ref{prop:def of bal}, $\nu=\mathrm{Bal}(\mu|_{U},\partial U)$.
\end{proof}

The following lemma is the analogue of Lemma \ref{lemma: moment when U is bounded} for unbounded $U$.

\begin{lemma}\label{lemma: moment when U is unbounded}
Suppose $U$ is unbounded and satisfies Assumption \ref{ass:U} (with $S$ replaced by $\C$), $U^{c}$ is connected, $\partial U$ is piecewise smooth, and that $p_{\mu|_{U}}$ is continuous in a neighborhood of $\partial U$. Suppose also that $\nu$ is supported on $\partial U$ with a bounded and piecewise continuous density with respect to the arclength measure on $\partial U$ and
\begin{align}\label{eqn:moment unbounded}
& \begin{cases}
\ds \int_{\partial U}(z-z_{0})^{-n} d\nu(z) = \int_{ U}(z-z_{0})^{-n}d\mu(z), \quad \mbox{for all } n \in \N, \\[0.25cm]
\ds \int_{\partial U}\log |z-z_{0}|  d\nu(z) = \int_{ U} \log |z-z_{0}| d\mu(z) - c^{\hspace{0.02cm}\mu}_{U},
\end{cases}
\end{align}
for some $z_{0}\notin \overline{U}$. Then $\nu=\mathrm{Bal}(\mu|_{U},\partial U)$.
\end{lemma}
\begin{remark}
The second equation in \eqref{eqn:moment unbounded} is equivalent to $p_{\nu}(z_{0}) = p_{\mu|_{U}}(z_{0}) + c^{\hspace{0.02cm}\mu}_{U}$ and is in practice more challenging to verify than the first equation. Note also that this difficulty is not present in \eqref{matching moments general U bounded} for the case where $U$ is bounded.
\end{remark}
\begin{proof}
Since $z_{0}\notin \overline{U}$, there exists $R>0$ such that $\{z:|z-z_{0}|< R\}\subset \overline{U}^{c}$. Fix $z$ with $|z-z_{0}|<R$. Using \eqref{eqn:moment unbounded} (with $z$ replaced by $w$), we see that the series 
\begin{align*}
\re \log(w-z) = \re \log (w-z_{0}) - \re \sum_{n=1}^{+\infty}  \frac{(z-z_{0})^n}{n\, (w-z_{0})^n}
\end{align*}
is uniformly convergent for $w\in U$, and thus $p_{\nu}(z)=p_{\mu|_{U}}(z)+c^{\hspace{0.02cm}\mu}_{U}$. The rest of the proof is identical to the proof of Lemma \ref{lemma: moment when U is bounded}.
\end{proof}

\subsection{Green function}\label{subsection: green function method for nu}
The other method to obtain $\mathrm{Bal}(\mu|_{U},\partial U)$ involves the Green function $g_{U}(z,w)$ of $U$ with a pole at $w \in U$. The definition of $g_{U}(z,\infty)$ (for $U$ unbounded) is given in Definition \ref{def:green with pole at inf}. Here is the definition of $g_{U}(z,w)$ for finite $w$ (see also e.g. \cite[Section II.4]{SaTo}).
\begin{definition}\label{def:green with pole at a}
Suppose that $U\subset (\C\cup \{\infty\})$ is open, connected and $\mathrm{cap}(\partial U)>0$. Then $g_{U}(z,w)$ is the unique function satisfying the following conditions:
\begin{itemize}
\item \vspace{-0.15cm} $g_{U}(z,w)$ is nonnegative and harmonic in $U\setminus \{w\}$ and bounded as $z$ stays away from $w$,
\item \vspace{-0.15cm} $g_{U}(z,w)-\frac{1}{2\pi}\log \frac{1}{|z-w|} = \bigO(1)$ as $z\to w$,
\item \vspace{-0.15cm} $\lim_{z\to x, z\in U}g_{U}(z,w)=0$ for q.e. $x\in \partial U$.
\end{itemize}
\end{definition}
Standard formulas for Green functions that will be used in this work are:
\begin{align*}
& g_{U}(z,w) = \frac{1}{2\pi} \log \bigg| \frac{1-z\overline{w}}{z-w} \bigg|, & & \mbox{if } U = \{z:|z|<1\} \mbox{ or if } U = \{z:|z|>1\}, \\
& g_{U}(z,w) = \frac{1}{2\pi}\log \bigg| \frac{z-\overline{w}}{z-w} \bigg|, & & \mbox{if } U = \{z:\im z >0\}.
\end{align*}
Moreover, given an open connected set $U$ with $\mathrm{cap}(\partial U)>0$, if $\varphi, \phi$ are conformal maps from $U$ to respectively $\{z:|z|<1\}$ and $\{z:\im z >0\}$, then
\begin{align}\label{green function with general conformap mapping}
g_{U}(z,w) = \frac{1}{2\pi} \log \bigg| \frac{1-\varphi(z)\overline{\varphi(w)}}{\varphi(z)-\varphi(w)} \bigg| = \frac{1}{2\pi} \log \bigg| \frac{\phi(z)-\overline{\phi(w)}}{\phi(z)-\phi(w)} \bigg|.
\end{align}
The following result is standard in potential theory but was not used to solve large gap problems in \cite{AR2017, A2018}. It expresses $\mathrm{Bal}(\mu|_{U},\partial U)$ in terms of $g_{U}$. 
\begin{theorem}\label{thm:dnu in terms of green general}
Let $\mu$ be a finite measure with compact support in $\C$. Suppose $U$ is open, connected, with compact and piecewise $C^{2}$ boundary. Then
\begin{align*}
d\nu(z) =  \bigg( \int_{U} \frac{\partial g_{U}(w,z)}{\partial \mathbf{n}_{z}}d\mu(w) \bigg)|dz|, \qquad \mbox{for all } z\in (\partial U)_{\star}
\end{align*}
where $\nu = \mathrm{Bal}(\mu|_{U},\partial U)$, $|dz|$ is the arclength measure on $\partial U$, $(\partial U)_{\star}$ consists of all points $z\in \partial U$ such that $\partial U$ is smooth in a neighborhood of $z$, and $\mathbf{n}_{z}$ is the normal at $z$ pointing inwards $U$. 
\end{theorem}
\begin{proof}
The proof directly follows from the identity $\nu = \int_{U} \mathrm{Bal}(\delta_{\zeta},\partial U)d\mu(\zeta)$ where $\delta_{\zeta}$ is the dirac measure at $\zeta$ (see \cite[eq between (II.5.2) and (II.5.3)]{SaTo}), and \cite[Theorem II.4.11]{SaTo}. (\cite[Theorem II.4.11]{SaTo} is stated for $C^{2}$ boundary but can readily be extended to piecewise $C^{2}$ boundary. Indeed, the proof of \cite[Theorem II.4.11]{SaTo} is based on \cite[(II.4.20) and (II.4.21)]{SaTo} and these equations are proved in \cite{SaTo} for piecewise $C^{2}$ boundary.)
\end{proof}

\section{General rotation-invariant potentials}\label{section:general rotation-invariant pot}
In this section, we prove Theorems \ref{thm:centered disk} (i), \ref{thm:regular annulus} (i), \ref{thm:unbounded annulus}, \ref{thm:g sector nu} (i) and \ref{thm:g sector C} (i). 
We thus assume that $Q(z) = g(|z|)$ for some $g: [0,+\infty)\to \R \cup \{+\infty\}$ such that Assumption \ref{ass:Q} holds, and such that $S = \{z : |z| \in \mathrm{S} \}$ with $\mathrm{S} := [r_{0},r_{1}]\cup [r_{2},r_{3}] \ldots \cup [r_{2\ell},r_{2\ell+1}]$ for some $0 \leq r_{0} < r_{1} < \ldots < r_{2\ell}<r_{2\ell+1}<+\infty$. We further assume that $g \in C^{2}(\mathrm{S}\setminus \{0\})$, so that by \eqref{mu radially symmetric intro 2} we have 
\begin{align}\label{mu radially symmetric proof}
d\mu(z) = \frac{\Delta Q(z)}{4\pi} d^{2}z = d\mu_{\mathrm{rad}}(r) \frac{d\theta}{2\pi}, \qquad d\mu_{\mathrm{rad}}(r):=\frac{r}{2} \big( g''(r) + \frac{1}{r}g'(r) \big) dr,
\end{align}
where $z=re^{i\theta}$, $r\geq 0$ and $\theta \in (-\pi,\pi]$.

\medskip (The proofs of Theorems \ref{thm:centered disk} (ii), \ref{thm:regular annulus} (ii), \ref{thm:g sector nu} (ii) and \ref{thm:g sector C} (ii) about the spherical ensemble are similar: it suffices to replace $g(r)$ by $\log(1+r^{2})$, $\ell$ by $0$, $r_{0}$ by $0$ and $r_{1}$ by $+\infty$.) 
\subsection{The disk centered at $0$: proof of Theorem \ref{thm:centered disk}}

Suppose $U=\{z:|z|<a\}$ for some $a\in [r_{2k},r_{2k+1}]\setminus \{0\}$, $k\in \{0,\ldots,\ell\}$. 

\subsubsection{Balayage measure: proof of \eqref{nu for centered disk}}
Define $\nu$ for $z \in \partial U$ by $d\nu(z) = \kappa\frac{d\theta}{2\pi}$, where $\theta \in (-\pi,\pi]$ and $\kappa = \int_{0}^{a}d\mu_{\mathrm{rad}}(r)$. It is easy to check that $\nu(\partial U)=\mu(U)$. Also, using the well-known formula (see e.g. \cite[Chapter 0, page 22]{SaTo})
\begin{align}\label{well-known formula}
\int_{0}^{2\pi}\log\frac{1}{|z-re^{i\theta}|}\frac{d\theta}{2\pi}=\left\{\begin{array}{lr}\log
\frac{1}{r}, & \mbox{ if } |z|\leq r,\\
\log \frac{1}{|z|}, & \mbox{ if
} |z|\geq r,
\end{array}\right. \qquad \mbox{for all } z\in \C, \; r>0,
\end{align}
we obtain, for all $z\in U^{c}$, that
\begin{align*}
& p_{\mu|_{U}}(z) = \int_{U} \log \frac{1}{|z-w|} d\mu(w) = \int_{0}^{a} \int_{0}^{2\pi} \log \frac{1}{|z-re^{i\theta}|} \frac{d\theta}{2\pi} d\mu_{\mathrm{rad}}(r) = \kappa \log \tfrac{1}{|z|} , \\
& p_{\nu}(z) = \int_{\partial U} \log \frac{1}{|z-w|} d\nu(w) = \kappa\int_{0}^{2\pi} \log \frac{1}{|z-a e^{i\theta}|} \frac{d\theta}{2\pi} = \kappa \log \tfrac{1}{|z|}.
\end{align*}
Hence, by Proposition \ref{prop:def of bal}, $\nu=\mathrm{Bal}(\mu|_{U},\partial U)$.

\subsubsection{The constant $C$: proof of \eqref{C in thm disk rot inv}}
Using now Theorem \ref{thm:general pot} (i), we have
\begin{align*}
C & = \frac{\beta}{4} \bigg( \int_{\partial U} g(a)d\nu(z) - \int_{U}g(|z|)d\mu(z) \bigg) = \frac{\beta}{4} \bigg( g(a)\nu(\partial U) -  \int_{0}^{a}g(r) d\mu_{\mathrm{rad}}(r) \bigg),
\end{align*}
which yields \eqref{C in thm disk rot inv}. 
\subsection{The annulus centered at $0$: proof of Theorem \ref{thm:regular annulus}}
Suppose $U = \{z: \rho_{1} < |z| < \rho_{2} \}$ for some $0<\rho_{1}<\rho_{2}$, $\rho_{1}\in [r_{2k},r_{2k+1}]$, $\rho_{2}\in [r_{2q},r_{2q+1}]$ with $k,q\in \{0,\ldots,\ell\}$.  

\subsubsection{Balayage measure: proof of \eqref{nu for centered annulus}.}
The proof of \eqref{nu for centered annulus} is similar to \cite[Example 3.4]{A2018}. Define $\nu$ as in \eqref{nu for centered annulus} for some $\lambda,\kappa \in \R$. We will show that $\nu = \mathrm{Bal}(\mu|_{U},\partial U)$, provided that $\kappa$ and $\lambda$ are as indicated in the statement of Theorem \ref{thm:regular annulus}. Since $\mu(U) = \int_{\rho_{1}}^{\rho_{2}}d\mu_{\mathrm{rad}}(r)$ and $\nu(\partial U) = \kappa$, to fulfill $\mu(U) = \nu(\partial U)$ we must take $\kappa = \int_{\rho_{1}}^{\rho_{2}}d\mu_{\mathrm{rad}}(r)$. Also, using \eqref{well-known formula} we get for $|z|\leq \rho_{1}$ that
\begin{align*}
& p_{\mu|_{U}}(z) = \int_{U} \log \frac{1}{|z-w|} d\mu(w) = \int_{\rho_{1}}^{\rho_{2}} \int_{0}^{2\pi} \log \frac{1}{|z-re^{i\theta}|} \frac{d\theta}{2\pi} d\mu_{\mathrm{rad}}(r) = \int_{\rho_{1}}^{\rho_{2}} (\log \tfrac{1}{r} ) d\mu_{\mathrm{rad}}(r), \\
& p_{\nu}(z) = \int_{\partial U} \log \frac{1}{|z-w|} d\nu(w) = (1-\lambda)\kappa\int_{0}^{2\pi} \log \frac{1}{|z-\rho_{2}e^{i\theta}|} \frac{d\theta}{2\pi} + \lambda \kappa \int_{0}^{2\pi} \log \frac{1}{|z-\rho_{1}e^{i\theta}|} \frac{d\theta}{2\pi} \\
& \hspace{0.81cm} = (1-\lambda)\kappa\log \frac{1}{\rho_{2}}  + \lambda \kappa \log \frac{1}{\rho_{1}}.
\end{align*}
In order for $p_{\mu|_{U}}(z) = p_{\nu}(z)$ to hold for all $|z| \leq \rho_{1}$, we must take
\begin{align}\label{lambda in proof disk}
\lambda = \frac{\kappa \log \rho_{2} - \int_{\rho_{1}}^{\rho_{2}} (\log r ) d\mu_{\mathrm{rad}}(r)}{\kappa \log(\rho_{2}/\rho_{1})}.
\end{align}
Similarly, for $|z|\geq \rho_{2}$, we get
\begin{align*}
& p_{\mu|_{U}}(z) = \int_{U} \log \frac{1}{|z-w|} d\mu(w) = \int_{\rho_{1}}^{\rho_{2}} \int_{0}^{2\pi} \log \frac{1}{|z-re^{i\theta}|} \frac{d\theta}{2\pi} d\mu_{\mathrm{rad}}(r) = \kappa \log \tfrac{1}{|z|}, \\
& p_{\nu}(z) = \int_{\partial U} \log \frac{1}{|z-w|} d\nu(w) = (1-\lambda)\kappa\int_{0}^{2\pi} \log \frac{1}{|z-\rho_{2}e^{i\theta}|} \frac{d\theta}{2\pi} + \lambda \kappa \int_{0}^{2\pi} \log \frac{1}{|z-\rho_{1}e^{i\theta}|} \frac{d\theta}{2\pi} \\
& \hspace{0.81cm} = (1-\lambda)\kappa\log \frac{1}{|z|}  + \lambda \kappa \log \frac{1}{|z|} = \kappa \log \tfrac{1}{|z|}.
\end{align*}
Hence $p_{\mu|_{U}}(z) = p_{\nu}(z)$ for all $|z| \geq \rho_{2}$. By Proposition \ref{prop:def of bal}, the measure $\nu$ defined in \eqref{nu for centered annulus}, with $\kappa = \int_{\rho_{1}}^{\rho_{2}}d\mu_{\mathrm{rad}}(r)$ and $\lambda$ as in \eqref{lambda in proof disk}, satisfies $\nu = \mathrm{Bal}(\mu|_{U},\partial U)$.

\subsubsection{The constant $C$: proof of \eqref{C in thm annulus rot inv}}
Using Theorem \ref{thm:general pot} (i), we obtain
\begin{align*}
C & = \frac{\beta}{4}\bigg(  g(\rho_{2}) \int_{\{|z|=\rho_{2}\}}d\nu(z) + g(\rho_{1}) \int_{\{|z|=\rho_{1}\}}d\nu(z) - \int_{\rho_{1}}^{\rho_{2}}g(r)d\mu_{\mathrm{rad}}(r) \bigg) \\
& = \frac{\beta}{4}\bigg(  (1-\lambda)\kappa g(\rho_{2})  + \lambda \kappa g(\rho_{1}) - \int_{\rho_{1}}^{\rho_{2}}g(r)d\mu_{\mathrm{rad}}(r) \bigg),
\end{align*}
which is \eqref{C in thm annulus rot inv}. 

\subsection{The complement of a disk centered at $0$: proof of Theorem \ref{thm:unbounded annulus}}
Suppose $U=\{z:|z|>a\}$ for some $a\in [r_{2k},r_{2k+1}]\setminus \{0\}$, $k\in \{0,\ldots,\ell\}$. 
\subsubsection{Balayage measure: proof of \eqref{nu for unbounded annulus}}
Define $\nu$ for $z \in \partial U$ by $d\nu(z) = \kappa\frac{d\theta}{2\pi}$, where $\theta \in (-\pi,\pi]$ and $\kappa = \int_{a}^{r_{2\ell+1}}d\mu_{\mathrm{rad}}(r)$. The definition of $\kappa$ implies that $\nu(\partial U)=\mu(U)$. Also, using \eqref{well-known formula}, we obtain, for all $z\in U^{c}$, that
\begin{align*}
& p_{\mu|_{U}}(z) = \int_{U} \log \frac{1}{|z-w|} d\mu(w) = \int_{a}^{+\infty} \int_{0}^{2\pi} \log \frac{1}{|z-re^{i\theta}|} \frac{d\theta}{2\pi} d\mu_{\mathrm{rad}}(r) = \int_{a}^{r_{2\ell+1}}(\log \tfrac{1}{r})d\mu_{\mathrm{rad}}(r), \\
& p_{\nu}(z) = \int_{\partial U} \log \frac{1}{|z-w|} d\nu(w) = \kappa\int_{0}^{2\pi} \log \frac{1}{|z-a e^{i\theta}|} \frac{d\theta}{2\pi} = \kappa \log \tfrac{1}{a}.
\end{align*}
Moreover, it is easy to check using Definition \ref{def:green with pole at inf} that $g_{U}(z,\infty) = \frac{1}{2\pi} \log\tfrac{|z|}{a}$. Thus
\begin{align*}
c^{\hspace{0.02cm}\mu}_{U} = 2\pi\int_{U}g_{U}(z,\infty)d\mu(z) = \int_{a}^{r_{2\ell+1}} (\log \tfrac{r}{a} ) d\mu_{\mathrm{rad}}(r) = \kappa \log \tfrac{1}{a} - \int_{a}^{r_{2\ell+1}}(\log \tfrac{1}{r})d\mu_{\mathrm{rad}}(r),
\end{align*}
which implies that $p_{\nu}(z) = p_{\mu|_{U}}(z) + c^{\hspace{0.02cm}\mu}_{U}$ for all $z \in U^{c}$. By Proposition \ref{prop:def of bal}, $\nu=\mathrm{Bal}(\mu|_{U},\partial U)$.
\subsubsection{The constant $C$: proof of \eqref{C for unbounded annulus}} 

Using Theorem \ref{thm:general pot} (i), we obtain
\begin{align*}
C = \frac{\beta}{4} \bigg( g(a)\int_{\partial U}d\nu(z) + 2c^{\hspace{0.02cm}\mu}_{U} - \int_{a}^{r_{2\ell+1}}g(r)d\mu_{\mathrm{rad}}(r) \bigg) = \frac{\beta}{4} \int_{a}^{r_{2\ell+1}} \big(g(a)-g(r)+2\log \tfrac{r}{a}\big)d\mu_{\mathrm{rad}}(r),
\end{align*}
which is \eqref{C for unbounded annulus}.

\subsection{The circular sector centered at $0$}
Let $p\in [1,\infty)$, $a\in (0,r_{1})$ and $U=\{re^{i\theta} : 0<r<a, \, 0<\theta < \frac{2\pi}{p}\}$.

\subsubsection{Balayage measure: proof of Theorem \ref{thm:g sector nu}}
\begin{figure}[h]
\begin{center}
\begin{tikzpicture}[master]
\node at (0,0) {\includegraphics[height=3cm]{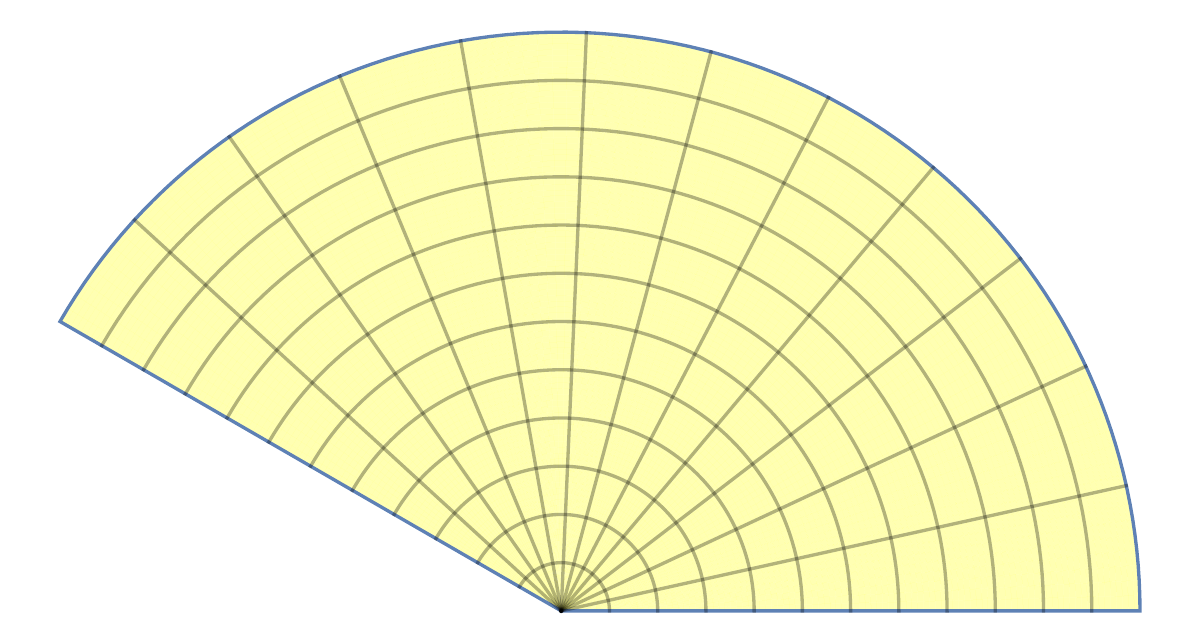}};
\draw[fill] (2.5,-1.35) circle (0.04);
\node at (2.5,-1.5) {\footnotesize $a$};

\draw[fill] (-2.5,0) circle (0.04);
\node at (-2.95,0) {\footnotesize $ae^{\frac{2\pi i}{p}}$};

\draw[fill] (-0.18,-1.34) circle (0.04);
\node at (-0.18,-1.55) {\footnotesize $0$};
\end{tikzpicture} \hspace{1cm}
\begin{tikzpicture}[slave]
\node at (0,0) {\includegraphics[height=4cm]{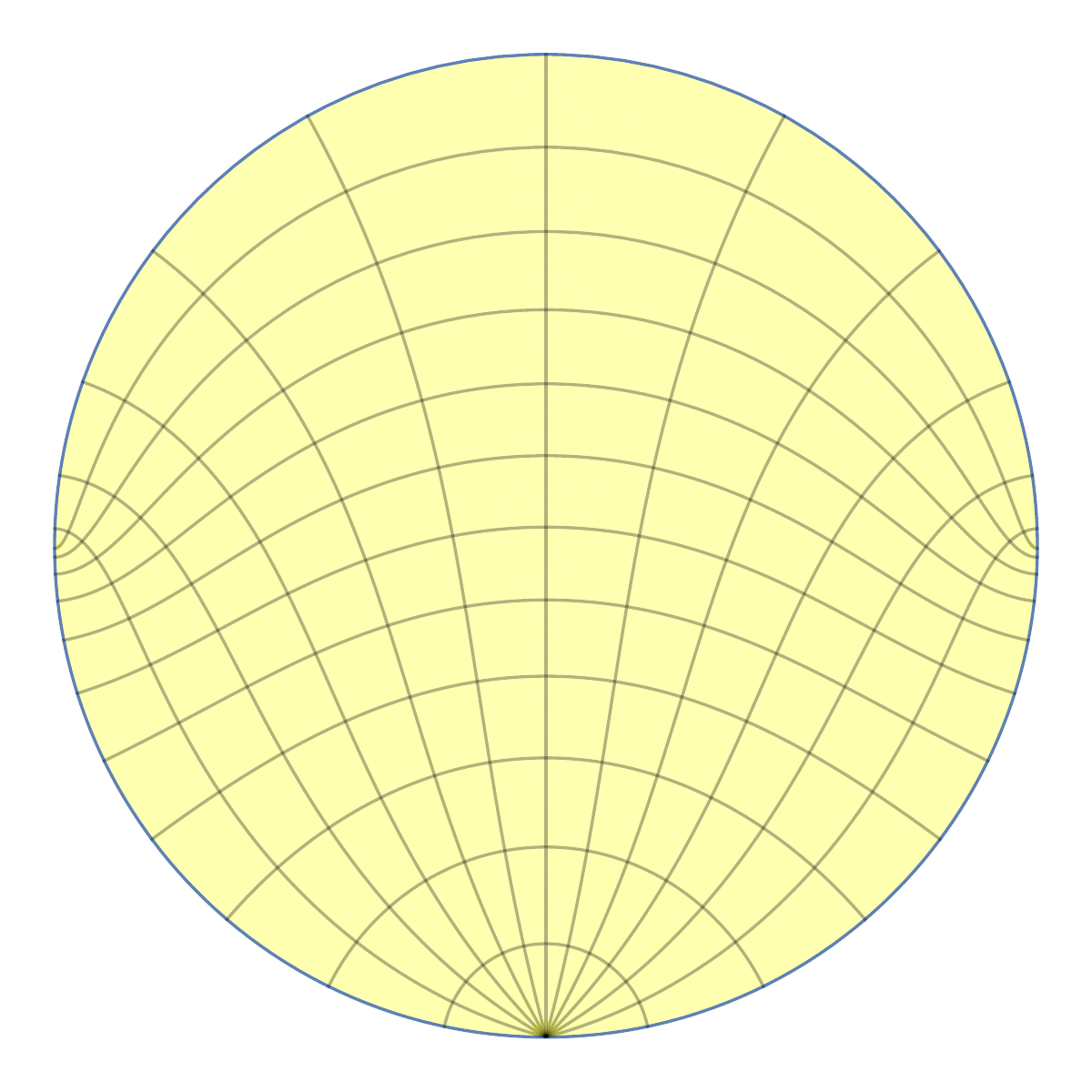}};
\node at (-3.5,1.4) {$\varphi$};
\coordinate (P) at ($(-3.5, 0.5) + (30:1cm and 0.6cm)$);
\draw[thick, black, -<-=0.45] ($(-3.5, 0.5) + (30:1cm and 0.6cm)$(P) arc (30:150:1cm and 0.6cm);
\draw[fill] (1.8,0) circle (0.04);
\node at (2,0) {\footnotesize $1$};
\draw[fill] (-1.8,0) circle (0.04);
\node at (-2.1,0) {\footnotesize $-1$};
\draw[fill] (0,-1.8) circle (0.04);
\node at (0,-2) {\footnotesize $-i$};
\end{tikzpicture}
\end{center}
\caption{\label{fig:conformal map from circular sector of disk} The conformal map $\varphi$ from $U$ to the unit disk.}
\end{figure}
By \cite{MathStackExchange}, the map $\tilde{\varphi}(z) := -i \frac{z^{2}+2iz+1}{z^{2}-2iz+1}$ is a conformal map from the upper-half of the unit disk $\{r e^{i\theta}:r\in (0,1), \; \theta \in (0,\pi) \}$ to the unit disk. Hence
\begin{align*}
\varphi(z) := -i \frac{(\frac{z}{a})_{(0)}^{p} + 2i (\frac{z}{a})_{(0)}^{\frac{p}{2}} + 1}{(\frac{z}{a})_{(0)}^{p} - 2i (\frac{z}{a})_{(0)}^{\frac{p}{2}}+1},
\end{align*}
where $z_{(0)}^{\frac{p}{2}}:=|z|^{\frac{p}{2}}e^{i \frac{p}{2} \arg_{0}(z)}$ and $\arg_{0}(z) \in (0,2\pi)$, is a conformal map from $U$ to $\{z:|z|<1\}$, see also Figure \ref{fig:conformal map from circular sector of disk}. Hence, by the first identity in \eqref{green function with general conformap mapping}, $g_{U}(z,w) = \frac{1}{2\pi} \log \big| \frac{1-\varphi(z)\overline{\varphi(w)}}{\varphi(z)-\varphi(w)} \big|$. Let $\nu := \mathrm{Bal}(\mu|_{U},\partial U)$. Using now Theorem \ref{thm:dnu in terms of green general}, for $r \in (0,a)$ we obtain
\begin{align}
& \frac{d\nu(r)}{dr} = \frac{d\nu(r e^{\frac{2\pi i}{p}})}{dr} = \int_{0}^{a}\int_{0}^{\frac{2\pi}{p}} \bigg[\frac{1}{r}\frac{d}{d\theta} g_{U}(x \, e^{i\alpha},r e^{i\theta}) \bigg]_{\theta=0} \frac{d\alpha}{2\pi}d\mu_{\mathrm{rad}}(x) \nonumber \\
& = \frac{p \, r^{\frac{p}{2}-1}(a^{p}-r^{p})}{a^{2p}+6a^{p}r^{p}+r^{2p}} \re \int_{0}^{a}\int_{0}^{\frac{2\pi}{p}} \frac{a^{2p} + a^{p}[e^{ip\alpha}x^{p}+4e^{\frac{ip\alpha}{2}}x^{\frac{p}{2}}r^{\frac{p}{2}}+r^{p}]+e^{ip\alpha}r^{p}x^{p} }{2\pi i (r^{\frac{p}{2}}-e^{\frac{i p \alpha}{2}}x^{\frac{p}{2}})(a^{p}-e^{\frac{ip\alpha}{2}}x^{\frac{p}{2}}r^{\frac{p}{2}})} \frac{d\alpha}{2\pi}d\mu_{\mathrm{rad}}(x), \label{lol49}
\end{align}
while for $z = ae^{i\theta}$ with $\theta \in (0,\tfrac{2\pi}{p})$, we get
\begin{align}
& \frac{d\nu(z)}{d\theta} = a \int_{0}^{a}\int_{0}^{\frac{2\pi}{p}} \bigg[(-1)\frac{d}{dr} g_{U}(x\,e^{i\alpha},r e^{i\theta}) \bigg]_{r =a} \frac{d\alpha}{2\pi}d\mu_{\mathrm{rad}}(x) \nonumber \\
& = \frac{p \sin(\frac{p\theta}{2})}{3+\cos(p\theta)} \re \int_{0}^{a}\int_{0}^{\frac{2\pi}{p}} \frac{(a^{p}+e^{ip\alpha}x^{p})(1+e^{ip\theta})+4e^{\frac{i p (\alpha+\theta)}{2}}a^{\frac{p}{2}}x^{\frac{p}{2}}}{2\pi i (e^{\frac{i p \theta}{2}}a^{\frac{p}{2}}-e^{\frac{i p \alpha}{2}}x^{\frac{p}{2}})(a^{\frac{p}{2}}-e^{\frac{i p (\alpha+\theta)}{2}}x^{\frac{p}{2}})} \frac{d\alpha}{2\pi}d\mu_{\mathrm{rad}}(x). \label{lol48}
\end{align}

\subsubsection*{Simplification of \eqref{lol49}: proof of \eqref{g sector nu segment}--\eqref{g sector nu segment 2}}

\begin{lemma}\label{lemma:simplif int segment}
Fix $r,x\in (0,a)$, $r\neq x$. Then
\begin{multline*}
I_{r}(x) := \int_{0}^{\frac{2\pi}{p}} \frac{a^{2p} + a^{p}[e^{ip\alpha}x^{p}+4e^{\frac{ip\alpha}{2}}x^{\frac{p}{2}}r^{\frac{p}{2}}+r^{p}]+e^{ip\alpha}r^{p}x^{p} }{2\pi i (r^{\frac{p}{2}}-e^{\frac{i p \alpha}{2}}x^{\frac{p}{2}})(a^{p}-e^{\frac{ip\alpha}{2}}x^{\frac{p}{2}}r^{\frac{p}{2}})} d\alpha \\
= \begin{cases}
\ds \frac{2(a^{2p}+6a^{p}r^{p}+r^{2p})}{\pi p r^{\frac{p}{2}}(a^{p}-r^{p})} \big( \arctanh((\tfrac{x}{r})^{\frac{p}{2}})-\arctanh((\tfrac{r}{a})^{\frac{p}{2}}(\tfrac{x}{a})^{\frac{p}{2}}) \big)-i\frac{a^{p}+r^{p}}{r^{\frac{p}{2}} p}, & \mbox{if } x<r, \\[0.3cm]
\ds \frac{2(a^{2p}+6a^{p}r^{p}+r^{2p})}{\pi p r^{\frac{p}{2}}(a^{p}-r^{p})} \big( \arctanh((\tfrac{r}{x})^{\frac{p}{2}}) + \tfrac{\pi i}{2} -\arctanh((\tfrac{r}{a})^{\frac{p}{2}}(\tfrac{x}{a})^{\frac{p}{2}}) \big)-i\frac{a^{p}+r^{p}}{r^{\frac{p}{2}} p}, & \mbox{if } x>r.
\end{cases}
\end{multline*}
\end{lemma}
\begin{proof}
We first do the change of variables $\tilde{\alpha} = \frac{1}{p}\alpha$, so that the integral is from $0$ to $2\pi$. Then we do the change of variables $z=e^{i\tilde{\alpha}}$. This yields
\begin{align*}
I_{r}(x) = \int_{\mathbb{S}^{1}} \frac{a^{2p} + a^{p}(z x^{p} + 4 z_{(0)}^{\frac{1}{2}}r^{\frac{p}{2}}x^{\frac{p}{2}} + r^{p}) + zr^{p}x^{p}}{2\pi i (r^{\frac{p}{2}}-z_{(0)}^{\frac{1}{2}}x^{\frac{p}{2}})(a^{p}-z_{(0)}^{\frac{1}{2}}r^{\frac{p}{2}}x^{\frac{p}{2}})} \frac{1}{p} \frac{dz}{iz},
\end{align*}
where $\mathbb{S}^{1}$ is the unit circle oriented in the counterclockwise direction. The integrand has a branch cut along $[0,+\infty)$, and simple poles at $z=z_{1}:=(\frac{r}{x})^{p}+i0_{+}$ and $z=z_{2}:=(\frac{a}{r})^{p}(\frac{a}{x})^{p}+i0_{+}$ (note that $z_{1}$ and $z_{2}$ lie on the side of the cut from the upper half-plane). Suppose that $x<r$, and let $M>z_{2}$. By deforming $\mathbb{S}^{1}$ into $M\mathbb{S}^{1}$, we pick up a branch cut contribution along $(1,M)+i0_{+}$ and $(1,M)-i0_{+}$, and two half-residues at $z_{1}$ and $z_{2}$. More precisely, 
\begin{align}
I_{r}(x) & =  \int_{M \, \mathbb{S}^{1}} \frac{a^{2p} + a^{p}(z x^{p} + 4 z_{(0)}^{\frac{1}{2}}r^{\frac{p}{2}}x^{\frac{p}{2}} + r^{p}) + zr^{p}x^{p}}{2\pi i (r^{\frac{p}{2}}-z_{(0)}^{\frac{1}{2}}x^{\frac{p}{2}})(a^{p}-z_{(0)}^{\frac{1}{2}}r^{\frac{p}{2}}x^{\frac{p}{2}})} \frac{1}{p} \frac{dz}{iz} \nonumber \\
& + \dashint_{1}^{M} \frac{a^{2p} + a^{p}(y x^{p} + 4 \sqrt{y} r^{\frac{p}{2}}x^{\frac{p}{2}} + r^{p}) + yr^{p}x^{p}}{2\pi i (r^{\frac{p}{2}}-\sqrt{y} x^{\frac{p}{2}})(a^{p}-\sqrt{y} r^{\frac{p}{2}}x^{\frac{p}{2}})} \frac{1}{p} \frac{dy}{iy} -\pi i \mbox{Res}_{z_{1}} - \pi i \mbox{Res}_{z_{2}} \nonumber \\
& - \int_{1}^{M} \frac{a^{2p} + a^{p}(y x^{p} - 4 \sqrt{y} r^{\frac{p}{2}}x^{\frac{p}{2}} + r^{p}) + yr^{p}x^{p}}{2\pi i (r^{\frac{p}{2}}+\sqrt{y} x^{\frac{p}{2}})(a^{p}+\sqrt{y} r^{\frac{p}{2}}x^{\frac{p}{2}})} \frac{1}{p} \frac{dy}{iy}, \label{lol50}
\end{align}
where $\dashint$ stands for the principal value integral, and $\mbox{Res}_{z_{1}}, \mbox{Res}_{z_{2}}$ are given by
\begin{align}
& \mbox{Res}_{z_{1}} = \mbox{Res}\bigg( \frac{a^{2p} + a^{p}(z x^{p} + 4 z_{(0)}^{\frac{1}{2}}r^{\frac{p}{2}}x^{\frac{p}{2}} + r^{p}) + zr^{p}x^{p}}{2\pi i (r^{\frac{p}{2}}-z_{(0)}^{\frac{1}{2}}x^{\frac{p}{2}})(a^{p}-z_{(0)}^{\frac{1}{2}}r^{\frac{p}{2}}x^{\frac{p}{2}})} \frac{1}{p} \frac{1}{iz}, z=z_{1} \bigg) = \frac{a^{2p}+6a^{p}r^{p}+r^{2p}}{p\pi r^{\frac{p}{2}}(a^{p}-r^{p})}, \label{lol52} \\
& \mbox{Res}_{z_{2}} = \mbox{Res}\bigg( \frac{a^{2p} + a^{p}(z x^{p} + 4 z_{(0)}^{\frac{1}{2}}r^{\frac{p}{2}}x^{\frac{p}{2}} + r^{p}) + zr^{p}x^{p}}{2\pi i (r^{\frac{p}{2}}-z_{(0)}^{\frac{1}{2}}x^{\frac{p}{2}})(a^{p}-z_{(0)}^{\frac{1}{2}}r^{\frac{p}{2}}x^{\frac{p}{2}})} \frac{1}{p} \frac{1}{iz}, z=z_{2} \bigg) = \frac{a^{2p}+6a^{p}r^{p}+r^{2p}}{-p\pi r^{\frac{p}{2}}(a^{p}-r^{p})}. \label{lol53} 
\end{align}
On the other hand,
\begin{multline*}
\lim_{M\to\infty} \int_{M \, \mathbb{S}^{1}} \frac{a^{2p} + a^{p}(z x^{p} + 4 z_{(0)}^{\frac{1}{2}}r^{\frac{p}{2}}x^{\frac{p}{2}} + r^{p}) + zr^{p}x^{p}}{2\pi i (r^{\frac{p}{2}}-z_{(0)}^{\frac{1}{2}}x^{\frac{p}{2}})(a^{p}-z_{(0)}^{\frac{1}{2}}r^{\frac{p}{2}}x^{\frac{p}{2}})} \frac{1}{p} \frac{dz}{iz} \\
= \lim_{M\to \infty} \int_{M \, \mathbb{S}^{1}} \frac{a^{p}x^{p} + r^{p}x^{p}}{2\pi i (x^{\frac{p}{2}}r^{\frac{p}{2}}x^{\frac{p}{2}})} \frac{1}{p} \frac{dz}{iz} = -i\frac{a^{p}+r^{p}}{p \, r^{\frac{p}{2}}}.
\end{multline*}
The integrals $\int_{1}^{M}$ and $\dashint_{1}^{M}$ can be evaluated using primitives. Indeed, it is easy to check that
\begin{multline*}
\frac{d}{dy} \bigg( \frac{2(a^{2p}+6a^{p}r^{p}+r^{2p})\log(\frac{r^{\frac{p}{2}}+x^{\frac{p}{2}}\sqrt{y}}{a^{p}+r^{\frac{p}{2}}x^{\frac{p}{2}}\sqrt{y}})-(a^{2p}-r^{2p})\log y}{2\pi p r^{\frac{p}{2}}(a^{p}-r^{p})} \bigg) \\
= \frac{a^{2p} + a^{p}(y x^{p} - 4 \sqrt{y} r^{\frac{p}{2}}x^{\frac{p}{2}} + r^{p}) + yr^{p}x^{p}}{2\pi i (r^{\frac{p}{2}}+\sqrt{y} x^{\frac{p}{2}})(a^{p}+\sqrt{y} r^{\frac{p}{2}}x^{\frac{p}{2}})} \frac{1}{p} \frac{1}{iy},
\end{multline*}
which allows to evaluate the integral $\int_{1}^{M}$ explicitly. The regularized integral $\dashint_{1}^{M}$ can also be evaluated explicitly using similar primitives. Taking then $M\to +\infty$, and using $\arctanh(y) = \frac{1}{2}\log(\frac{1+y}{1-y})$, we find
\begin{align}
& \lim_{M\to \infty}  \dashint_{1}^{M} \hspace{-0.1cm} \bigg( \frac{a^{2p} + a^{p}(y x^{p} + 4 \sqrt{y} r^{\frac{p}{2}}x^{\frac{p}{2}} + r^{p}) + yr^{p}x^{p}}{2\pi i (r^{\frac{p}{2}}-\sqrt{y} x^{\frac{p}{2}})(a^{p}-\sqrt{y} r^{\frac{p}{2}}x^{\frac{p}{2}})p} - \frac{a^{2p} + a^{p}(y x^{p} - 4 \sqrt{y} r^{\frac{p}{2}}x^{\frac{p}{2}} + r^{p}) + yr^{p}x^{p}}{2\pi i (r^{\frac{p}{2}}+\sqrt{y} x^{\frac{p}{2}})(a^{p}+\sqrt{y} r^{\frac{p}{2}}x^{\frac{p}{2}})p} \bigg) \frac{dy}{iy} \nonumber \\
& = \frac{2(a^{2p}+6a^{p}r^{p}+r^{2p})}{\pi p r^{\frac{p}{2}}(a^{p}-r^{p})} \big( \arctanh((\tfrac{x}{r})^{\frac{p}{2}})-\arctanh((\tfrac{r}{a})^{\frac{p}{2}}(\tfrac{x}{a})^{\frac{p}{2}}) \big). \label{51}
\end{align}
Combining \eqref{lol50}, \eqref{lol52}, \eqref{lol53}, and \eqref{51}, we find the claim for $x<r$. The proof for $x>r$ is similar (now $z_{1}$ is inside the unit circle) and we omit further details.
\end{proof}
Combining \eqref{lol49} with Lemma \ref{lemma:simplif int segment}, we directly obtain \eqref{g sector nu segment}. Then \eqref{g sector nu segment 2} follows from the facts that $d\mu_{\mathrm{rad}}(x)/dx$ is continuous on $\mathrm{S}\setminus\{0\}$ and that the series
\begin{align}\label{series of arctanh}
\arctanh(y) = \sum_{m=0}^{+\infty} \frac{y^{2m+1}}{2m+1}
\end{align}
converges uniformly for $y$ in compact subsets of $[0,1)$.
\subsubsection*{Simplification of \eqref{lol48}: proof of \eqref{g sector nu arxtanh}--\eqref{g sector nu}}
It is possible to prove \eqref{g sector nu arxtanh}--\eqref{g sector nu} by rewriting the integral over $\alpha$ in \eqref{lol48} as an integral in the complex plane (in a similar way as we proved \eqref{g sector nu segment}--\eqref{g sector nu segment 2}). However, we give here a different proof using Fourier series, in a similar spirit as in \cite[Example 6.6]{AR2017}. 

\bigskip Note that $z_{(0)}^{n}$ is harmonic in $U$ and continuous on $\overline{U}$ for any $p>1$ and $n\in [0,+\infty)$ (not only for integers $n$), where we recall that $z^{n}_{(0)}:=|z|^{n}e^{i\arg_{0}z}$ with $\arg_{0}z \in (0,2\pi)$. Thus, by Lemma \ref{lemma:further properties of nu} (iii),
\begin{align}\label{moments znp0p}
& \int_{\partial U} z_{(0)}^{n}d\nu(z) = \int_{U}z_{(0)}^{n} d\mu(z), \qquad \mbox{for all } n \in [0,+\infty).
\end{align}
Since the right-hand sides of \eqref{lol49} and \eqref{lol48} are continuous for $p\in [1,+\infty)$, \eqref{moments znp0p} also holds with $p=1$ provided that $\partial U$ is defined as in Remark \ref{boundary of U if p=1}. Let $h(\theta)$ be the right-hand side of \eqref{lol48}. Recall from \eqref{g sector nu segment} that $d\nu(z)=d\nu(e^{\frac{2\pi i}{p}}z)$ for all $z\in (0,a)$. Thus, for all $n \in [0,+\infty)$,
\begin{align}
& (1+e^{\frac{2\pi n i}{p}}) \int_{0}^{a} t^{n} d\nu(t) + \int_{0}^{\frac{2\pi}{p}} a^{n}e^{in\theta}h(\theta)d\theta = \int_{0}^{\frac{2\pi}{p}} \int_{0}^{a} r^{n}e^{in\theta} d\mu_{\mathrm{rad}}(r) \frac{d\theta}{2\pi} \nonumber \\
& \hspace{2.5cm} = \int_{0}^{a} r^{n} d\mu_{\mathrm{rad}}(r) \times \begin{cases}
\frac{1}{p}, & \mbox{if } n =0, \\
0, & \mbox{if } n \mbox{ is a multiple of } p \mbox{ and } n \neq 0, \\
\frac{i(1-e^{\frac{2\pi ni}{p}})}{2\pi n}, & \mbox{otherwise}.
\end{cases} \label{moments circular sector}
\end{align}
Comparing the imaginary parts of both sides, we get
\begin{align*}
\frac{\sin (\frac{2\pi n}{p})}{a^{n}}\int_{0}^{a} t^{n} d\nu(t) + \int_{0}^{\frac{2\pi}{p}} \sin(n \theta) h(\theta)d\theta = \frac{1}{a^{n}}\int_{0}^{a} r^{n} d\mu_{\mathrm{rad}}(r) \times \begin{cases}
0, & \mbox{if } n \in p\N, \\
\frac{\sin(\frac{\pi n}{p})^{2}}{\pi n}, & \mbox{otherwise}.
\end{cases}
\end{align*}
In particular, if $n=\frac{p}{2}\N$, then
\begin{align}\label{lol4}
\int_{0}^{\frac{2\pi}{p}} \sin(n \theta) \; h(\theta)d\theta = \frac{1}{a^{n}}\int_{0}^{a} r^{n} d\mu_{\mathrm{rad}}(r) \times \begin{cases}
0, & \mbox{if } n \in p \N, \\
\frac{1}{\pi n}, & \mbox{if }n=\frac{p}{2}+p\N.
\end{cases}
\end{align}
A direct analysis of \eqref{lol48} shows that $h$ is continuous on $(0,\frac{2\pi}{p})$. Let $\hat{h}$ be the Fourier series of the odd extension of $h$ from $(0,\frac{2\pi}{p})$ to $(-\frac{2\pi}{p},\frac{2\pi}{p})$. Thus $\hat{h}(\theta) = \sum_{m=1}^{+\infty} h_{m} \sin(\frac{p}{2}mx)$ for some coefficients $h_{m}\in \R$. Furthermore, it is easy to see from \eqref{lol48} that $h$ (and therefore $\hat{h}$) satisfies $h(\frac{\pi}{p}+\theta)=h(\frac{\pi}{p}-\theta)$ for all $\theta \in (0,\frac{\pi}{p})$. This extra symmetry implies that $h_{m}=0$ whenever $m$ is even. Hence
\begin{align*}
h(\theta) = \hat{h}(\theta) = \sum_{m=0}^{+\infty}  h_{2m+1} \sin((pm+\tfrac{p}{2})\theta), \qquad \mbox{for all } \theta \in (0,\tfrac{\pi}{p}).
\end{align*}
Using 
\begin{align*}
\int_{0}^{\frac{2\pi}{p}} \sin((k_{1}p+\tfrac{p}{2}) \theta)\sin((k_{2}p+\tfrac{p}{2}) \theta) d\theta = \frac{2}{p} \int_{0}^{\pi} \sin((2k_{1}+1)\alpha) \sin((2k_{2}+1) \alpha)d\alpha = \begin{cases}
0, & \mbox{if } k_{1} \neq k_{2} \\
\frac{\pi}{p}, & \mbox{if } k_{1}=k_{2}
\end{cases}
\end{align*}
together with \eqref{lol4}, we obtain
\begin{align}\label{h as a Fourier series smaller sector}
h(\theta) = \frac{1}{\pi^{2}} \sum_{m=0}^{+\infty}  \frac{\sin((pm+\frac{p}{2})\theta)}{(m+\frac{1}{2}) \, a^{pm+\frac{p}{2}}}\int_{0}^{a} r^{pm+\frac{p}{2}} d\mu_{\mathrm{rad}}(r), \qquad \theta \in (0,\tfrac{2\pi}{p}),
\end{align}
which is \eqref{g sector nu}. Then \eqref{g sector nu arxtanh} follows using \eqref{series of arctanh}. 
This finishes the proof of Theorem \ref{thm:g sector nu}.
\subsubsection{The constant C: proof of Theorem \ref{thm:g sector C}}
In the statement of Theorem \ref{thm:g sector C} (i), the assumption $p\geq 2$ is made to ensure that $U$ satisfies Assumption \ref{ass:U2} (this is required to use Theorem \ref{thm:general pot}). 

\medskip \noindent Since $d\nu(z)=d\nu(e^{\frac{2\pi i}{p}}z)$ for all $z\in (0,a)$, we have
\begin{align}
& \int_{\partial U} g(|z|) d\nu(z) = 2 \int_{0}^{a} g(r) d\nu(r) + g(a) \int_{0}^{\frac{2\pi}{p}} h(\theta)d\theta, \nonumber  \\
& \int_{0}^{\frac{2\pi}{p}} h(\theta)d\theta = \frac{2p}{\pi^{2}} \sum_{m=0}^{+\infty} \frac{1}{(pm+\frac{p}{2})^{2}} \frac{1}{a^{pm+\frac{p}{2}}} \int_{0}^{a} r^{pm+\frac{p}{2}} d\mu_{\mathrm{rad}}(r) \label{lol7}
\end{align}
where we recall that $h$ is given by \eqref{h as a Fourier series smaller sector}. Now \eqref{C in thm cir sector} directly follows from Theorem \ref{thm:general pot}.

\medskip \noindent The integral $\int_{0}^{a} g(r) d\nu(r)$ above is difficult to simplify further using \eqref{g sector nu segment 2}. Suppose that $g(r) = \sum_{k\in \mathcal{I}} g_{k}r^{k}$ uniformly for all $r \in [0,a]$, where $\mathcal{I}\subset [0,+\infty)$ is countable with no accumulation points and such that $\mathcal{I}\cap (\frac{p}{2}+p\N) = \emptyset$ (thus $g$ must not contain powers of the form $r^{k}$ for some $k \in \frac{p}{2}+p\N$). For later convenience we also assume without loss of generality that $0\notin \mathcal{I}$. Then we can simplify $\int_{0}^{a} g(r) d\nu(r)$ as follows. First, taking the real part of \eqref{moments circular sector} yields (for any $k \notin \frac{p}{2}+p\N$, $k\geq 0$)
\begin{align*}
& \int_{0}^{a} r^{k} d\nu(r) =  \frac{-a^{k}}{1+\cos\tfrac{2\pi k}{p}} \int_{0}^{\frac{2\pi}{p}} \cos (k\theta)h(\theta)d\theta  \nonumber \\
&  + \frac{1}{1+\cos\tfrac{2\pi k}{p}} \int_{0}^{a} r^{k} d\mu_{\mathrm{rad}}(r) \times \begin{cases}
\frac{1}{p}, & \mbox{if } k =0, \\
0, & \mbox{if } k \mbox{ is a multiple of } p \mbox{ and } k \neq 0, \\
\frac{\re[i(1-e^{\frac{2\pi ki}{p}})]}{2\pi k}, & \mbox{otherwise}.
\end{cases}
\end{align*}
Multiplying by $g_{k}$ and then summing over $k \in \mathcal{I}$ yields
\begin{align*}
\int_{0}^{a} g(r) d\nu(r) &  =  \sum_{k\in \mathcal{I}} \bigg\{ \frac{-g_{k}a^{k}}{1+\cos\tfrac{2\pi k}{p}} \int_{0}^{\frac{2\pi}{p}} \cos (k\theta)h(\theta)d\theta + (\tan\tfrac{\pi k}{p}) \frac{g_{k}}{2k\pi}  \int_{0}^{a} r^{k} d\mu_{\mathrm{rad}}(r) \bigg\}.
\end{align*}
Using then \eqref{h as a Fourier series smaller sector}, we get
\begin{align*}
\int_{0}^{a} g(r) d\nu(r) &  =  \sum_{k\in \mathcal{I}} \bigg\{ \frac{p g_{k}a^{k}}{\pi^{2}} \bigg[ \sum_{m=0}^{+\infty} \frac{a^{-(pm+\frac{p}{2})}}{(k+pm+\frac{p}{2})(k-pm-\frac{p}{2})} \int_{0}^{a} r^{pm+\frac{p}{2}} d\mu_{\mathrm{rad}}(r) \bigg]  \nonumber \\
& + (\tan\tfrac{\pi k}{p}) \frac{g_{k}}{2k\pi}  \int_{0}^{a} r^{k} d\mu_{\mathrm{rad}}(r) \bigg\}. 
\end{align*}
Using the identity $\frac{\tan(\pi y)}{y} = -\frac{2}{\pi}\sum_{m=0}^{+\infty} \frac{1}{(y+m+\frac{1}{2})(y-m-\frac{1}{2})}$, the above simplifies to
\begin{align}\label{lol8}
\int_{0}^{a} g(r) d\nu(r) =  \sum_{k\in \mathcal{I}} \frac{p g_{k}a^{k}}{\pi^{2}} \bigg[ \sum_{m=0}^{+\infty} \frac{1}{(k+pm+\frac{p}{2})(k-pm-\frac{p}{2})} \bigg( \frac{\int_{0}^{a} r^{pm+\frac{p}{2}} d\mu_{\mathrm{rad}}(r)}{a^{pm+\frac{p}{2}}} - \frac{\int_{0}^{a} r^{k} d\mu_{\mathrm{rad}}(r)}{a^{k}} \bigg) \bigg]
\end{align}
Substituting \eqref{lol8} in \eqref{C in thm cir sector} yields \eqref{g sector C}. Also, it is clear from \eqref{C in thm cir sector} and \eqref{g sector nu segment}--\eqref{g sector nu} that $C$ is continuous for $p\in [2,+\infty)$. Thus, if $\mathcal{I}\cap (\frac{p}{2}+p\N) \neq \emptyset$, then $C$ can be obtained by first replacing $p$ by $p'$ in the right-hand side of \eqref{g sector C}, and then taking $p'\to p$. This finishes the proof of Theorem \ref{thm:g sector C} (i). 

\medskip The proof of Theorem \ref{thm:g sector C} (ii) is similar (it suffices to replace $g(r)$ by $\log(1+r^{2})$, and to use the fact that $\log(1+r^{2}) = \sum_{k=1}^{+\infty} \frac{(-1)^{k+1}}{k}r^{2k}$ converges uniformly on $[0,a]$ if $a<1$). 

\medskip  The proof of Theorem \ref{thm:g sector C} (iii) (corresponding to the case $Q(z) = |z|^{2}$) is a particular case of the proof of Theorem \ref{thm:g sector C} (i), except for the fact that we are not allowed to use Theorem \ref{thm:general pot} for $p<2$. Instead we use \cite[Theorem 1.2 (A)]{A2018} (or \cite[Theorem 1.4]{AR2017} for $\beta=2$), using the fact that $U$ satisfies Remark \ref{remark:AR star domain} (b). (The assumption $a\neq 1$ is required to ensure that the sets $U_{m}$ in Remark \ref{remark:AR star domain} satisfy $U_{m}\subset S=\{z:|z|\leq 1\}$.)

\section{The elliptic Ginibre point process}\label{section:elliptic Ginibre and several hole regions}
In this section, we prove Theorems \ref{thm:general U EG}, \ref{thm:Elliptic disk C} and \ref{thm:Ginibre triangle C}.
\subsection{Proof of Theorem \ref{thm:general U EG}}
Suppose $U\subset \C$ is bounded and such that Assumption \ref{ass:U} holds (with $S$ replaced by $\C$), and define $\mu:=\frac{d^{2}z}{\pi(1-\tau^{2})}$, $\mathrm{supp} \, \mu := \C$, $\nu := \mathrm{Bal}(\mu|_{\zeta_{0}+\rho e^{i\theta_{0}}U},\zeta_{0}+\rho e^{i\theta_{0}}\partial U)$. By Proposition \ref{prop:def of bal}, 
\begin{align*}
\int_{\zeta_{0}+\rho e^{i\theta_{0}}\partial U} \log\frac{1}{|z-w|}d\nu(w) = \int_{\zeta_{0}+\rho e^{i\theta_{0}} U} \log\frac{1}{|z-w|}d\mu(w) \qquad \mbox{for q.e. } z \notin \zeta_{0}+\rho e^{i\theta_{0}}U.
\end{align*}
Using the change of variables $z=\zeta_{0}+\rho e^{i\theta_{0}}z'$, $w=\zeta_{0}+\rho e^{i\theta_{0}}w'$, $d\mu(w) = \rho^{2} d\mu(w')$, and the fact that $\nu(\zeta_{0}+\rho e^{i\theta_{0}}\partial U) = \mu(\zeta_{0}+\rho e^{i\theta_{0}} U)$, we get
\begin{align*}
\int_{\partial U} \log\frac{1}{|z'-w'|}d\nu(\zeta_{0}+\rho e^{i\theta_{0}}w') = \int_{U} \log\frac{1}{|z'-w'|} \; \rho^{2} d\mu(w') \qquad \mbox{for q.e. } z' \notin U.
\end{align*}
Hence, by the uniqueness part of Proposition \ref{prop:def of bal},
\begin{align}\label{dnu relation in EG}
\frac{d\nu(\zeta_{0}+\rho e^{i\theta_{0}}w')}{\rho^{2}} = d\nu_{0}(w'), \qquad w'\in \partial U,
\end{align}
where $\nu_{0} := \mathrm{Bal}(\mu|_{U},\partial U)$. Let $\tau \in [0,1)$ and define
\begin{align*}
C(\tau,\zeta_{0},\rho,\theta_{0}) = \frac{\beta}{4}\bigg(  \int_{\zeta_{0}+\rho e^{i\theta_{0}}\partial U}Q(z)d\nu(z) - \int_{\zeta_{0}+\rho e^{i\theta_{0}}U}Q(z) \frac{d^{2}z}{\pi(1-\tau^{2})} \bigg),
\end{align*}
where $Q(z):=\frac{1}{1-\tau^{2}}\big( |z|^{2}-\tau \, \re z^{2} \big)$. By Lemma \ref{lemma:further properties of nu} (iii),
\begin{align}\label{lol95}
\int_{\partial U} h(z') d\nu_{0}|_{\tau=0}(z') = \int_{U} h(z') \frac{d^{2}z'}{\pi}
\end{align}
holds for each function $h$ continuous on $\overline{U}$ and harmonic on $U$. Using $z=\zeta_{0}+\rho e^{i\theta_{0}}z'$ and \eqref{dnu relation in EG}, we get
\begin{align*}
C(\tau,\zeta_{0},\rho,\theta_{0}) & = \frac{\beta}{4}\frac{\rho^{2}}{1-\tau^{2}} \bigg(  \int_{\partial U}Q(\zeta_{0}+\rho e^{i\theta_{0}}z')d\nu_{0}|_{\tau=0}(z') - \int_{U}Q(\zeta_{0}+\rho e^{i\theta_{0}}z') \frac{d^{2}z'}{\pi} \bigg), \\
& = \frac{\beta}{4}\frac{\rho^{2}}{1-\tau^{2}} \bigg(  \int_{\partial U}\frac{\rho^{2}|z'|^{2}}{1-\tau^{2}} d\nu_{0}|_{\tau=0}(z') - \int_{U} \frac{\rho^{2}|z'|^{2}}{1-\tau^{2}} \frac{d^{2}z'}{\pi} \bigg) \\
& = \frac{\rho^{4}}{(1-\tau^{2})^{2}}C(0,0,1,0),
\end{align*}
where for the second equality we have used \eqref{lol95} with $h(z') := Q(\zeta_{0}+\rho e^{i\theta_{0}}z') - \frac{\rho^{2}|z'|^{2}}{1-\tau^{2}}$ ($h$ is harmonic on $U$ and continuous on $\overline{U}$ because $U$ is assumed to be bounded). The proof is complete.
\subsection{Proof of Theorems \ref{thm:Elliptic disk C} and \ref{thm:Ginibre triangle C}}
As mentioned in Remark \ref{remark:tau=theta0:z0=0}, the constants $C$ in \eqref{C ellipse EG}, \eqref{C annulus EG}, \eqref{C cardioid EG} were already known for $\tau=\theta_{0}=\zeta_{0}=0$ and $\beta=2$ from \cite{AR2017}. Also, the constant $C$ in \eqref{C circular sector EG} is already proved for $\tau=\theta_{0}=\zeta_{0}=0$ in Theorem \ref{thm:g sector C} (see also Remark \ref{remark:circular sector and ML}). Note also that each of the four hole regions $U$ considered in the statement of Theorem \ref{thm:Elliptic disk C} satisfies Assumptions \ref{ass:U} and \ref{ass:U2}. Hence Theorem \ref{thm:Elliptic disk C} directly follows by combining Theorems \ref{thm:general U EG} with \ref{thm:general pot}.

Similarly, since Theorem \ref{thm:Ginibre triangle C} was already proved for $\tau=0$, $\beta=2$, $\zeta_{0}=0$ and $\theta_{0} = \frac{\pi}{3}$ in \cite{AR2017}, Theorem \ref{thm:Ginibre triangle C} directly follows by combining Theorems \ref{thm:general U EG} with \ref{thm:general pot}.
\section{The disk}\label{section:disk}
In this section, $U=\{z:|z-z_{0}|<a \}$, $z_{0}:=x_{0}+iy_{0}$ for some $x_{0},y_{0}\in \R$ and $a>0$. Fix $b \in \N_{>0}$, and let $\mu$ be as in \eqref{mu S ML}, i.e. $d\mu(z) = \frac{b^{2}}{\pi}|z|^{2b-2}d^{2}z$ and $S:=\mathrm{supp}\,\mu =\{z:|z|\leq b^{-\frac{1}{2b}}\}$. Since $\mu$ is rotation-invariant, we will assume without loss of generality that $y_{0}=0$ and $x_{0}\geq 0$. 
\subsection{Balayage measure: proof of Theorem \ref{thm:ML disk nu}}\label{subsec:nu disk}
Suppose that $x_{0}\geq 0$ and $a>0$ are such that $U\subset S$ (i.e. $x_{0}+a \leq b^{-\frac{1}{2b}}$). The goal of this subsection is to find an explicit expression for $\nu:=\mathrm{Bal}(\mu|_{U},\partial U)$. We will obtain $\nu$ using the first method described in Subsection \ref{section:ansatz method}, i.e. by making an ansatz for $\nu$ and verifying that the moments of $\nu$ and $\mu$ match (using Lemma \ref{lemma: moment when U is bounded}).
\begin{lemma}\label{lemma:moments of mu DISK}
For any $n \in \N$, we have
\begin{align*}
\int_{U} z^{n} d\mu(z) = \sum_{j=0}^{b-1} \binom{n}{j} x_{0}^{n-j} f_{j},
\end{align*}
where
\begin{align*}
f_{j} = 2b^{2} \sum_{\substack{k=0 \\ k-j \, \mathrm{even}}}^{b-1} \binom{b-1}{k} \binom{k}{\frac{k-j}{2}} x_{0}^{k} \sum_{m=0}^{b-1-k}\binom{b-1-k}{m}\frac{x_{0}^{2(b-1-k-m)}a^{2+j+k+2m}}{2+j+k+2m}.
\end{align*}
\end{lemma}
\begin{proof}
By definitions of $\mu$, $U$, we have 
\begin{align*}
\int_{U} z^{n} d\mu(z) & = \int_{U} z^{n} \frac{b^{2}}{\pi}|z|^{2b-2}d^{2}z = \int_{0}^{2\pi}\int_{0}^{a} (x_{0}+re^{i\theta})^{n} \frac{b^{2}}{\pi}|x_{0}+re^{i\theta}|^{2b-2}rdrd\theta \\
& = \int_{0}^{2\pi} \int_{0}^{a} \sum_{j=0}^{n} \binom{n}{j} x_{0}^{n-j} r^{j} e^{i j \theta} \frac{b^{2}}{\pi}\bigg( x_{0}^{2}+r^{2}+r x_{0}(e^{i\theta}+e^{-i\theta}) \bigg)^{b-1}r dr d\theta
\end{align*}
By expanding further, using that $b\in \N_{>0}$, we obtain
\begin{align*}
& \int_{U} z^{n} d\mu(z) = \int_{0}^{2\pi} \int_{0}^{a} \sum_{j=0}^{n} \binom{n}{j} x_{0}^{n-j} r^{j} e^{i j \theta} \frac{b^{2}}{\pi} \sum_{k=0}^{b-1} \binom{b-1}{k}r^{k}x_{0}^{k}(e^{i\theta}+e^{-i\theta})^{k}(r^{2}+x_{0}^{2})^{b-1-k} r dr d\theta \\
	& = \int_{0}^{2\pi} \int_{0}^{a} \sum_{j=0}^{n} \binom{n}{j} x_{0}^{n-j} r^{j} e^{i j \theta} \frac{b^{2}}{\pi} \sum_{k=0}^{b-1} \binom{b-1}{k}r^{k}x_{0}^{k}\sum_{\ell=0}^{k}\binom{k}{\ell}e^{(2\ell-k)i\theta}(r^{2}+x_{0}^{2})^{b-1-k} r dr d\theta.
\end{align*}
Using that $\int_{0}^{2\pi} e^{ij\theta}e^{(2\ell-k)i\theta}d\theta$ is equal to $2\pi$ if $\ell = \frac{k-j}{2}$ and is $0$ otherwise, the above becomes
\begin{align*}
\int_{U} z^{n} d\mu(z) & = 2b^{2} \sum_{j=0}^{b-1} \binom{n}{j} x_{0}^{n-j} \sum_{\substack{k=0 \\ k-j \, \mathrm{even}}}^{b-1} \binom{b-1}{k} \binom{k}{\frac{k-j}{2}} x_{0}^{k} \int_{0}^{a}  r^{1+j+k} (r^{2}+x_{0}^{2})^{b-1-k} dr.
\end{align*}
The claim is found after evaluating the remaining integral $\int_{0}^{a}$ explicitly (using again that $b\in \N_{>0}$).
\end{proof}
\begin{lemma}\label{lemma:moments of nu DISK}
Let $c_{0},c_{1},\ldots,c_{b-1}\in \R$, and let $\hat{\nu}$ be the (possibly signed) measure supported on $\partial U$ defined by $d\hat{\nu}(z) := \big( c_{0} + 2 \sum_{\ell=1}^{b-1} c_{\ell} \cos (\ell \theta) \big)\frac{d\theta}{\pi}$ for $z=x_{0}+a e^{i\theta}$, $\theta \in [0,2\pi)$. For $n \in \N$, we have
\begin{align}\label{lol51}
\int_{\partial U} z^{n} d\hat{\nu}(z) = 2 \sum_{\ell=0}^{b-1} c_{\ell} \binom{n}{\ell} x_{0}^{n-\ell} a^{\ell}.
\end{align}
\end{lemma}
\begin{proof}
Since $\partial U = \{z\in \C:|z-x_{0}|=a\}$, 
\begin{align*}
& \int_{\partial U} z^{n} d\hat{\nu}(z) = \int_{0}^{2\pi} (x_{0}+ae^{i\theta})^{n} \bigg( c_{0} + 2 \sum_{\ell=1}^{b-1} c_{\ell} \cos (\ell \theta) \bigg)\frac{d\theta}{\pi} \\
& = \int_{0}^{2\pi} \sum_{j=0}^{n}\binom{n}{j}x_{0}^{n-j}a^{j}e^{ij\theta} \bigg( c_{0} + \sum_{\ell=1}^{b-1} c_{\ell} (e^{\ell i \theta}+e^{-\ell i \theta}) \bigg)\frac{d\theta}{\pi} = \int_{0}^{2\pi} \sum_{j=0}^{n}\binom{n}{j}x_{0}^{n-j}a^{j}e^{ij\theta} \sum_{\ell=0}^{b-1} c_{\ell} e^{-\ell i \theta} \frac{d\theta}{\pi},
\end{align*}
and \eqref{lol51} follows.
\end{proof}
We now finish the proof of Theorem \ref{thm:ML disk nu}. Let $\hat{\nu}$ be as in Lemma \ref{lemma:moments of nu DISK} with $c_{0},c_{1},\ldots,c_{b-1}$ given by \eqref{def of cj DISK intro}. Since $c_{\ell} = \frac{f_{\ell}}{2a^{\ell}}$, it follows from Lemmas \ref{lemma:moments of mu DISK} and \ref{lemma:moments of nu DISK} that $\int_{\partial U}z^n d\hat{\nu}(z) = \int_{ U}z^nd\mu(z)$ for all $n \in \N$. By Lemma \ref{lemma: moment when U is bounded} and Remark \ref{remark: pmu bounded on dU}, $\hat{\nu}=\mathrm{Bal}(\mu|_{U},\partial U)$.

\subsection{The constant C: proof of Theorem \ref{thm:ML disk C}}\label{subsec:C disk EG}
In this subsection, $b \in \N_{>0}$, $d\mu(z) := \frac{b^{2}}{\pi}|z|^{2b-2}d^{2}z$, $S:=\mathrm{supp}\,\mu =\{z:|z|\leq b^{-\frac{1}{2b}}\}$, and $x_{0}\geq 0$, $a>0$ are such that $U:=\{z:|z-x_{0}|<a \}\subset S$ (i.e. $x_{0}+a \leq b^{-\frac{1}{2b}}$).
\begin{lemma}
\begin{align}\label{int |z|2b dmu DISK}
\int_{U} |z|^{2b} d\mu(z) = b^{2} \sum_{j=0}^{b-1} \binom{2b-1}{2j} \binom{2j}{j} \sum_{k=0}^{2b-1-2j}\binom{2b-1-2j}{k} \frac{x_{0}^{2(2b-1-j-k)}a^{2(1+j+k)}}{1+j+k}.
\end{align}
\end{lemma}
\begin{proof}
Using the change of variables $z=x_{0}+re^{i\theta}$, $r\geq 0$, $\theta \in [0,2\pi)$, we obtain
\begin{align*}
& \int_{U} |z|^{2b} d\mu(z) = \int_{U} |z|^{2b} \frac{b^{2}}{\pi} |z|^{2b-2}d^{2}z = \frac{b^{2}}{\pi} \int_{0}^{2\pi} \int_{0}^{a} | x_{0} + r e^{i\theta}|^{4b-2} r dr d\theta \\
& = \frac{b^{2}}{\pi} \int_{0}^{2\pi} \int_{0}^{a} \big( x_{0}^{2}+r^{2}+rx_{0}(e^{i\theta} + e^{-i\theta}) \big)^{2b-1}r dr d\theta \\
& = \frac{b^{2}}{\pi} \sum_{j=0}^{2b-1} \binom{2b-1}{j}x_{0}^{j} \sum_{k=0}^{j} \binom{j}{k} \int_{0}^{a}r^{1+j}(x_{0}^{2}+r^{2})^{2b-1-j}dr \int_{0}^{2\pi}e^{i\theta(2k-j)}d\theta.
\end{align*}
After evaluating the $\theta$-integral, changing the indice of summation $j \to 2j$, and then performing the $r$-integral, we find \eqref{int |z|2b dmu DISK}.
\end{proof}
\begin{lemma}
Let $\nu = \mathrm{Bal}(\mu|_{U},\partial U)$. Then
\begin{align}\label{int |z|2b dnu DISK}
\int_{\partial U} |z|^{2b} d\nu(z) = 2\sum_{j=0}^{b} \binom{b}{j}a^{j}x_{0}^{j}(x_{0}^{2}+a^{2})^{b-j} \bigg[ c_{0} \binom{j}{\frac{j}{2}} \mathbf{1}_{j \, \mathrm{even}} + 2 \sum_{\ell=1}^{b-1} c_{\ell} \binom{j}{\frac{j+\ell}{2}} \mathbf{1}_{j+\ell \, \mathrm{even}} \bigg],
\end{align}
where $c_{0},\ldots,c_{b-1}$ are given by \eqref{def of cj DISK intro}.
\end{lemma}
\begin{proof}
Recall from Theorem \ref{thm:ML disk nu} that $d\nu(z) = \big( c_{0} + 2 \sum_{\ell=1}^{b-1} c_{\ell} \cos (\ell \theta) \big)\frac{d\theta}{\pi}$ for $z=x_{0}+a e^{i\theta}$, $\theta \in [0,2\pi)$, where $c_{0},\ldots,c_{b-1}$ are given by \eqref{def of cj DISK intro}. Since $b\in \N_{>0}$,
\begin{align*}
& \int_{\partial U} |z|^{2b} d\nu(z) = \int_{0}^{2\pi} \big( x_{0}^{2} + a^{2} + a x_{0} (e^{i\theta}+e^{-i\theta}) \big)^{b} \bigg( c_{0} + \sum_{\ell=1}^{b-1} c_{\ell} (e^{i\ell\theta}+e^{-i\ell\theta}) \bigg)\frac{d\theta}{\pi} \\
& = \int_{0}^{2\pi} \sum_{j=0}^{b} \binom{b}{j}a^{j}x_{0}^{j}(e^{i\theta}+e^{-i\theta})^{j}(x_{0}^{2}+a^{2})^{b-j} \bigg( c_{0} + \sum_{\ell=1}^{b-1} c_{\ell} (e^{i\ell\theta}+e^{-i\ell\theta}) \bigg)\frac{d\theta}{\pi} \\
& = \int_{0}^{2\pi} \sum_{j=0}^{b} \binom{b}{j}a^{j}x_{0}^{j}\sum_{k=0}^{j}\binom{j}{k}e^{i\theta(2k-j)} (x_{0}^{2}+a^{2})^{b-j} \bigg( c_{0} + \sum_{\ell=1}^{b-1} c_{\ell} (e^{i\ell\theta}+e^{-i\ell\theta}) \bigg)\frac{d\theta}{\pi} \\
& = 2\sum_{j=0}^{b} \binom{b}{j}a^{j}x_{0}^{j}(x_{0}^{2}+a^{2})^{b-j} \bigg[ c_{0} \binom{j}{\frac{j}{2}} \mathbf{1}_{j \, \mathrm{even}} + \sum_{\ell=1}^{b-1} c_{\ell} \bigg( \binom{j}{\frac{j+\ell}{2}} + \binom{j}{\frac{j-\ell}{2}} \bigg) \mathbf{1}_{j+\ell \, \mathrm{even}} \bigg].
\end{align*}
Using $\binom{j}{\frac{j+\ell}{2}} = \binom{j}{\frac{j-\ell}{2}}$, we find \eqref{int |z|2b dnu DISK}.
\end{proof}
Note that $U$ satisfies Assumptions \ref{ass:U} and \ref{ass:U2}. Theorem \ref{thm:ML disk C} now follows directly from \eqref{int |z|2b dmu DISK}, \eqref{int |z|2b dnu DISK} and Theorem \ref{thm:general pot} (i).

\section{The ellipse centered at $0$}\label{section:ellipse centered}
In this section, $U=\{z:(\frac{\re z}{a})^{2}+(\frac{\im z}{c})^{2}<1 \}$ for some $a>0$, $c>0$. 

\subsection{Balayage measure: proof of Theorem \ref{thm:ML ellipse nu}}
Fix $b \in \N_{>0}$, and let $\mu$ be as in \eqref{mu S ML}, i.e. $d\mu(z) = \frac{b^{2}}{\pi}|z|^{2b-2}d^{2}z$ and $S:=\mathrm{supp}\,\mu =\{z:|z|\leq b^{-\frac{1}{2b}}\}$. 
Suppose that $0<a,c\leq b^{-\frac{1}{2b}}$, so that $U := \{z:(\frac{\re z}{a})^{2}+(\frac{\im z}{c})^{2}<1\} \subset S$. Let $\alpha := \frac{a+c}{2}$ and $\gamma := \frac{a-c}{2}$. In this subsection, we obtain an explicit expression for $\nu:=\mathrm{Bal}(\mu|_{U},\partial U)$. While it is in principle possible to use Theorem \ref{thm:dnu in terms of green general}, in practice this method leads to an expression for $\nu$ in terms of elliptic integrals that is difficult to simplify. Instead we use the first method described in Subsection \ref{section:ansatz method}, i.e. by making an ansatz for $\nu$ and verifying that the moments of $\nu$ and $\mu$ match (using Lemma \ref{lemma: moment when U is bounded}). 

\medskip The most important technical challenge that we have to overcome in this subsection is to invert explicitly triangular matrices of size $(b-k+1)\times (b-k+1)$ for all $k\in \{0,1,\ldots,b-1\}$ (thus for large $b\in \N_{>0}$, we have to invert large matrices). This is needed in the proof of Lemma \ref{lemma:dpkp}, and this is in fact the only part in the proof where treating general $b\in \N_{>0}$ is clearly harder than treating a finite number of values of $b$, say $b\in \{1,2,3,4\}$.

\medskip We start by making the ansatz that $\nu(z) = [c_{0}+2\sum_{\ell=1}^{b}c_{\ell} \cos (2\ell \theta) ] \frac{d\theta}{\pi}$ for $z=a\cos \theta + c \sin \theta\in \partial U$ and some $c_{0},\ldots,c_{b}\in \R$. If this ansatz is true, then the moments of $\nu$ are given by the following lemma.
\begin{lemma}\label{lemma:moments of nu ellipse}
Let $\hat{\nu}$ be a measure supported on $\partial U$ and of the form $\hat{\nu}(z) = [c_{0}+2\sum_{\ell=1}^{b}c_{\ell} \cos (2\ell \theta) ] \frac{d\theta}{\pi}$ with $z=a\cos \theta + c \sin \theta\in \partial U$, $\theta\in [0,2\pi)$, and for some $\{c_{\ell}\}_{\ell=0}^{b} \subset \R$. Then
\begin{align}
& \int_{\partial U} z^{n} d\hat{\nu}(z) = \begin{cases}
\ds 2 \alpha^{\frac{n}{2}} \gamma^{\frac{n}{2}} \bigg( c_{0} \binom{n}{n/2} + \sum_{\ell=1}^{b} c_{\ell} \binom{n}{\frac{n}{2}-\ell} (\alpha^{\ell}\gamma^{-\ell} + \alpha^{-\ell}\gamma^{\ell}) \bigg), & \mbox{if } n \mbox{ is even}, \\
0, & \mbox{if } n \mbox{ is odd},
\end{cases} \label{moments in terms of the cj}
\end{align}
where $\binom{n}{\frac{n}{2}-\ell}:=0$ if $\ell \geq \frac{n}{2}+1$.
\end{lemma}
\begin{proof}
Using the change of variables $z=a\cos \theta + i c \sin \theta = \alpha e^{i\theta} + \gamma e^{-i\theta}$, $\theta \in [0,2\pi)$, we get
\begin{align*}
\int_{\partial U} z^{n} d\hat{\nu}(z) & = \int_{0}^{2\pi} (\alpha e^{i\theta} + \gamma e^{-i\theta})^{n} \bigg[c_{0} + \sum_{\ell=1}^{b}c_{\ell} (e^{2\ell i \theta} + e^{-2\ell i \theta})\bigg] \frac{d\theta}{\pi} \\
& = \sum_{j=0}^{n} \binom{n}{j}\alpha^{j}\gamma^{n-j} \int_{0}^{2\pi} \bigg[c_{0}e^{(2j-n)i\theta} + \sum_{\ell=1}^{b}c_{\ell} (e^{(2j+2\ell-n)i\theta} + e^{(2j-2\ell-n)i\theta})\bigg] \frac{d\theta}{\pi}.
\end{align*}
The claim follows after performing the $\theta$-integral and using that $\binom{n}{\frac{n}{2}-\ell}=\binom{n}{\frac{n}{2}+\ell}$.
\end{proof}
The goal is to choose $c_{0},\ldots,c_{b}$ in Lemma \ref{moments in terms of the cj} so that $\int_{\partial U} z^{n} d\hat{\nu}(z) = \int_{U} z^{n} d\mu(z)$ for all $n\in \N$ (by Lemma \ref{lemma: moment when U is bounded} this would imply $\hat{\nu}=\nu$). We now turn to the computation of $\int_{U} z^{n} d\mu(z)$. Recall that $\alpha = \frac{a+c}{2}$ and $\gamma = \frac{a-c}{2}$.
\begin{lemma}\label{lemma: moments of mu ellipse, first lemma}
For any $n\in \N$, we have $\int_{U} z^{n} d\mu(z)=0$ if $n$ is odd, while if $n$ is even we have
\begin{align}
\int_{U} z^{n} d\mu(z) & = 2 \alpha^{\frac{n}{2}}\gamma^{\frac{n}{2}} \frac{ (\alpha^{2}-\gamma^{2}) \, b^{2}}{2} \sum_{\ell=0}^{b-1} \binom{b-1}{\ell} (\alpha^{2}+\gamma^{2})^{b-1-\ell} \nonumber \\
&  \times \sum_{\substack{k=0 \, \mathrm{or} \, 1 \\ \ell-k \, \mathrm{even}}}^{\ell} \binom{\ell}{\frac{\ell-k}{2}} \frac{\binom{n}{\frac{n}{2}+k}}{\frac{n}{2}+b} \frac{1+\mathbf{1}_{k \neq 0}}{2} (\alpha^{\ell+k} \gamma^{\ell-k} + \alpha^{\ell-k} \gamma^{\ell+k}), \label{moments of mu ellipse 1}
\end{align}
where $\mathbf{1}_{k \neq 0}=1$ if $k \neq 0$ and $\mathbf{1}_{k \neq 0}=0$ if $k=0$.
\end{lemma}
\begin{proof}
With the change of variables $z=ar \cos \theta + i c r \sin \theta$, $0 \leq \theta \leq 2\pi$, $0 \leq r < 1$, we have $d^{2}z = acrdrd\theta$, and thus
\begin{align*}
\int_{U} z^{n} d\mu(z) & = \int_{U} z^{n} \frac{b^{2}}{\pi}|z|^{2b-2}d^{2}z  \\
& = a c \frac{b^{2}}{\pi} \int_{0}^{1} r^{n+2b-1} dr  \int_{0}^{2\pi} (a \cos \theta + i c \sin \theta)^{n} |a\cos \theta + i c \sin \theta|^{2b-2} d\theta \\
& =  \frac{a c \, b^{2}}{(n+2b)\pi} \int_{0}^{2\pi} (a \cos \theta + i c \sin \theta)^{n} |a\cos \theta + i c \sin \theta|^{2b-2} d\theta.
\end{align*}
Rewriting $a\cos \theta + i c \sin \theta = \alpha e^{i\theta} + \gamma e^{-i\theta}$, $|a\cos \theta + i c \sin \theta|^{2} = \alpha \gamma (e^{2i\theta}+e^{-2i\theta}) + \alpha^{2}+\gamma^{2}$, and using that $b\in \N_{>0}$, we obtain
\begin{align*}
& \int_{U} z^{n} d\mu(z) =  \frac{(\alpha^{2}-\gamma^{2}) \, b^{2}}{(n+2b)\pi} \int_{0}^{2\pi} \Big( \alpha e^{i\theta} + \gamma e^{-i\theta} \Big)^{n} \Big( \alpha \gamma (e^{2i\theta}+e^{-2i\theta}) + \alpha^{2}+\gamma^{2} \Big)^{b-1} d\theta \\
& = \frac{(\alpha^{2}-\gamma^{2}) \, b^{2}}{(n+2b)\pi} \sum_{\ell=0}^{b-1} \binom{b-1}{\ell} \alpha^{\ell} \gamma^{\ell} (\alpha^{2}+\gamma^{2})^{b-1-\ell} \int_{0}^{2\pi} (e^{2i\theta} + e^{-2i\theta})^{\ell}  \sum_{j=0}^{n} \binom{n}{j} \alpha^{j} \gamma^{n-j} e^{ij\theta}e^{-i(n-j)\theta}   d\theta.
\end{align*}
Expanding $(e^{2i\theta} + e^{-2i\theta})^{\ell}$, we then get
\begin{align*}
\int_{U} z^{n} d\mu(z) & = \frac{(\alpha^{2}-\gamma^{2}) \, b^{2}}{(n+2b)\pi} \sum_{\ell=0}^{b-1} \binom{b-1}{\ell} \alpha^{\ell} \gamma^{\ell} (\alpha^{2}+\gamma^{2})^{b-1-\ell} \\
& \times \sum_{m=0}^{\ell}\binom{\ell}{m} \sum_{j=0}^{n} \binom{n}{j} \alpha^{j} \gamma^{n-j} \int_{0}^{2\pi}  e^{2(2m-\ell)i\theta}   e^{-i(n-2j)\theta}   d\theta.
\end{align*}
From this identity, we see that $\int_{U} z^{n} d\mu(z)=0$ if $n$ is odd (because the $\theta$-integral vanish for all $\ell,m,j$). If $n$ is even, then the $\theta$-integral is non-zero only for $j = \frac{n}{2}+\ell-2m$, and thus
\begin{align*}
\int_{U} z^{n} d\mu(z) & = 2 \alpha^{\frac{n}{2}}\gamma^{\frac{n}{2}} \frac{ (\alpha^{2}-\gamma^{2}) \, b^{2}}{n+2b} \sum_{\ell=0}^{b-1} \binom{b-1}{\ell} \alpha^{\ell} \gamma^{\ell} (\alpha^{2}+\gamma^{2})^{b-1-\ell} \\
& \times \sum_{m=0}^{\ell}\binom{\ell}{m} \binom{n}{\frac{n}{2}+\ell-2m} \alpha^{\ell-2m} \gamma^{2m-\ell}.
\end{align*}
Note that $-(b-1) \leq \ell-2m \leq b-1$ for all $\ell \in \{0,\ldots,b-1\}$ and $m\in \{0,\ldots,\ell\}$. Hence, by changing indices $k=\ell-2m$, we get
\begin{align*}
\int_{U} z^{n} d\mu(z) & = 2 \alpha^{\frac{n}{2}}\gamma^{\frac{n}{2}} \frac{ (\alpha^{2}-\gamma^{2}) \, b^{2}}{2} \sum_{\ell=0}^{b-1} \binom{b-1}{\ell} (\alpha^{2}+\gamma^{2})^{b-1-\ell}  \sum_{\substack{k=-\ell \\ k+\ell \, \mathrm{even}}}^{\ell} \binom{\ell}{\frac{\ell-k}{2}} \frac{\binom{n}{\frac{n}{2}+k}}{\frac{n}{2}+b} \alpha^{\ell+k} \gamma^{\ell-k}.
\end{align*}
Since $\binom{\ell}{\frac{\ell-k}{2}}=\binom{\ell}{\frac{\ell+k}{2}}$, we can regroup the terms corresponding to $k$ and $-k$, and we obtain \eqref{moments of mu ellipse 1}.
\end{proof}
Comparing \eqref{moments in terms of the cj} with \eqref{moments of mu ellipse 1}, it is still not clear whether we can choose $c_{0},\ldots,c_{b}$ in Lemma \ref{lemma:moments of nu ellipse} so that $\int_{U} z^{n} d\mu(z) = \int_{\partial U} z^{n} d\hat{\nu}(z)$ for all $n \in \N_{>0}$. The following key lemma will allow us to rewrite \eqref{moments of mu ellipse 1} as $\int_{U} z^{n} d\mu(z) = 2 \alpha^{\frac{n}{2}}\gamma^{\frac{n}{2}} \sum_{j=0}^{b} e_{j} \binom{n}{\frac{n}{2}-j}$, where the $e_{j}$'s are defined in \eqref{def of ej}. As mentioned earlier, the most challenging part in the proof of Lemma \ref{lemma:dpkp} is to invert explicitly matrices of size $(b-k+1)\times (b-k+1)$ with $k \in \{0,1,\ldots,b-1\}$.
\begin{lemma}\label{lemma:dpkp}
Let $b \in \N_{>0}$ and $k \in \{0,1,\ldots,b-1\}$. Then
\begin{align}\label{identity with djk}
\frac{1}{\frac{n}{2}+b}\binom{n}{\frac{n}{2}-k} = \sum_{\ell=0}^{b} d_{\ell}^{(k)} \binom{n}{\frac{n}{2}-\ell}, \qquad \mbox{for all } n \in \N,
\end{align}
with $d_{0}^{(k)} = d_{1}^{(k)} = \ldots = d_{k-1}^{(k)} = 0$ and for $\ell \in \{k,k+1,\ldots,b\}$, $d_{\ell}^{(k)}$ is defined in \eqref{def of d ell pkp}, and where $\binom{n}{m}:=\frac{\Gamma(n+1)}{\Gamma(m+1)\Gamma(n-m)}$.
\end{lemma}
\begin{proof}
In fact we will even prove that \eqref{identity with djk} holds for all $n\in \C\setminus (-1-2\N)$. Recall that $z\mapsto \Gamma(z)$ is a holomorphic function in $\C \setminus (-\N)$ with no zeros and with simple poles at each negative integer (see e.g. \cite[Section 5.2]{NIST}). Thus, as a function of $n \in \C$, the left-hand side of \eqref{identity with djk}, namely
\begin{align*}
\frac{1}{\frac{n}{2}+b} \frac{\Gamma(n+1)}{\Gamma(\frac{n}{2}-k+1)\Gamma(\frac{n}{2}+k+1)}
\end{align*}
has simple poles at $-1-2\N$ and simple zeros at $\{2k-2,2k-4,\ldots,0\} \cup (-2k-2-2\N)\setminus \{-2b\}$. On the other hand, $\binom{n}{\frac{n}{2}-\ell}$ has simple poles at $-1-2\N$ and simple zeros at $\{2\ell-2,2\ell-4,\ldots,0\} \cup (-2\ell-2-2\N)$.

\medskip Since $\binom{n}{\frac{n}{2}-\ell}$ has simple zeros at $-2\ell-2-2\N$, it is clear that the right-hand side of \eqref{identity with djk} vanishes at $-2b-2-2\N$ for any choices of $\smash{d_{0}^{(k)},\ldots,d_{b}^{(k)}}$.

\medskip Suppose that there exist $d_{0}^{(k)},\ldots,d_{b}^{(k)}$ such that \eqref{identity with djk} holds for all $n \in \C\setminus (-1-2\N)$. Then evaluating \eqref{identity with djk} successively at $n=0,2,4,\ldots,2k-2$ yields $d_{0}^{(k)} = 0$, $d_{1}^{(k)} = 0$, $\ldots$, and finally $d_{k-1}^{(k)} = 0$. In other words, $d_{0}^{(k)} = d_{1}^{(k)} = \ldots = d_{k-1}^{(k)} = 0$ ensures that the right-hand side of \eqref{identity with djk} vanishes at $n=0,2,4,\ldots,2k-2$. 

\medskip Let $R(n)$ be the right-hand side of \eqref{identity with djk} with $d_{0}^{(k)} = d_{1}^{(k)} = \ldots = d_{k-1}^{(k)} = 0$, and let  $L(n)$ be the left-hand side of \eqref{identity with djk}. We have shown above that $\frac{R(n)}{L(n)}$ has at most simple poles at $\{-2k-2,-2k-4,\ldots,-2b+2\}$, and no other poles except possibly at $\infty$. Also, since (see \cite[5.11.1]{NIST})
\begin{align}
\binom{n}{\frac{n}{2}-\ell} = 2^{n}\frac{\sqrt{2}}{\sqrt{\pi n}} \bigg( 1 - \frac{2\ell^{2}+\frac{1}{4}}{n} + \bigO(n^{-2}) \bigg), \qquad \mbox{as } n \to +\infty,
\end{align}
for any fixed $\ell\in \{k,\ldots,b\}$, we have $\frac{R(n)}{L(n)} = \bigO(n)$ as $n \to + \infty$. Thus, by Liouville's theorem, $\frac{R(n)}{L(n)}\prod_{r=1}^{b-k-1}(n+2k+2r)$ is a polynomial in $n$ of degree at most $b-k$. 

\medskip The most obvious approach is to try to find $d_{k}^{(k)},\ldots,d_{b}^{(k)}$ such that $R(n) = 0$ for each $n \in \{-2k-2,-2k-4,\ldots,-2b+2\}$ (this gives $b-k-1$ conditions) and such that $R(n)-L(n) = \bigO(n^{-1})$ as $n \to +\infty$ (this gives two other conditions). Using the fact that $\mbox{Res}(\Gamma(z),z=-m)=\frac{(-1)^{m}}{m!}$ for each $m\in -\N$ (see \cite[5.2.1]{NIST}), we obtain after simplification the linear system
\begin{align*}
\begin{cases}
\ds \sum_{\ell=b-r}^{b} (-1)^{b-r-\ell} \binom{b-r+\ell-1}{r+\ell-b} d_{\ell}^{(k)} = 0, & \mbox{for each } r \in \{1,\ldots,b-k-1\}, \\
\ds \sum_{\ell=k}^{b} d_{\ell}^{(k)} = 0, \qquad \mbox{ and } \qquad \sum_{\ell=k}^{b} \ell^{2}d_{\ell}^{(k)} = -1.
\end{cases}
\end{align*}
However, the matrix $\tilde{A}$ associated with this linear system is not triangular, and we have not managed to invert $\tilde{A}$ for arbitrary values of $b\in \N_{>0}$ and $k\in \{0,\ldots,b-1\}$. 

\medskip Hence we use another approach. By evaluating \eqref{identity with djk} at $n=2q$ for $q \in \{k,k+1,\ldots,b\}$, we get
\begin{align*}
\frac{1}{q+b}\binom{2q}{q-k} = \sum_{\ell=k}^{q} d_{\ell}^{(k)} \binom{2q}{q-\ell}, \qquad q = k,k+1,\ldots,b.
\end{align*}
This can be rewritten as $A\vec{d}^{(k)} = \vec{f}$, where
\begin{align*}
A = \begin{pmatrix}
\binom{2k}{0} & 0 & 0 & \dots & 0 \\
\binom{2k+2}{1} & \binom{2k+2}{0} & 0 & \dots & 0 \\
\binom{2k+4}{2} & \binom{2k+4}{1} & \binom{2k+4}{0} & \dots & 0 \\
\vdots & \vdots & \vdots & \ddots & \vdots \\
\binom{2b}{b-k} & \binom{2b}{b-k-1} & \binom{2b}{b-k-2} & \dots & \binom{2b}{0} 
\end{pmatrix}, \quad \vec{d}^{(k)} = \begin{pmatrix}
d_{k}^{(k)} \\ d_{k+1}^{(k)} \\ \vdots \\ d_{b}^{(k)}
\end{pmatrix}, \quad \vec{f} = \begin{pmatrix}
\frac{1}{k+b}\binom{2k}{0} \\
\frac{1}{k+1+b}\binom{2k+2}{1} \\
\frac{1}{k+2+b}\binom{2k+4}{2} \\
\vdots \\
\frac{1}{2b}\binom{2b}{b-k}
\end{pmatrix}.
\end{align*}
Since all diagonal elements of $A$ are $1$'s, $A_{k}$ is invertible and $\vec{d}^{(k)}$ is unique. Define $B=(B_{ij})_{i,j=0}^{b-k}$ by
\begin{align*}
B_{ij} = \begin{cases}
(-1)^{i-j} \frac{2(k+i)}{(i-j)!} \frac{(2k+i+j-1)!}{(2k+2j)!}, & \mbox{if } i\geq j, \\
0, & \mbox{if } i< j.
\end{cases}
\end{align*}
We claim that $B=A^{-1}$. To verify this, we will prove that $BA=I$ (proving $AB=I$ by a direct approach is more complicated). Since $A,B$ are both lower triangular, $BA$ is also lower triangular. Now, take $p,q\in \{0,\ldots,b-k\}$ with $q\leq p$. We have
\begin{align}
[BA]_{p,q} & = \sum_{j=q}^{p} B_{p,j}A_{j,q} = \sum_{j=q}^{p} (-1)^{p-j} \frac{2(k+p)}{(p-j)!} \frac{(2k+p+j-1)!}{(2k+2j)!} \binom{2k+2j}{j-q} \nonumber \\
& = 2(k+p) \sum_{j=q}^{p} (-1)^{p-j} \frac{(2k+p+j-1)!}{(2k+j+q)!(p-j)!(j-q)!} \nonumber \\
& = 2(k+p)(-1)^{p-q} \sum_{j=0}^{p-q} (-1)^{j} \frac{(2k+p+q+j-1)!}{(2k+j+2q)!(p-q-j)!j!} \nonumber \\
& = 2(k+p)(-1)^{p-q} \sum_{j=0}^{p-q} (-1)^{j} \binom{p-q}{j} \frac{(2k+p+q+j-1)!}{(2k+j+2q)!(p-q)!}. \label{lol11}
\end{align}
We will need the following fact (see e.g. \cite{Gould}): for any $m \in \N$ and any polynomial $P$ of degree $\leq m-1$, 
\begin{align}\label{lol12}
\sum_{j=0}^{m} (-1)^{j} \binom{m}{j} P(j) = 0.
\end{align}
For $q<p$, it is easy to check that
\begin{align*}
\frac{(2k+p+q+j-1)!}{(2k+j+2q)!(p-q)!}
\end{align*}
is a polynomial in $j$ of degree $p-q-1$. Combining \eqref{lol11} and \eqref{lol12}, we conclude that $[BA]_{p,q}=0$ for $q<p$. Finally, if $q=p$, then by \eqref{lol11} we have
\begin{align*}
[BA]_{p,p} = 2(k+p) \frac{(2k+p+p-1)!}{(2k+2p)!} = 1.
\end{align*}
Then the formula $\vec{d}^{(k)} = B\vec{f}$ yields \eqref{def of d ell pkp}. Thus, with $d_{0}^{(k)},\ldots, d_{b}^{(k)}$ as in the statement, we have $R(n)=L(n)$ for each $n \in \{2k,2k+2,\ldots,2b\}$. Since $\frac{R(n)}{L(n)}\prod_{r=1}^{b-k-1}(n+2k+2r)$ is a polynomial in $n$ of degree of at most $b-k$, we infer that $R(n) \equiv L(n)$. This completes the proof.  
\end{proof}
The following lemma is an immediate corollary of Lemmas \ref{lemma: moments of mu ellipse, first lemma} and \ref{lemma:dpkp}.
\begin{lemma}\label{lemma: moments of mu ellipse, second lemma}
For any $n\in \N$, we have $\int_{U} z^{n} d\mu(z)=0$ if $n$ is odd, while if $n$ is even we have
\begin{align}\label{moments in terms of the ej}
\int_{U} z^{n} d\mu(z) = 2 \alpha^{\frac{n}{2}}\gamma^{\frac{n}{2}} \sum_{j=0}^{b} e_{j} \binom{n}{\frac{n}{2}-j},
\end{align}
where $e_{j}$ is defined in \eqref{def of ej}.
\end{lemma}

Let $\hat{\nu}$ be as in Lemma \ref{lemma:moments of nu ellipse} with $c_{0},c_{1},\ldots,c_{b}$ given by \eqref{def of ej}. Since $c_{0}=e_{0}$ and $c_{\ell} = \frac{e_{\ell}}{\alpha^{\ell}\gamma^{-\ell}+\alpha^{-\ell}\gamma^{\ell}}$ for $\ell=1,\ldots,b$, it follows from Lemmas \ref{lemma:moments of nu ellipse} and \ref{lemma: moments of mu ellipse, second lemma} that $\int_{\partial U}z^n d\hat{\nu}(z) = \int_{ U}z^nd\mu(z)$ for all $n \in \N$. By Lemma \ref{lemma: moment when U is bounded} and Remark \ref{remark: pmu bounded on dU}, $\hat{\nu}=\mathrm{Bal}(\mu|_{U},\partial U)$, which finishes the proof of Theorem \ref{thm:ML ellipse nu}.

\subsection{The constant $C$: proof of Theorem \ref{thm:ML ellipse C}}
In this subsection, $b \in \N_{>0}$, $d\mu(z) := \frac{b^{2}}{\pi}|z|^{2b-2}d^{2}z$, $S:=\mathrm{supp}\,\mu =\{z:|z|\leq b^{-\frac{1}{2b}}\}$, and $a,c\in (0,b^{-\frac{1}{2b}}]$ so that $U := \{z:(\frac{\re z}{a})^{2}+(\frac{\im z}{c})^{2}<1\} \subset S$.
\begin{lemma}
\begin{align}\label{int |z|2b dmu ELLIPSE}
\int_{U} |z|^{2b} d\mu(z) = (\alpha^{2}-\gamma^{2}) \frac{b}{2} \sum_{j=0}^{b-1} \binom{2b-1}{2j}\binom{2j}{j} \alpha^{2j}\gamma^{2j}(\alpha^{2}+\gamma^{2})^{2b-1-2j}.
\end{align}
\end{lemma}
\begin{proof}
With $z=ar \cos \theta + i c r \sin \theta$, $0 \leq \theta \leq 2\pi$, $0 \leq r < 1$, we have $d^{2}z = acrdrd\theta$, and thus
\begin{align*}
\int_{U} |z|^{2b} d\mu(z) & = ac \frac{b^{2}}{\pi} \int_{0}^{1}\int_{0}^{2\pi} |ar\cos \theta + i c r \sin \theta|^{4b-2} r d\theta dr \\
& = ac \frac{b}{4\pi} \int_{0}^{2\pi} \big( \alpha \gamma (e^{2i\theta} + e^{-2i\theta}) + \alpha^{2} + \gamma^{2} \big)^{2b-1} d\theta.
\end{align*}
Since $b\in \N_{>0}$,
\begin{align*}
\int_{U} |z|^{2b} d\mu(z) & = ac \frac{b}{4\pi} \sum_{j=0}^{2b-1} \binom{2b-1}{j} \alpha^{j} \gamma^{j} (\alpha^{2}+\gamma^{2})^{2b-1-j} \int_{0}^{2\pi} (e^{2i\theta} + e^{-2i \theta})^{j} d\theta \\
& = ac \frac{b}{4\pi} \sum_{j=0}^{2b-1} \binom{2b-1}{j} \alpha^{j} \gamma^{j} (\alpha^{2}+\gamma^{2})^{2b-1-j} \sum_{k=0}^{j} \binom{j}{k}\int_{0}^{2\pi} e^{2i\theta(2k-j)} d\theta.
\end{align*}
If $j$ is odd, all integrals in the $k$-sum vanish, while for $j$ even only the integral corresponding to $k=\frac{j}{2}$ is non-vanishing. After simplification, we obtain \eqref{int |z|2b dmu ELLIPSE}.
\end{proof}

\begin{lemma}
Let $\nu = \mathrm{Bal}(\mu|_{U},\partial U)$. Then
\begin{align}\label{int |z|2b dnu ELLIPSE}
\int_{\partial U} |z|^{2b} d\nu(z) = \sum_{j=0}^{b} \binom{b}{j} \alpha^{j} \gamma^{j} (\alpha^{2}+\gamma^{2})^{b-j} 2 \bigg( c_{0} \binom{j}{\frac{j}{2}} \mathbf{1}_{j \, \mathrm{even}} + \sum_{\ell=1}^{b} 2 c_{\ell} \binom{j}{\frac{j+\ell}{2}} \mathbf{1}_{j+\ell \, \mathrm{even}} \bigg),
\end{align}
where $c_{0},\ldots,c_{b}$ are given by \eqref{def of ej}.
\end{lemma}
\begin{proof}
Recall from Theorem \ref{thm:ML ellipse nu} that $d\nu(z) = \big( c_{0} + 2 \sum_{\ell=1}^{b} c_{\ell} \cos (2\ell \theta) \big)\frac{d\theta}{\pi}$ for $z=a\cos \theta + i c \sin \theta$, $\theta \in [0,2\pi)$, where $c_{0},\ldots,c_{b}$ are given by \eqref{def of ej}. Since $b\in \N_{>0}$,
\begin{align*}
& \int_{\partial U} |z|^{2b}d\nu(z) = \int_{0}^{2\pi} \big( \alpha \gamma (e^{2i\theta} + e^{-2i\theta}) + \alpha^{2} + \gamma^{2} \big)^{b}  \bigg[c_{0}+\sum_{\ell=1}^{b}c_{\ell} (e^{2\ell i \theta} + e^{-2\ell i \theta}) \bigg] \frac{d\theta}{\pi} \\
& = \sum_{j=0}^{b} \binom{b}{j} \alpha^{j} \gamma^{j} (\alpha^{2}+\gamma^{2})^{b-j} \sum_{k=0}^{j} \binom{j}{k} \bigg( c_{0} \int_{0}^{2\pi}e^{2(2k-j)i\theta}\frac{d\theta}{\pi} + \sum_{\ell=1}^{b} c_{\ell} \int_{0}^{2\pi} e^{2(2k-j)i\theta}(e^{2\ell i \theta} + e^{-2\ell i \theta})\frac{d\theta}{\pi} \bigg),
\end{align*}
and after performing the $\theta$-integrals we find \eqref{int |z|2b dnu ELLIPSE}.
\end{proof}

Note that $U$ satisfies Assumptions \ref{ass:U} and \ref{ass:U2}. Theorem \ref{thm:ML ellipse C} now follows directly from \eqref{int |z|2b dmu ELLIPSE}, \eqref{int |z|2b dnu ELLIPSE} and Theorem \ref{thm:general pot} (i).

\section{The triangle}\label{section:triangle}
In this section, we will prove Theorem \ref{thm:Ginibre triangle nu}. In particular, we will show that $\mathrm{Bal}(d^{2}z|_{U},\partial U)$ has a polynomial density with respect to the arclength measure on $\partial U$ in the case where $U$ is an equilateral triangle.

\medskip In fact, we will first consider the following general setting: $a>0$, $U=aP$, where $P$ is the polygon with $p$ vertices ($p\geq 3$) at $\kappa^{-\frac{1}{2}},\kappa^{\frac{1}{2}},\ldots,\kappa^{p-1-\frac{1}{2}}$, where $\kappa^{\frac{1}{2}} := e^{\frac{\pi i}{p}}$, and $\mu$ is supported on $U$ and defined by $d\mu(z) = \frac{b^{2}}{\pi}|z|^{2b-2}d^{2}z$ with $b\in \N_{>0}$. The region $P$ can be written as
\begin{align}\label{def of P}
P = \bigcup_{j=0}^{p-1} \{r(t \kappa^{j-\frac{1}{2}} + (1-t)\kappa^{j+\frac{1}{2}}): 0 \leq r < 1, \; 0 \leq t \leq 1\}.
\end{align}
Suppose $z=x+iy = ar(t \kappa^{j-\frac{1}{2}} + (1-t) \kappa^{j+\frac{1}{2}})$ for some $j \in \{0,1,\ldots,p-1\}$. Then 
\begin{align}\label{change of var polygons}
d^{2}z = dxdy = \sin (\tfrac{2\pi}{p})a^{2}r  drdt.
\end{align}

\medskip In this section, we will try to obtain an explicit expression for $\nu$ using the first method described in Subsection \ref{section:ansatz method}. More precisely, we will make the ansatz that the density of $\nu$ is a polynomial with respect to the arclength measure on $\partial U$. However, we will manage to prove that this ansatz is correct only if $p=3$ and $b=1$. In particular, in the following settings, the analysis of this section is inconclusive but strongly suggests that $\mathrm{Bal}(\mu|_{U},\partial U)$ does not have a polynomial density with respect to $\partial U$ (see also Remark \ref{remark:triangle is polynomial}):
\begin{itemize}
\item[(i)] $\frac{d\mu(z)}{d^{2}z} = |z|^{2b-2}$ with $b\in \{2,3,\ldots\}$ and $U$ is an equilateral triangle centered at $0$, 
\item[(ii)] $\frac{d\mu(z)}{d^{2}z} = |z|^{2b-2}$ with $b\in \N_{>0}$ and $U$ is a regular polygon centered at $0$ with at least four edges.
\end{itemize} 
Note that Conjecture \ref{conj:behavior of balayage measure} also predicts that $\nu$ does not have polynomial density with respect to $\partial U$ in the above case (ii).
\subsection*{Balayage measure: proof of Theorem \ref{thm:Ginibre triangle nu}}\label{subsection: balayage triangle}
Here $a>0$, $U=aP$ with $P$ as in \eqref{def of P}, $p\geq 3$, $\mu$ is defined by $d\mu(z) = \frac{b^{2}}{\pi}|z|^{2b-2}d^{2}z$ with $b\in \N_{>0}$, and $\nu = \mathrm{Bal}(\mu|_{U},\partial U)$. 
\begin{lemma}\label{lemma:moments mu general polygon}
Let $n \in \N$. Then $\int_{U} z^{n} d\mu(z)=0$ if $n$ is not a multiple of $p$, and
\begin{align}\label{moments ML polygons}
\int_{U} z^{pn} d\mu(z) = \frac{(-1)^{n}a^{pn} }{p n+2b} \sum_{q=0}^{2b-2} \frac{\tilde{f}_{q}}{pn+q+1} = (-1)^{n}a^{pn} \sum_{q=0}^{2b-1} \frac{f_{q}}{pn+q+1},
\end{align}
where $f_{q}=\frac{\tilde{f}_{q}}{2b-1-q}$ for $q=0,1,\ldots,2b-2$, $f_{2b-1} = -\sum_{q'=0}^{2b-2} \frac{\tilde{f}_{q'}}{2b-1-q'}$, and
\begin{align*}
& \tilde{f}_{q} = \frac{a^{2b}b^{2}p\sin(\frac{2\pi}{p})}{\pi} (\cos \tfrac{\pi}{p})^{2(b-1)-q}\frac{\sin \frac{(q+1)\pi}{p}}{\sin \frac{\pi}{p}}\sum_{k=\lceil \frac{q}{2} \rceil}^{b-1}(-1)^{q+k}\binom{2k}{q}\binom{b-1}{k}.
\end{align*}
In particular, for $b=1$, we have
\begin{align}\label{lol64}
\int_{U} z^{pn} d\mu(z) = \frac{p(-1)^{n} \sin (\tfrac{2\pi}{p})a^{2+pn}}{(pn+1)(pn+2)\pi}.
\end{align}
\end{lemma}
\begin{proof}
Using  \eqref{def of P}, \eqref{change of var polygons}, we get
\begin{align*}
& \int_{U} z^{n} d\mu(z) = \int_{U} z^{n} \frac{b^{2}}{\pi}|z|^{2b-2}d^{2}z \\
& = \sin (\tfrac{2\pi}{p})a^{2} \sum_{j=0}^{p-1} \int_{0}^{1}\int_{0}^{1} \big(ar(t \kappa^{j-\frac{1}{2}} + (1-t) \kappa^{j+\frac{1}{2}})\big)^{n} \frac{b^{2}}{\pi} \big|ar(t \kappa^{j-\frac{1}{2}} + (1-t) \kappa^{j+\frac{1}{2}})\big|^{2b-2} r drdt \\
& = \sin (\tfrac{2\pi}{p})a^{2} \int_{0}^{1}\int_{0}^{1} \big(ar(t \kappa^{-\frac{1}{2}} + (1-t) \kappa^{\frac{1}{2}})\big)^{n} \frac{b^{2}}{\pi} \big|ar(t \kappa^{-\frac{1}{2}} + (1-t) \kappa^{\frac{1}{2}})\big|^{2b-2} rdrdt \sum_{j=0}^{p-1} \kappa^{jn}.
\end{align*}
The above sum is clearly $0$ unless $n$ is a multiple of $p$. Next, using that $b\in \N_{>0}$,
\begin{align*}
& \int_{U} z^{pn} d\mu(z) = \frac{b^{2}}{\pi} p \sin (\tfrac{2\pi}{p})a^{2b+pn} \int_{0}^{1}\int_{0}^{1} r^{pn+2b-1}\big(t \kappa^{-\frac{1}{2}} + (1-t) \kappa^{\frac{1}{2}}\big)^{pn}  \big| t \kappa^{-\frac{1}{2}} + (1-t) \kappa^{\frac{1}{2}} \big|^{2b-2} drdt \\
& = \frac{b^{2}p \sin (\tfrac{2\pi}{p})a^{2b+pn}}{(pn+2b)\pi}  \int_{0}^{1} \big(t \kappa^{-\frac{1}{2}} + (1-t) \kappa^{\frac{1}{2}}\big)^{pn}  \big| t \kappa^{-\frac{1}{2}} + (1-t) \kappa^{\frac{1}{2}} \big|^{2b-2} dt \\
& = \frac{b^{2}p \sin (\tfrac{2\pi}{p})a^{2b+pn}}{2(pn+2b)\pi}  \int_{-1}^{1} \big( \cos \tfrac{\pi}{p}  + i u \sin \tfrac{\pi}{p} \big)^{pn}  \big( (\cos \tfrac{\pi}{p})^{2} + u^{2} (\sin \tfrac{\pi}{p})^{2} \big)^{b-1} du \\
& = \frac{b^{2}p \sin (\tfrac{2\pi}{p})a^{2b+pn}}{2(pn+2b)\pi}  \int_{-1}^{1} \sum_{j=0}^{pn}\binom{pn}{j} ( \cos \tfrac{\pi}{p} )^{pn-j} i^{j} u^{j} ( \sin \tfrac{\pi}{p} )^{j}  \sum_{k=0}^{b-1} \binom{b-1}{k} ( \cos \tfrac{\pi}{p} )^{2(b-1-k)} u^{2k} ( \sin \tfrac{\pi}{p} )^{2k} du.
\end{align*}
After performing the $u$-integral, the above can be rewritten as
\begin{align}\label{lol3}
& \int_{U} z^{pn} d\mu(z) = \frac{b^{2}p \sin (\tfrac{2\pi}{p})a^{2b+pn}}{(pn+2b)\pi}   \sum_{k=0}^{b-1} \binom{b-1}{k}(\cos \tfrac{\pi}{p})^{pn+2(b-1)} \sum_{\substack{j=0 \\ j \, \mathrm{even}}}^{pn} \binom{pn}{j}  \frac{i^{j}(\tan \tfrac{\pi}{p})^{j+2k}}{j+2k+1}.
\end{align}
To simplify the above expression further, note that
\begin{align*}
\sum_{\substack{j=0 \\ j \, \mathrm{even}}}^{m} \binom{m}{j}c^{j} = \frac{(1-c)^{m}+(1+c)^{m}}{2}, \qquad \mbox{for all } m \in \N, \; c \in \C.
\end{align*}
Multiplying the above by $c^{2k}$, then integrating over $c$ from $0$ to an arbitrary $c\in \C\setminus \{0\}$, and then dividing by $c$, we get
\begin{align*}
\sum_{\substack{j=0 \\ j \, \mathrm{even}}}^{m} \binom{m}{j}\frac{c^{j+2k}}{j+2k+1} = \sum_{q=0}^{2k} \binom{2k}{q} \frac{(-1)^{q}}{m+q+1} \frac{(1+c)^{m+q+1}-(1-c)^{m+q+1}}{2c}.
\end{align*}
Using the above with $m=pn$ and $c=i\tan \frac{\pi}{p}$ yields (after simplifications)
\begin{align}\label{lol2}
(\cos \tfrac{\pi}{p})^{pn+2k} \sum_{\substack{j=0 \\ j \, \mathrm{even}}}^{pn} \binom{pn}{j}  \frac{i^{j}(\tan \tfrac{\pi}{p})^{j+2k}}{j+2k+1} = \sum_{q=0}^{2k} \binom{2k}{q} \frac{(-1)^{n+q+k}}{pn+q+1} (\cos \tfrac{\pi}{p})^{2k-q} \frac{\sin \frac{(q+1)\pi}{p}}{\sin \frac{\pi}{p}}.
\end{align}
The advantage of the above formula is that the right-hand side contains finitely many terms as $n \to +\infty$. Substituting \eqref{lol2} in \eqref{lol3} yields 
\begin{align*}
& \int_{U} z^{pn} d\mu(z) = \frac{b^{2}p \sin (\tfrac{2\pi}{p})a^{2b+pn}}{(pn+2b)\pi}   \sum_{k=0}^{b-1} \binom{b-1}{k} \sum_{q=0}^{2k} \binom{2k}{q} \frac{(-1)^{n+q+k}}{pn+q+1} (\cos \tfrac{\pi}{p})^{2(b-1)-q} \frac{\sin \frac{(q+1)\pi}{p}}{\sin \frac{\pi}{p}}.
\end{align*}
Switching the order of summation, we find the first identity in \eqref{moments ML polygons}. Using then the decomposition by parts
\begin{align*}
\frac{1}{p n+2b} \frac{1}{pn+q+1} = \frac{1}{2b-1-q} \bigg( \frac{1}{pn+q+1}-\frac{1}{pn+2b} \bigg),
\end{align*}
we obtain the second identity in \eqref{moments ML polygons}.
\end{proof}
\medskip Let us make the ansatz that $d\nu(z) = \sum_{\ell=0}^{M}c_{\ell}y^{2\ell}dy$ for $z=a\cos \frac{\pi}{p}+iy$, $y\in [-a\sin \frac{\pi}{p}, a\sin \frac{\pi}{p}]$ and some $M\in \N$, $c_{0},\ldots,c_{M}\in \R$. This ansatz will be proven to be correct only for $p=3$ (the triangle) and $b=1$ (the uniform measure).
\begin{lemma}\label{lemma: moment nu hat triangle}
Let $M\in \N$ and $c_{0},\ldots,c_{M}\in \R$. Let $\hat{\nu}$ be the measure supported on $\partial U$ defined for $z=a\cos \frac{\pi}{p}+iy$, $y\in [-a\sin \frac{\pi}{p}, a\sin \frac{\pi}{p}]$ by $d\hat{\nu}(z) = \sum_{\ell=0}^{M}c_{\ell}y^{2\ell}dy$ and by $d\hat{\nu}(z\kappa^{j})=d\hat{\nu}(z)$ for all $j\in \{0,1,\ldots,p-1\}$. For $n \in \N$, we have $\int_{\partial U} z^{n}d\hat{\nu}(z) = 0$ if $n$ is not a multiple of $p$, and
\begin{align}\label{lol}
\int_{\partial U} z^{pn}d\hat{\nu}(z) = (-1)^{n}a^{pn}  \sum_{q=0}^{2M} \frac{g_{q}}{pn+q+1},
\end{align}
where
\begin{align*}
g_{q} = 2p \frac{\sin \frac{(q+1)\pi}{p}}{(\cos \frac{\pi}{p})^{q}} \sum_{\ell=\lceil \frac{q}{2}\rceil}^{M} c_{\ell} a^{2\ell+1}\binom{2\ell}{q}(-1)^{q+\ell} (\cos \tfrac{\pi}{p})^{2\ell}.
\end{align*}
\end{lemma}
\begin{proof}
If $n$ is not a multiple of $p$, then 
\begin{align*}
\int_{\partial U} z^{n}d\hat{\nu}(z) = (1+\kappa^{n}+\ldots+\kappa^{(p-1)n})\int_{a\cos \frac{\pi}{p}+ i [-a\sin \frac{\pi}{p}, a\sin \frac{\pi}{p}]} z^{n}d\nu(z) = 0.
\end{align*}
It remains to prove \eqref{lol}. For any $n\in \N$,
\begin{align*}
& \int_{\partial U} z^{pn}d\hat{\nu}(z) = p \int_{-a\sin \frac{\pi}{p}}^{a\sin \frac{\pi}{p}} \big( a \cos \tfrac{\pi}{p} + i y \big)^{pn} \sum_{\ell=0}^{M}c_{\ell}y^{2l}dy \\
& = p  \int_{-1}^{1} a^{pn} \big( \cos \tfrac{\pi}{p} + i y \sin \frac{\pi}{p} \big)^{pn} \sum_{\ell=0}^{M}c_{\ell}a^{2\ell+1}(\sin \tfrac{\pi}{p})^{2\ell+1} y^{2l}dy \\
& = 2p a^{pn} \sum_{\ell=0}^{M}c_{\ell}a^{2\ell+1}(\sin \tfrac{\pi}{p})^{2\ell+1} (\cos \tfrac{\pi}{p})^{pn}\sum_{\substack{j=0 \\ j \, \mathrm{even}}}^{pn} \binom{pn}{j} \frac{i^{j} (\tan \frac{\pi}{p})^{j}}{j+2\ell+1} dy.
\end{align*}
Using \eqref{lol2} in the form
\begin{align*}
(\cos \tfrac{\pi}{p})^{pn} \sum_{\substack{j=0 \\ j \, \mathrm{even}}}^{pn} \binom{pn}{j}  \frac{i^{j}(\tan \tfrac{\pi}{p})^{j}}{j+2\ell+1} = \sum_{q=0}^{2\ell} \binom{2\ell}{q} \frac{(-1)^{n+q+\ell}}{pn+q+1} (\tan \tfrac{\pi}{p})^{-2\ell}(\cos \tfrac{\pi}{p})^{-q} \frac{\sin \frac{(q+1)\pi}{p}}{\sin \frac{\pi}{p}},
\end{align*}
we get
\begin{align*}
& \int_{\partial U} z^{pn}d\hat{\nu}(z) = 2p a^{pn} \sum_{\ell=0}^{M}c_{\ell}a^{2\ell+1} \sum_{q=0}^{2\ell} \binom{2\ell}{q} \frac{(-1)^{n+q+\ell}}{pn+q+1} (\cos \tfrac{\pi}{p})^{2\ell} \frac{\sin \frac{(q+1)\pi}{p}}{(\cos \tfrac{\pi}{p})^{q}}.
\end{align*}
Switching the order of summation, we find the claim.
\end{proof}

Let $\hat{\nu}$ be as in Lemma \ref{lemma: moment nu hat triangle} for some $M\in \N$ and $c_{0},c_{1},\ldots,c_{M}\in \R$. Comparing \eqref{moments ML polygons} with \eqref{lol}, we see that in order for $\int_{\partial U}z^n d\hat{\nu}(z) = \int_{ U}z^nd\mu(z)$ to hold for all $n \in \N$, we must have $M=b$ and
\begin{align}\label{system for fq gq}
g_{2b} = 0, \qquad g_{q} = f_{q} \quad \mbox{for all } q \in \{0,\ldots,2b-1\}.
\end{align}
The equation $g_{2b} = 0$ is automatically satisfied if $\sin \frac{(2b+1)\pi}{p}=0$, i.e. if $p=2b+1$. Even with the choice $p=2b+1$, the system \eqref{system for fq gq} still contains $2M$ linear equations for the $M+1$ coefficients $c_{0},\ldots,c_{M}$. We have been able to solve this system only for
\begin{align*}
M=b, \qquad p=2b+1, \qquad 2M=M+1,
\end{align*} 
i.e. for $b=1$ and $p=3$. In this case, \eqref{system for fq gq} is satisfied with the choice $c_{0} = \frac{3a}{8\pi}$, $c_{1} = - \frac{1}{2\pi a}$, i.e. when $\frac{d\hat{\nu}(z)}{dy} := \frac{3a^{2}-4y^{2}}{8\pi a (1-\tau^{2})}$ with $z=a\cos \frac{\pi}{p}+iy$, $y\in [-a\sin \frac{\pi}{p}, a\sin \frac{\pi}{p}]$. By Lemma \ref{lemma: moment when U is bounded} and Remark \ref{remark: pmu bounded on dU}, $\hat{\nu}=\mathrm{Bal}(\mu|_{U},\partial U)$, which finishes the proof of Theorem \ref{thm:Ginibre triangle nu} for $\theta_{0}=0$ and $\zeta_{0}=0$. Since $d^{2}z$ is both rotation-invariant and translation-invariant, $\nu:=\mathrm{Bal}(d^{2}z|_{\zeta_{0}+e^{i\theta_{0}}aP},\zeta_{0}+e^{i\theta_{0}}a \, \partial P)$ and $\nu_{0}:=\mathrm{Bal}(d^{2}z|_{aP},a \,\partial P)$ are related by
\begin{align*}
d\nu(z) = d\nu_{0}(e^{-i\theta_{0}}(z-\zeta_{0})), \qquad \mbox{for all } z \in \zeta_{0} + e^{i\theta_{0}}a \, \partial P,
\end{align*}
(see also \eqref{dnu relation in EG}) which finishes the proof of Theorem \ref{thm:Ginibre triangle nu}.

\section{The rectangle}\label{section: rectangle}
Let $a_{2}>a_{1}$, $c_{2}>c_{1}$, and define $U = \{z: a_{1} < \re z < a_{2}, \; c_{1} < \im z < c_{2}\}$. 

\subsection{Balayage measure in terms of elliptic functions}\label{subsection:balayage square conformal}

\begin{figure}[h]
\begin{center}
\hspace{-1cm}\begin{tikzpicture}[master]
\node at (0,0) {\includegraphics[height=3.3cm]{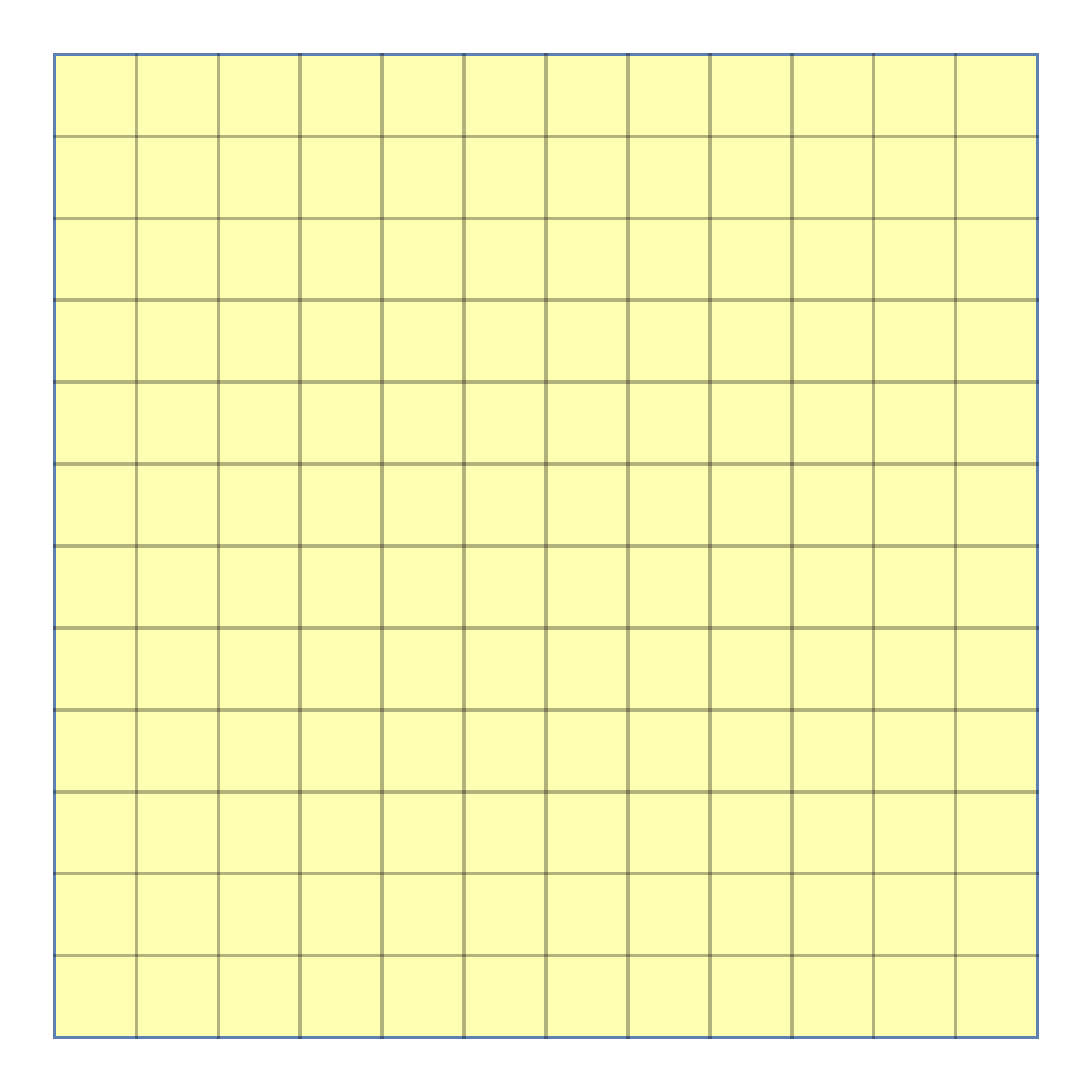}};
\draw[fill] (1.48,-1.48) circle (0.04);
\node at (1.48,-1.7) {\footnotesize $\frac{c}{2}-i\frac{c}{2}$};

\draw[fill] (1.48,1.48) circle (0.04);
\node at (1.48,1.7) {\footnotesize $\frac{c}{2}+i\frac{c}{2}$};

\end{tikzpicture} \hspace{4.2cm}
\begin{tikzpicture}[slave]
\node at (0,0) {\includegraphics[height=3cm]{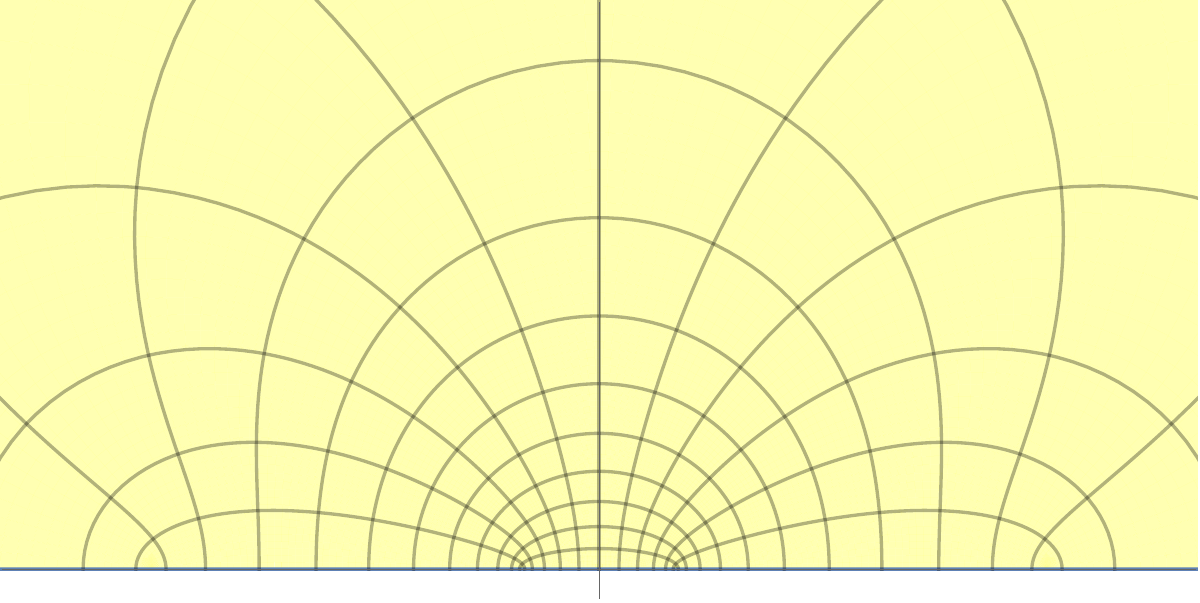}};
\node at (-4.8,1.8) {$\phi$};
\coordinate (P) at ($(-4.8, 0.9) + (30:1cm and 0.6cm)$);
\draw[thick, black, -<-=0.45] ($(-4.8, 0.9) + (30:1cm and 0.6cm)$(P) arc (30:150:1cm and 0.6cm);
\draw[fill] (2.2,-1.34) circle (0.04);
\node at (2.2,-1.65) {\footnotesize $\frac{1}{\kappa}$};
\draw[fill] (0.375,-1.34) circle (0.04);
\node at (0.375,-1.65) {\footnotesize $1$};
\end{tikzpicture}
\end{center}
\caption{\label{fig:conformal map from square to upper-half plane} The conformal map $\phi$ from $U$ to the upper-half plane.}
\end{figure}

Let $\nu := \mathrm{Bal}(\frac{d^{2}z}{\pi}|_{U},\partial U)$. In this subsection, we give another expression for $\nu$ than the one given in Theorem \ref{thm:Ginibre rectangle nu} (see also Remark \ref{remark: elliptic for square}). For simplicity, here we only focus on the special case where $\mu=\frac{d^{2}z}{\pi}$ and $a_{2}=c_{2}=-a_{1}=-c_{1}=\frac{c}{2}$ for some $c>0$. 

Recall from \cite[22.2.3]{NIST} that the Jacobi elliptic function $\mathrm{sn}$ is defined for $\kappa \in (0,1)$ and $z\in \C$ by
\begin{align*}
& \mathrm{sn}(z,\kappa) := \frac{\theta_{3}(0,q)}{\theta_{2}(0,q)}\frac{\theta_{1}(\zeta,q)}{\theta_{4}(\zeta,q)}, \quad \zeta := \frac{\pi z}{2K(\kappa)}, \quad q := \exp\bigg(\hspace{-0.15cm}-\hspace{-0.05cm}\frac{\pi K'(\kappa)}{K(\kappa)}\bigg), \quad K(\kappa) := \int_{0}^{1}\frac{dt}{\sqrt{1-t^{2}}\sqrt{1-\kappa^{2}t^{2}}} \\
& K'(\kappa) := \int_{0}^{1}\frac{dt}{\sqrt{1-t^{2}}\sqrt{1-(1-\kappa^{2})t^{2}}}, \qquad  \theta_{3}(z,q) := 1+2\sum_{n=1}^{+\infty}q^{n^{2}} \cos(2nz).
\end{align*}
The definitions $\theta_{1},\theta_{2}$ and $\theta_{4}$ are similar to that of $\theta_{3}$ (see e.g. \cite[20.2.1--20.2.4]{NIST}) and we omit them here. As a function of $z$ with $\kappa \in (0,1)$ fixed, $\mathrm{sn}(z,k)$ is doubly periodic with simple poles and simple zeros (see e.g. \cite[Chapter 22]{NIST} for further properties). Define also
\begin{align}\label{def of phi square}
\phi(z) := \mathrm{sn} \big( K(\kappa)\frac{z + i\tfrac{c}{2}}{c/2}, \kappa \big)
\end{align}
where the dependence of $\phi$ in $c$ and $\kappa \in (0,1)$ is omitted.
\begin{proposition}\label{prop:Ginibre square nu}
Let $\kappa \in (0,1)$ be the unique solution to
\begin{align}\label{eq kappa}
\int_{0}^{1} \frac{du}{\sqrt{1-(1-\kappa^{2})u^{2}}\sqrt{1-u^{2}}} = 2\int_{0}^{1} \frac{du}{\sqrt{1-\kappa^{2}u^{2}}\sqrt{1-u^{2}}}, \qquad (\kappa\approx 0.171573)
\end{align}
and let $\phi$ be as in \eqref{def of phi square}.
Let $\nu := \mathrm{Bal}(\mu|_{U},\partial U)$ with $\mu:=\frac{d^{2}z}{\pi}$. For $z=\frac{c}{2}+iy$, $y\in (-\frac{c}{2}, \frac{c}{2})$, we have
\begin{align}
\frac{d\nu(z)}{dy} & = \re \Bigg\{ \int_{-\frac{c}{2}}^{\frac{c}{2}} \int_{-\frac{c}{2}}^{\frac{c}{2}} \frac{- \phi'(\frac{c}{2}+i y)}{\pi \big(  \phi(\tilde{x}+i \tilde{y})-\phi(\frac{c}{2}+ i y) \big)} \frac{d\tilde{x}d\tilde{y}}{\pi} \Bigg\} \label{lol54} \\
& = \bigg(\frac{c^{2}}{2 K(\kappa)}\bigg)^{2} \re \bigg\{ \int_{\{w:\im w > 0\}} \frac{- \phi'(\frac{c}{2}+i y)}{\pi \big( w-\phi(\frac{c}{2}+ i y) \big)}  \frac{d^{2}w}{\pi |1-w^{2}| \; |1-\kappa^{2}w^{2}|} \bigg\} \label{lol55}
\end{align}
and $d\nu(i^{j}z)=d\nu(z)$ for all $j\in \{0,1,2,3\}$ and $z\in \frac{c}{2}+ i (-\frac{c}{2}, \frac{c}{2})$.
\end{proposition}
\begin{proof}
Let $\kappa\in (0,1)$ and consider the function
\begin{align*}
f(w) = \int_{0}^{w} \frac{-du}{\sqrt{u-\frac{1}{\kappa}}\sqrt{u+\frac{1}{\kappa}}\sqrt{u-1}\sqrt{u+1}}, \qquad \im w > 0,
\end{align*}
where each branch is the principal one and the contour integral lies in the upper-half plane. The function $f$ can also be rewritten in terms of the incomplete elliptic integral of the first kind (see e.g. \cite[eq 8.111]{GRtable} or \cite[eq 19.2.4]{NIST}). This is a particular case of a Schwarz–Christoffel mapping: $f$ is a conformal map from the open upper-half plane to the interior of a rectangle with vertices at $f(\pm \frac{1}{\kappa} + i 0_{+})$, $f(\pm 1 + i 0_{+})$. The rectangle is a square if $\kappa$ is such that $f(1+i 0_{+})-f(-1+i 0_{+}) = -i(f(\frac{1}{\kappa}+i 0_{+})-f(1+i 0_{+}))$, i.e. if $\kappa$ is the unique solution to \eqref{eq kappa}. (\eqref{eq kappa} clearly determines $\kappa$ uniquely since the left-hand side decreases from $+\infty$ to $\frac{\pi}{2}$ as $\tau$ increases from $0$ to $1$, while the right-hand side increases from $\pi$ to $+\infty$.) From now on we assume that $\kappa\in (0,1)$ is the unique solution to \eqref{eq kappa}. Therefore the function
\begin{align*}
h(w) := -\frac{c}{2}i + \frac{c}{2} \frac{f(w)}{f(1+i0_{+})}, \qquad \im w >0,
\end{align*}
is a conformal map from the open upper half-plane to the open square of side length $c$ centered at $0$ ($h(\pm \frac{1}{\kappa} + i 0_{+}) = \pm \frac{c}{2} + i \frac{c}{2}$ and $h(\pm 1 + i 0_{+}) = \pm \frac{c}{2} - i \frac{c}{2}$). By e.g. \cite[22.15.11]{NIST} and \cite[8.14]{GRtable}, the inverse function of $f$ is $\mathrm{sn}(\frac{z}{\kappa},\kappa)$; thus the inverse function of $h$ is given by
\begin{align*}
\phi(z) := \mathrm{sn} \big( \tfrac{f(1+i 0_{+})}{\kappa}\frac{w + i\tfrac{c}{2}}{c/2}, \kappa \big),
\end{align*}
which is equivalent to \eqref{def of phi square}. 
Hence $\phi$ is a conformal map from $U$ to the open upper half-plane $\{z:\im z >0\}$, see also Figure \ref{fig:conformal map from square to upper-half plane}. By \eqref{green function with general conformap mapping}, $g_{U}(z,w) = \frac{1}{2\pi} \log \big| \frac{\phi(z)-\overline{\phi(w)}}{\phi(z)-\phi(w)} \big|$, and thus, by Theorem \ref{thm:dnu in terms of green general}, for $z=\frac{c}{2}+iy$, $y\in (-\frac{c}{2},\frac{c}{2})$, we have
\begin{align*}
\frac{d\nu(z)}{dy} & = \frac{-1}{2\pi} \re \int_{-\frac{c}{2}}^{\frac{c}{2}} \int_{-\frac{c}{2}}^{\frac{c}{2}} \bigg( \frac{d}{dz} \log  \frac{\phi(\tilde{x}+i\tilde{y})-\overline{\phi(z)}}{\phi(\tilde{x}+i\tilde{y})-\phi(z)} \bigg) \bigg|_{z=\frac{c}{2}+i y} \frac{d\tilde{x}d\tilde{y}}{\pi} \\
& = \re \int_{-\frac{c}{2}}^{\frac{c}{2}} \int_{-\frac{c}{2}}^{\frac{c}{2}} i\frac{\im \big( \overline{\phi(\frac{c}{2}+iy)} \phi'(\frac{c}{2}+iy) \big) - \phi(\tilde{x}+i\tilde{y})\im \phi'(\frac{c}{2}+iy)}{\pi (\phi(\tilde{x}+i\tilde{y}) - \overline{\phi(\frac{c}{2}+iy)})(\phi(\tilde{x}+i\tilde{y}) - \phi(\frac{c}{2}+iy))} \frac{d\tilde{x}d\tilde{y}}{\pi}.
\end{align*}
Simplifying the above expression, using that $\phi(\frac{c}{2}+iy) \in \R$ and $\phi'(\frac{c}{2}+iy)\in i \R$, we find \eqref{lol54}. The second identity \eqref{lol55} is obtained after making the change of variables $w = \phi(\tilde{x}+i\tilde{y})$.
\end{proof}

\subsection{Balayage measure: proof of Theorems \ref{thm:Ginibre rectangle nu}, \ref{thm:ML rectangle nu} and Corollary \ref{coro:ML square nu}}\label{subsection: balayage rectangle}
We start with the proof of Theorem \ref{thm:ML rectangle nu}. Fix $b \in \N_{>0}$, and let $\mu$ be as in \eqref{mu S ML}, i.e. $d\mu(z) = \frac{b^{2}}{\pi}|z|^{2b-2}d^{2}z$ and $S:=\mathrm{supp}\,\mu =\{z:|z|\leq b^{-\frac{1}{2b}}\}$. Let $a_{2}>a_{1}>0$, $c_{2}>c_{1}>0$, and define $U = \{z: a_{1} < \re z < a_{2}, \; c_{1} < \im z < c_{2}\}$. Suppose that $U\subset S$. The goal of this subsection is to find an explicit expression for $\nu:=\mathrm{Bal}(\mu|_{U},\partial U)$. We will obtain $\nu$ using the method described in Subsection \ref{subsection: green function method for nu}. In the previous subsection, we computed $\nu|_{b=1}$ using an expression for $g_{U}$ involving a conformal map from $U$ to the open upper half-plane. Here we obtain more explicit results using another expression for $g_{U}$ obtained via separation of variables in e.g. \cite{Duffy}. 

\medskip In \cite[Section 6.3]{Duffy}, an expression for $g_{U}$ is given in the case $a_{1}=c_{1}=0$. A straightforward adaptation of \cite[Section 6.3]{Duffy} yields the following expression for $g_{U}$, valid for all $a_{1}<a_{2}$ and $c_{1}<c_{2}$: 
\begin{align}\label{series representation for gU in the case of a square}
g_{U}(x+iy,\xi+i\eta) = \sum_{q=1}^{+\infty}\sum_{m=1}^{+\infty} A_{qm} \sin \bigg( \frac{x-a_{1}}{a_{2}-a_{1}} q \pi  \bigg) \sin \bigg( \frac{y-c_{1}}{c_{2}-c_{1}} m \pi  \bigg),
\end{align}
where
\begin{align}
& A_{qm} := \frac{4}{(a_{2}-a_{1})(c_{2}-c_{1})\lambda_{qm}} \sin \bigg( \frac{\xi-a_{1}}{a_{2}-a_{1}} q \pi  \bigg) \sin \bigg( \frac{\eta-c_{1}}{c_{2}-c_{1}} m \pi  \bigg), \nonumber \\
& \lambda_{qm} := \frac{q^{2}\pi^{2}}{(a_{2}-a_{1})^{2}} + \frac{m^{2}\pi^{2}}{(c_{2}-c_{1})^{2}}. \label{def of lambda qm}
\end{align}
By Theorem \ref{thm:dnu in terms of green general}, for $z=a_{2}+i\eta$, $\eta \in (c_{1},c_{2})$, we have
\begin{align*}
\frac{d\nu(z)}{d\eta} = (-1)\int_{U} \frac{d}{d\xi}g_{U}(x+iy,\xi+i\eta)\bigg|_{\xi=a_{2}}d\mu(x+iy).
\end{align*}
For each fixed $x\in (a_{1},a_{2})$ and $y,\eta\in(c_{1},c_{2})$, the derivative $\frac{d}{d\xi}g_{U}(x+iy,\xi+i\eta)\big|_{\xi=a_{2}}$ can be computed by differentiating the series \eqref{series representation for gU in the case of a square} term by term, and thus
\begin{align}
\frac{d\nu(z)}{d\eta} = \int_{U} \bigg[ &\sum_{q=1}^{+\infty}\sum_{m=1}^{+\infty}  \frac{4}{(a_{2}-a_{1})(c_{2}-c_{1})\lambda_{qm}} \frac{-q\pi}{a_{2}-a_{1}}\cos ( q \pi  ) \sin \bigg( \frac{\eta-c_{1}}{c_{2}-c_{1}} m \pi  \bigg) \nonumber \\
& \times \sin \bigg( \frac{x-a_{1}}{a_{2}-a_{1}} q \pi  \bigg) \sin \bigg( \frac{y-c_{1}}{c_{2}-c_{1}} m \pi  \bigg) \bigg] \frac{b^{2}}{\pi}(x^{2}+y^{2})^{b-1}dxdy. \label{lol57}
\end{align} 
Let $s_{q,m}(x,y)$ be the summand in the above integral. Since $\lambda_{qm} = \bigO(q^{2}+m^{2})$ as $q,m\to+\infty$, for all $(x,y)\in U$ and all sufficiently large $M_{1},M_{2}\in \N_{>0}$, we have the bound
\begin{align}\label{lol56}
\bigg|\sum_{q=1}^{M_{1}}\sum_{m=1}^{M_{2}} s_{q,m}(x,y) \bigg| \leq C \bigg| \sum_{q=1}^{M_{1}} \frac{(-1)^{q}}{q} \sin \bigg( \frac{x-a_{1}}{a_{2}-a_{1}} q \pi  \bigg) \bigg| \leq C \frac{\pi}{2}(1+2\gamma_{\mathrm{WG}})
\end{align}
for some constant $C>0$ and $\gamma_{\mathrm{WG}} \approx 0.0895$ is the Wilbraham--Gibbs constant. For the last inequality above, we have used that $\sum_{q=1}^{+\infty}\frac{(-1)^{q}}{q}\sin(\tilde{x}q\pi)$ is the Fourier series of the $1$-periodic function defined for $\tilde{x}\in (0,1)$ by $-\frac{\pi}{2}\tilde{x}$. Since the right-hand side of \eqref{lol56} is clearly integrable on $U$, by Lebesgue's dominated convergence theorem we can interchange $\int_{U}$ with $\sum_{q=1}^{+\infty}\sum_{m=1}^{+\infty}$ in \eqref{lol57}, so that
\begin{align}
& \frac{d\nu(z)}{d\eta} = \sum_{q=1}^{+\infty}\sum_{m=1}^{+\infty} \frac{4}{(a_{2}-a_{1})(c_{2}-c_{1})\lambda_{qm}} \frac{-q\pi}{a_{2}-a_{1}}(-1)^{q} \sin \bigg( \frac{\eta-c_{1}}{c_{2}-c_{1}} m \pi  \bigg) I_{q,m}   \label{lol58} \\
& I_{q,m}:= \int_{a_{1}}^{a_{2}}\int_{c_{1}}^{c_{2}} \sin \bigg( \frac{x-a_{1}}{a_{2}-a_{1}} q \pi  \bigg) \sin \bigg( \frac{y-c_{1}}{c_{2}-c_{1}} m \pi  \bigg)  \frac{b^{2}}{\pi}(x^{2}+y^{2})^{b-1}dxdy \nonumber \\
& \hspace{0.74cm} = \sum_{\ell=0}^{b-1} \frac{b^{2}}{\pi}\binom{b-1}{\ell} \int_{a_{1}}^{a_{2}} x^{2\ell} \sin \bigg( \frac{x-a_{1}}{a_{2}-a_{1}} q \pi  \bigg)  dx \int_{c_{1}}^{c_{2}} y^{2(b-1-\ell)} \sin \bigg( \frac{y-c_{1}}{c_{2}-c_{1}} m \pi  \bigg) dy. \label{def of Iqm}
\end{align}
Using \cite[2.633]{GRtable}, we obtain
\begin{align}
& \int_{a_{1}}^{a_{2}} x^{2\ell} \sin \bigg( \frac{x-a_{1}}{a_{2}-a_{1}} q \pi  \bigg)  dx = \sum_{\substack{v_{1}=0 \\v_{1} \, \mathrm{even} }}^{2\ell} \frac{(2\ell)!}{(2\ell-v_{1})!} \bigg( \frac{a_{2}-a_{1}}{q\pi} \bigg)^{v_{1}+1} (-1)^{1+\frac{v_{1}}{2}} \big( (-1)^{q}a_{2}^{2\ell-v_{1}} - a_{1}^{2\ell-v_{1}} \big), \nonumber \\
& \int_{c_{1}}^{c_{2}} y^{2(b-1-\ell)} \sin \bigg( \frac{y-c_{1}}{c_{2}-c_{1}} m \pi  \bigg) dy = \sum_{\substack{v_{2}=0 \\v_{2} \, \mathrm{even} }}^{2(b-1-\ell)} \frac{(2(b-1-\ell))!}{(2(b-1-\ell)-v_{2})!} \bigg( \frac{c_{2}-c_{1}}{m \pi} \bigg)^{v_{2}+1} (-1)^{1+\frac{v_{2}}{2}} \nonumber \\
& \hspace{5.7cm} \times \big( (-1)^{m}c_{2}^{2(b-1-\ell)-v_{2}} - c_{1}^{2(b-1-\ell)-v_{2}} \big). \label{int xp sin}
\end{align}
After substituting the above and \eqref{def of lambda qm} in \eqref{def of Iqm} and \eqref{lol58}, and regrouping the terms, we find
\begin{align*}
\frac{d\nu(z)}{d\eta} = \sum_{m=1}^{+\infty} \tilde{\mathcal{C}}_{b,m}^{R}\sin \bigg( \frac{\eta-c_{1}}{c_{2}-c_{1}} m \pi \bigg),
\end{align*}
where
\begin{align}
& \tilde{\mathcal{C}}_{b,m}^{R} = \frac{-4b^{2}}{(c_{2}-c_{1})\pi^{2}} \sum_{\ell=0}^{b-1} \binom{b-1}{\ell} \label{lol62} \\
& \times \bigg\{ \sum_{\substack{v_{1}=0 \\v_{1} \, \mathrm{even} }}^{2\ell} \frac{(2\ell)!}{(2\ell-v_{1})!} \bigg( \frac{a_{2}-a_{1}}{\pi} \bigg)^{v_{1}+1} (-1)^{\frac{v_{1}}{2}} \sum_{q=0}^{+\infty}\bigg( \frac{a_{2}^{2\ell-v_{1}}}{q^{2} + \frac{(a_{2}-a_{1})^{2}m^{2}}{(c_{2}-c_{1})^{2}}} \frac{1}{q^{v_{1}}}  - \frac{(-1)^{q}a_{1}^{2\ell-v_{1}}}{q^{2} + \frac{(a_{2}-a_{1})^{2}m^{2}}{(c_{2}-c_{1})^{2}}}\frac{1}{q^{v_{1}}} \bigg) \bigg\} \nonumber \\
& \times \bigg\{ \sum_{\substack{v_{2}=0 \\v_{2} \, \mathrm{even} }}^{2(b-1-\ell)} \frac{(2(b-1-\ell))!}{(2(b-1-\ell)-v_{2})!} \bigg( \frac{c_{2}-c_{1}}{m \pi} \bigg)^{v_{2}+1} (-1)^{\frac{v_{2}}{2}} \big( (-1)^{m}c_{2}^{2(b-1-\ell)-v_{2}} - c_{1}^{2(b-1-\ell)-v_{2}} \big) \bigg\}. \nonumber
\end{align}
In order to simplify $\tilde{\mathcal{C}}_{b,m}^{R}$, we will need the following lemma.
\begin{lemma}
For any $v_{1}\in 2\N$ and $m \in (0,+\infty)$, we have
\begin{align}
& \sum_{q=1}^{+\infty} \frac{1}{q^{2}+m^{2}} \frac{1}{q^{v_{1}}} = \frac{(-1)^{\frac{v_{1}}{2}}}{m^{v_{1}}} \frac{m \pi \coth(m\pi)-1}{2m^{2}} + \sum_{j=0}^{\frac{v_{1}}{2}-1} \frac{(-1)^{j}}{m^{2+2j}}\zeta(v_{1}-2j), \label{lol59} \\
& \sum_{q=1}^{+\infty} \frac{1}{q^{2}+m^{2}} \frac{(-1)^{q}}{q^{v_{1}}} = \frac{(-1)^{\frac{v_{1}}{2}}}{m^{v_{1}}} \frac{m \pi \frac{1}{\sinh(m\pi)}-1}{2m^{2}} + \sum_{j=0}^{\frac{v_{1}}{2}-1} \frac{(-1)^{j}}{m^{2+2j}} \bigg( \frac{1}{2^{v_{1}-2j-1}}-1 \bigg) \zeta(v_{1}-2j), \label{lol60} \\
& \sum_{\substack{q=1 \\ q\, \mathrm{odd}}}^{+\infty} \frac{1}{q^{2}+m^{2}} \frac{1}{q^{v_{1}}} = \frac{(-1)^{\frac{v_{1}}{2}}}{m^{v_{1}}} \frac{\pi \tanh(\frac{m\pi}{2})}{4m} + \sum_{j=0}^{\frac{v_{1}}{2}-1} \frac{(-1)^{j}}{m^{2+2j}} \bigg( 1-\frac{1}{2^{v_{1}-2j}} \bigg)\zeta(v_{1}-2j), \label{lol61}
\end{align}
where $\zeta$ is the Riemann zeta function (see e.g. \cite[Chapter 25]{NIST} for the definition of $\zeta$).
\end{lemma}
\begin{proof}
Taking the difference between \eqref{lol59} and \eqref{lol60} and then dividing by $2$ yields \eqref{lol61}. Hence it is enough to prove two of these identities. For $v_{1}=0$, \eqref{lol59} and \eqref{lol61} follow respectively from \cite[1.421.4 and 1.421.2]{GRtable}. To prove \eqref{lol59} and \eqref{lol61} for general $v_{1}\in 2\N$, we use the following decomposition
\begin{align}\label{lol63}
\frac{1}{q^{2}+m^{2}} \frac{1}{q^{v_{1}}} = \frac{\frac{(-1)^{\frac{v_{1}}{2}}}{m^{v_{1}}}}{q^{2}+m^{2}} + \sum_{j=0}^{\frac{v_{1}}{2}-1} \frac{\frac{\frac{d^{j}}{d\tilde{q}^{j}}\frac{1}{\tilde{q}+m^{2}}|_{\tilde{q}=0}}{j!}}{q^{v_{1}-2j}} = \frac{\frac{(-1)^{\frac{v_{1}}{2}}}{m^{v_{1}}}}{q^{2}+m^{2}} + \sum_{j=0}^{\frac{v_{1}}{2}-1} \frac{\frac{(-1)^{j}}{m^{2+2j}}}{q^{v_{1}-2j}}.
\end{align}
Summing \eqref{lol63}, and simplifying using \eqref{lol59} with $v_{1}=0$, \eqref{lol61} with $v_{1}=0$, and the identities
\begin{align*}
& \sum_{q=1}^{+\infty} \frac{1}{q^{v_{1}}} = \zeta(v_{1}), \qquad \sum_{\substack{q=2 \\ q \, \mathrm{even}}}^{+\infty} \frac{1}{q^{v_{1}}} = \sum_{q=1}^{+\infty} \frac{1}{(2q)^{v_{1}}} = \frac{1}{2^{v_{1}}}\zeta(v_{1}), \\
& \sum_{\substack{q=1 \\ q \, \mathrm{odd}}}^{+\infty} \frac{1}{q^{v_{1}}} = \sum_{q=1}^{+\infty} \frac{1}{q^{v_{1}}} - \sum_{\substack{q=2 \\ q \, \mathrm{even}}}^{+\infty} \frac{1}{q^{v_{1}}} = \bigg( 1-\frac{1}{2^{v_{1}}} \bigg) \zeta(v_{1}),
\end{align*}
we find \eqref{lol59}, \eqref{lol61} for general $v_{1}\in 2\N$.
\end{proof}
Substituting \eqref{lol59} and \eqref{lol60} in \eqref{lol62}, and simplifying, we find that $\tilde{\mathcal{C}}_{b,m}^{R} = \mathcal{C}_{b,m}^{R}$, where $\mathcal{C}_{b,m}^{R}$ is as in Theorem \ref{thm:ML rectangle nu}. The density of $\nu$ along the left, the top and the bottom sides can be obtained similarly. This finishes the proof of Theorem \ref{thm:ML rectangle nu}.

\medskip \begin{proof}[Proof of Corollary \ref{coro:ML square nu}]
With $a_{2}=c_{2}=-a_{1}=-c_{1}$ (and general $b$), the coefficients in the statement of Theorem \ref{thm:ML rectangle nu} satisfy $\mathcal{C}_{b,m}^{R} = \mathcal{C}_{b,m}^{L} = \mathcal{C}_{b,m}^{T} = \mathcal{C}_{b,m}^{B}$ for all $m\in \N_{>0}$ and $\mathcal{C}_{b,m}^{R} = 0$ if $m$ is even. Using the functional relation $\coth (y) - \frac{1}{\sinh(y)} \equiv \tanh(\frac{y}{2})$, it is easy to check that $\mathcal{C}_{b,2m+1}^{R} = (-1)^{m}\mathcal{C}_{b,m}$, where $\mathcal{C}_{b,m}$ is as in the statement of Corollary \ref{coro:ML square nu}. Since $\sin(\frac{y+\frac{c}{2}}{c}(1+2m)\pi) = (-1)^{m} \cos(\frac{y}{c}(1+2m)\pi)$ for all $y\in \R$, \eqref{square right} is equivalent to \eqref{density square general b}. 
\end{proof} 

\begin{proof}[Proof of Theorem \ref{thm:Ginibre rectangle nu}]
With $b=1$ (and general $a_{1},a_{2},c_{1},c_{2}$), the coefficients in the statement of Theorem \ref{thm:ML rectangle nu} satisfy $\mathcal{C}_{b,m}^{R} = \mathcal{C}_{b,m}^{L} = \mathcal{C}_{b,m}^{T} = \mathcal{C}_{b,m}^{B} = 0$ if $m$ is even, while if $m$ is odd we have
\begin{align*}
& \mathcal{C}_{b,m}^{R} = \mathcal{C}_{b,m}^{L} = \frac{4(c_{2}-c_{1})}{\pi^{3}} \frac{\tanh(\frac{(a_{2}-a_{1})m}{(c_{2}-c_{1})}\frac{\pi}{2})}{m^{2}}, \qquad \mathcal{C}_{b,m}^{T} = \mathcal{C}_{b,m}^{B} = \frac{4(a_{2}-a_{1})}{\pi^{3}} \frac{\tanh(\frac{(c_{2}-c_{1})m}{(a_{2}-a_{1})}\frac{\pi}{2})}{m^{2}},
\end{align*}
where we have again used the functional relation $\coth (y) - \frac{1}{\sinh(y)} \equiv \tanh(\frac{y}{2})$. Now \eqref{square right left b1} and \eqref{square top bottom b1} directly follow from Theorem \ref{thm:ML rectangle nu} with $b=1$.
\end{proof}

\subsection{Balayage measure: behavior near a corner}\label{subsection: behavior near a corner}
Let $\mu := \frac{d^{2}z}{\pi}$, $a_{2}>a_{1}>0$, $c_{2}>c_{1}>0$, and define $U = \{z: a_{1} < \re z < a_{2}, \; c_{1} < \im z < c_{2}\}$. Recall from \eqref{square right left b1} that the measure $\nu:=\mathrm{Bal}(\mu|_{U},\partial U)$ is given for $z=a_{2}+iy$, $y\in (c_{1},c_{2})$ by
\begin{align}\label{lol94}
\frac{d\nu(z)}{dy} = \sum_{m=0}^{+\infty} \frac{4(c_{2}-c_{1})}{\pi^{3}} \frac{\tanh(\frac{a_{2}-a_{1}}{c_{2}-c_{1}}(1+2m)\frac{\pi}{2})}{(1+2m)^{2}} \sin \bigg( \frac{y-c_{1}}{c_{2}-c_{1}} (1+2m) \pi \bigg).
\end{align}
Conjecture \ref{conj:behavior of balayage measure} (with $p=4$ and $b=1$) predicts that $\frac{d\nu(z)}{dy} \asymp (y-c_{1})\log(y-c_{1})$ as $y\to c_{1}$, $y>c_{1}$ (see also Remark \ref{remark:vanishing of balayage near a corner}). In this subsection, we show that this is indeed the case. For $y\in (c_{1},c_{2})$, define $M=\lfloor \frac{1}{y-c_{1}}\rfloor \in \N$. Letting $y\to c_{1}$ in \eqref{lol94}, we get
\begin{align*}
\frac{d\nu(z)}{dy} & = (y-c_{1})\sum_{m=0}^{M} \frac{4}{\pi^{2}} \frac{\tanh(\frac{a_{2}-a_{1}}{c_{2}-c_{1}}(1+2m)\frac{\pi}{2})}{1+2m} +\bigO\bigg( (y-c_{1})^{3} \sum_{m=0}^{M} (1+2m) + \sum_{m=M+1}^{+\infty}  \frac{1}{(1+2m)^{2}} \bigg),
\end{align*}
where we have used that $|\sin x |\leq 1$ and $|\tanh x| \leq 1$ for all $x \geq 0$. Since $M=\lfloor \frac{1}{y-c_{1}}\rfloor$, we have
\begin{align*}
(y-c_{1})^{3} \sum_{m=0}^{M} (1+2m) + \sum_{m=M+1}^{+\infty}  \frac{1}{(1+2m)^{2}} \lesssim (y-c_{1})^{3} M^{2} + \frac{1}{M} \lesssim y-c_{1}.
\end{align*}
Using also that $\tanh x = 1 + \bigO(e^{-2x})$ as $x\to +\infty$, we obtain
\begin{align*}
\frac{d\nu(z)}{dy} & = (y-c_{1})\sum_{m=0}^{M} \frac{4}{\pi^{2}} \frac{1}{1+2m} +\bigO( y-c_{1} ), \qquad \mbox{as } y \to c_{1}.
\end{align*}
A first-order Riemann sum analysis shows that $\sum_{m=0}^{M} \frac{1}{1+2m} = \frac{\log M}{2} + \bigO(1)$ as $M\to + \infty$. Hence,
\begin{align*}
\frac{d\nu(z)}{dy} = \frac{2}{\pi^{2}} (y-c_{1}) \log \frac{1}{y-c_{1}} +\bigO( y-c_{1} ), \qquad \mbox{as } y \to c_{1},
\end{align*}
as desired.
\subsection{Proof of Theorem \ref{thm: some nice series}}
Recall that $T_{v,\alpha}$ is defined for $v\in [0,+\infty)$ and $\alpha \in (0,+\infty)$ in \eqref{def of Tva}. In this subsection, we obtain a simplified expression for $T_{v,\alpha}$ when $v\in 2\N$ and $\alpha \in (0,+\infty)$. Let $\mu$ be supported on $S:=\{z:|z|\leq 1\}$ and given by $d\mu(z) = \frac{d^{2}z}{\pi}$. Let $a_{2}=-a_{1}>0$ and $c_{2}=-c_{1}>0$ be such that $U=\{z: \re z\in (a_{1},a_{2}), \im z \in (c_{1},c_{2})\}\subset S$, and let $\nu := \mathrm{Bal}(\mu|_{U},\partial U)$. Then $\nu$ is given by Theorem \ref{thm:Ginibre rectangle nu} with $\tau=0$. By Lemma \ref{lemma:further properties of nu} (iii), we have 
\begin{align}\label{moments match for the square}
\int_{\partial U} z^{n} d\nu(z) = \int_{U} z^{n} d\mu(z), \qquad \mbox{for any } n \in \N.
\end{align}
In what follows, we let $\alpha := \frac{a_{2}-a_{1}}{c_{2}-c_{1}} = \frac{a_{2}}{c_{2}} \in (0,+\infty)$. 
\begin{lemma}
Let $n \in \N$. Then $\int_{\partial U} z^{n} d\nu(z) = 0$ if $n$ is odd, and
\begin{align}\label{lol69}
\int_{U} z^{2n} d\mu(z) = \frac{4 a_{2}^{2n+2}}{\pi (2n+1)(2n+2)}\frac{(\sqrt{\alpha} + \frac{i}{\sqrt{\alpha}})^{2+2n}-(\sqrt{\alpha} - \frac{i}{\sqrt{\alpha}})^{2+2n}}{2i \alpha^{1+n}}.
\end{align}
\end{lemma}
\begin{proof}
A direct computation gives
\begin{align*}
\int_{U} z^{n} d\mu(z) & = \int_{a_{1}}^{a_{2}}\int_{c_{1}}^{c_{2}} (x+iy)^{n} \frac{dxdy}{\pi} = \sum_{j=0}^{n} \binom{n}{j} i^{j} \int_{a_{1}}^{a_{2}}\int_{c_{1}}^{c_{2}} y^{j}x^{n-j} \frac{dxdy}{\pi} \\
& = \frac{1}{\pi} \sum_{j=0}^{n} \binom{n}{j} i^{j} \frac{a_{2}^{n-j+1}-a_{1}^{n-j+1}}{n-j+1} \frac{c_{2}^{j+1}-c_{1}^{j+1}}{j+1}.
\end{align*}
Since $a_{2}=-a_{1}>0$ and $c_{2}=-c_{1}>0$, the above is $0$ if $n$ is odd, while for $n$ even we have
\begin{align*}
& \int_{U} z^{n} d\mu(z) = \frac{4}{\pi} \sum_{\substack{j=0 \\ j \, \mathrm{even}}}^{n} \binom{n}{j} i^{j} \frac{a_{2}^{n-j+1}}{n-j+1}\frac{c_{2}^{j+1}}{j+1} = \frac{4}{\pi (n+2)(n+1)} \sum_{\substack{j=0 \\ j \, \mathrm{even}}}^{n} \binom{n+2}{j+1} i^{j} a_{2}^{n-j+1}c_{2}^{j+1} \\
& = \frac{4 a_{2}^{n+2} (-i)}{\pi (n+2)(n+1)} \sum_{\substack{j=0 \\ j \, \mathrm{even}}}^{n} \binom{n+2}{j+1} \bigg(\frac{i}{\alpha}\bigg)^{j+1} = \frac{4 a_{2}^{n+2}}{\pi (n+2)(n+1)} \frac{(1+\frac{i}{\alpha})^{2+n}-(1-\frac{i}{\alpha})^{2+n}}{2i},
\end{align*}
and \eqref{lol69} follows.
\end{proof}

We now evaluate $\int_{\partial U} z^{n} d\nu(z)$ explicitly.
\begin{lemma}
Let $n \in \N$. Then $\int_{\partial U} z^{n} d\nu(z)=0$ if $n$ is odd, and 
\begin{align}\label{lol68}
\int_{\partial U} z^{2n} d\nu(z) = \frac{a_{2}^{2n+2}}{\pi} \sum_{v=0}^{n} \frac{(2n)!}{(2n-2v)!} \frac{4^{3+v}}{\pi^{3+2v}} \frac{(\sqrt{\alpha} - \frac{i}{\sqrt{\alpha}})^{2n-2v} + (\sqrt{\alpha} + \frac{i}{\sqrt{\alpha}})^{2n-2v}}{\alpha^{n+1}} T_{2v,\alpha}
\end{align}
\end{lemma}
\begin{proof}
Let $n \in \N$. By Theorem \ref{thm:Ginibre rectangle nu} with $\tau=0$, we have
\begin{align}
& \int_{\partial U} z^{n} d\nu(z) = \int_{c_{1}}^{c_{2}} \big( (a_{1}+iy)^{n} + (a_{2}+iy)^{n} \big) d\nu(a_{2}+iy) + \int_{a_{1}}^{a_{2}} \big( (x+ic_{1})^{n} + (x+ic_{2})^{n} \big) d\nu(x+ic_{2}) \nonumber \\
& = \int_{c_{1}}^{c_{2}} \sum_{j=0}^{n} \binom{n}{j} i^{j} y^{j} (a_{1}^{n-j} + a_{2}^{n-j}) \sum_{m=0}^{+\infty} \frac{4(c_{2}-c_{1})}{\pi^{3}}\frac{\tanh(\alpha(1+2m)\frac{\pi}{2})}{(1+2m)^{2}} \sin \bigg( \frac{y-c_{1}}{c_{2}-c_{1}} (1+2m)\pi \bigg) dy \nonumber \\
& + \int_{a_{1}}^{a_{2}} \sum_{j=0}^{n} \binom{n}{j} i^{n-j} x^{j} (c_{1}^{n-j} + c_{2}^{n-j}) \sum_{m=0}^{+\infty} \frac{4(a_{2}-a_{1})}{\pi^{3}}\frac{\tanh(\alpha^{-1}(1+2m)\frac{\pi}{2})}{(1+2m)^{2}} \sin \bigg( \frac{x-a_{1}}{a_{2}-a_{1}} (1+2m)\pi \bigg) dx. \label{lol73}
\end{align}
Since $c_{2}=-c_{1}$ and $a_{2}=-a_{1}$, for any $m\in \N$ and any $j$ odd, we have
\begin{align}\label{lol72}
\int_{c_{1}}^{c_{2}} y^{j} \sin \bigg( \frac{y-c_{1}}{c_{2}-c_{1}} (1+2m)\pi \bigg) dy = 0, \qquad \int_{a_{1}}^{a_{2}} x^{j} \sin \bigg( \frac{x-a_{1}}{a_{2}-a_{1}} (1+2m)\pi \bigg) dx = 0.
\end{align}
Hence only the terms corresponding to $j$ even  contribute in \eqref{lol73}. Since $a_{1}^{n-j} + a_{2}^{n-j}=0=c_{1}^{n-j} + c_{2}^{n-j}$ whenever $j$ is even and $n$ is odd, we conclude that $\int_{\partial U} z^{n} d\nu(z)=0$ if $n$ is odd. From now we assume that $n$ is even. For $j$ even,  the left-hand sides in \eqref{lol72} can be evaluated explicitly using \eqref{int xp sin}, and after simplification (using again $a_{2}=-a_{1}$ and $c_{2}=-c_{1}$) we get
\begin{align*}
& \int_{\partial U} z^{n} d\nu(z) =  \sum_{\substack{j=0 \\ j \, \mathrm{even}}}^{n} \binom{n}{j} i^{j} a_{2}^{n-j} \sum_{m=0}^{+\infty} \frac{32c_{2}}{\pi^{3}}\frac{\tanh(\alpha(1+2m)\frac{\pi}{2})}{(1+2m)^{2}} \sum_{\substack{v=0 \\ v \, \mathrm{even}}} \frac{j!(-1)^{\frac{v}{2}} c_{2}^{j-v}}{(j-v)!} \bigg( \frac{2c_{2}}{(1+2m)\pi} \bigg)^{v+1}  \\
& + \sum_{\substack{j=0 \\ j \, \mathrm{even}}}^{n} \binom{n}{j} i^{n-j} c_{2}^{n-j} \sum_{m=0}^{+\infty} \frac{32a_{2}}{\pi^{3}}\frac{\tanh(\alpha^{-1}(1+2m)\frac{\pi}{2})}{(1+2m)^{2}} \sum_{\substack{v=0 \\ v \, \mathrm{even}}} \frac{j!(-1)^{\frac{v}{2}} a_{2}^{j-v}}{(j-v)!} \bigg( \frac{2a_{2}}{(1+2m)\pi} \bigg)^{v+1}.
\end{align*}
Using $\sum_{\substack{j=0 \\ j \, \mathrm{even}}}^{n} \sum_{\substack{v=0 \\ v \, \mathrm{even}}}^{j} = \sum_{\substack{v=0 \\ v \, \mathrm{even}}}^{n} \sum_{\substack{j=v \\ j \, \mathrm{even}}}^{n}$, regrouping the sums and changing indices, we get
\begin{align*}
& \int_{\partial U} z^{n} d\nu(z) = \frac{a_{2}^{n+2}}{\pi} \sum_{\substack{v=0 \\ v \, \mathrm{even}}}^{n}  \frac{n!(-1)^{\frac{v}{2}}}{(n-v)!} \frac{2^{6+v}}{\pi^{3+v}} \sum_{\substack{j=0 \\ j \, \mathrm{even}}}^{n-v} \binom{n-v}{j} \\
& \hspace{1cm} \times \bigg( \frac{i^{j+v}}{\alpha^{j+v+2}} \sum_{m=0}^{+\infty} \frac{\tanh(\alpha(1+2m)\frac{\pi}{2})}{(1+2m)^{3+v}} + \frac{i^{n-v-j}}{\alpha^{n-v-j}} \sum_{m=0}^{+\infty} \frac{\tanh(\alpha^{-1}(1+2m)\frac{\pi}{2})}{(1+2m)^{3+v}} \bigg).
\end{align*}
After simplifying the above using
\begin{align*}
\sum_{\substack{j=0 \\ j \, \mathrm{even}}}^{n-v} \binom{n-v}{j} \beta^{j} = \frac{(1+\beta)^{n-v} + (1-\beta)^{n-v}}{2}, \qquad \beta \in \C,
\end{align*}
we find \eqref{lol68}.
\end{proof}

Theorem \ref{thm: some nice series} now directly follows by combining \eqref{moments match for the square}, \eqref{lol69} and \eqref{lol68}.

\subsection{The constant $C$: proof of Theorem \ref{thm:Ginibre rectangle C}}
Let $\tau \in [0,1)$, $Q(z)=\frac{1}{1-\tau^{2}}\big( |z|^{2}-\tau \, \re z^{2} \big) = \frac{x^{2}}{1+\tau}+\frac{y^{2}}{1-\tau}$, $z=x+iy$, $x,y\in \R$. Recall that $\mu$ and $S$ are defined in \eqref{mu S EG}, and suppose that $a_{1},a_{2},c_{1},c_{2}$ are such that $U\subset S$. Recall also from Theorem \ref{thm:Ginibre rectangle nu} that $\nu := \mathrm{Bal}(\mu|_{U},\partial U)$ is given by \eqref{square right left b1} and \eqref{square top bottom b1}.
\begin{lemma}
\begin{align}\label{int Q d2z RECTANGLE}
\int_{U}Q(z) d\mu(z) = \frac{1}{\pi(1-\tau^{2})} \bigg( \frac{(a_{2}^{3}-a_{1}^{3})(c_{2}-c_{1})}{3(1+\tau)} + \frac{(a_{2}-a_{1})(c_{2}^{3}-c_{1}^{3})}{3(1-\tau)} \bigg).
\end{align}
\end{lemma}
\begin{proof}
Substituting the definitions of $Q$, $\mu$ and $U$, we get
\begin{align*}
\int_{U}Q(z) d\mu(z) = \int_{a_{1}}^{a_{2}}\int_{c_{1}}^{c_{2}} \bigg( \frac{x^{2}}{1+\tau} + \frac{y^{2}}{1-\tau} \bigg) \frac{dxdy}{\pi(1-\tau^{2})},
\end{align*}
and the claim follows after a straightforward computation.
\end{proof}

\begin{lemma}
\begin{align}
& \int_{U}Q(z) d\nu(z) = \frac{1}{\pi(1-\tau^{2})} \bigg\{ \frac{(a_{2}-a_{1})(c_{2}-c_{1})}{2}\bigg( \frac{a_{1}^{2}+a_{2}^{2}}{1+\tau} + \frac{c_{2}^{2}+c_{1}^{2}}{1-\tau} \bigg) \nonumber \\
& - \frac{32}{\pi^{5}} \frac{(c_{2}-c_{1})^{4}}{1-\tau} \sum_{m=0}^{+\infty} \frac{\tanh(\frac{a_{2}-a_{1}}{c_{2}-c_{1}}(1+2m)\frac{\pi}{2})}{(1+2m)^{5}} - \frac{32}{\pi^{5}} \frac{(a_{2}-a_{1})^{4}}{1+\tau} \sum_{m=0}^{+\infty} \frac{\tanh(\frac{c_{2}-c_{1}}{a_{2}-a_{1}}(1+2m)\frac{\pi}{2})}{(1+2m)^{5}} \bigg\}. \label{lol71}
\end{align}
\end{lemma}
\begin{proof}
Since $d\nu(a_{2}+iy)=d\nu(a_{1}+iy)$ and $d\nu(x+ic_{2})=d\nu(x+ic_{1})$, we have
\begin{align*}
& \int_{U}Q(z) d\nu(z) \\
& = \int_{c_{1}}^{c_{2}} \bigg( Q(a_{2}+iy) + Q(a_{1}+iy) \bigg) d\nu(a_{2}+iy) + \int_{a_{1}}^{a_{2}} \bigg( Q(x+ic_{2}) + Q(x+ic_{1}) \bigg) d\nu(x+ic_{2}) \\
& = \int_{c_{1}}^{c_{2}} \bigg( \frac{a_{1}^{2}+a_{2}^{2}}{1+\tau} + \frac{2y^{2}}{1-\tau} \bigg) d\nu(a_{2}+iy) + \int_{a_{1}}^{a_{2}} \bigg( \frac{2x^{2}}{1+\tau} + \frac{c_{1}^{2}+c_{2}^{2}}{1-\tau} \bigg) d\nu(x+ic_{2}).
\end{align*}
Substituting \eqref{square right left b1} and \eqref{square top bottom b1} then yields
\begin{align*}
& \int_{U}Q(z) d\nu(z) \\
& = \sum_{m=0}^{+\infty} \frac{4(c_{2}-c_{1})}{\pi^{3}(1-\tau^{2})} \frac{\tanh(\frac{a_{2}-a_{1}}{c_{2}-c_{1}}(1+2m)\frac{\pi}{2})}{(1+2m)^{2}} \int_{c_{1}}^{c_{2}} \bigg( \frac{a_{1}^{2}+a_{2}^{2}}{1+\tau} + \frac{2y^{2}}{1-\tau} \bigg) \sin \bigg( \frac{y-c_{1}}{c_{2}-c_{1}}(1+2m)\pi \bigg) dy \\
& + \sum_{m=0}^{+\infty} \frac{4(a_{2}-a_{1})}{\pi^{3}(1-\tau^{2})} \frac{\tanh(\frac{c_{2}-c_{1}}{a_{2}-a_{1}}(1+2m)\frac{\pi}{2})}{(1+2m)^{2}} \int_{a_{1}}^{a_{2}} \bigg( \frac{2x^{2}}{1+\tau} + \frac{c_{1}^{2}+c_{2}^{2}}{1-\tau} \bigg) \sin \bigg( \frac{x-a_{1}}{a_{2}-a_{1}}(1+2m)\pi \bigg) dx.
\end{align*}
Using now \eqref{int xp sin} and simplifying, we get
\begin{align}
& \int_{U}Q(z) d\nu(z) = \frac{1}{\pi(1-\tau^{2})} \bigg\{ \frac{8(a_{2}-a_{1})(c_{2}-c_{1})}{\pi^{3}}\bigg( \frac{a_{1}^{2}+a_{2}^{2}}{1+\tau} + \frac{c_{2}^{2}+c_{1}^{2}}{1-\tau} \bigg) \nonumber \\
& \times \bigg( \frac{c_{2}-c_{1}}{a_{2}-a_{1}} \sum_{m=0}^{+\infty} \frac{\tanh(\frac{a_{2}-a_{1}}{c_{2}-c_{1}}(1+2m)\frac{\pi}{2})}{(1+2m)^{3}} + \frac{a_{2}-a_{1}}{c_{2}-c_{1}} \sum_{m=0}^{+\infty} \frac{\tanh(\frac{c_{2}-c_{1}}{a_{2}-a_{1}}(1+2m)\frac{\pi}{2})}{(1+2m)^{3}} \bigg) \nonumber \\
& - \frac{32}{\pi^{5}} \frac{(c_{2}-c_{1})^{4}}{1-\tau} \sum_{m=0}^{+\infty} \frac{\tanh(\frac{a_{2}-a_{1}}{c_{2}-c_{1}}(1+2m)\frac{\pi}{2})}{(1+2m)^{5}} - \frac{32}{\pi^{5}} \frac{(a_{2}-a_{1})^{4}}{1+\tau} \sum_{m=0}^{+\infty} \frac{\tanh(\frac{c_{2}-c_{1}}{a_{2}-a_{1}}(1+2m)\frac{\pi}{2})}{(1+2m)^{5}} \bigg\}. \label{lol70}
\end{align}
By \cite[0.243.3]{GRtable} (or by Theorem \ref{thm: some nice series} with $v=0$ and $\alpha := \frac{a_{2}-a_{1}}{c_{2}-c_{1}}$),
\begin{align*}
\frac{c_{2}-c_{1}}{a_{2}-a_{1}} \sum_{m=0}^{+\infty} \frac{\tanh(\frac{a_{2}-a_{1}}{c_{2}-c_{1}}(1+2m)\frac{\pi}{2})}{(1+2m)^{3}} + \frac{a_{2}-a_{1}}{c_{2}-c_{1}} \sum_{m=0}^{+\infty} \frac{\tanh(\frac{c_{2}-c_{1}}{a_{2}-a_{1}}(1+2m)\frac{\pi}{2})}{(1+2m)^{3}} = \frac{\pi^{3}}{16},
\end{align*}
and substituting the above in \eqref{lol70} yields \eqref{lol71}.
\end{proof}
Note that $U$ satisfies Assumptions \ref{ass:U} and \ref{ass:U2}. Combining \eqref{int Q d2z RECTANGLE} and \eqref{lol71} with Theorem \ref{thm:general pot} (i), we infer that $\mathbb{P}(\# \{z_{j}\in U\} = 0) = \exp \big( -C n^{2}+o(n^{2}) \big)$ as $n\to+\infty$, where $C$ is given by
\begin{align}
 C & = \frac{(a_{2}-a_{1})^{2}(c_{2}-c_{1})^{2}}{\pi(1-\tau^{2})} \bigg( \frac{\alpha}{6(1+\tau)} + \frac{\alpha^{-1}}{6(1-\tau)} - \frac{32}{\pi^{5}} \frac{\alpha^{-2}}{1-\tau} \sum_{m=0}^{+\infty} \frac{\tanh(\alpha (1+2m)\frac{\pi}{2})}{(1+2m)^{5}} \nonumber \\
&  - \frac{32}{\pi^{5}} \frac{\alpha^{2}}{1+\tau} \sum_{m=0}^{+\infty} \frac{\tanh(\alpha^{-1}(1+2m)\frac{\pi}{2})}{(1+2m)^{5}} \bigg). \label{lol74}
\end{align}
where we recall that $\alpha := \frac{a_{2}-a_{1}}{c_{2}-c_{1}}$. The fact that $C$ in \eqref{lol74} can be rewritten as in the statement of Theorem \ref{thm:Ginibre rectangle C} is a consequence of the identity
\begin{align*}
T_{2,\alpha} = \pi^{5} \frac{\frac{1}{\alpha}-\alpha}{384},
\end{align*}
which follows from Theorem \ref{thm: some nice series}. (The fact that $C$ in \eqref{lol74} can be rewritten as in the statement of Theorem \ref{thm:Ginibre rectangle C} can also be seen as direct consequence of Theorem \ref{thm:general U EG}.) This finishes the proof of Theorem \ref{thm:Ginibre rectangle C}.

\subsection{The constant $C$: proof of Theorem \ref{thm:ML rectangle C} and Corollary \ref{coro:ML square C}}
In this subsection, $b \in \N_{>0}$, $d\mu(z) := \frac{b^{2}}{\pi}|z|^{2b-2}d^{2}z$, $S:=\mathrm{supp}\,\mu =\{z:|z|\leq b^{-\frac{1}{2b}}\}$, and $a_{2}>a_{1}$, $c_{2}>c_{1}$ are such that $U := \{z:\re z \in (a_{1},a_{2}), \im z \in (c_{1},c_{2})\} \subset S$. Let $\nu := \mathrm{Bal}(\mu|_{U},\partial U)$, and recall that $\nu$ is given in Theorem \ref{thm:ML rectangle nu}.
\begin{lemma}\label{lemma:int |z|2b dmu RECTANGLE}
\begin{align}\label{int |z|2b dmu RECTANGLE}
\int_{U} |z|^{2b} d\mu(z) =  \frac{b^{2}}{\pi} \sum_{j=0}^{2b-1} \binom{2b-1}{j} \frac{a_{2}^{2j+1}-a_{1}^{2j+1}}{2j+1} \frac{c_{2}^{2(2b-1-j)+1}-c_{1}^{2(2b-1-j)+1}}{2(2b-1-j)+1}.
\end{align}
\end{lemma}
\begin{proof}
By definition of $U$ and $\mu$,
\begin{align*}
& \int_{U} |z|^{2b} d\mu(z) = \int_{U} |z|^{2b} \frac{b^{2}}{\pi} |z|^{2b-2}d^{2}z = \frac{b^{2}}{\pi} \int_{a_{1}}^{a_{2}} \int_{c_{1}}^{c_{2}} (x^{2}+y^{2})^{2b-1} dx dy.
\end{align*}
After expanding the above right-hand side (using that $b\in \N_{>0}$), we find \eqref{int |z|2b dmu RECTANGLE}.
\end{proof}
\begin{lemma}
For $a_{1},a_{2},c_{1},c_{2}$ such that $U\subset S$, we have
\begin{align}
& \int_{\partial U} |z|^{2b} d\nu(z) = \sum_{k=0}^{b} \binom{b}{k} \sum_{v=0}^{k} \frac{(2k)!(-1)^{1+v}}{(2k-2v)!}  \bigg\{ \bigg( \frac{c_{2}-c_{1}}{\pi} \bigg)^{2v+1} \nonumber \\
& \times \sum_{m=1}^{+\infty} \Big( a_{2}^{2(b-k)} \mathcal{C}_{b,m}^{R} + a_{1}^{2(b-k)} \mathcal{C}_{b,m}^{L}   \Big) \bigg( c_{2}^{2k-2v} \frac{(-1)^{m}}{m^{2v+1}} - \frac{c_{1}^{2k-2v}}{m^{2v+1}} \bigg) \nonumber \\
& + \bigg( \frac{a_{2}-a_{1}}{\pi} \bigg)^{2v+1} \sum_{m=1}^{+\infty} \Big( c_{2}^{2(b-k)} \mathcal{C}_{b,m}^{T}  + c_{1}^{2(b-k)} \mathcal{C}_{b,m}^{B}   \Big) \bigg( a_{2}^{2k-2v} \frac{(-1)^{m}}{m^{2v+1}} - \frac{a_{1}^{2k-2v}}{m^{2v+1}} \bigg) \bigg\}, \label{int |z|2b dnu RECTANGLE general b}
\end{align}
where the coefficients $\mathcal{C}_{b,m}^{R}$, $\mathcal{C}_{b,m}^{L}$, $\mathcal{C}_{b,m}^{T}$ and $\mathcal{C}_{b,m}^{B}$ are as in the statement of Theorem \ref{thm:ML rectangle nu}. 

\medskip \noindent If $a_{2}=c_{2}=-a_{1}=-c_{1}=\frac{c}{2}$ for some $c>0$, then \eqref{int |z|2b dnu RECTANGLE general b} simplifies to
\begin{align}
& \int_{\partial U} |z|^{2b} d\nu(z) = \frac{32b^{2}c^{4b}}{\pi^{4}4^{2b-1}} \sum_{j=0}^{b} \binom{b}{j} \sum_{v=0}^{j} \frac{(2j)!(-1)^{v}4^{v}}{(2j-2v)!\pi^{2v}} \sum_{\ell=0}^{b-1} \binom{b-1}{\ell} \sum_{v_{1}=0}^{\ell}\sum_{v_{2}=0}^{b-1-\ell} \frac{(2\ell)!}{(2\ell-2v_{1})!} \nonumber \\
& \times \frac{(2b-2-2\ell)!}{(2b-2-2\ell-2v_{2})!} \frac{4^{v_{1}+v_{2}}(-1)^{v_{2}}}{\pi^{2v_{1}+2v_{2}}} \bigg\{ \mathbf{1}_{v_{1}+v_{2}\, \mathrm{even}} T_{2(v_{1}+v_{2}+v)} + \sum_{q=0}^{v_{1}-1} \frac{4(-1)^{v_{1}+q}}{\pi}\bigg( 1- \frac{1}{4^{2+v_{2}+q+v}} \bigg) \nonumber \\
& \times \zeta\big(2(2+v_{2}+q+v)\big) \bigg( 1- \frac{1}{4^{v_{1}-q}} \bigg) \zeta\big(2(v_{1}-q)\big) \bigg\}, \label{int |z|2b dnu SQUARE general b}
\end{align}
where $T_{2v}$ is defined in \eqref{def of Tv}.
\end{lemma}
\begin{proof}
By definition of $U$,
\begin{multline} \label{lol75}
\int_{\partial U} |z|^{2b} d\nu(z) = \int_{c_{1}}^{c_{2}} (a_{2}^{2}+y^{2})^{b}d\nu(a_{2}+iy) + \int_{c_{1}}^{c_{2}} (a_{1}^{2}+y^{2})^{b}d\nu(a_{1}+iy) \\
+ \int_{a_{1}}^{a_{2}} (x^{2}+c_{2}^{2})^{b}d\nu(x+ic_{2}) + \int_{a_{1}}^{a_{2}} (x^{2}+c_{1}^{2})^{b}d\nu(x+ic_{1}).
\end{multline}
Using the expression \eqref{square right} for $d\nu(a_{2}+iy)$, we get
\begin{align*}
& \int_{c_{1}}^{c_{2}} (a_{2}^{2}+y^{2})^{b}d\nu(a_{2}+iy) = \sum_{m=1}^{+\infty} \mathcal{C}_{b,m}^{R} \int_{c_{1}}^{c_{2}} (a_{2}^{2}+y^{2})^{b} \sin \bigg( \frac{y-c_{1}}{c_{2}-c_{1}} m \pi \bigg) dy \\
& = \sum_{k=0}^{b} \binom{b}{k} a_{2}^{2(b-k)} \sum_{m=1}^{+\infty} \mathcal{C}_{b,m}^{R} \int_{c_{1}}^{c_{2}} y^{2k} \sin \bigg( \frac{y-c_{1}}{c_{2}-c_{1}} m \pi \bigg) dy \\
& = \sum_{k=0}^{b} \binom{b}{k} a_{2}^{2(b-k)} \sum_{m=1}^{+\infty} \mathcal{C}_{b,m}^{R} \sum_{v=0}^{k} \frac{(2k)!(-1)^{1+v}}{(2k-2v)!} \bigg( \frac{c_{2}-c_{1}}{\pi} \bigg)^{2v+1} \bigg( c_{2}^{2k-2v} \frac{(-1)^{m}}{m^{2v+1}} - \frac{c_{1}^{2k-2v}}{m^{2v+1}} \bigg),
\end{align*}
where for the last equality we have used \eqref{int xp sin}. The last three integrals on the right-hand side of \eqref{lol75} can be simplified similarly, and substituting these simplified expressions in \eqref{lol75} yields \eqref{int |z|2b dnu RECTANGLE general b}.

\medskip If $a_{2}=c_{2}=-a_{1}=-c_{1}=\frac{c}{2}$, we have $\mathcal{C}_{b,m}^{R} = \mathcal{C}_{b,m}^{L} = \mathcal{C}_{b,m}^{T} = \mathcal{C}_{b,m}^{B}$ for all $m\in \N_{>0}$ and $\mathcal{C}_{b,m}^{R} = 0$ if $m$ is even. Using the functional relation $\coth (y) - \frac{1}{\sinh(y)} \equiv \tanh(\frac{y}{2})$, it is easy to check that $\mathcal{C}_{b,2m+1}^{R} = (-1)^{m}\mathcal{C}_{b,m}$, where $\mathcal{C}_{b,m}$ is as in the statement of Corollary \ref{coro:ML square nu}. Substituting these identities in \eqref{int |z|2b dnu RECTANGLE general b} yields after  simplication \eqref{int |z|2b dnu SQUARE general b}. 
\end{proof}

Note that $U$ satisfies Assumptions \ref{ass:U} and \ref{ass:U2}. Combining \eqref{int |z|2b dmu RECTANGLE} and \eqref{int |z|2b dnu RECTANGLE general b} with Theorem \ref{thm:general pot} (i), we infer that $\mathbb{P}(\# \{z_{j}\in U\} = 0) = \exp \big( -C n^{2}+o(n^{2}) \big)$ as $n\to+\infty$, where $C$ is given by \eqref{lol76}. This finishes the proof of Theorem \ref{thm:ML rectangle C}. 

\medskip Corolary \ref{coro:ML square C} follows similarly, using \eqref{int |z|2b dnu SQUARE general b} instead of \eqref{int |z|2b dnu RECTANGLE general b}.

\section{The complement of an ellipse centered at $0$}\label{section:complement ellipse}

In this section, $\tau \in [0,1)$, $Q(z)=\frac{1}{1-\tau^{2}}\big( |z|^{2}-\tau \, \re z^{2} \big) = \frac{x^{2}}{1+\tau}+\frac{y^{2}}{1-\tau}$, $z=x+iy$, $x,y\in \R$, and the measure $\mu$ and its support $S$ are defined in \eqref{mu S EG}. We also suppose that $a\in (0,1+\tau]$ and $c\in (0,1-\tau]$, so that $U^{c}\subset S$, where $U := \{z: \frac{(\re z)^{2}}{a^{2}}+\frac{(\im z)^{2}}{c^{2}}>1\}$.

\subsection{Balayage measure: proof of Theorem \ref{thm:EG complement ellipse nu}}

Let $\nu := \mathrm{Bal}(\mu|_{U},\partial U)$. We will obtain an explicit expression for $\nu$ using the method described in Subsection \ref{section:ansatz method}. Since $U$ is unbounded and $0\notin \overline{U}$, we will use Lemma \ref{lemma: moment when U is unbounded} with $z_{0}=0$. In particular, we are seeking for a measure $\nu$ supported on $\partial U$ and satisfying 
\begin{align}\label{eqn:moment unbounded with mass}
& \begin{cases}
\ds \int_{\partial U}z^{-n} d\nu(z) = \int_{ U}z^{-n}d\mu(z), \quad \mbox{for all } n \in \N, \\[0.3cm]
\ds \int_{\partial U}\log |z|  d\nu(z) = \int_{ U} \log |z| d\mu(z) - c^{\hspace{0.02cm}\mu}_{U},
\end{cases}
\end{align}
where $c^{\hspace{0.02cm}\mu}_{U} = 2\pi\int_{U} g_{U}(z,\infty) d\mu(z)$.
From now on we focus on the case $a> c$ and $\tau>0$ (the analysis below can easily be adapted to the other cases involving either $a<c$, $a=c$ or $\tau=0$). 

Define $\alpha = \frac{a+c}{2}$, $\gamma=\frac{a-c}{2}$, $\tau_{\star} = \frac{2\tau}{1+\tau^{2}}$, $\tau_{a,c} = \frac{a^{2}-c^{2}}{a^{2}+c^{2}}$ and
\begin{align}\label{def of Ri and Re}
R_{i}(\theta) = \sqrt{\frac{a^{2}+c^{2}}{2}} \frac{\sqrt{1-\tau_{a,c}^{2}}}{\sqrt{1-\tau_{a,c}\cos (2\theta)}}, \qquad R_{e}(\theta) = \sqrt{1+\tau^{2}} \frac{\sqrt{1-\tau_{\star}^{2}}}{\sqrt{1-\tau_{\star} \cos(2\theta)}}.
\end{align}
We will use the following two parametrizations of the set $\partial U$:
\begin{align}\label{two param of partial U when Uc ellipse}
\partial U = \{a\cos \theta + i c \sin \theta: \theta \in (-\pi,\pi]\} = \{R_{i}(\theta)e^{i\theta} : \theta \in (-\pi,\pi]\}.
\end{align}
The parametrization of $S\cap U$ in polar coordinates is given by
\begin{align}\label{param polar ellipses}
S\cap U = \{r e^{i\theta}: \theta \in (-\pi,\pi], r\in (R_{i}(\theta),R_{e}(\theta)]\}.
\end{align}
Define also 
\begin{align}\label{def of ri and re}
r_{i} := \frac{1-\sqrt{1-\tau_{\smash{a,c}}^{2}}}{\tau_{a,c}} = \frac{a-c}{a+c},  \qquad r_{e} := \frac{1-\sqrt{1-\tau_{\star}^{2}}}{\tau_{\star}} = \tau,
\end{align}
so that we can write, for all $\theta \in (-\pi,\pi]$,
\begin{align}
& \sqrt{1-\tau_{\smash{a,c}}\cos \theta} = \sqrt{\frac{\tau_{a,c}}{2}}e^{-i\frac{\theta}{2}} \sqrt{e^{i\theta}-r_{i}} \sqrt{\frac{1}{r_{i}}-e^{i\theta}}, \quad \sqrt{1-\tau_{\star}\cos \theta} = \sqrt{\frac{\tau_{\star}}{2}} e^{-i\frac{\theta}{2}} \sqrt{e^{i\theta}-r_{e}}\sqrt{\frac{1}{r_{e}}-e^{i\theta}}, \label{branch cuts with tau} \\
& \sqrt{1-\tau_{\smash{a,c}}\cos (2\theta)} = \sqrt{\frac{\tau_{a,c}}{2}}e^{-i\theta} \sqrt{e^{i\theta}-\sqrt{r_{i}}}\sqrt{e^{i\theta}+\sqrt{r_{i}}} \sqrt{\frac{1}{\sqrt{r_{i}}}-e^{i\theta}}\sqrt{e^{i\theta}+\frac{1}{\sqrt{r_{i}}}}, \label{branch cuts with tau 2} \\
& \sqrt{1-\tau_{\star}\cos (2\theta)} = \sqrt{\frac{\tau_{\star}}{2}}e^{-i\theta} \sqrt{e^{i\theta}-\sqrt{r_{e}}}\sqrt{e^{i\theta}+\sqrt{r_{e}}} \sqrt{\frac{1}{\sqrt{r_{e}}}-e^{i\theta}}\sqrt{\frac{1}{e^{i\theta}+\sqrt{r_{e}}}}, \label{branch cuts with tau 3}
\end{align}
where the principal branches for the roots are used (below, unless otherwise specified, all branches for the logarithms and roots are the principal ones). Since $\tau_{\star},\tau_{a,c}\in (0,1)$, we have $r_{i},r_{e}\in (0,1)$. In this section, to simplify the computations we will use repetitively the following simple but important identity: for any $f$ integrable on $\mathbb{S}^{1}$, we have
\begin{align}\label{f 2 theta to f 1 theta}
\int_{-\pi}^{\pi}f(e^{2i\theta})d\theta = \int_{-\pi}^{\pi}f(e^{i\theta})d\theta.
\end{align}
We start with a preliminary lemma.
\begin{lemma}\label{lemma:some integrals}
The following relations hold:
\begin{align}
& \int_{-\pi}^{\pi}e^{-2i\theta}\log R_{e}(\theta) \; d\theta = \pi r_{e}, \label{lol27} \\
& \int_{-\pi}^{\pi}e^{-2i\theta}\log R_{i}(\theta) \; d\theta = \pi r_{i}, \label{lol28} \\
& \int_{-\pi}^{\pi} R_{e}(\theta)^{2}d\theta = 2 \pi (1-\tau^{2}), \label{lol29} \\
& \int_{-\pi}^{\pi} R_{i}(\theta)^{2}d\theta = 2 \pi a c, \label{lol30} \\
& \int_{-\pi}^{\pi}e^{-i\theta n}R_{e}(\theta)^{2-n}d\theta = 0 = \int_{-\pi}^{\pi}e^{-i\theta n}R_{i}(\theta)^{2-n}d\theta, \qquad \mbox{for all } n \in \N_{>0}. \label{lol31}
\end{align}
\end{lemma}
\begin{proof}
\underline{Proof of \eqref{lol27}.} Using \eqref{def of Ri and Re}, \eqref{branch cuts with tau}, and \eqref{f 2 theta to f 1 theta}, we get
\begin{align*}
& \int_{-\pi}^{\pi}e^{-2i\theta}\log R_{e}(\theta) \; d\theta = \int_{-\pi}^{\pi}e^{-i\theta}\log R_{e}(\tfrac{\theta}{2}) \; d\theta = \int_{-\pi}^{\pi} e^{-i\theta} \log \frac{\sqrt{2}\sqrt{1+\tau^{2}}\sqrt{\frac{1}{\tau_{\star}}-\tau_{\star}}e^{\frac{i\theta}{2}}}{\sqrt{e^{i\theta}-r_{e}}\sqrt{\frac{1}{r_{e}}-e^{i\theta}}} \; d\theta ,
\end{align*}
where we recall that the principal branch of the log is used. A long but direct analysis shows that the above log can be split so that $\int_{-\pi}^{\pi}e^{-2i\theta}\log R_{e}(\theta) \; d\theta = I_{1}+I_{2}+I_{3}+I_{4}$, where
\begin{align*}
& I_{1} := \log \Big(\sqrt{2}\sqrt{1+\tau^{2}}\sqrt{\frac{1}{\tau_{\star}}-\tau_{\star}}\Big) \int_{-\pi}^{\pi} e^{-i\theta}  d\theta = 0, \qquad I_{2} := \int_{-\pi}^{\pi} e^{-i\theta} \frac{i\theta}{2} = \pi, \\
& I_{3} := -\frac{1}{2}\int_{-\pi}^{\pi} e^{-i\theta}\log(e^{i\theta}-r_{e})d\theta = -\frac{1}{2}\int_{\mathbb{S}^{1}} \log(z-r_{e})\frac{dz}{iz^{2}} = -\frac{1}{2}2\pi i \int_{-\infty}^{-1} \frac{dx}{ix^{2}} = -\pi, \\
& I_{4} := - \frac{1}{2}\int_{-\pi}^{\pi}e^{-i\theta}\log(\tfrac{1}{r_{e}}-e^{i\theta})d\theta = -\frac{1}{2} \int_{\mathbb{S}^{1}} \log(\tfrac{1}{r_{e}}-z)\frac{dz}{iz^{2}} = \pi r_{e},
\end{align*}
and $\mathbb{S}^{1}$ is the unit circle oriented counterclockwise. For $I_{3}$, we have first use the change of variables $z=e^{i\theta}$. Since $\frac{\log(z-r_{e})}{z^{2}}$ is analytic outside $\mathbb{S}^{1}$ except on $(-\infty,-1]$ (recall that $r_{e}\in (0,1)$), we have then deformed $\mathbb{S}^{1}$ into $C_{R}(0)\cup [-R+i 0_{+},-1+i 0_{+}] \cup [-1-i 0_{+},-R-i 0_{+}]$, where $C_{R}(0)$ is the circle centered at $0$ of radius $R>0$ and oriented counterclockwise. The integral along $C_{R}$ vanishes as $R\to \infty$, while the integral along $[-R+i 0_{+},-1+i 0_{+}] \cup [-1-i 0_{+},-R-i 0_{+}]$ yields $-\frac{1}{2}2\pi i \int_{-\infty}^{-1} \frac{dx}{ix^{2}}$. For $I_{4}$, since $\log(\tfrac{1}{r_{e}}-z)\frac{dz}{iz^{2}}$ is analytic inside $\mathbb{S}^{1}\setminus \{0\}$, the integral can be evaluated explicitly by a simple residue computation, using that
\begin{align*}
-\frac{1}{2}\log(\tfrac{1}{r_{e}}-z) \frac{1}{iz^{2}} = \frac{\log r_{e}}{2iz^{2}} - \frac{ir_{e}}{2z} + \bigO(1), \qquad \mbox{as } z \to 0.
\end{align*}
\underline{Proof of \eqref{lol28}.} The proof is similar to the proof of \eqref{lol27}.

\noindent \underline{Proof of \eqref{lol29}.} Using \eqref{def of Ri and Re}, \eqref{branch cuts with tau} and \eqref{f 2 theta to f 1 theta}, we get
\begin{align*}
\int_{-\pi}^{\pi} R_{e}(\theta)^{2} \; d\theta & = \int_{-\pi}^{\pi} \frac{2(1+\tau^{2})(\frac{1}{\tau_{\star}}-\tau_{\star})}{e^{-i\theta}(e^{i\theta}-r_{e})(\frac{1}{r_{e}}-e^{i\theta})}d\theta  = \int_{\mathbb{S}^{1}} \frac{2(1+\tau^{2})(\frac{1}{\tau_{\star}}-\tau_{\star})}{i(z-r_{e})(\frac{1}{r_{e}}-z)}dz = 4\pi (1+\tau^{2}) \frac{\frac{1}{\tau_{\star}}-\tau_{\star}}{\frac{1}{r_{e}}-r_{e}},
\end{align*}
where the last equality follows from a residue computation. The last expression can be further simplified using \eqref{def of ri and re} and $\tau_{\star} = \frac{2\tau}{1+\tau^{2}}$, and we find \eqref{lol29}.

\noindent \underline{Proof of \eqref{lol30}.} The proof is similar to the proof of \eqref{lol29}.

\noindent \underline{Proof of \eqref{lol31}.} We only prove the first identity, the proof of the second one being similar. Note that for $n$ odd, the integrand $e^{-i\theta n}R_{e}(\theta)^{2-n}$ is not of the form $f(e^{2i\theta})$ for some integrable function $f$, so we cannot use \eqref{f 2 theta to f 1 theta} (and also not \eqref{branch cuts with tau}) to simplify the analysis. Using \eqref{def of Ri and Re}, \eqref{branch cuts with tau 2} and \eqref{branch cuts with tau 3}, we obtain
\begin{align*}
& \int_{-\pi}^{\pi} e^{-i\theta n} R_{e}(\theta)^{2-n}d\theta = \int_{-\pi}^{\pi} e^{-i\theta n} \Bigg( \frac{\sqrt{2}\sqrt{1+\tau^{2}}\sqrt{\frac{1}{\tau_{\star}}-\tau_{\star}}}{e^{-i\theta } \sqrt{e^{i\theta}-\sqrt{r_{e}}}\sqrt{e^{i\theta}+\sqrt{r_{e}}} \sqrt{\frac{1}{\sqrt{r_{e}}}-e^{i\theta}}\sqrt{e^{i\theta}+\frac{1}{\sqrt{r_{e}}}}  } \Bigg)^{2-n} d\theta \\
& = \int_{\mathbb{S}^{1}} \Bigg( \frac{\sqrt{z-\sqrt{r_{e}}}\sqrt{z+\sqrt{r_{e}}} \sqrt{\frac{1}{\sqrt{r_{e}}}-z}\sqrt{z+\frac{1}{\sqrt{r_{e}}}}}{\sqrt{2}\sqrt{1+\tau^{2}}\sqrt{\frac{1}{\tau_{\star}}-\tau_{\star}}  } \Bigg)^{n-2} \frac{dz}{i z^{2n-1}}.
\end{align*}
If $n$ is even, $n \neq 0$, the integrand is analytic on $\C\setminus \{0\}$ and is $\bigO(z^{-3})$ as $z \to \infty$. Hence, by deforming the contour towards $\infty$, we conclude that the integral vanishes. This proves \eqref{lol31} for $n$ even. For $n$ odd, the integrand is analytic on $\{z:|z|\geq 1\}\setminus \big( (-\infty,-\frac{1}{\sqrt{r_{e}}}] \cup [\frac{1}{\sqrt{r_{e}}},+\infty) \big)$. Therefore, by deforming the contour towards $\infty$, we pick up some contributions along $(-\infty,-\frac{1}{\sqrt{r_{e}}}]\pm i 0_{+}$ and $[\frac{1}{\sqrt{r_{e}}},+\infty)\pm i 0_{+}$:
\begin{align*}
& \int_{-\pi}^{\pi} e^{-i\theta n} R_{e}(\theta)^{2-n}d\theta = 2(I_{1}+I_{2}), \quad I_{1} := \int_{-\infty}^{-\frac{1}{\sqrt{e}}} \Bigg( \frac{(-1)\sqrt{x^{2}-r_{e}} \; i \sqrt{x^{2}-\frac{1}{r_{e}}}}{\sqrt{2}\sqrt{1+\tau^{2}}\sqrt{\frac{1}{\tau_{\star}}-\tau_{\star}}  } \Bigg)^{n-2} \frac{dx}{i x^{2n-1}} \\
& I_{2} := \int_{\frac{1}{\sqrt{e}}}^{+\infty} \Bigg( \frac{\sqrt{x^{2}-r_{e}} \; (-i) \sqrt{x^{2}-\frac{1}{r_{e}}}}{\sqrt{2}\sqrt{1+\tau^{2}}\sqrt{\frac{1}{\tau_{\star}}-\tau_{\star}}  } \Bigg)^{n-2} \frac{dx}{i x^{2n-1}}.
\end{align*}
Note that the integrands of $I_{1}$ and $I_{2}$ are the same, and are odd functions. Hence $I_{1}=-I_{2}$, from which \eqref{lol31} follows also for $n$ odd.
\end{proof}
\begin{lemma}\label{lemma:moments of mu ELLIPSE complement}
For $n \in \N$, we have
\begin{align}\label{inverse moments EG}
& \int_{U} z^{-n} d\mu(z) = \begin{cases}
\ds \frac{1-\tau^{2}-ac}{1-\tau^{2}}, & \mbox{if } n = 0, \\[0.25cm]
\ds \frac{r_{e}-r_{i}}{1-\tau^{2}}, & \mbox{if } n=2, \\
0, & \mbox{otherwise}.
\end{cases} 
\end{align}
\end{lemma}
\begin{proof}
For any $r\in \N$, using \eqref{param polar ellipses} we get
\begin{align*}
\int_{U} z^{-n} d\mu(z) & = \int_{-\pi}^{\pi} \int_{R_{i}(\theta)}^{R_{e}(\theta)} r^{1-n}e^{-in\theta} \frac{drd\theta}{\pi(1-\tau^{2})} \\
& = \begin{cases}
\ds \frac{1}{\pi(1-\tau^{2})}\int_{-\pi}^{\pi}e^{-i\theta n} \frac{R_{e}(\theta)^{2-n}-R_{i}(\theta)^{2-n}}{2-n}d\theta, & \mbox{if } n \neq 2, \\[0.25cm]
\ds \frac{1}{\pi(1-\tau^{2})} \int_{-\pi}^{\pi}e^{-2i\theta} \log \frac{R_{e}(\theta)}{R_{i}(\theta)} \; d\theta, & \mbox{if } n=2.
\end{cases}
\end{align*}
Then \eqref{inverse moments EG} with $n=2$ follows from \eqref{lol27}--\eqref{lol28}; \eqref{inverse moments EG} with $n=0$ follows from \eqref{lol29}--\eqref{lol30}; and \eqref{inverse moments EG} with $n\neq 0,2$ follows from \eqref{lol31}.
\end{proof}
\begin{lemma}\label{lemma:moments of nu ELLIPSE complement}
Let $c_{0},c_{1}\in \R$, and let $\hat{\nu}$ be the measure supported on $\partial U$ defined by $d\hat{\nu}(z) = \big( c_{0} + 2 c_{1} \cos (2 \theta) \big)d\theta$ for $z=a \cos \theta + i c \sin \theta$, $\theta \in (-\pi,\pi]$. For $n \in \N$, we have
\begin{align}\label{moments nu complement ellipse}
\int_{\partial U} z^{-n} d\hat{\nu}(z) = \begin{cases}
2\pi c_{0}, & \mbox{if } n = 0, \\
\ds \frac{8\pi c_{1}}{(a+c)^{2}}, & \mbox{if } n=2, \\
0, & \mbox{otherwise}.
\end{cases}
\end{align}
\end{lemma}
\begin{proof}
With $z_{0} := \frac{i\sqrt{a-c}}{\sqrt{a+c}}$, we have
\begin{align}\label{lol77}
\int_{\partial U} z^{-n} d\nu(z) = \int_{-\pi}^{\pi} \frac{c_{0}+c_{1}(e^{2i\theta}+e^{-2i\theta})}{\big( \frac{a+c}{2e^{i\theta}}(e^{i\theta}-z_{0})(e^{i\theta}+z_{0}) \big)^{n}} d\theta = \bigg( \frac{2}{a+c} \bigg)^{n}\int_{\mathbb{S}^{1}} \frac{c_{0}+c_{1}(z^{2}+z^{-2})}{(z-z_{0})^{n}(z-z_{0})^{n}} \frac{z^{n-1}}{i}dz.
\end{align}
Since $|z_{0}|<1$, the integrand is analytic in $\{z:|z|\geq 1\}$ and as $z\to \infty$ it behaves as
\begin{align}\label{lol78}
\begin{cases}
ic_{1}z-\frac{ic_{0}}{z}+\bigO(z^{-2}), & \mbox{if } n=0, \\[0.1cm]
-\frac{2i c_{1}}{a+c}+\bigO(z^{-2}), & \mbox{if } n=1,
\end{cases} \qquad\qquad \begin{cases}
-\frac{4ic_{1}}{(a+c)^{2}z}+\bigO(z^{-2}), & \mbox{if } n=2, \\[0.1cm]
\bigO(z^{-2}), & \mbox{if } n \geq 3.
\end{cases}
\end{align}
Hence, by deforming the contour of integration in the right-hand side of \eqref{lol77} towards $\infty$, and using \eqref{lol78}, we find \eqref{moments nu complement ellipse}.
\end{proof}
For $z=a \cos \theta + i c \sin \theta$, $\theta \in (-\pi,\pi]$, let $d\hat{\nu}(z) = \big( c_{0} + 2 c_{1} \cos (2 \theta) \big)d\theta$, where $c_{0}$ and $c_{1}$ are given by \eqref{def of c0 c1 complement ellipse intro}. By Lemmas \ref{lemma:moments of mu ELLIPSE complement} and \ref{lemma:moments of nu ELLIPSE complement}, we have that $\hat{\nu}$ satisfies the first equation in \eqref{eqn:moment unbounded with mass} for all $n \in \N$ (with $\nu$ replaced by $\hat{\nu}$). In order to conclude that $\hat{\nu} = \mathrm{Bal}(\mu|_{U},\partial U)$, we still have to show that $\hat{\nu}$ satisfies the second equation in \eqref{eqn:moment unbounded with mass}. This is, by far, the most complicated part of the proof of Theorem \ref{thm:EG complement ellipse nu}. We will proceed by computing $c^{\hspace{0.02cm}\mu}_{U}$, $\int_{\partial U}\log |z|  d\hat{\nu}(z)$ and $\int_{ U} \log |z| d\mu(z)$ separately. We start with $c^{\hspace{0.02cm}\mu}_{U}$ (Lemma \ref{lemma: cQ unbounded annulus} will also be used later on to prove Theorem \ref{thm:EG complement ellipse C}.)

\begin{lemma}\label{lemma: cQ unbounded annulus}
\begin{align}\label{cQ simplif}
c^{\hspace{0.02cm}\mu}_{U} = \frac{a^{2}}{4(1+\tau)} + \frac{c^{2}}{4(1-\tau)} - \frac{1}{2} + \log\bigg( \frac{2}{a+c} \bigg).
\end{align}
\end{lemma}
\begin{proof}
Recall from Proposition \ref{prop:def of bal} that $c^{\hspace{0.02cm}\mu}_{U}=2\pi\int_{U}g_{U}(z,\infty)d\mu(z)$. Let $V:=\{z:|z|>1\}$. Since $g_{V}(z,\infty) = \frac{1}{2\pi}\log |z|$ (see e.g. \cite[Section II.4]{SaTo}), and since 
\begin{align}\label{def of varphi}
\varphi(z) = \frac{z + \sqrt{z-\sqrt{a^{2}-c^{2}}} \sqrt{z+\sqrt{a^{2}-c^{2}}}}{a+c}
\end{align}
is a conformal map from $U$ to $V$ (see e.g. \cite{KS2008} and Figure \ref{fig:conformal map from exterior of ellipse}), we have
\begin{align*}
g_{U}(z,\infty) = \frac{1}{2\pi}\log |\varphi(z)|.
\end{align*}
\begin{figure}
\begin{center}
\begin{tikzpicture}[master]
\node at (0,0) {\includegraphics[width=7.35cm]{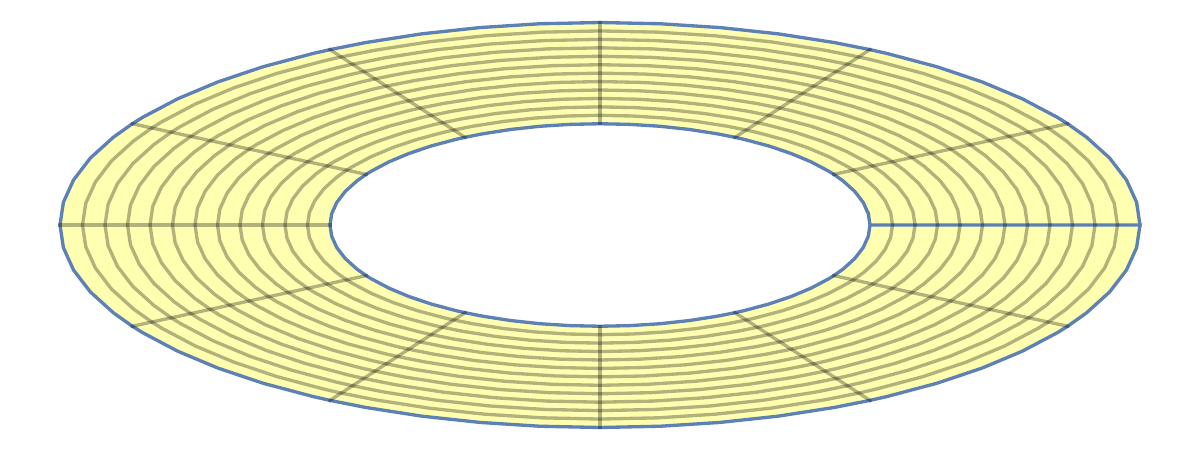}};
\draw[fill] (1.65,0) circle (0.05);
\draw[fill] (0,0.6) circle (0.05);
\node at (1.5,0) {\footnotesize $a$};
\node at (-0.02,0.45) {\footnotesize $ic$};
\end{tikzpicture}
\begin{tikzpicture}[slave]
\node at (0,0) {\includegraphics[width=7.35cm]{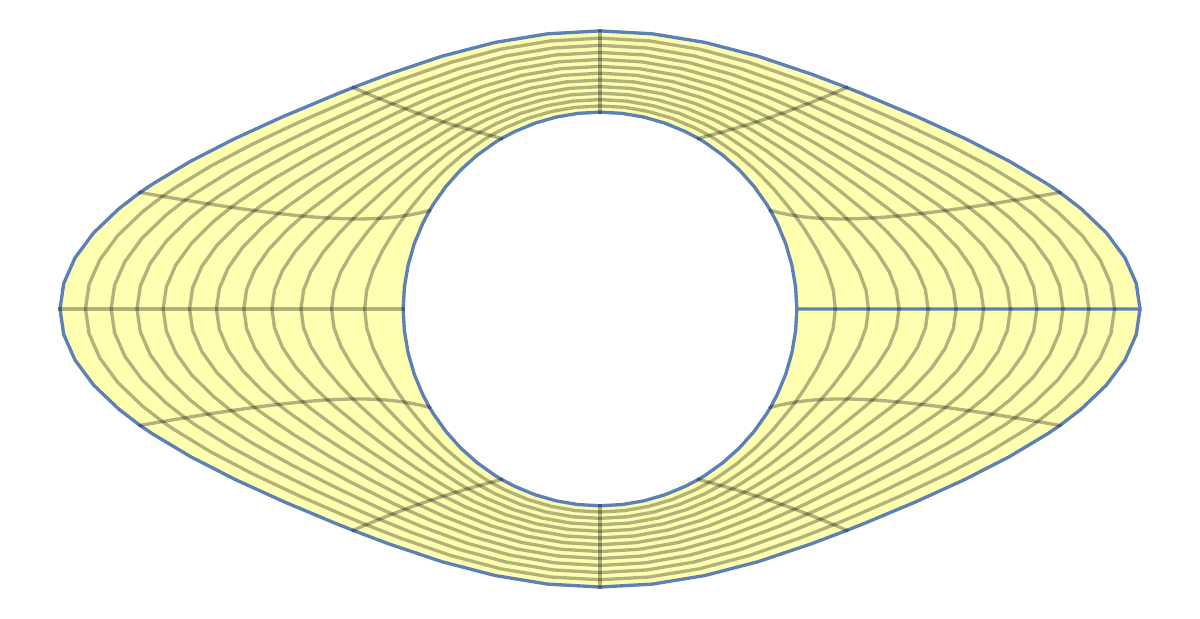}};
\node at (-4,1.4) {$\varphi$};
\coordinate (P) at ($(-4, 0.5) + (30:1cm and 0.6cm)$);
\draw[thick, black, -<-=0.45] ($(-4, 0.5) + (30:1cm and 0.6cm)$(P) arc (30:150:1cm and 0.6cm);
\draw[fill] (1.2,0) circle (0.05);
\node at (1.05,0) {\footnotesize $1$};
\draw[fill] (0,1.2) circle (0.05);
\node at (0,1) {\footnotesize $i$};
\end{tikzpicture}
\end{center}
\caption{\label{fig:conformal map from exterior of ellipse} The conformal map $\varphi$ from $U$ to the exterior of the unit disk.}
\end{figure}
Hence, by \eqref{param polar ellipses},
\begin{align*}
c^{\hspace{0.02cm}\mu}_{U} = \int_{U} \log |\varphi(z)| d\mu(z) = \frac{1}{\pi(1-\tau^{2})} \re \int_{-\pi}^{\pi} \int_{R_{i}(\theta)}^{R_{e}(\theta)} r \log \varphi(r e^{i\theta}) dr d\theta.
\end{align*}
Quite remarkably, $r\mapsto r \log \varphi(r e^{i\theta})$ admits an explicit primitive: indeed, it can be checked that
\begin{align*}
r \log \varphi(r e^{i\theta}) = \frac{d}{dr} \Bigg\{ & - \frac{e^{-i\theta}r}{4} \sqrt{r e^{i\theta} - \sqrt{a^{2}-c^{2}}}\sqrt{r e^{i\theta} + \sqrt{a^{2}-c^{2}}} - \frac{a^{2}-c^{2}}{4 \, e^{2i\theta}} \log(a+c) \\
& + \bigg( \frac{r^{2}}{2} - \frac{a^{2}-c^{2}}{4 \, e^{2i\theta}} \bigg) \log \varphi(r e^{i\theta})  \Bigg\}.
\end{align*}
Using this primitive, we obtain $c^{\hspace{0.02cm}\mu}_{U} = \frac{1}{\pi(1-\tau^{2})}\re (I_{1}+I_{2}+I_{3}+I_{4})$, where
\begin{align*}
& I_{1} = -\int_{-\pi}^{\pi} \frac{e^{-i\theta}R_{e}(\theta)}{4} \sqrt{R_{e}(\theta)e^{i\theta} - \sqrt{a^{2}-c^{2}}} \sqrt{R_{e}(\theta)e^{i\theta} + \sqrt{a^{2}-c^{2}}} \; d\theta, \\
& I_{2} = \int_{-\pi}^{\pi} \bigg( \frac{R_{e}(\theta)^{2}}{2} - \frac{a^{2}-c^{2}}{4 \, e^{2i\theta}} \bigg) \log \varphi(R_{e}(\theta)e^{i\theta}) \; d\theta, \\
& I_{3} = \int_{-\pi}^{\pi} \frac{e^{-i\theta}R_{i}(\theta)}{4} \sqrt{R_{i}(\theta)e^{i\theta} - \sqrt{a^{2}-c^{2}}} \sqrt{R_{i}(\theta)e^{i\theta} + \sqrt{a^{2}-c^{2}}} \; d\theta, \\
& I_{4} = -\int_{-\pi}^{\pi} \bigg( \frac{R_{i}(\theta)^{2}}{2} - \frac{a^{2}-c^{2}}{4 \, e^{2i\theta}} \bigg) \log \varphi(R_{i}(\theta)e^{i\theta}) \; d\theta.
\end{align*}
Since $\varphi$ maps $\partial U$ to the unit circle, $I_{3}$ and $I_{4}$ are simpler than their counterparts $I_{1}$ and $I_{2}$. We will evaluate these integrals in order of increasing difficulty: first $I_{3}$, then $I_{1}$, then $I_{4}$, and finally $I_{2}$.

\noindent \underline{Analysis of $I_{3}$.} Using \eqref{f 2 theta to f 1 theta} and \eqref{def of Ri and Re}, we get
\begin{align}
I_{3} & = \int_{-\pi}^{\pi} \frac{R_{i}(\theta)}{4} \sqrt{R_{i}(\theta)^{2}-(a^{2}-c^{2})e^{-2i\theta}}d\theta = \int_{-\pi}^{\pi} \frac{R_{i}(\frac{\theta}{2})}{4} \sqrt{R_{i}(\tfrac{\theta}{2})^{2}-(a^{2}-c^{2})e^{-i\theta}}d\theta \nonumber \\
& = \frac{1}{4} \frac{a^{2}+c^{2}}{2} \int_{-\pi}^{\pi} \frac{1-\tau_{a,c}^{2}}{1-\tau_{a,c} \cos \theta} \sqrt{1-(a^{2}-c^{2})\frac{2}{a^{2}+c^{2}}\frac{1-\tau_{a,c}\cos \theta}{1-\tau_{a,c}^{2}}e^{-i\theta}}d\theta. \label{lol32}
\end{align}
Since $\tau_{a,c} = \frac{a^{2}-c^{2}}{a^{2}+c^{2}}$, we have
\begin{align}\label{lol33}
1-\frac{2}{a^{2}+c^{2}}\frac{(a^{2}-c^{2})}{1-\tau_{a,c}^{2}}\bigg( 1-\tau_{a,c} \frac{z+z^{-1}}{2} \bigg)\frac{1}{z} = \frac{(a^{2}(z-1)+c^{2}(z+1))^{2}}{4a^{2}c^{2}z^{2}},
\end{align}
and thus, by squaring the first identity in \eqref{branch cuts with tau} and substituting it in \eqref{lol32}, we get
\begin{align*}
I_{3} & = \frac{1}{4} \frac{a^{2}+c^{2}}{2} \int_{-\pi}^{\pi} \frac{1-\tau_{a,c}^{2}}{-\frac{\tau_{a,c}}{2e^{i\theta}} (e^{i\theta}-r_{i})(e^{i\theta}-\frac{1}{r_{i}})} \frac{a^{2}(e^{i\theta}-1)+c^{2}(e^{i\theta}+1)}{2ac e^{i\theta}}d\theta \\
& = - \frac{a^{2}+c^{2}}{4} \int_{-\pi}^{\pi} \frac{\frac{1}{\tau_{a,c}}-\tau_{a,c}}{(z-r_{i})(z-\frac{1}{r_{i}})} \frac{a^{2}(z-1)+c^{2}(z+1)}{2ac}\frac{dz}{iz}.
\end{align*}
The integrand is analytic in $\{|z|\leq 1\}\setminus \{0,r_{i}\}$. Hence, by deforming the contour towards $0$, we pick up some residues at $0$ and $r_{i}$. After a computation (using again $\tau_{a,c} = \frac{a^{2}-c^{2}}{a^{2}+c^{2}}$ to simplify), we get
\begin{align} \label{I3}
I_{3} = \frac{\pi}{2}ac.
\end{align}

\noindent \underline{Analysis of $I_{1}$.} In a similar way as for \eqref{lol32}, we obtain
\begin{align*}
I_{1} = - \frac{1}{4} \int_{-\pi}^{\pi} \frac{(1+\tau_{\star}^{2})(1-\tau_{\star}^{2})}{1-\tau_{\star} \cos \theta} \sqrt{1-\frac{a^{2}-c^{2}}{(1+\tau^{2})(1-\tau_{\star}^{2})}(1-\tau_{\star} \cos \theta)e^{-i\theta}}d\theta.
\end{align*}
Define
\begin{align}\label{def of zac}
z_{a,c} = \frac{(1+\tau^{2})}{2((1-\tau^{2})^{2}+a^{2}\tau-c^{2}\tau)} \bigg( a^{2}-c^{2} + \frac{1-\tau^{2}}{1+\tau^{2}}\sqrt{(a^{2}-c^{2})(a^{2}-c^{2}-4\tau)} \bigg).
\end{align}
A long but direct computation shows that $|z_{a,c}|<1$ for all $\tau \in [0,1)$, $a\in (0,1+\tau]$ and $c\in (0,1-\tau]$. Recall that we only treat the case $a>c$ in this proof. Hence it readily follows from \eqref{def of zac} that $z_{a,c} \in \R$ only if $a^{2}-c^{2}-4\tau\geq 0$, while $z_{a,c} \notin \R$ if $a^{2}-c^{2}-4\tau< 0$. In the latter case, for definiteness we choose the branch of the square root in \eqref{def of zac} so that $z_{a,c}$ lies in the upper half-plane. Both cases $z_{a,c} \in \R$ and $z_{a,c} \notin \R$ can be handled similarly, but the situation when $z_{a,c} \notin \R$ requires more care with branch cuts, and here we will only treat the case $z_{a,c} \notin \R$. For $z_{a,c} \notin \R$, we can rewrite $I_{1}$ as
\begin{align*}
I_{1} & = \int_{-\pi}^{\pi} \frac{1}{2} \frac{(1+\tau^{2})(\frac{1}{\tau_{\star}}-\tau_{\star})}{(e^{i\theta}-r_{e})(e^{i\theta}-\frac{1}{r_{e}})} \sqrt{1+\frac{\tau(a^{2}-c^{2})}{(1-\tau^{2})^{2}}} (e^{i\theta}-z_{a,c})^{1/2}_{-\frac{\pi}{2}} (e^{i\theta}-\overline{z_{a,c}})^{1/2}_{-\frac{\pi}{2}} d\theta \\
& = \int_{\mathbb{S}^{1}} \frac{1}{2} \frac{(1+\tau^{2})(\frac{1}{\tau_{\star}}-\tau_{\star})}{(z-r_{e})(z-\frac{1}{r_{e}})} \sqrt{1+\frac{\tau(a^{2}-c^{2})}{(1-\tau^{2})^{2}}} (z-z_{a,c})^{1/2}_{-\frac{\pi}{2}} (z-\overline{z_{a,c}})^{1/2}_{-\frac{\pi}{2}} \frac{dz}{iz},
\end{align*}
where $(z)^{1/2}_{-\frac{\pi}{2}} := \sqrt{|z|}e^{\frac{i}{2}\arg_{-\frac{\pi}{2}}(z)}$ and $\arg_{-\frac{\pi}{2}}(z) \in [-\frac{\pi}{2},\frac{3\pi}{2})$. It is easy to check that the integrand is $\bigO(z^{-2})$ as $z \to \infty$. Furthermore, since $|z_{a,c}|<1$, the integrand is analytic in $\{z:|z|\geq 1\}\setminus \{\frac{1}{r_{e}}\}$. Hence, by deforming the contour towards $\infty$, we only pick up a residue at $\frac{1}{r_{e}}$:
\begin{align}
I_{1} & = \frac{1}{2} \frac{(1+\tau^{2})(\frac{1}{\tau_{\star}}-\tau_{\star})}{\frac{1}{r_{e}}-r_{e}} \sqrt{1+\frac{\tau(a^{2}-c^{2})}{(1-\tau^{2})^{2}}} (\tfrac{1}{r_{e}}-z_{a,c})^{1/2}_{-\frac{\pi}{2}} (\tfrac{1}{r_{e}}-\overline{z_{a,c}})^{1/2}_{-\frac{\pi}{2}} \frac{1}{i \frac{1}{r_{e}}}(-2\pi i) \nonumber \\
& = -\pi \frac{(1+\tau^{2})(\frac{1}{\tau_{\star}}-\tau_{\star})}{\frac{1}{r_{e}}-r_{e}} \sqrt{1+\frac{\tau(a^{2}-c^{2})}{(1-\tau^{2})^{2}}} r_{e} \sqrt{(\tfrac{1}{r_{e}}-z_{a,c})(\tfrac{1}{r_{e}}-\overline{z_{a,c}})} = - \frac{\pi}{2}(1-\tau^{2}), \label{I1}
\end{align}
where for the last equality we have used \eqref{def of ri and re}. If $z_{a,c}\in \R$, using a similar analysis we also obtain $I_{1} = - \frac{\pi}{2}(1-\tau^{2})$.

\noindent \underline{Analysis of $I_{4}$.} Using \eqref{def of varphi}, we get
\begin{align*}
I_{4} & = - \int_{-\pi}^{\pi} \bigg( \frac{R_{i}(\theta)^{2}}{2} - \frac{a^{2}-c^{2}}{4e^{2i\theta}} \bigg) \log \bigg( e^{i\theta} \frac{R_{i}(\theta)+\sqrt{R_{i}(\theta)^{2}-(a^{2}-c^{2})e^{-2i\theta}}}{a+c} \bigg) \; d\theta \\
& = - \int_{-\pi}^{\pi} \bigg( \frac{R_{i}(\theta)^{2}}{2} - \frac{a^{2}-c^{2}}{4e^{2i\theta}} \bigg) \bigg[ i \theta + \log \bigg(  \frac{R_{i}(\theta)+\sqrt{R_{i}(\theta)^{2}-(a^{2}-c^{2})e^{-2i\theta}}}{a+c} \bigg) \bigg] \; d\theta.
\end{align*}
Since $R_{i}(\theta)=R_{i}(-\theta)$, $\int_{-\pi}^{\pi} \theta R_{i}(\theta)^{2}d\theta = 0$. Also, by a direct computation, $\int_{-\pi}^{\pi} \frac{i \theta d\theta}{e^{2i \theta}} = -\pi$. Thus, using also \eqref{f 2 theta to f 1 theta}, \eqref{def of Ri and Re} and \eqref{lol33}, we obtain
\begin{align}
I_{4} & = -\frac{a^{2}-c^{2}}{4}\pi - \int_{-\pi}^{\pi} \bigg( \frac{R_{i}(\frac{\theta}{2})^{2}}{2}-\frac{a^{2}-c^{2}}{4e^{i\theta}} \bigg) \log \bigg(  \frac{R_{i}(\frac{\theta}{2})+\sqrt{R_{i}(\frac{\theta}{2})^{2}-(a^{2}-c^{2})e^{-i\theta}}}{a+c} \bigg) d\theta \nonumber \\
& = -\frac{a^{2}-c^{2}}{4}\pi - \int_{-\pi}^{\pi} \bigg( \frac{a^{2}+c^{2}}{2}\frac{1-\tau_{a,c}^{2}}{2(1-\tau_{a,c}\cos \theta)}-\frac{a^{2}-c^{2}}{4e^{i\theta}} \bigg) \nonumber \\
& \times  \log \bigg( \sqrt{\frac{a^{2}+c^{2}}{2}} \frac{\sqrt{1-\tau_{a,c}^{2}}}{\sqrt{1-\tau_{a,c}\cos\theta}} \frac{1+ \frac{a^{2}(e^{i\theta}-1)+c^{2}(e^{i\theta}+1)}{2ac e^{i\theta}}}{a+c} \bigg) d\theta. \label{lol34}
\end{align}
Using \eqref{branch cuts with tau} and
\begin{align*}
1+ \frac{a^{2}(z-1)+c^{2}(z+1)}{2ac z} = \bigg( 1+\frac{a}{2c}+\frac{c}{2a} \bigg) \frac{1}{z} (z-w_{a,c}), \qquad w_{a,c}:= \frac{a^{2}-c^{2}}{(a+c)^{2}}\in [0,1),
\end{align*}
we then find
\begin{align}
& I_{4} = -\frac{a^{2}-c^{2}}{4}\pi - \int_{-\pi}^{\pi} \bigg( \frac{a^{2}+c^{2}}{2} \frac{1-\tau_{a,c}^{2}}{2 \frac{-\tau_{a,c}}{2e^{i\theta}}(e^{i\theta}-r_{i})(e^{i\theta}-\frac{1}{r_{i}})} - \frac{a^{2}-c^{2}}{4e^{i\theta}} \bigg) \nonumber \\
& \times \log \bigg( \sqrt{\frac{a^{2}+c^{2}}{2}} \frac{e^{i \frac{\theta}{2}}\sqrt{1-\tau_{a,c}^{2}}}{\sqrt{\frac{\tau_{a,c}}{2}}\sqrt{e^{i\theta}-r_{i}}\sqrt{\frac{1}{r_{i}}-e^{i\theta}}} \frac{(1+\frac{a}{2c}+\frac{c}{2a})\frac{1}{e^{i\theta}}(e^{i\theta}-w_{a,c})}{a+c} \bigg) d\theta \nonumber \\
& = \int_{-\pi}^{\pi} \bigg( \frac{a^{2}+c^{2}}{2} \frac{(\frac{1}{\tau_{a,c}}-\tau_{a,c})e^{i\theta}}{(e^{i\theta}-r_{i})(e^{i\theta}-\frac{1}{r_{i}})} + \frac{a^{2}-c^{2}}{4e^{i\theta}} \bigg) \nonumber \\
& \times \log \bigg( \sqrt{\frac{a^{2}+c^{2}}{2}} \frac{\sqrt{1-\tau_{a,c}^{2}}}{\sqrt{\frac{\tau_{a,c}}{2}}\sqrt{e^{i\theta}-r_{i}}\sqrt{\frac{1}{r_{i}}-e^{i\theta}}} \frac{(1+\frac{a}{2c}+\frac{c}{2a})\frac{1}{e^{i\theta}}(e^{i\theta}-w_{a,c})}{a+c} \bigg) d\theta. \label{lol80}
\end{align}
A long but direct analysis shows that, for all $\theta \in (-\pi,\pi]$,
\begin{align}
& \log \bigg( \sqrt{\frac{a^{2}+c^{2}}{2}} \frac{\sqrt{1-\tau_{a,c}^{2}}}{\sqrt{\frac{\tau_{a,c}}{2}}\sqrt{e^{i\theta}-r_{i}}\sqrt{\frac{1}{r_{i}}-e^{i\theta}}} \frac{(1+\frac{a}{2c}+\frac{c}{2a})\frac{1}{e^{i\theta}}(e^{i\theta}-w_{a,c})}{a+c} \bigg) = \log(e^{i\theta}-w_{a,c})-\log(e^{i\theta}) \nonumber \\
& + \log \bigg( \sqrt{a^{2}+c^{2}} \sqrt{\frac{1}{\smash{\tau_{a,c}}}-\tau_{a,c}} \frac{1+\frac{a}{2c}+\frac{c}{2a}}{a+c} \bigg) - \frac{1}{2} \big( \log(e^{i\theta}-r_{i})+\log(\tfrac{1}{r_{i}}-e^{i\theta}) \big), \label{lol79}
\end{align}
where the principal branch is used for the logarithms. After substituting \eqref{lol79} in \eqref{lol80}, and then changing variables $z=e^{i\theta}$, we obtain that $I_{4} = I_{4,1} + I_{4,2} + I_{4,3}$, where
\begin{align*}
& I_{4,1} = \log \bigg( \sqrt{a^{2}+c^{2}} \sqrt{\frac{1}{\smash{\tau_{a,c}}}-\tau_{a,c}} \frac{1+\frac{a}{2c}+\frac{c}{2a}}{a+c} \bigg) \int_{\mathbb{S}^{1}} \bigg( \frac{a^{2}+c^{2}}{2} \frac{(\frac{1}{\tau_{a,c}}-\tau_{a,c})z}{(z-r_{i})(z-\frac{1}{r_{i}})} + \frac{a^{2}-c^{2}}{4z} \bigg)\frac{dz}{iz}, \\
& I_{4,2} = \int_{\mathbb{S}^{1}} \bigg( \frac{a^{2}+c^{2}}{2} \frac{(\frac{1}{\tau_{a,c}}-\tau_{a,c})z}{(z-r_{i})(z-\frac{1}{r_{i}})} + \frac{a^{2}-c^{2}}{4z} \bigg) \frac{-1}{2} \big( \log(z-r_{i})+\log(\tfrac{1}{r_{i}}-z) \big)\frac{dz}{iz}, \\
& I_{4,3} = \int_{\mathbb{S}^{1}} \bigg( \frac{a^{2}+c^{2}}{2} \frac{(\frac{1}{\tau_{a,c}}-\tau_{a,c})z}{(z-r_{i})(z-\frac{1}{r_{i}})} + \frac{a^{2}-c^{2}}{4z} \bigg) \big( \log(z-w_{a,c})-\log(z) \big)\frac{dz}{iz}.
\end{align*}
The integrand of $I_{1}$ is analytic in $\C \setminus \{0,r_{i},\frac{1}{r_{i}}\}$, and has no residue at $0$. Hence, by deforming the contour towards $0$, we pick up a residue at $r_{i}$ and obtain
\begin{align}\label{I41}
I_{4,1} = -(a^{2}+c^{2})\pi \frac{\frac{1}{\tau_{a,c}}-\tau_{a,c}}{\frac{1}{r_{i}-r_{i}}} \log \bigg( \sqrt{a^{2}+c^{2}} \sqrt{\frac{1}{\smash{\tau_{a,c}}}-\tau_{a,c}} \frac{1+\frac{a}{2c}+\frac{c}{2a}}{a+c} \bigg) = - a c \pi \log \frac{a+c}{\sqrt{a^{2}-c^{2}}},
\end{align}
where for the last equality we have used \eqref{def of ri and re} and $\tau_{a,c} = \frac{a^{2}-c^{2}}{a^{2}+c^{2}}$. The integrand of $I_{4,3}$ is analytic in $\{z:|z|\geq 1\}\setminus \{\frac{1}{r_{i}}\}$ and is $\bigO(z^{-3})$ as $z\to \infty$. Hence, by deforming the contour towards $\infty$, we obtain
\begin{align}\label{I43}
I_{4,3} = - a c \pi \log(1-r_{i}w_{a,c}) = - a c \pi \log \frac{4ac}{(a+c)^{2}}.
\end{align}
It only remains to evaluate $I_{4,2}$. For this, we write $I_{4,2} = I_{4,2}^{(1)}+I_{4,2}^{(2)}+I_{4,3}^{(3)}$, where
\begin{align}
& I_{4,2}^{(1)} = -\int_{\mathbb{S}^{1}} \frac{a^{2}-c^{2}}{4z^{2}}\big( \log(z-r_{i})+\log(\tfrac{1}{r_{i}}-z) \big) \frac{dz}{2i}, \nonumber \\
& I_{4,2}^{(2)} = \int_{\mathbb{S}^{1}} \frac{a^{2}+c^{2}}{2} \frac{-(\frac{1}{\tau_{a,c}}-\tau_{a,c})}{(z-r_{i})(z-\frac{1}{r_{i}})} \log(z-r_{i})\frac{dz}{2i}, \;\;\; I_{4,2}^{(3)} = \int_{\mathbb{S}^{1}} \frac{a^{2}+c^{2}}{2} \frac{-(\frac{1}{\tau_{a,c}}-\tau_{a,c})}{(z-r_{i})(z-\frac{1}{r_{i}})} \log(\tfrac{1}{r_{i}}-z)\frac{dz}{2i}. \label{def of I42p2p and I42p3p}
\end{align}
For $I_{4,2}^{(1)}$, we deform the contour towards $\infty$, and we pick up a contribution from the branch cuts along $(-\infty,-1)$ and $(\frac{1}{r_{i}},+\infty)$:
\begin{align}\label{I42p1p}
I_{4,2}^{(1)} = \int_{-\infty}^{-1} \frac{a^{2}-c^{2}}{4x^{2}}2\pi i \frac{-dx}{2i} + \int_{\frac{1}{r_{i}}}^{+\infty} \frac{a^{2}-c^{2}}{4x^{2}}(-2\pi i) \frac{-dx}{2i} = -\frac{a^{2}-c^{2}}{4}\pi (1-r_{i}).
\end{align}
The integrand of $I_{4,2}^{(3)}$ is analytic in $\{z:|z|\leq 1\}\setminus \{r_{i}\}$, so by deforming the contour towards $0$ we get
\begin{align}\label{I42p3p}
I_{4,2}^{(3)} = \frac{a^{2}+c^{2}}{2}\pi \frac{\frac{1}{\tau_{a,c}}-\tau_{a,c}}{\frac{1}{r_{i}}-r_{i}} \log(\tfrac{1}{r_{i}}-r_{i}).
\end{align}
For $I_{4,2}^{(2)}$, by deforming the contour towards $\infty$, we pick up a contribution along the branch cut $(-\infty,-1)$, and also a residue at $\frac{1}{r_{i}}$:
\begin{align}
I_{4,2}^{(2)} & = \frac{a^{2}+c^{2}}{2}\pi \frac{\frac{1}{\tau_{a,c}}-\tau_{a,c}}{\frac{1}{r_{i}}-r_{i}} \log(\tfrac{1}{r_{i}}-r_{i}) + \int_{-\infty}^{-1} \frac{a^{2}+c^{2}}{2} \frac{\frac{1}{\tau_{a,c}}-\tau_{a,c}}{(x-r_{i})(x-\frac{1}{r_{i}})} 2\pi i \frac{-1}{2i}dx \nonumber \\
& = \frac{a^{2}+c^{2}}{2}\pi \frac{\frac{1}{\tau_{a,c}}-\tau_{a,c}}{\frac{1}{r_{i}}-r_{i}} \log(\tfrac{1}{r_{i}}-r_{i}) + \frac{a^{2}+c^{2}}{2}\pi \frac{\frac{1}{\tau_{a,c}}-\tau_{a,c}}{\frac{1}{r_{i}}-r_{i}} \log(r_{i}). \label{I42p2p}
\end{align}
Substituting \eqref{I42p1p}, \eqref{I42p3p} and \eqref{I42p2p} in $I_{4,2} = I_{4,2}^{(1)}+I_{4,2}^{(2)}+I_{4,3}^{(3)}$, and simplifying using \eqref{def of ri and re} and $\tau_{a,c} = \frac{a^{2}-c^{2}}{a^{2}+c^{2}}$, we get
\begin{align}\label{I42}
I_{4,2} = -\frac{\pi}{2}(a-c)c + \pi \frac{ac}{2} \log \bigg( \frac{16a^{2}c^{2}}{(a-c)(a+c)^{3}} \bigg).
\end{align}
Substituting now \eqref{I41}, \eqref{I43} and \eqref{I42} in $I_{4} = I_{4,1} + I_{4,2} + I_{4,3}$, we get
\begin{align}\label{I4}
I_{4} = -\frac{(a-c)c\pi}{2}.
\end{align}

\noindent \underline{Analysis of $I_{2}$.} Recall that $z_{a,c}$ is defined in \eqref{def of zac} and satisfies $|z_{a,c}|<1$. Just as in the analysis of $I_{1}$, we will only treat here the case where $z_{a,c}\notin \R$ (the case $z_{a,c}\in \R$ is simpler). In a similar way as for \eqref{lol34}, we get
\begin{align}
I_{2} & = \frac{a^{2}-c^{2}}{4}\pi + \int_{-\pi}^{\pi} \bigg( \frac{R_{e}(\frac{\theta}{2})^{2}}{2}-\frac{a^{2}-c^{2}}{4e^{i\theta}} \bigg) \log \bigg(  \frac{R_{e}(\frac{\theta}{2})+\sqrt{R_{e}(\frac{\theta}{2})^{2}-(a^{2}-c^{2})e^{-i\theta}}}{a+c} \bigg) d\theta \nonumber \\
& = \frac{a^{2}-c^{2}}{4}\pi + \int_{-\pi}^{\pi} \bigg( \frac{(1+\tau^{2})(1-\tau_{\star}^{2})}{2(1-\tau_{\star}\cos \theta)}-\frac{a^{2}-c^{2}}{4e^{i\theta}} \bigg) \nonumber \\
& \times  \log \bigg( \frac{\sqrt{1+\tau^{2}}\sqrt{1-\tau_{\star}^{2}}}{\sqrt{1-\tau_{\star}\cos\theta}} \frac{1+ \sqrt{1+\frac{\tau(a^{2}-c^{2})}{(1-\tau^{2})^{2}}} \frac{1}{e^{i\theta}} (e^{i\theta}-z_{a,c})_{-\frac{\pi}{2}}^{1/2}(e^{i\theta}-\overline{z_{a,c}})_{-\frac{\pi}{2}}^{1/2} }{a+c} \bigg) d\theta, \label{lol35}
\end{align}
where we recall that $(z)^{1/2}_{-\frac{\pi}{2}} := \sqrt{|z|}e^{\frac{i}{2}\arg_{-\frac{\pi}{2}}(z)}$ and $\arg_{-\frac{\pi}{2}}(z) \in [-\frac{\pi}{2},\frac{3\pi}{2})$. Following the analysis of $I_{4}$, we split the above $\log$ in three parts, and then make the change of variables $z=e^{i\theta}$. This yields $I_{2} = I_{2,1} + I_{2,2} + I_{2,3}$, where
\begin{align*}
& I_{2,1} = -\log \bigg( \frac{\sqrt{2}\sqrt{1+\tau^{2}}\sqrt{\frac{1}{\tau_{\star}}-\tau_{\star}}}{a+c} \bigg) \int_{\mathbb{S}^{1}} \bigg( \frac{(1+\tau^{2})(\frac{1}{\tau_{\star}}-\tau_{\star})}{(z-r_{e})(z-\frac{1}{r_{e}})} + \frac{a^{2}-c^{2}}{4z^{2}} \bigg)\frac{dz}{i}, \\
& I_{2,2} = \int_{\mathbb{S}^{1}} \bigg( \frac{(1+\tau^{2})(\frac{1}{\tau_{\star}}-\tau_{\star})}{(z-r_{e})(z-\frac{1}{r_{e}})} + \frac{a^{2}-c^{2}}{4z^{2}} \bigg) \frac{1}{2} \big( \log(z-r_{e})+\log(\tfrac{1}{r_{e}}-z) \big)\frac{dz}{i}, \\
& I_{2,3} = \int_{\mathbb{S}^{1}} \bigg( \frac{(1+\tau^{2})(\frac{1}{\tau_{\star}}-\tau_{\star})}{(z-r_{e})(z-\frac{1}{r_{e}})} + \frac{a^{2}-c^{2}}{4z^{2}} \bigg) \log \tilde{\varphi}(z)\frac{dz}{i},
\end{align*}
where
\begin{align*}
\tilde{\varphi}(z) := 1+ \sqrt{1+\frac{\tau(a^{2}-c^{2})}{(1-\tau^{2})^{2}}} \frac{1}{z} (z-z_{a,c})_{-\frac{\pi}{2}}^{1/2}(z-\overline{z_{a,c}})_{-\frac{\pi}{2}}^{1/2}.
\end{align*}
Note that in \eqref{lol35} there were several possible splittings of the log at our disposal (before making the change of variables $z=e^{i\theta}$). The important advantage of the splitting $I_{2} = I_{2,1} + I_{2,2} + I_{2,3}$ is that $I_{2,3}$ involves $\log \tilde{\varphi}$, and that $\log \tilde{\varphi}$ is analytic in $\{z:|z|\geq 1\}$. Indeed, a direct analysis shows that
\begin{align*}
\tilde{\varphi}(re^{i\theta}) = 1+\sqrt{1+\frac{a^{2}-c^{2}}{\tau(\frac{1}{\tau}-\tau)^{2}re^{i\theta}}\Big(re^{i\theta} + \frac{1}{re^{i\theta}} - (\tau^{-1}+\tau)\Big)}, \qquad \mbox{for all } r\geq 1, \; \theta \in (-\pi,\pi],
\end{align*}
where the principal branch is used for the square root. In particular this shows that $\re \tilde{\varphi}(z) \geq 1$ for all $z \in \{z\in  \C:|z|\geq 1\}$, and thus that $\log \tilde{\varphi}$ is analytic in $\{z:|z|\geq 1\}$. The analysis of $I_{2,1}$ is similar to that of $I_{4,1}$ and we find
\begin{align*}
I_{2,1} = \pi (1-\tau^{2}) \log \bigg( \frac{1-\tau^{2}}{(a+c)\sqrt{\tau}} \bigg).
\end{align*}
For $I_{2,2}$, in a similar way as for $I_{4,2}$ we find 
\begin{align*}
I_{2,2} = \frac{\pi}{4}(a^{2}-c^{2})(1-r_{e}) - \pi \frac{1-\tau^{2}}{2} \log \frac{(1-r_{e}^{2})^{2}}{r_{e}}.
\end{align*}
For $I_{2,3}$, since $\tilde{\varphi}$ is analytic in $\{z:|z|\geq 1\}$, and since the integrand is $\bigO(z^{-2})$ as $z\to \infty$, by deforming the contour towards $\infty$ we get
\begin{align*}
I_{2,3} & = \frac{(1+\tau^{2})(\frac{1}{\tau_{\star}}-\tau_{\star})}{-(\frac{1}{r_{e}}-r_{e})} (\log \tilde{\varphi}(\tfrac{1}{r_{e}})) \; \frac{1}{i}(-2\pi i) \\
& = 2\pi \frac{(1+\tau^{2})(\frac{1}{\tau_{\star}}-\tau_{\star})}{\frac{1}{r_{e}}-r_{e}} \log \bigg( 1+ \sqrt{1+\frac{\tau(a^{2}-c^{2})}{(1-\tau^{2})^{2}}} r_{e} \sqrt{(\tfrac{1}{r_{e}}-z_{a,c})(\tfrac{1}{r_{e}}-\overline{z_{a,c}})} \bigg) = \pi (1-\tau^{2}) \log 2.
\end{align*}
By substituting the above expressions in $I_{2} = I_{2,1} + I_{2,2} + I_{2,3}$, and simplifying, we obtain
\begin{align}\label{I2}
I_{2} = \frac{\pi}{4}(a^{2}-c^{2})(1-\tau)+\pi(1-\tau^{2})\log \bigg( \frac{2}{a+c} \bigg).
\end{align}
Finally, substituting \eqref{I3}, \eqref{I1}, \eqref{I4} and \eqref{I2} in $c^{\hspace{0.02cm}\mu}_{U} = \frac{1}{\pi(1-\tau^{2})}\re (I_{1}+I_{2}+I_{3}+I_{4})$, and simplifying, we get \eqref{cQ simplif}.
\end{proof}
Our next goal is to find explicit expressions for $\int_{\partial U}\log |z|  d\hat{\nu}(z)$ and $\int_{ U} \log |z| d\mu(z)$.
\begin{lemma}\label{lemma:int log |z|  dnu(z)}
Let $c_{0},c_{1} \in \R$, and let $\hat{\nu}$ be the measure supported on $\partial U$ defined by $d\nu(z) = \big( c_{0} + 2 c_{1} \cos (2 \theta) \big)d\theta$ for $z=a \cos \theta + i c \sin \theta$, $\theta \in (-\pi,\pi]$. Then
\begin{align}\label{lol37}
\int_{\partial U}\log |z| \, d\hat{\nu}(z) = 2\pi \bigg( c_{0} \log \frac{a+c}{2} + c_{1} \frac{a-c}{a+c} \bigg).
\end{align}
\end{lemma}
\begin{proof}
Let $r_{0} := \frac{a-c}{a+c}\in (0,1)$. Using the definition of $\hat{\nu}$ and \eqref{f 2 theta to f 1 theta}, we obtain
\begin{align}
& \int_{\partial U}\log |z| \; d\hat{\nu}(z) = \frac{1}{2} \int_{-\pi}^{\pi} \log\big(a^{2}(\cos \theta)^{2}+c^{2}(\sin \theta)^{2}\big) \; (c_{0}+2c_{1}\cos (2\theta))d\theta \nonumber \\
& = \frac{1}{2}\int_{-\pi}^{\pi} \log \bigg( \frac{a^{2}-c^{2}}{4e^{2i\theta}}(e^{2i\theta}+r_{0})(e^{2i\theta}+\tfrac{1}{r_{0}}) \bigg)  \; (c_{0}+2c_{1}\cos (2\theta))d\theta \nonumber \\
& = \frac{1}{2}\int_{-\pi}^{\pi} \log \bigg( \frac{a^{2}-c^{2}}{4e^{i\theta}}(e^{i\theta}+r_{0})(e^{i\theta}+\tfrac{1}{r_{0}}) \bigg) \; (c_{0}+2c_{1}\cos \theta)d\theta = I_{1}+I_{2}+I_{3}, \label{lol36}
\end{align}
where
\begin{align*}
& I_{1} =  \frac{1}{2} \log  \frac{a^{2}-c^{2}}{4} \; \int_{-\pi}^{\pi}   \big( c_{0} + c_{1}(e^{i\theta}+e^{-i\theta}) \big)d\theta   = \pi c_{0} \log  \frac{a^{2}-c^{2}}{4}, \\
& I_{2} = \frac{1}{2} \int_{-\pi}^{\pi} (-i\theta)  \big( c_{0} + c_{1}(e^{i\theta}+e^{-i\theta}) \big)d\theta = 0, \\
& I_{3} = \frac{1}{2} \int_{-\pi}^{\pi} \big( \log(e^{i\theta}+r_{0}) + \log(e^{i\theta}+\tfrac{1}{r_{0}})\big) \big( c_{0} + c_{1}(e^{i\theta}+e^{-i\theta}) \big)d\theta.
\end{align*}
For $I_{3}$, after the change of variables $z=e^{i\theta}$, we deform the contour towards $0$. In this way we pick up a branch cut contribution along $(-1,-r_{0})$, and a residue contribution at $0$:
\begin{align*}
I_{3} = -\frac{1}{2}\int_{-1}^{-r_{0}} 2\pi i \bigg( c_{0} + c_{1} \big( x+\frac{1}{x} \big) \bigg) \frac{dx}{ix} - \frac{i c_{1} (\frac{1}{r_{0}}+r_{0})}{2}2\pi i = \pi (2c_{1}r_{0}-c_{0}\log r_{0}).
\end{align*}
Substituting the above expressions in \eqref{lol36}, we obtain \eqref{lol37}.
\end{proof}

\begin{lemma}\label{lemma: int Re^2 log Re dtheta}
The following relations hold:
\begin{align}
& \int_{-\pi}^{\pi} R_{e}(\theta)^{2}\log R_{e}(\theta) d\theta = 0, \label{lol38} \\
& \int_{-\pi}^{\pi} R_{i}(\theta)^{2}\log R_{i}(\theta) d\theta = 2\pi ac \log \frac{a+c}{2}. \label{lol39}
\end{align}
\end{lemma}
\begin{proof}
We only do the proof for \eqref{lol39} (the proof for \eqref{lol38} is similar). Using \eqref{def of Ri and Re}, \eqref{branch cuts with tau}, \eqref{f 2 theta to f 1 theta}, $\int_{-\frac{\pi}{2}}^{\frac{\pi}{2}} \theta R_{i}(\frac{\theta}{2})d\theta=0$, and the change of variables $z=e^{i\theta}$, we get
\begin{align*}
\int_{-\pi}^{\pi} R_{i}(\theta)^{2}\log R_{i}(\theta) d\theta = \frac{a^{2}+c^{2}}{2} \log \Big( (a^{2}+c^{2}) \big( \tfrac{1}{\tau_{a,c}}-\tau_{a,c} \big) \Big) \int_{\mathbb{S}^{1}} \frac{-(\frac{1}{\tau_{a,c}}-\tau_{a,c})}{(z-r_{i})(z-\frac{1}{r_{i}})}\frac{dz}{i} - 2 I_{4,2}^{(2)} - I_{4,2}^{(3)},
\end{align*}
where $I_{4,2}^{(2)}, I_{4,2}^{(3)}$ are as in \eqref{def of I42p2p and I42p3p}. The first integral above can be evaluated by a simple residue computation, while explicit expressions for $I_{4,2}^{(2)}, I_{4,2}^{(3)}$ have already been obtained in \eqref{I42p3p} and \eqref{I42p2p}. Using then \eqref{def of ri and re} and $\tau_{a,c} = \frac{a^{2}-c^{2}}{a^{2}+c^{2}}$ to simplify, we obtain \eqref{lol39}.
\end{proof}

\begin{lemma}\label{lemma:int log|z| dmu}
\begin{align}\label{int log|z| dmu}
\int_{ U} \log |z| \, d\mu(z) = - \frac{1-\tau^{2} - a c +2 a c \log \frac{a+b}{2}}{2(1-\tau^{2})}.
\end{align}
\end{lemma}
\begin{proof}
Since $\frac{d}{dr}(\frac{r^{2}\log r}{2} - \frac{r^{2}}{4}) = r \log r$, we have
\begin{align*}
& \int_{U} \log |z| d\mu(z) = \int_{-\pi}^{\pi}\int_{R_{i}(\theta)}^{R_{e}(\theta)} r \log r  \; \frac{1-\tau^{2}}{\pi}drd\theta \\
& = \frac{1-\tau^{2}}{\pi} \int_{-\pi}^{\pi} \bigg( \frac{R_{e}(\theta)^{2}\log R_{e}(\theta)}{2} - \frac{R_{i}(\theta)^{2}\log R_{i}(\theta)}{2} - \frac{R_{e}(\theta)^{2}}{4} + \frac{R_{i}(\theta)^{2}}{4} \bigg) d\theta.
\end{align*}
Substituting \eqref{lol29}, \eqref{lol30}, \eqref{lol38} and \eqref{lol39}, we find \eqref{int log|z| dmu}.
\end{proof}

Let $\hat{\nu}$ be as in Lemma \ref{lemma:int log |z|  dnu(z)} with $c_{0},c_{1}$ are given by \eqref{def of c0 c1 complement ellipse intro}. It is easy to check using Lemmas \ref{lemma: cQ unbounded annulus}, \ref{lemma:int log |z|  dnu(z)} and \ref{lemma:int log|z| dmu} that $\hat{\nu}$ satisfies $\int_{\partial U}\log |z|  d\hat{\nu}(z) = \int_{ U} \log |z| d\mu(z) - c^{\hspace{0.02cm}\mu}_{U}$. Moreover, it follows from Lemmas \ref{lemma:moments of mu ELLIPSE complement} and \ref{lemma:moments of nu ELLIPSE complement} that $\int_{\partial U}z^{-n} d\hat{\nu}(z) = \int_{ U}z^{-n}d\mu(z)$ holds for all $n\in \N$. In particular, $\hat{\nu}$ satisfies \eqref{eqn:moment unbounded with mass}. By Lemma \ref{lemma: moment when U is unbounded} and Remark \ref{remark: pmu bounded on dU}, $\hat{\nu}=\mathrm{Bal}(\mu|_{U},\partial U)$, which concludes the proof of Theorem \ref{thm:EG complement ellipse nu}.

\subsection{The constant $C$: proof of Theorem \ref{thm:EG complement ellipse C}}

\begin{lemma}\label{lemma: int (1-cos)R4 dtheta}
The following relations hold:
\begin{align}
& \int_{-\pi}^{\pi} \big(1-\tau \cos (2\theta)\big) R_{e}(\theta)^{4}d\theta = 2\pi (1-\tau^{2})^{2}, \label{lol40} \\
& \int_{-\pi}^{\pi} \big(1-\tau \cos (2\theta)\big) R_{i}(\theta)^{4}d\theta = a c \pi \big( a^{2}(1-\tau) + c^{2}(1+\tau) \big). \label{lol41}
\end{align}
\end{lemma}
\begin{proof}
We only do the proof for \eqref{lol41} (the proof for \eqref{lol40} is identical upon replacing $a$ and $c$ by $1+\tau$ and $1-\tau$, respectively). By \eqref{def of Ri and Re}, \eqref{f 2 theta to f 1 theta} and \eqref{branch cuts with tau},
\begin{align*}
& \int_{-\pi}^{\pi} \big(1-\tau \cos (2\theta)\big) R_{i}(\theta)^{4}d\theta = \int_{-\pi}^{\pi} \big(1-\tau \cos (\theta)\big) R_{i}(\tfrac{\theta}{2})^{4}d\theta \\
& = \bigg( \frac{a^{2}+c^{2}}{2}(1-\tau_{a,c}^{2}) \bigg)^{2} \int_{-\pi}^{\pi} \frac{1-\tau \cos \theta}{(1-\tau_{a,c}\cos\theta)^{2}}d\theta = \bigg( (a^{2}+c^{2}) \big( \tfrac{1}{\tau_{a,c}} - \tau_{a,c} \big) \bigg)^{2} \int_{\mathbb{S}^{1}} \frac{z-\tau \frac{z^{2}+1}{2}}{(z-r_{i})^{2}(\frac{1}{r_{i}}-z)^{2}} \frac{dz}{i}.
\end{align*}
The integrand is analytic inside the unit disk, except for a double pole at $z=r_{i}$. Using the residue
\begin{align*}
\mbox{Res} \bigg( \frac{z-\tau \frac{z^{2}+1}{2}}{(z-r_{i})^{2}(\frac{1}{r_{i}}-z)^{2}}; z=r_{i} \bigg) = \frac{r_{i}^{2}(1+r_{i}^{2}-2r_{i}\tau)}{(1-r_{i}^{2})^{3}},
\end{align*}
and simplifying using \eqref{def of ri and re} and $\tau_{a,c} = \frac{a^{2}-c^{2}}{a^{2}+c^{2}}$, we find \eqref{lol41}.
\end{proof}

Recall that $\tau \in [0,1)$, $Q(z)=\frac{1}{1-\tau^{2}}\big( |z|^{2}-\tau \, \re z^{2} \big) = \frac{x^{2}}{1+\tau}+\frac{y^{2}}{1-\tau}$, $z=x+iy$, $x,y\in \R$, and that $\mu$ and $S$ are defined in \eqref{mu S EG}. Recall also that $\hat{\nu}=\mathrm{Bal}(\mu|_{U},\partial U)$ is given by Theorem \ref{thm:EG complement ellipse nu}.
\begin{lemma}\label{lemma:int Q dmu complement ellipse}
\begin{align}\label{int Q dmu complement ellipse}
\int_{U} Q(z) d\mu(z) = \frac{2(1-\tau^{2})^{2} - a c (a^{2}(1-\tau)+c^{2}(1+\tau))}{4(1-\tau^{2})^{2}}.
\end{align}
\end{lemma}
\begin{proof}
Using the parametrization \eqref{param polar ellipses} of $S\cap U$, we obtain
\begin{align*}
\int_{U} Q(z) d\mu(z) & = \frac{1}{\pi(1-\tau^{2})} \int_{-\pi}^{\pi} \int_{R_{i}(\theta)}^{R_{e}(\theta)} Q(re^{i\theta}) r \; dr d\theta \\
&  = \frac{1}{\pi(1-\tau^{2})} \int_{-\pi}^{\pi} \bigg( \frac{(\cos \theta)^{2}}{1+\tau} + \frac{(\sin \theta)^{2}}{1-\tau} \bigg) \frac{R_{e}(\theta)^{4}-R_{i}(\theta)^{4}}{4} d\theta \\
& = \frac{1}{4\pi(1-\tau^{2})^{2}}\int_{-\pi}^{\pi} \big(1-\tau \cos (2\theta)\big) \big( R_{e}(\theta)^{4} - R_{i}(\theta)^{4} \big) d\theta.
\end{align*}
Now \eqref{int Q dmu complement ellipse} directly follows from Lemma \ref{lemma: int (1-cos)R4 dtheta}.
\end{proof}

\begin{lemma}\label{lemma:int Q dnu complement ellipse}
Let $c_{0}$ and $c_{1}$ be as in \eqref{def of c0 c1 complement ellipse intro}. Then
\begin{align}\label{int Q dnu complement ellipse}
\int_{\partial U} Q(z) d\nu(z) =  \bigg( \frac{a^{2}}{1+\tau} + \frac{c^{2}}{1-\tau} \bigg) \pi c_{0} + \bigg( \frac{a^{2}}{1+\tau} - \frac{c^{2}}{1-\tau} \bigg) \pi c_{1}.
\end{align}
\end{lemma}
\begin{proof}
Using the first parametrization of $\partial U$ in \eqref{two param of partial U when Uc ellipse}, we get
\begin{align*}
\int_{\partial U} Q(z) d\nu(z) & = \int_{-\pi}^{\pi} Q(a\cos\theta + i c \sin \theta) \big( c_{0} + 2c_{1} \cos(2\theta) \big) d\theta \\
& = \int_{\mathbb{S}^{1}} \bigg( \frac{a^{2}}{1+\tau} \bigg( \frac{z+z^{-1}}{2} \bigg)^{2} + \frac{c^{2}}{1-\tau} \bigg( \frac{z-z^{-1}}{2i} \bigg)^{2} \bigg) \big( c_{0} + c_{1}(z^{2}+z^{-2}) \big) \frac{dz}{iz}.
\end{align*}
The integrand is analytic in $\{z:|z|\leq 1\}\setminus \{0\}$, and a residue computation yields
\begin{align*}
& \mbox{Res}\bigg( \bigg( \frac{a^{2}}{1+\tau} \bigg( \frac{z+z^{-1}}{2} \bigg)^{2} + \frac{c^{2}}{1-\tau} \bigg( \frac{z-z^{-1}}{2i} \bigg)^{2} \bigg) \big( c_{0} + c_{1}(z^{2}+z^{-2}) \big) \frac{1}{iz} , z=0 \bigg) \\
& = \bigg( \frac{a^{2}}{1+\tau} + \frac{c^{2}}{1-\tau} \bigg) \frac{ c_{0}}{2i} + \bigg( \frac{a^{2}}{1+\tau} - \frac{c^{2}}{1-\tau} \bigg) \frac{c_{1}}{\pi i},
\end{align*}
from which we directly obtain \eqref{int Q dnu complement ellipse}.
\end{proof}
Note that $U$ satisfies Assumptions \ref{ass:U} and \ref{ass:U2}. Combining the formulas \eqref{cQ simplif}, \eqref{int Q dmu complement ellipse} and \eqref{int Q dnu complement ellipse} with  Theorem \ref{thm:general pot} (i), we infer that $\mathbb{P}(\# \{z_{j}\in U\} = 0) = \exp \big( -C n^{2}+o(n^{2}) \big)$ as $n\to+\infty$, where $C$ is given by \eqref{EG complement ellipse C}.  This finishes the proof of Theorem \ref{thm:EG complement ellipse C}.

\section{The complement of a disk}\label{section:complement disk}

In this section, $\tau \in [0,1)$, $Q(z)=\frac{1}{1-\tau^{2}}\big( |z|^{2}-\tau \, \re z^{2} \big) = \frac{x^{2}}{1+\tau}+\frac{y^{2}}{1-\tau}$, $z=x+iy$, $x,y\in \R$, and the measure $\mu$ and its support $S$ are defined in \eqref{mu S EG}. We also suppose that $x_{0},y_{0}\in \R$ and $a>0$ are such that $U^{c}\subset S$, where $U := \{z: |z-(x_{0}+iy_{0})| > a\}$. In particular, $(\frac{x_{0}}{1+\tau})^{2}+(\frac{y_{0}}{1-\tau})^{2}<1$.

\subsection{Balayage measure: proof of Theorem \ref{thm:EG complement disk nu}}
Let $\nu := \mathrm{Bal}(\mu|_{U},\partial U)$. By Theorem \ref{thm:dnu in terms of green general}, for all $w=x_{0}+iy_{0}+ae^{i\alpha} \in \partial U$, $\alpha \in (-\pi,\pi]$, we have
\begin{align*}
d\nu(w) =  \bigg( \int_{U} \frac{\partial g_{U}(z,w)}{\partial \mathbf{n}_{w}}d\mu(z) \bigg)|dw|,
\end{align*}
where $|dw|=ad\alpha$ is the arclength measure on $\partial U$, and $\frac{\partial}{\partial \mathbf{n}_{w}}$ is the normal at $w$ pointing inwards $U$. The set $S\cap U$ can be parametrized as
\begin{align}\label{param S cap U when Uc is disk}
S\cap U = \{x_{0}+iy_{0} + re^{i\theta}: \theta  \in (-\pi,\pi], r \in (a,R(e^{i\theta})]\},
\end{align}
where we define $R(e^{i\theta})$ as the right-hand side of \eqref{def of R Uc is a disk}. Hence
\begin{align}\label{lol81}
\frac{d\nu(w)}{a\, d\alpha} = \int_{-\pi}^{\pi}\int_{a}^{R(e^{i\theta})} \frac{d}{dx}g_{U}(x_{0}+iy_{0}+re^{i\theta},x_{0}+iy_{0}+x e^{i\alpha})\bigg|_{x =a}\frac{rdrd\theta}{\pi(1-\tau^{2})},
\end{align}
where $\mathbf{n}_{w}$ is the normal vector on $\partial U$ at $w$ pointing inwards $U$. Using the first equality in \eqref{green function with general conformap mapping} with $\varphi(z)=\frac{z-(x_{0}+iy_{0})}{a}$, we get
\begin{align}\label{lol82}
g_{U}(x_{0}+iy_{0}+re^{i\theta},x_{0}+iy_{0}+r_{w}e^{i\alpha}) = \frac{1}{2\pi}\re \log \frac{1-\frac{re^{i\theta}}{a}\frac{r_{w}e^{-i\alpha}}{a}}{\frac{re^{i\theta}}{a} - \frac{r_{w}e^{i\alpha}}{a}},
\end{align}
where the principal branch of the logarithm is used. Substituting \eqref{lol82} in \eqref{lol81} yields
\begin{align}\label{lol83}
\frac{d\nu(w)}{a\, d\alpha} = \re \int_{-\pi}^{\pi}\int_{a}^{R(e^{i\theta})} \frac{1}{2\pi a}\frac{re^{i\theta}+ae^{i\alpha}}{re^{i\theta}-ae^{i\alpha}} \frac{rdrd\theta}{\pi(1-\tau^{2})}.
\end{align}
The integrand in \eqref{lol83} admits an explicit primitive,
\begin{align}\label{lol84}
\frac{d}{dr}\bigg( \frac{re^{i(\alpha-\theta)}}{\pi^{2}} + \frac{r^{2}}{4a\pi^{2}} + \frac{a e^{2i(\alpha-\theta)}}{\pi^{2}} \log(re^{i(\theta-\alpha)} - a) \bigg) = \frac{1}{2\pi a}\frac{re^{i\theta}+ae^{i\alpha}}{re^{i\theta}-ae^{i\alpha}} \frac{r}{\pi}.
\end{align}
Using \eqref{lol84} to compute the $r$-integral in \eqref{lol83}, we get 
\begin{align}\label{dnu over dalpha}
\frac{d\nu(w)}{a\, d\alpha} = \re (I_{1}+I_{2}+I_{3}),
\end{align}
where
\begin{align}
& I_{1} := \frac{e^{i\alpha}}{1-\tau^{2}} \int_{-\pi}^{\pi} \frac{e^{-i\theta}(R(e^{i\theta})-a)}{\pi^{2}}d\theta, \qquad I_{2} := \frac{1}{1-\tau^{2}} \int_{-\pi}^{\pi} \frac{R(e^{i\theta})^{2}-a^{2}}{4a\pi^{2}}d\theta, \label{I1I2 disk} \\
& I_{3} := \frac{e^{2i\alpha}}{1-\tau^{2}} \int_{-\pi}^{\pi} \frac{a e^{-2i\theta}}{\pi^{2}} \bigg( \log\Big( e^{i(\theta-\alpha)}\frac{R(e^{i\theta})}{a}-1 \Big) - \log\Big( e^{i(\theta-\alpha)} - 1 \Big) \bigg) d\theta, \label{I3 disk}
\end{align}
where the principal branch is used for the logarithms. We need an analytic continuation of $R$. To avoid long discussions on the parameters $\tau,x_{0},y_{0},a$, in this subsection we will only treat the case $\tau,x_{0},y_{0}>0$ (the cases when $\tau=0$ and/or $w_{0},y_{0}\leq 0$ are similar, but require slight modifications of the branch cut structure of $R(z)$ defined in \eqref{def of R(z) disk} below). Let us define
\begin{align*}
& w_{0} := \frac{2(1+\tau^{2})-x_{0}^{2}-y_{0}^{2}+2\sqrt{\Delta}}{4\tau-(x_{0}+iy_{0})^{2}}, & & w_{1} := \frac{2(1+\tau^{2})-x_{0}^{2}-y_{0}^{2}-2\sqrt{\Delta}}{4\tau-(x_{0}+iy_{0})^{2}}, \\
& \Delta := (1-\tau^{2})^{2}-x_{0}^{2}(1-\tau)^{2} - y_{0}^{2}(1+\tau)^{2}.
\end{align*}

\begin{figure}[h]
\begin{center}
\begin{tikzpicture}[master]
\node at (0,0) {\includegraphics[width=10cm]{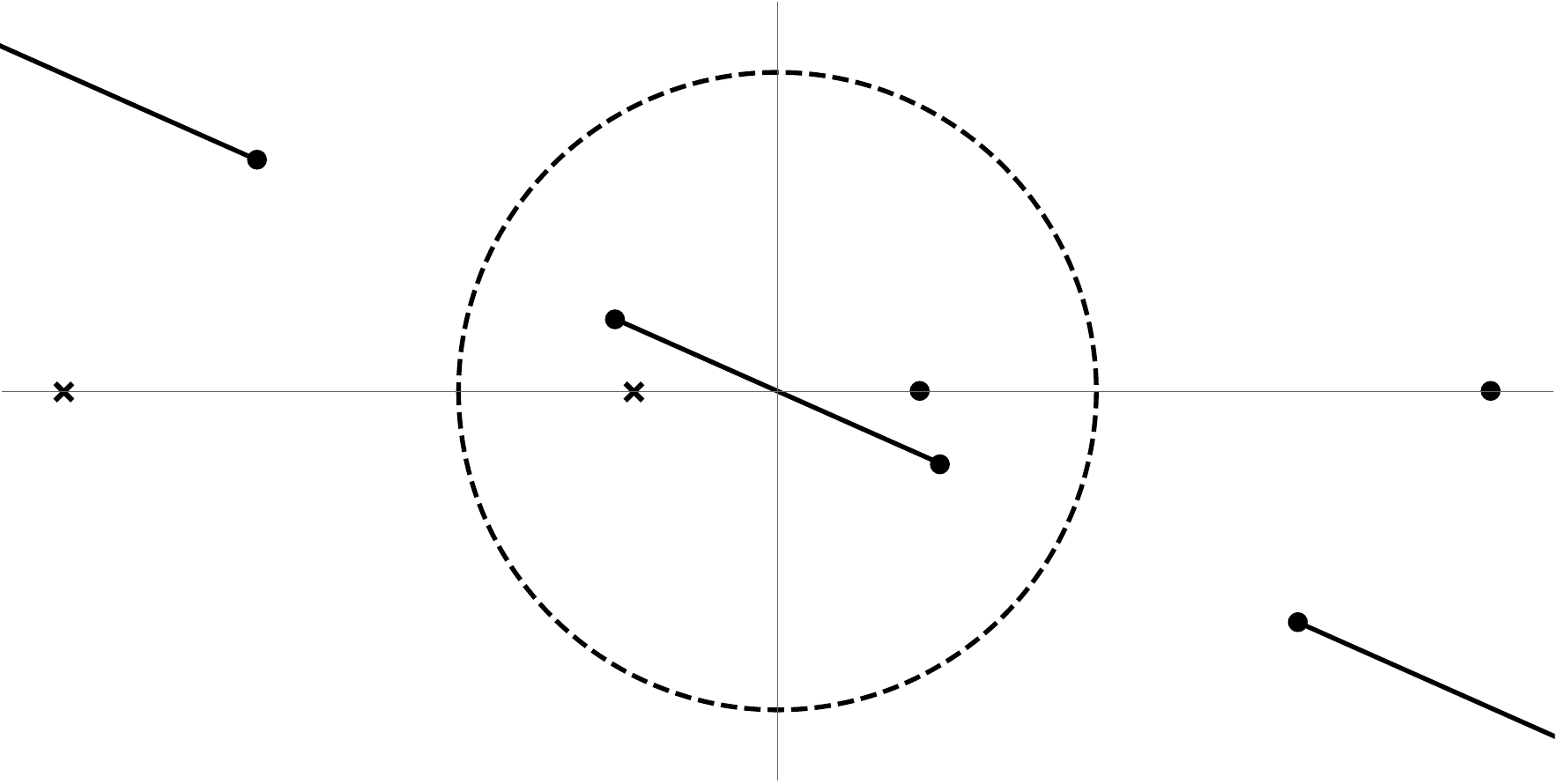}};
\node at (0.9,0.25) {\footnotesize $\sqrt{\tau}$};
\node at (-1,-0.25) {\footnotesize $-\sqrt{\tau}$};

\node at (4.5,0.3) {\footnotesize $\frac{1}{\sqrt{\tau}}$};
\node at (-4.6,0.3) {\footnotesize $-\frac{1}{\sqrt{\tau}}$};

\node at (-31:1.33) {\footnotesize $\sqrt{w_{1}}$};
\node at (147:1.33) {\footnotesize $-\sqrt{w_{1}}$};

\node at (-21:3.46) {\footnotesize $\sqrt{w_{0}}$};
\node at (161:3.55) {\footnotesize $-\sqrt{w_{0}}$};
\end{tikzpicture}
\end{center}
\caption{\label{fig:domain of analyticity of R} The function $R(z)$ in \eqref{def of R(z) disk} has poles at $-\frac{1}{\sqrt{\tau}}$, $-\sqrt{\tau}$ (crosses), and has discontinuities along the cuts $(-\sqrt{w_{0}}\infty,-\sqrt{w_{0}})$, $(-\sqrt{w_{1}},\sqrt{w_{1}})$ and $(\sqrt{w_{1}},\sqrt{w_{1}}\infty)$. The unit circle is the dashed curve. The figure has been made with $\tau=0.2$, $x_{0}=0.67$ and $y_{0}=0.45$.}
\end{figure}

Since $y_{0}>0$, $w_{0},w_{1}$ are well-defined, and since $x_{0},y_{0}>0$, we have $\im (w_{0})<0$. Note also that $\arg(w_{0})=\arg(w_{1})$. For any $\tau,x_{0},y_{0}>0$ such that $(\frac{x_{0}}{1+\tau})^{2}+(\frac{y_{0}}{1-\tau})^{2}<1$, it is easy to check that $0<|w_{1}|<1<|w_{0}|$ and $\Delta>0$. We will use the following analytic continuation of $R$:
\begin{align}\label{def of R(z) disk}
R(z) = z \frac{(1-\tau^{2})^{2}\big( \sqrt{4\tau - (x_{0}-iy_{0})^{2}} \frac{(z+\sqrt{w_{0}})^{1/2}_{\star}(z+\sqrt{w_{1}})^{1/2}_{\star}(z-\sqrt{w_{1}})^{1/2}_{\star}(\sqrt{w_{0}}-z)^{1/2}_{\star}}{1-\tau^{2}} - \frac{x_{0}(z^{2}+1)}{(1+\tau)^{2}} - \frac{y_{0}(z^{2}-1)}{(1-\tau)^{2}i} \big)}{2(-\tau)(z-\sqrt{\tau})(z+\sqrt{\tau})(z-\frac{1}{\sqrt{\tau}})(z+\frac{1}{\sqrt{\tau}})}
\end{align}
where $(z)^{1/2}_{*} := \sqrt{|z|}e^{\frac{i}{2}\arg_{\star}(z)}$, $\arg_{\star}(z)$ has a branch cut along the line $(0,-\sqrt{w_{1}}\infty)$ and is such that $\arg_{\star}(1)=0$, and all other branches in \eqref{def of R(z) disk} are the principal ones (in particular $\sqrt{w_{0}},\sqrt{w_{1}}$ are in the lower-right quadrant, and $\sqrt{4\tau - (x_{0}-iy_{0})^{2}}$ is in the upper-right quadrant). A necessary condition for the numerator of $R(z)$ to vanish is
\begin{align*}
(4\tau - (x_{0}-iy_{0})^{2}) \frac{(z^{2}-w_{1}) (w_{0}-z^{2})}{(1-\tau^{2})^{2}} = \bigg( \frac{x_{0}(z^{2}+1)}{(1+\tau)^{2}} + \frac{y_{0}(z^{2}-1)}{(1-\tau)^{2}i} \bigg)^{2}.
\end{align*}
A computation shows that the above equation is satisfied if and only if $z \in \{-\frac{1}{\sqrt{\tau}},-\sqrt{\tau},\sqrt{\tau},\frac{1}{\sqrt{\tau}}\}$. However, using the definition of $\smash{(z)_{\star}^{1/2}}$ and our assumption that $x_{0}>0$, we infer that
\begin{align*}
& \sqrt{4\tau - (x_{0}-iy_{0})^{2}} \frac{(z+\sqrt{w_{0}})^{1/2}_{\star}(z+\sqrt{w_{1}})^{1/2}_{\star}(z-\sqrt{w_{1}})^{1/2}_{\star}(\sqrt{w_{0}}-z)^{1/2}_{\star}}{1-\tau^{2}} \\
&  = \bigg( \frac{x_{0}(z^{2}+1)}{(1+\tau)^{2}} + \frac{y_{0}(z^{2}-1)}{(1-\tau)^{2}i} \bigg) \times  \begin{cases} 
-1, & \mbox{if } z\in \{-\frac{1}{\sqrt{\tau}},-\sqrt{\tau}\}, \\
1, & \mbox{if } z \in \{\sqrt{\tau},\frac{1}{\sqrt{\tau}}\}.
\end{cases}
\end{align*}
This implies that the numerator of $R(z)$ has simple zeros at $z=\sqrt{\tau}$ and $z=\frac{1}{\sqrt{\tau}}$, and no other zeros. Hence $R(z)$ is analytic in $\C \setminus \big( (-\sqrt{w_{0}}\infty,-\sqrt{w_{0}})\cup (-\sqrt{w_{1}},\sqrt{w_{1}}) \cup (\sqrt{w_{1}},\sqrt{w_{1}}\infty) \cup \{-\frac{1}{\sqrt{\tau}},-\sqrt{\tau}\}\big)$, see also Figure \ref{fig:domain of analyticity of R}.

Now, we come back to the problem of evaluating $I_{1}, I_{2}, I_{3}$ in \eqref{I1I2 disk}--\eqref{I3 disk}. Since $\int_{-\pi}^{\pi}e^{-i\theta}d\theta=0$, for $I_{1}$ we have
\begin{align}\label{lol85}
I_{1} = \frac{e^{i\alpha}}{1-\tau^{2}} \int_{-\pi}^{\pi} \frac{e^{-i\theta}R(e^{i\theta})}{\pi^{2}}d\theta = \frac{e^{i\alpha}}{\pi^{2}(1-\tau^{2})} \int_{\mathbb{S}^{1}} \frac{R(z)}{z}\frac{dz}{iz} = \frac{e^{i\alpha}}{\pi^{2}(1-\tau^{2})} ( I_{1,1} + I_{1,2}),
\end{align}
where $\mathbb{S}^{1}$ is the unit circle oriented counterclockwise, and
\begin{align*}
& I_{1,1} := \int_{\mathbb{S}^{1}} \frac{(1-\tau^{2})^{2}(-\frac{x_{0}(z^{2}+1)}{(1+\tau)^{2}}-\frac{y_{0}(z^{2}-1)}{(1-\tau)^{2}i})}{2(-\tau)(z-\sqrt{\tau})(z+\sqrt{\tau})(z-\frac{1}{\sqrt{\tau}})(z+\frac{1}{\sqrt{\tau}})} \frac{dz}{iz}, \\
& I_{1,2} := \int_{\mathbb{S}^{1}} \frac{(1-\tau^{2})\sqrt{4\tau - (x_{0}-iy_{0})^{2}} (z+\sqrt{w_{0}})^{1/2}_{\star}(z+\sqrt{w_{1}})^{1/2}_{\star}(z-\sqrt{w_{1}})^{1/2}_{\star}(\sqrt{w_{0}}-z)^{1/2}_{\star}}{2(-\tau)(z-\sqrt{\tau})(z+\sqrt{\tau})(z-\frac{1}{\sqrt{\tau}})(z+\frac{1}{\sqrt{\tau}})} \frac{dz}{iz}.
\end{align*}
The integral $I_{1,1}$ can be evaluated by deforming $\mathbb{S}^{1}$ towards $0$: we pick up residues at $0$, $\sqrt{\tau}$ and $-\sqrt{\tau}$, and after simplifying we find
\begin{align*}
I_{1,1} = - \pi \big( x_{0}(1-\tau)-iy_{0}(1+\tau) \big).
\end{align*}
For $I_{1,2}$, we also deform $\mathbb{S}^{1}$ towards $0$: here we pick up two residues at $\sqrt{\tau}$ and $-\sqrt{\tau}$ (they cancel each other), and a branch cut contribution along the segment $[-\sqrt{w_{1}},\sqrt{w_{1}}]$:
\begin{align}
I_{1,2} = (-2) \dashint_{[-\sqrt{w_{1}},\sqrt{w_{1}}]} &  \frac{(1-\tau^{2})\sqrt{4\tau - (x_{0}-iy_{0})^{2}}  }{2(-\tau)(z-\sqrt{\tau})(z+\sqrt{\tau})(z-\frac{1}{\sqrt{\tau}})(z+\frac{1}{\sqrt{\tau}})} \nonumber \\
& \times |w_{0}-z^{2}|e^{\frac{i}{2}\arg(\sqrt{w_{0}}+z)}e^{\frac{i}{2}\arg(\sqrt{w_{0}}-z)} \, |z^{2}-w_{1}|^{2} e^{\frac{i}{2}(\pi + 2\arg(\sqrt{w_{1}}))} \frac{dz}{iz}, \label{lol47}
\end{align}
where $\dashint$ stands for principal value integral. Since $\arg \sqrt{w_{0}} = \arg \sqrt{w_{1}}$, the above integrand is odd, and thus $I_{1,2} = 0$. Substituting the above expressions for $I_{1,1}$ and $I_{1,2}$ in \eqref{lol85} yields
\begin{align}\label{lol45}
I_{1} = \frac{-x_{0}(1-\tau)+iy_{0}(1+\tau)}{\pi(1-\tau^{2})}e^{i\alpha}.
\end{align}
For $I_{2}$, we have
\begin{align*}
I_{2} = \frac{1}{1-\tau^{2}} \int_{-\pi}^{\pi} \frac{R(e^{i\theta})^{2}}{4a\pi^{2}}d\theta - \frac{1}{1-\tau^{2}} \int_{-\pi}^{\pi} \frac{a^{2}}{4a\pi^{2}}d\theta = \frac{1}{1-\tau^{2}} \frac{1}{4a\pi^{2}} \int_{\mathbb{S}^{1}} R(z)^{2}\frac{dz}{iz} - \frac{1}{1-\tau^{2}} \frac{a}{2\pi}.
\end{align*}
From \eqref{def of R(z) disk}, we infer that $\frac{R(z)^{2}}{iz}$ has double poles at $z=-\frac{1}{\sqrt{\tau}}$ and $z=-\sqrt{\tau}$, and no other poles. Hence, by deforming the contour towards $0$, we pick up a residue at $-\sqrt{\tau}$, and an integral along $[-\sqrt{w_{1}},\sqrt{w_{1}}]$. In a similar way as for $I_{1,2}$, we find (by odd symmetry of the integrand) that the integral along $[-\sqrt{w_{1}},\sqrt{w_{1}}]$ vanishes. After computing the residue at $-\sqrt{\tau}$ and simplifying, we obtain
\begin{align}\label{lol46}
I_{2} = \frac{1}{2a\pi} - \frac{1}{1-\tau^{2}} \frac{a}{2\pi}.
\end{align}
For $I_{3}$, since $R(e^{i\theta})\geq a$ for all $\theta \in (-\pi,\pi]$, we can write $I_{3} = I_{3,1} + I_{3,2} + I_{3,3}$, where
\begin{align}
& I_{3,1} := \frac{e^{2i\alpha}}{1-\tau^{2}} \int_{-\pi}^{\pi} \frac{a\, e^{-2i\theta}}{\pi^{2}} \log \frac{R(e^{i\theta})}{a} \;d\theta, \nonumber \\
& I_{3,2} := \frac{1}{1-\tau^{2}} \int_{-\pi}^{\pi}\frac{a\, e^{2i(\alpha-\theta)}}{\pi^{2}} \log \bigg( e^{i(\theta-\alpha)} - \frac{a}{R(e^{i\theta})} \bigg) d\theta , \label{I32 lol} \\
& I_{3,3} := \frac{1}{1-\tau^{2}} \int_{-\pi}^{\pi}\frac{a  \, e^{2i(\alpha-\theta)} }{\pi^{2}} (-1)\log \big( e^{i(\theta-\alpha)} - 1 \big) d\theta, \nonumber
\end{align}
and the principal branches for the logarithms are used. For $I_{3,3}$, by changing variables $\theta \to \theta+\alpha$, then changing variables $e^{i\theta}=z$, and then deforming the contour towards $\infty$, we obtain
\begin{align}\label{lol43}
I_{3,3} = \frac{-a}{\pi^{2}(1-\tau^{2})i}  \int_{\mathbb{S}^{1}} \frac{\log(z-1)}{z^{3}}dz = \frac{-a}{\pi^{2}(1-\tau^{2})i}  \int_{-\infty}^{-1} \frac{2\pi i}{x^{3}}dx = \frac{1}{1-\tau^{2}} \frac{a}{\pi}.
\end{align}
Suppose first that the strict inequality $R(e^{i\theta})>a$ holds for all $\theta \in (-\pi,\pi]$. Then
\begin{align*}
\log \bigg( e^{i(\theta-\alpha)} - \frac{a}{R(e^{i\theta})} \bigg) = \log(e^{i(\theta-\alpha)}) - \sum_{j=1}^{+\infty} \frac{a^{j}}{j e^{i j(\theta-\alpha)}R(e^{i\theta})^{j}}
\end{align*}
uniformly for $\theta \in (-\pi,\pi]$, and thus
\begin{align}\label{lol42}
I_{3,2} = \frac{1}{1-\tau^{2}} \int_{-\pi}^{\pi}\frac{a\, e^{2i(\alpha-\theta)}}{\pi^{2}} \log ( e^{i(\theta-\alpha)}) d\theta - \frac{1}{1-\tau^{2}} \sum_{j=1}^{+\infty} \frac{a^{j+1}e^{i\alpha(2+j)}}{j\pi^{2}} \int_{-\pi}^{\pi} \frac{e^{-i\theta(2+j)}}{R(e^{i\theta})^{j}} d\theta.
\end{align}
The first integral can easily be computed:
\begin{align*}
\frac{1}{1-\tau^{2}} \int_{-\pi}^{\pi}\frac{a\, e^{2i(\alpha-\theta)}}{\pi^{2}} \log ( e^{i(\theta-\alpha)}) d\theta = \frac{1}{1-\tau^{2}} \int_{-\pi}^{\pi}\frac{a\, e^{-2i\theta}}{\pi^{2}} i\theta d\theta = -\frac{1}{1-\tau^{2}} \frac{a}{\pi}.
\end{align*}
For the other integrals in \eqref{lol42}, we write
\begin{align*}
\int_{-\pi}^{\pi} \frac{e^{-i\theta(2+j)}}{R(e^{i\theta})^{j}} d\theta = \int_{\mathbb{S}^{1}} \frac{1}{R(z)^{j}z^{3+j}} \frac{dz}{i}, \qquad j \in \N_{>0}.
\end{align*}
Since $R(z) = \bigO(z)$ as $z\to \infty$ (see \eqref{def of R(z) disk}), by deforming the contour towards $\infty$, we obtain an integral along $(-\sqrt{w_{0}}\infty,-\sqrt{w_{0}})\cup (\sqrt{w_{0}},\sqrt{w_{0}}\infty)$. In a similar way as for $I_{1,2}$ we can show that the integrand is odd, and therefore $\int_{-\pi}^{\pi} \frac{e^{-i\theta(2+j)}}{R(e^{i\theta})^{j}} d\theta = 0$ for each $j \in \N_{>0}$. Thus, if $R(e^{i\theta})>a$ holds for all $\theta \in (-\pi,\pi]$, we have
\begin{align}\label{I32 again lol}
I_{3,2} = - \frac{1}{1-\tau^{2}} \frac{a}{\pi}.
\end{align}
If $R(e^{i\theta})=a$ for some $\theta \in (-\pi,\pi]$, then by replacing $a$ by $a'$ in the above analysis and letting $a'\nearrow a$, using the fact that $I_{3,2}$ in \eqref{I32 lol} is continuous for $a \in [0,\min_{\theta \in (-\pi,\pi]} R(e^{i\theta})]$, we conclude that \eqref{I32 again lol} also holds in this case. 

It only remains to evaluate $I_{3,1}$. Using \eqref{def of R(z) disk} with $z=e^{i\theta}$, we obtain 
\begin{align}\label{splitting of I31 disk}
I_{3,1} = \frac{e^{2i\alpha}}{1-\tau^{2}} \frac{a}{\pi^{2}} (I_{3,1}^{(1)}+I_{3,1}^{(2)} + I_{3,1}^{(3)} + I_{3,1}^{(4)}),
\end{align}
where
\begin{align*}
& I_{3,1}^{(1)} := \log \bigg( \frac{(1-\tau^{2})^{2}}{2a\tau} \bigg) \int_{-\pi}^{\pi} e^{-2i\theta} d\theta = 0, \\
& I_{3,1}^{(2)} := \int_{-\pi}^{\pi} e^{-2i\theta} i \theta d\theta = - \pi, \\
& I_{3,1}^{(3)} := \int_{-\pi}^{\pi} e^{-2i\theta} \log \phi(e^{i\theta}) d \theta, \\
& \phi(z) :=  \frac{\frac{\sqrt{4\tau - (x_{0}+iy_{0})^{2}} (z+\sqrt{w_{0}})^{1/2}_{\star}(z+\sqrt{w_{1}})^{1/2}_{\star}(z-\sqrt{w_{1}})^{1/2}_{\star}(\sqrt{w_{0}}-z)^{1/2}_{\star}}{1-\tau^{2}}-\frac{x_{0}(z^{2}+1)}{(1+\tau)^{2}}-\frac{y_{0}(z^{2}-1)}{(1-\tau)^{2}i}}{\sqrt{z-\sqrt{\tau}}\sqrt{z+\sqrt{\tau}}\sqrt{z-\frac{1}{\sqrt{\tau}}}\sqrt{z+\frac{1}{\sqrt{\tau}}} }, \\
& I_{3,1}^{(4)} := \int_{-\pi}^{\pi} e^{-2i\theta} \bigg( -\frac{1}{2}\log(e^{i\theta}-\sqrt{\tau}) -\frac{1}{2}\log(e^{i\theta}+\sqrt{\tau}) - \frac{1}{2}\log(\tfrac{1}{\sqrt{\tau}}-e^{i\theta})- \frac{1}{2}\log(e^{i\theta}+\tfrac{1}{\sqrt{\tau}}) \bigg) d\theta.
\end{align*}
A direct analysis shows that
\begin{align*}
\theta \mapsto \phi(e^{i\theta}) = \bigg( \frac{e^{i\theta}(1-\tau^{2})^{2}}{2\tau \sqrt{e^{i\theta}-\sqrt{\tau}}\sqrt{e^{i\theta}+\sqrt{\tau}}\sqrt{e^{i\theta}-\frac{1}{\sqrt{\tau}}}\sqrt{e^{i\theta}+\frac{1}{\sqrt{\tau}}}} \bigg)^{-1}R(e^{i\theta})
\end{align*}
is real valued and positive for all $\theta \in (-\pi,\pi)$. For $I_{3,1}^{(4)}$, we change variables $z=e^{i\theta}$, then deform the contour towards $\infty$:
\begin{align*}
I_{3,1}^{(4)} & = -\int_{\mathbb{S}^{1}} \bigg( \log(z-\sqrt{\tau}) + \log(z+\sqrt{\tau}) + \log(\tfrac{1}{\sqrt{\tau}}-z) + \log(z+\tfrac{1}{\sqrt{\tau}}) \bigg) \frac{dz}{2iz^{3}} \\
& = -2\pi \int_{-\infty}^{-1} \frac{dx}{x^{3}} - \pi \int_{-\infty}^{-\frac{1}{\sqrt{\tau}}} \frac{dx}{x^{3}} + \pi \int_{\frac{1}{\sqrt{\tau}}}^{+\infty} \frac{dx}{x^{3}} = \pi (1+\tau).
\end{align*}
We now turn to the analysis of $I_{3,1}^{(3)}$. A long but direct computation shows that
\begin{align*}
\log \phi(e^{i\theta}) + \log \phi(e^{i(\theta+\pi)}) = \log \frac{4\tau(1-\frac{x_{0}^{2}}{(1+\tau)^{2}}-\frac{y_{0}^{2}}{(1-\tau)^{2}})}{(1-\tau^{2})^{2}}, \qquad \mbox{for all } \theta \in (-\pi,\pi].
\end{align*}
This important identity allows us to prove that $I_{3,1}^{(3)}$ vanishes:
\begin{align*}
I_{3,1}^{(3)} & = \frac{1}{2}\int_{-\pi}^{\pi} e^{-2i\theta} \log \phi(e^{i\theta}) d \theta + \frac{1}{2} \int_{-\pi}^{\pi} e^{-2i\theta} \log \phi(e^{i(\theta+\pi)}) d \theta \\
& = \frac{1}{2}  \log \bigg( \frac{4\tau(1-\frac{x_{0}^{2}}{(1+\tau)^{2}}-\frac{y_{0}^{2}}{(1-\tau)^{2}})}{(1-\tau^{2})^{2}} \bigg) \int_{-\pi}^{\pi} e^{-2i\theta}d\theta = 0.
\end{align*}
Substituting the simplified expressions for $I_{3,1}^{(1)},I_{3,1}^{(2)},I_{3,1}^{(3)},I_{3,1}^{(4)}$ in \eqref{splitting of I31 disk} yields $I_{3,1} = \frac{e^{2i\alpha}}{\frac{1}{\tau}-\tau} \frac{a}{\pi}$. By combining this identity with \eqref{lol43}, \eqref{lol42} and $I_{3} = I_{3,1} + I_{3,2} + I_{3,3}$, we find
\begin{align}\label{lol44}
I_{3} = I_{3,1} = \frac{e^{2i\alpha}}{\frac{1}{\tau}-\tau} \frac{a}{\pi}.
\end{align}
Substituting \eqref{lol45}, \eqref{lol46} and \eqref{lol44} in \eqref{dnu over dalpha}, we obtain \eqref{def of dnu complement disk}. This finishes the proof of Theorem \ref{thm:EG complement disk nu}.
\subsection{The constant $C$: proof of Theorem \ref{thm:EG complement disk C}}
\begin{lemma}\label{lemma:J1J2J3}
The following relations hold:
\begin{align}
& J_{1} := \int_{-\pi}^{\pi} R(e^{i\theta})^{2}d\theta = 2\pi (1-\tau^{2}), \label{J1} \\
& J_{2} := \int_{-\pi}^{\pi} \bigg( \frac{2x_{0}\cos \theta}{1+\tau} + \frac{2y_{0}\sin \theta}{1-\tau} \bigg) \frac{R(e^{i\theta})^{3}}{3} d\theta = -2\pi \big( x_{0}^{2}(1-\tau) + y_{0}^{2}(1+\tau) \big), \label{J2} \\
& J_{3} := \int_{-\pi}^{\pi} \bigg( \frac{(\cos \theta)^{2}}{1+\tau} + \frac{(\sin \theta)^{2}}{1-\tau} \bigg) \frac{R(e^{i\theta})^{4}}{4}d\theta = \pi \bigg( x_{0}^{2}(1-\tau) + y_{0}^{2}(1+\tau) + \frac{1-\tau^{2}}{2} \bigg). \label{J3}
\end{align}
\end{lemma}
\begin{remark}
Interestingly, $J_{1}$ is independent of $x_{0}$ and $y_{0}$.
\end{remark}
\begin{proof}
Using the change of variables $z=e^{i\theta}$, we get
\begin{align*}
& J_{1} = \int_{\mathbb{S}^{1}} R(z)^{2} \frac{dz}{iz}, \qquad J_{2} = \int_{\mathbb{S}^{1}} \bigg( \frac{2x_{0}\frac{z+z^{-1}}{2}}{1+\tau} + \frac{2y_{0}\frac{z-z^{-1}}{2i}}{1-\tau} \bigg) \frac{R(z)^{3}}{3} \frac{dz}{iz}, \\
& J_{3} = \int_{\mathbb{S}^{1}} \bigg( \frac{(\frac{z+z^{-1}}{2})^{2}}{1+\tau} + \frac{(\frac{z-z^{-1}}{2i})^{2}}{1-\tau} \bigg) \frac{R(z)^{4}}{4}\frac{dz}{iz},
\end{align*}
where $\mathbb{S}^{1}$ is the unit circle oriented counterclockwise. Since $R(z) = \bigO(z)$ as $z \to 0$, all integrands remain bounded at $0$, and thus, by \eqref{def of R(z) disk}, all integrands are analytic in $\{z:|z|\leq 1\}\setminus \big( (-\sqrt{w_{1}},\sqrt{w_{1}}) \cup \{-\sqrt{\tau}\} \big)$. By deforming the contour towards $0$, for $J_{1}$ we obtain
\begin{align}\label{J1 res and integral}
J_{1} = 2\pi i \; \mbox{Res}\bigg( \frac{R(z)^{2}}{iz} , z=-\sqrt{\tau} \bigg) + \int_{[-\sqrt{w_{1}},\sqrt{w_{1}}]} \frac{R_{-}(z)^{2}-R_{+}(z)^{2}}{iz}dz,
\end{align}
where $[-\sqrt{w_{1}},\sqrt{w_{1}}]$ is the line segment joining $-\sqrt{w_{1}}$ and $\sqrt{w_{1}}$. In general $R_{-}\neq R_{+}$ and $R_{-}\neq -R_{+}$. However, since $\arg(\sqrt{w_{0}}) = \arg(\sqrt{w_{1}})$, a direct analysis (similar to that of \eqref{lol47}) shows that both $\frac{R_{-}(z)^{2}}{iz}$ and $\frac{R_{+}(z)^{2}}{iz}$ are odd on $[-\sqrt{w_{1}},\sqrt{w_{1}}]$, so that
\begin{align*}
\int_{[-\sqrt{w_{1}},\sqrt{w_{1}}]} \frac{R_{-}(z)^{2}}{iz}dz = \int_{[-\sqrt{w_{1}},\sqrt{w_{1}}]} \frac{R_{+}(z)^{2}}{iz}dz = 0.
\end{align*}
Hence $J_{1}$ is equal to the residue in \eqref{J1 res and integral}, and after a direct computation we obtain \eqref{J1}. In a similar way, by deforming the contours of $J_{2}, J_{3}$ towards $0$, we find that $J_{2}, J_{3}$ can each be expressed as the sum of a residue at $-\sqrt{\tau}$ and an integral along $[-\sqrt{w_{1}},\sqrt{w_{1}}]$. By odd symmetry of the integrands, the integrals along $[-\sqrt{w_{1}},\sqrt{w_{1}}]$ vanish, and after computing the residues at $-\sqrt{\tau}$ we find \eqref{J2} and \eqref{J3}.
\end{proof}
Recall that for $z=x+iy$, $x,y\in \R$, we have $Q(z)=\frac{1}{1-\tau^{2}}\big( |z|^{2}-\tau \, \re z^{2} \big) = \frac{x^{2}}{1+\tau}+\frac{y^{2}}{1-\tau}$.
\begin{lemma}\label{lemma:int Q dmu complement disk}
\begin{align}\label{int Q dmu complement disk}
\int_{U} Q(z) d\mu(z) = \frac{1}{1-\tau^{2}} \bigg( \frac{1-\tau^{2}}{2}-a^{2}\bigg( \frac{x_{0}^{2}}{1+\tau} + \frac{y_{0}^{2}}{1-\tau} \bigg) - \frac{a^{4}}{2(1-\tau^{2})} \bigg).
\end{align}
\end{lemma}
\begin{proof}
Using the definition \eqref{mu S EG} of $\mu$ and the polar parametrization \eqref{param S cap U when Uc is disk} of $S\cap U$, we obtain
\begin{align*}
& \int_{U} Q(z) d\mu(z) = \int_{-\pi}^{\pi} \int_{a}^{R(e^{i\theta})} Q(x_{0}+iy_{0}+re^{i\theta}) \frac{r \, dr d\theta}{\pi(1-\tau^{2})} \\
& = \hspace{-0.1cm} \int_{-\pi}^{\pi} \int_{a}^{R(e^{i\theta})} \hspace{-0.1cm} \bigg\{ \hspace{-0.05cm}\bigg( \frac{x_{0}^{2}}{1+\tau} + \frac{y_{0}^{2}}{1-\tau} \bigg)r + \hspace{-0.05cm} \bigg( \frac{2x_{0}\cos \theta}{1+\tau} + \frac{2y_{0}\sin \theta}{1-\tau} \bigg)r^{2} + \hspace{-0.05cm} \bigg( \frac{(\cos \theta)^{2}}{1+\tau} + \frac{(\sin \theta)^{2}}{1-\tau} \bigg) r^{3} \bigg\} \frac{\, dr d\theta}{\pi(1-\tau^{2})} \\
& = \frac{1}{\pi(1-\tau^{2})}\bigg( \frac{1}{2}\bigg( \frac{x_{0}^{2}}{1+\tau} + \frac{y_{0}^{2}}{1-\tau} \bigg) (J_{1}-2\pi a^{2}) + J_{2} + J_{3} - \frac{a^{4}\pi}{2(1-\tau^{2})} \bigg),
\end{align*}
where $J_{1},J_{2},J_{3}$ are defined in the statement of Lemma \ref{lemma:J1J2J3}. Substituting \eqref{J1}, \eqref{J2} and \eqref{J3} in the above expression, we find \eqref{int Q dmu complement disk}.
\end{proof}
Recall that $\nu := \mathrm{Bal}(\mu|_{U},\partial U)$ is given by \eqref{def of dnu complement disk}.
\begin{lemma}\label{lemma:int Q dnu complement disk}
\begin{align}\label{int Q dnu complement disk}
\int_{\partial U} Q(z) d\nu(z) = \frac{x_{0}^{2}(1-\tau) + y_{0}^{2}(1+\tau)  + a^{2} \big( 1-x_{0}^{2}\frac{3-2\tau}{1+\tau}-y_{0}^{2}\frac{3+2\tau}{1-\tau} \big) - a^{4}\frac{1+\tau^{2}}{1-\tau^{2}}}{1-\tau^{2}}.
\end{align}
\end{lemma}
\begin{proof}
Since $\partial U = \{x_{0}+iy_{0} + ae^{i\theta}: \theta  \in (-\pi,\pi]\}$, we have
\begin{align*}
\int_{\partial U} Q(z) d\nu(z) = \int_{-\pi}^{\pi} \bigg( \frac{(x_{0}+a\cos\theta)^{2}}{1+\tau} + \frac{(y_{0}+a\sin\theta)^{2}}{1-\tau} \bigg)d\nu(x_{0}+iy_{0}+ae^{i\theta}).
\end{align*}
Substituting \eqref{def of dnu complement disk} and performing the integral using primitives, we obtain \eqref{int Q dnu complement disk}.
\end{proof}

\begin{lemma}
\begin{align}\label{cQ simplif disk}
c^{\hspace{0.02cm}\mu}_{U} = \frac{1}{1-\tau^{2}}\int_{-\pi}^{\pi}R(e^{i\theta})^{2}\log R(e^{i\theta}) \frac{d\theta}{2\pi} - \frac{1}{2} + \frac{a^{2}}{2(1-\tau^{2})} - \log a.
\end{align}
\end{lemma}
\begin{proof}
It is well-known and easy to check (see e.g. \cite[Section II.4]{SaTo}) that $g_{V}(z,\infty) = \frac{1}{2\pi}\log |z|$, where $V:=\{z:|z|>1\}$. Since $\varphi(z) := \frac{z-(x_{0}+iy_{0})}{a}$ is a conformal map from $V$ to $U$, we have $g_{U}(z,\infty) =\frac{1}{2\pi} \log |\varphi(z)| = \frac{1}{2\pi} \log |\frac{z-(x_{0}+iy_{0})}{a}|$.  Hence, by Proposition \ref{prop:def of bal},
\begin{align*}
& c^{\hspace{0.02cm}\mu}_{U} = 2\pi \int_{U}g_{U}(z)d\mu(z) = 2\pi\int_{-\pi}^{\pi}\int_{a}^{R(e^{i\theta})}g_{U}(x_{0}+iy_{0}+re^{i\theta})\frac{r \, drd\theta}{\pi(1-\tau^{2})} \\
& = \frac{1}{\pi(1-\tau^{2})} \hspace{-0.05cm} \int_{-\pi}^{\pi} \hspace{-0.1cm} \bigg( \hspace{-0.05cm} \frac{R(e^{i\theta})^{2}\log R(e^{i\theta})}{2}-\frac{R(e^{i\theta})^{2}}{4}-\frac{a^{2}\log a}{2} + \frac{a^{2}}{4} \bigg)d\theta - \frac{\log a}{\pi(1-\tau^{2})} \int_{-\pi}^{\pi} \frac{R(e^{i\theta})^{2}-a^{2}}{2}d\theta,
\end{align*}
where for the last equality we have used $\frac{d}{dr}(\frac{r^{2}\log r}{2} - \frac{r^{2}}{4}) = r\log r$. Using \eqref{J1} to simplify the above expression, we find \eqref{cQ simplif disk}. 
\end{proof}

Note that $U$ satisfies Assumptions \ref{ass:U} and \ref{ass:U2}. Combining the formulas \eqref{cQ simplif disk}, \eqref{int Q dmu complement disk} and \eqref{int Q dnu complement disk} with Theorem \ref{thm:general pot} (i), we infer that $\mathbb{P}(\# \{z_{j}\in U\} = 0) = \exp \big( -C n^{2}+o(n^{2}) \big)$ as $n\to+\infty$, where $C$ is given by \eqref{EG complement disk C}. The fact that $\int_{-\pi}^{\pi}R(e^{i\theta})^{2}\log R(e^{i\theta}) \frac{d\theta}{2\pi}=0$ if $x_{0}=y_{0}=0$ was already proved in \eqref{lol38}. This finishes the proof of Theorem \ref{thm:EG complement disk C}.

\paragraph{Acknowledgements.} The author is grateful to Julian Mauersberger and Philippe Moreillon for valuable discussions and a careful reading. Support is acknowledged from the Swedish Research Council, Grant No. 2021-04626.

\end{document}